\documentclass[11pt,leqno]{book}

\usepackage{anysize}
\marginsize{3.5cm}{3.5cm}{3cm}{3cm}

\usepackage[]{hyperref}
\hypersetup{
    colorlinks=true,       
    linkcolor=red,          
    citecolor=blue,        
    filecolor=magenta,      
    urlcolor=cyan           
}

\usepackage{amsmath}
\usepackage{amsfonts,amssymb}
\usepackage{enumerate}
\usepackage{mathrsfs}
\usepackage{amsthm}
\usepackage{tikz}

\theoremstyle{plain}
\newtheorem{theo}{Theorem}[chapter]
\newtheorem*{theo*}{Theorem}
\newtheorem{prop}[theo]{Proposition}
\newtheorem{lemm}[theo]{Lemma}
\newtheorem{coro}[theo]{Corollary}
\newtheorem{assu}[theo]{Assumption}
\newtheorem{defi}[theo]{Definition}
\theoremstyle{definition}
\newtheorem{rema}[theo]{Remark}
\newtheorem{nota}[theo]{Notation}
\DeclareMathOperator{\cnx}{div}
\DeclareMathOperator{\cn}{div}
\DeclareMathOperator{\RE}{Re}

\DeclareMathOperator{\Hess}{Hess}
\DeclareMathOperator{\supp}{supp}
\DeclareMathOperator{\di}{d}

\DeclareSymbolFont{pletters}{OT1}{cmr}{m}{sl}
\DeclareMathSymbol{s}{\mathalpha}{pletters}{`s}
\def\ba{\begin{align}}
\def\bad{\begin{aligned}}
\def\be{\begin{equation}}
\def\ea{\end{align}}
\def\ead{\end{aligned}}
\def\ee{\end{equation}}
\def\e{\eqref}

\def\htilde{\tilde{h}}
\def\lDx#1{\langle D_x\rangle^{#1}\,}
\def\mU{\mathcal{U}}

\def\B{B }

\def\da{a}
\def\defn{\mathrel{:=}}
\def\deta{\eta}

\def\dzeta{\zeta}
\def\dB{B}
\def\dV{V}
\def\Deltax{\Delta}
\def\Deltayx{\Delta_{x,y}}

\def\eps{\varepsilon}
\def\eval{\arrowvert_{y=\eta}}

\def\la{\left\vert}
\def\lA{\left\Vert}
\def\le{\leq}
\def\les{\lesssim}
\def\leo{}

\def\ma{a}
\def\mez{\frac{1}{2}}
\def\partialx{\nabla}
\def\partialyx{\nabla_{x,y}}
\def\ra{\right\vert}
\def\rA{\right\Vert}

\def\tdm{\frac{3}{2}}
\def\tq{\frac{3}{4}}
\def\uq{\frac{1}{4}}

\def\xC{\mathbf{C}}
\def\xN{\mathbf{N}}
\def\xR{\mathbf{R}}

\def\K{\mathcal{K}}

\def\gain{\mu}

\def\zygmund#1{C^{{#1}}_*}
\def\holder#1{W^{{#1},\infty}}
\def\holdertdm{W^{r+\mez,\infty}}
\def\holdermez{\holder{\mez}}
\numberwithin{equation}{chapter}

\title{Strichartz estimates and the Cauchy problem for the gravity water waves equations}
\author{
T. Alazard, 
N. Burq, 
C. Zuily}
\date{\empty}

\begin{document}
\frontmatter
\maketitle

\begin{center}{\large \textbf{Abstract}}
\end{center}
\vspace{5mm}
This paper is devoted to the proof of a well-posedness result for the gravity 
water waves
equations, in arbitrary dimension and in fluid domains with general bottoms, when 
the initial velocity field is not necessarily Lipschitz. Moreover, for two-dimensional waves, we can 
consider solutions such that the curvature of the initial free surface does not belong to $L^2$. 

The proof is entirely based on the Eulerian formulation of the water waves equations, using microlocal analysis to 
obtain sharp Sobolev and H\"older estimates. 
We first prove tame estimates 
in Sobolev spaces depending linearly on H\"older norms and then we use 
the dispersive 
properties of the water-waves system, namely Strichartz estimates, 
to control these H\"older norms.

\tableofcontents

\mainmatter

\chapter{Introduction}

In this paper we consider the free boundary problem describing the 
motion of water waves over 
an incompressible, irrotational fluid flow. 
We are interested in the study of the possible 
emergence 
of singularities and would like to 
understand which quantities govern the boundedness of the solutions. 

We shall work in the Eulerian coordinate system where the  unknowns are the velocity field 
$v$ and the free surface elevation $\eta$. Namely, consider a 
simply connected domain, $\Omega$, located between a fixed bottom $\Gamma$ and a free  unknown surface $\Sigma$, 
given as a graph
$$
\Sigma= \{ (x,y) \in \xR^{d}\times \xR\,;\, y = \eta(t,x)\}.
$$
In this framework,  water waves are described by a system of coupled equations: the incompressible Euler equation 
in the interior of the domain and a kinematic equation describing the deformations of the domain. 
Moreover, the velocity will be assumed to be irrotational.  Thus we are interested in the following system 
\begin{equation}\label{WW}
\left\{ \begin{aligned}
&\partial_t v+ v\cdot \nabla_{x,y} v +  \nabla_{x,y} (P +g y)=0 \text{ in } \Omega, \\[0.5ex]
&\text{div}\, _{x,y} v =0,~ \text{curl}\, _{x,y} v =0 \text{ in } \Omega,\\[0.5ex]
&v\cdot\nu=0 \text{ on } \Gamma,\\[0.5ex]
&\partial_t \eta = (1+ |\nabla_x \eta |^2 ) ^{1/2} \, v\cdot n\text{ on } \Sigma,\\[0.5ex]
&P=0 \text{ on } \Sigma, 
\end{aligned} \right. 
\end{equation} 
 augmented with   initial data $(\eta_0, v_0)$ at time $t=0.$ Here  $n$ (resp.\ $\nu$) is the outward unit normal to the free surface (resp.\ the bottom), 
 $P$ is the pressure and $g>0$ is the acceleration of gravity. 
 
The first equation is the usual Euler equation in presence 
of a gravity force directed along the $y$ coordinate, the third one is 
the solid wall boundary condition at the bottom, the fourth equation 
describes the movement of the interface under the action of the fluid 
and ensures that the fluid particles initially at the interface remain 
at the interface, while the last one expresses the continuity 
of the pressure through the interface (no tension surface).
 
In \cite{ABZ3} we proved that the Cauchy problem for the system~\eqref{WW} is well-posed 
under the minimal assumptions that insure that at time $t=0$, in terms of 
Sobolev embeddings, the initial velocity ${v}_0$ 
is Lipschitz up to the boundary (see also the improvement to velocities whose derivatives are in BMO by Hunter-Ifrim-Tataru~\cite{HIT}). This Lipschitz regularity threshold for the velocity appears to be very natural. However, it has been known for some time (see the work by Bahouri-Chemin~\cite{BaCh} and Tataru~\cite{TataruNS}) that taking benefit of dispersive effects, it is possible to go beyond this threshold on some quasilinear wave-type systems. The goal of this article is  to show that such an improvement is also possible on the water-waves system.   

To describe our main result, we need to introduce some notations.
For $(\eta, v)$ as above, denote by $ V = v_x\mid_{\Sigma}, B= v_y \mid_{\Sigma}$ the horizontal and vertical components of the velocity field at the interface.  Since $v$ is irrotational, and incompressible, we can write $v = \nabla_{x,y} \phi$, 
with $\Delta_{x,y}  \phi = 0$ in the domain. We set $\psi = \phi \mid_{\Sigma}$.  
We shall prove (see Section~\ref{sec.1.4} for a more complete statement)
\begin{theo*}
Let 
\begin{equation*}
s>1+\frac{d}{2}-\gain, \qquad \text{ with }\begin{cases} \gain=\frac 1 {24} & \text{ if } d=1,\\ \gain =\frac 1 {12} & \text{ if } d\geq 2, \end{cases}
\end{equation*}
and set $\mathcal{H}^s=  H^{s+\mez}\times H^{s+\mez}\times (H^{s})^{d}\times H^{s}$ 
where $H^\sigma=H^{\sigma}( \xR^d)$.

Then for any initial data $(\eta_0, v_0)$ such that $(\eta_0, \psi_0,V_0, B_0) \in \mathcal{H}^s $ and  satisfying  the Taylor sign condition, there exist  $T>0$ and  a  solution $(\eta,v)$ of the water-waves system~\eqref{WW} (unique in a suitable space) such that 
$(\eta, \psi, V, B) \in C^0((-T, T); \mathcal{H}^s)$.
\end{theo*}
\begin{rema}\label{R1}
\begin{itemize}
\item The Taylor sign condition expresses the fact that the pressure increases going from the air into the fluid domain. It reads
$$
\partial_y P \arrowvert_{\Sigma} \leq c <0.
$$
It is always satisfied when there is no bottom (see Wu~\cite{WuJAMS}) or for small perturbations of flat bottoms (see Lannes~\cite{LannesJAMS}). Notice that the water-waves system is ill-posed when this condition is not satisfied (see Ebin~\cite{Ebin}).
\item The  curvature  $\kappa_0$   of 
the initial free surface involves two derivatives of $\eta_0$. Hence, we have $\kappa_0 \in H^{s- \tdm}$ which, according to our assumption on $s,$ can be negative in dimension $1$. This shows that, when $d=1$, no control on the  $L^2(\xR)$-norm 
of the curvature of 
the initial free surface is to be assumed.
\item In any dimension, in view of Sobolev embeddings, our assumptions require that $V_0, B_0$ belong to the H\" older space $W^{1-\mu,\infty}$. Consequently, our result applies to initial data for which   the initial velocity field is not Lipschitz.
\end{itemize}
\end{rema}

To prove this result, we follow several steps.  
The first one is (roughly speaking) to reduce the water waves equations to a quasilinear wave 
type equation of the form 
\begin{equation}\label{eq.reduction}
(\partial_t + T_V\cdot \nabla_x + i T_c)u=f,
\end{equation}
where $T_V$ is the paramultiplication by $V$, $T_c$ is a paradifferential operator of order $\mez$ 
and almost self-adjoint (such that $T_c^*-T_c$ is of order $0$) and $f$ is a remainder term in the paradifferential reduction. Actually, this reduction is not new. It was already  performed in 
our previous work~\cite{ABZ3} and  was based on two facts: the  
  Craig-Sulem-Zakharov reduction to a system on the boundary, 
introducing the Dirichlet-Neumann operator and   (following Lannes~\cite{LannesJAMS} Ê
and \cite{ABZ1,AM}) the use of  paradifferential analysis
to study the Dirichlet-Neumann operator in non smooth domains.
The second major step in the proof consists in proving that the solutions of the water waves system enjoy dispersive estimates (Strichartz-type inequalities).
For the equations with surface tension in the special case of dimension~$1$, Strichartz estimates were proved by 
Christianson, Hur and Staffilani in 
\cite{CHS} for smooth enough data and in~\cite{ABZ2} 
for the low regularity solutions constructed in \cite{ABZ1}. 
In the present context, the main difficulty will  consist in proving 
these dispersive estimates for gravity waves at a  lower level of regularity 
than the threshold where we proved the existence of the solutions in \cite{ABZ3}. This will be done  by constructing parametrices on small time intervals tailored to the size of the frequencies considered (in the spirit of the works by Lebeau~\cite{LebeauSc}, Bahouri-Chemin~\cite{BaCh}, Tataru~\cite{TataruNS}, Staffilani-Tataru~\cite{StTa}, and  Burq-G\'erard-Tzvetkov~\cite{BGT1}).

The   important new points in the present article with respect to our previous analysis are the following.
\begin{itemize}
\item To go beyond the analysis previously developed in~\cite{ABZ3} (under the assumption $s> 1+ \frac d 2$), 
we need to develop much more technically involved approaches, in order to work with very rough functions and domains (most parts in our analysis extend to  $s> \frac 1 2 +\frac d 2$). We believe that these results on the Dirichlet-Neumann operator in very rough domains can be of independent interest (see also the work by Dahlberg-Kenig~\cite{DK} and Craig-Schanz-Sulem~\cite{CSS}).
\item The {\em a priori} estimates we prove involve $L^\infty_t(\mathcal{H}^s_x)$ norms (energy estimates) and $L^2_t(C^\sigma_x)$ norms (dispersive estimates). For this,  we need to estimate the non linear (and non local)  remainder terms given by the paradifferential calculus using these norms. The loss of integrability in time ($L^2_t$) for the H\"older norms forces us to track down the precise dependence of the constants in our analysis and prove  {\em tame} estimates depending linearly on H\"older norms.
\item  A simpler model operator describing our system is $(\partial_t + V \cdot \nabla_x) + i |D_x|^{\mez}$ (while in presence of surface tension, $|D_x|^{\mez} $ is replaced by $|D_x|^{\tdm}$).  Consequently, the dispersive properties exhibited on the water-waves system without surface tension are generated by the {\em lower order term} in the equation (the principal part, being a simple transport equation, induces no dispersive effects). To our knowledge, this is the only known example of such phenomenon. 
\item The main point in the proof of dispersive estimates is the construction of a parametrix for solutions to~\eqref{eq.reduction}. One of the difficulties in the approach will be to get rid of this low regularity transport term by means of a parachange of variables (see the work by Alinhac~\cite{Alipara}). Once this reduction is performed, we finally have to handle the low regularity in the parametrix construction. 
\end{itemize}
In the rest of this section we shall describe more precisely the problem and outline the  steps of the proof on the main result.
\section{Equations and assumptions on the fluid domain}\label{S:1.1}

We consider the incompressible Euler equation in a 
time dependent fluid domain $\Omega$  
contained in a fixed domain $\mathcal{O}$, located between a free surface 
and a fixed bottom. 
We consider the general case where the bottom is arbitrary which means that  
the only assumption we shall make on the bottom is that it is separated from the free surface 
by a ÓstripÓ of fixed length. 

Namely, we assume that,
$$
\Omega=\left\{ (t,x,y)\in I\times \mathcal{O} \, :\, y < \eta(t,x)\right\},
$$
where $I\subset \xR_t$ and $\mathcal{O}\subset \xR^d\times \xR$ is a given open connected set. The spatial 
coordinates are $x\in \xR^d$ (horizontal) and $y\in \xR$ (vertical) with $d\ge 1$. 
We assume that the free surface
$$
\Sigma=\left\{ (t,x,y)\in I\times \xR^d\times \xR \, :\, y = \eta(t,x)\right\},
$$ 
is separated from the bottom $\Gamma=\partial\Omega\setminus \Sigma$ 
by a curved strip. This means that we study the case where there exists $h>0$ such that, for 
any $t$ in $I$, 
\begin{equation}\label{n1}
\left\{ (x,y)\in \xR^d\times \xR \, :\, \eta(t,x)-h < y < \eta(t,x)\right\} \subset \mathcal{O}.
\end{equation}
\textbf{Examples.} 
\begin{enumerate} 
\item $\mathcal{O}=\xR^d\times \xR$ corresponds to the infinite depth case (
$\Gamma=\emptyset$); 
\item The finite depth case corresponds to $\mathcal{O}=\{(x,y)\in \xR^d\times \xR\,:\, y>b(x)\}$ for some continuous function $b$ such that 
$\eta(t,x)-h>b(x)$ for any time $t$ (then $\Gamma=\{y=b(x)\}$). 
Notice that no regularity assumption is required on~$b$.
\item  
See the picture below.
\end{enumerate}

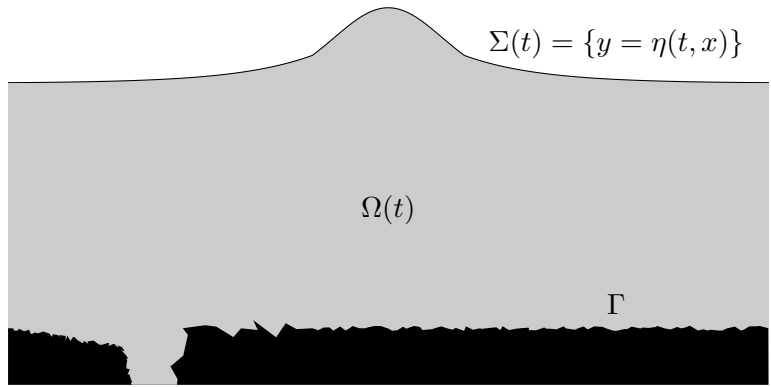
\usetikzlibrary{fadings}
\usetikzlibrary{decorations}
\usepgflibrary{decorations.pathmorphing}

\tikzfading[name=fade out, inner color=transparent!0,
  outer color=transparent!100]

\begin{center}
\begin{tikzpicture}[scale=1,samples=100]
    
\clip (-5,-3) rectangle (5,2);

    \filldraw [white!80!black,draw=black]
plot [domain=-6:-1] ({\x},{1+exp(\x)}) --
plot [domain=-1:1] ({\x},{1+exp(-\x*\x)}) --  
plot [domain=1:6] ({\x},{1+exp(-\x)}) -- (6,0) -- (-6,0) ;
    \fill[color=white!80!black,draw=none] (-12,-6) rectangle (9,1);
    \fill[black]
    decorate [decoration={random steps,segment length=1pt,amplitude=1pt}] 
    {(-5,-2.25) -- (-3.5,-2.5)}
    decorate [decoration={random steps,segment length=1pt,amplitude=1pt}] 
    {(-3.5,-2.5) -- (-3,-4.25)}
    decorate [decoration={random steps,segment length=5pt,amplitude=4pt}] 
    {-- (-2.7,-2.25) -- (-1,-2.25)}
    decorate [decoration={random steps,segment length=2pt,amplitude=1pt}] 
    {-- (-1,-2.25) -- (5,-2.25)}
    -- (5,-3) -- (-5,-3) -- (-5,-2.25);
     \node at (0,-1) [above] {$\Omega(t)$};
\node at (3,1.2) [above] {$\Sigma(t)=\{y=\eta(t,x)\}$};
\node at (3,-2.2) [above] {$\Gamma$};
\end{tikzpicture}  
\end{center}

We consider a potential flow such that the velocity $v$ is given by $v=\nabla_{x,y}\phi$ 
for some   function $\phi\colon \Omega\rightarrow \xR$, 
such that 
$\Deltayx\phi=0$. 
The system~\eqref{WW} reads 
\begin{equation}\label{n2}
\left\{
\begin{aligned}
&\partial_{t}\phi+\frac{1}{2}\la \partialyx\phi\ra^2  +P+g y = 0 \quad\text{ in }\Omega,\\
&\partial_{t} \eta = \partial_{y}\phi -\partialx\eta\cdot\partialx \phi \quad \text{ on } \Sigma, \\
&P = 0\quad  \text{ on } \Sigma, \\
&\partial_\nu \phi=0 \quad \text{ on }\Gamma,
\end{aligned}
\right.
\end{equation}
where as above $g>0$ is acceleration due to gravity, $P$ is the pressure and 
$\nu$ denotes the normal vector to $\Gamma$ (whenever it exists; 
for general domains, one solves the boundary value problem by a variational argument, 
see \cite{ABZ1,ABZ3}). 

\section{Regularity thresholds for the water waves}\label{S:21}

A well-known property of smooth solutions is that their energy 
is conserved
$$
\frac{d}{dt} 
\Big\{\mez \int_{\Omega(t)} \la \nabla_{x,y}\phi (t,x,y)\ra^2 \, dxdy + \frac{g}{2}\int_{\xR^d} \eta(t,x)^2 \, dx \Big\} =0.
$$
However, we do not know if weak solutions exist at this level of regularity 
(even the meaning of the equations is not clear). 
This is the only known coercive quantity (see~\cite{BO}).

Another regularity threshold is given by the scaling invariance which holds 
in the infinite depth case (that is when $\mathcal{O}=\xR^d\times \xR$). 
If $\phi$ and $\eta$ are solutions of the gravity water waves equations, then $\phi_\lambda$ and $\eta_\lambda$ defined by
$$
\phi_\lambda(t,x,y)=\lambda^{-3/2} \phi (\sqrt{\lambda}t,\lambda x , \lambda y),\quad 
\eta_\lambda(t,x)=\lambda^{-1} \eta(\sqrt{\lambda} t,\lambda x),
$$
solve the same system of equations. The (homogeneous) H\"older spaces invariant by this scaling (the scaling critical spaces) correspond to $\eta_0$ Lipschitz and $\phi_0$ in  $ \dot{W}^{3/2, \infty}$   
(one can replace the H\"older spaces by 
other spaces having the same invariance by scalings). 

According to the scaling argument, 
one could expect that the problem exhibits some kind of ``ill-posedness'' for initial data such that the free surface is not Lipschitz. 
See e.g.\ \cite{CaARMA,CCT} for such ill-posedness 
results for semi-linear equations. However, the water waves equations are not semi-linear and it 
is not clear whether the scaling argument is the only relevant regularity threshold to determine 
the optimal regularity in the analysis of the Cauchy problem (we refer the reader to 
the discussion in Section~$1.1.2$ of the recent result by Klainerman-Rodnianski-Szeftel~\cite{KRS}). 
In particular, it remains an open problem to prove an ill-posedness result for the gravity water waves equations. 
We refer to the recent paper by Chen, Marzuola, Spirn and Wright~\cite{CMSW} for a related result in the presence of surface tension. 

Several additional criterions have appeared in the mathematical analysis 
of the water waves equations. The first results on the water waves equations required very smooth initial data. 
The literature on the subject is now well established, 
starting with the pioneering work of Nalimov~\cite{Nalimov} (see also 
Yosihara~\cite{Yosihara} and Craig~\cite{Craig1985}) who showed the unique solvability in 
Sobolev spaces under a smallness assumption. 
Wu proved that the Cauchy problem is well posed 
without smallness assumption~(\cite{WuJAMS,WuInvent}). Several extensions of this result were obtained 
by various methods and many authors. We shall quote only some recent results on the local Cauchy problem: \cite{ABZ3,CL,CS,LannesJAMS,LindbladAnnals,MR,SZ}. 

To ensure that the particles flow is well defined, it seems natural to assume that the 
gradient of the velocity is bounded (or at least in BMO, see the work by  Hunter, Ifrim and Tataru in~\cite{HIT}). We refer to blow-up criteria by 
Christodoulou and Lindblad~\cite{ChLi} or Wang and Zhang~\cite{WaZh}.  
Below we shall construct solutions such that the velocity is still in $L^2((-T,T);W^{1,\infty})$
even though it is initially only  in $W^{1- \mu, \infty}$.

Finally, notice that though the above continuation criterions are most naturally stated in H\"older spaces, the use of $L^2$-based Sobolev spaces seems unavoidable (recall from the appendix of \cite{ABZ4} that 
the Cauchy problem for the linearized equations is ill-posed on H\"older spaces, as it exhibits a loss of $d/4$ derivatives). So let us rewrite the previous discussion in this framework. Firstly, 
the critical space for $\eta_0$ (resp.\ the trace $\underline{v}_0$ of the velocity at the free surface) is $\dot{H}^{1+\frac{d}{2}}(\xR^d)$ (resp.\ $\dot{H}^{\mez+\frac{d}{2}}(\xR^d)$). We proved in \cite{ABZ3} that 
the Cauchy problem is well-posed 
for initial data $(\eta_0,\underline{v}_0)$ 
in $H^{\tdm+\frac{d}{2}+\eps}(\xR^d)\times H^{1+\frac{d}{2}+\eps}(\xR^d)$ with $\eps>0$. 
This corresponds to the requirement that the initial velocity field should be Lipschitz. In this paper we shall prove that the Cauchy problem is well posed for initial data $(\eta_0,\underline{v}_0)$ belonging to 
$H^{\tdm+\frac{d}{2}-\delta}(\xR^d)\times H^{1+\frac{d}{2}-\delta}(\xR^d)$ for  
$0<\delta<\mu $. One important conclusion is that, 
in dimension $d=1$, one can consider initial free surface whose curvature does not belong to $L^2$. 

\vspace{-5mm}

\section{Reformulation of the equations}

Following Zakharov~(\cite{Zakharov1968}) 
and Craig and Sulem 
(\cite{CrSu}) 
we reduce the water waves equations 
to a system on the free surface. To do so, notice that since the velocity potential 
$\phi$ is harmonic, it is fully determined by the knowledge of $\eta$ and the knowledge of its trace at the free surface, denoted by $\psi$. 
Then one uses the 
Dirichlet-Neumann operator which maps a function defined on the free surface to the normal derivative of its harmonic extension. Namely, if $\psi=\psi(t,x) \in\xR$ is defined by 
$$
\psi(t,x)=\phi(t,x,\eta(t,x)),
$$
and if the Dirichlet-Neumann operator is defined by 
\begin{align*}
(G(\eta) \psi)  (t,x)&=
\sqrt{1+|\partialx\eta|^2}\,
\partial _n \phi\arrowvert_{y=\eta(t,x)}\\
&=(\partial_y \phi)(t,x,\eta(t,x))-\partialx \eta (t,x)\cdot (\partialx \phi)(t,x,\eta(t,x)),
\end{align*}
then one obtains the following system for two unknowns $(\eta,\psi)$ of the variables $(t,x)$,
\begin{equation}\label{n10}
\left\{
\begin{aligned}
&\partial_{t}\eta-G(\eta)\psi=0,\\[1ex]
&\partial_{t}\psi+g \eta
+ \smash{\frac{1}{2}\la\partialx \psi\ra^2  -\frac{1}{2}
\frac{\bigl(\partialx  \eta\cdot\partialx \psi +G(\eta) \psi \bigr)^2}{1+|\partialx  \eta|^2}}
= 0.
\end{aligned}
\right.
\end{equation}

We refer to \cite{ABZ1,ABZ3} for a precise construction of $G(\eta)$ 
in a domain with a general bottom. We also mention that, 
for general domains, we proved in \cite{Bertinoro} that 
if a solution $(\eta,\psi)$ of System~\eqref{n10} belongs 
to $C^0([0,T];H^{s+\mez}(\xR^d)\times H^{s+\mez}(\xR^d))$ 
for some $T>0$ and $s>1/2+d/2$, then one can define a velocity potential $\phi$ and 
a pressure $P$ satisfying \eqref{n2}. 
Below we shall always consider solutions such that  $(\eta,\psi)$ belongs 
to $C^0([0,T];H^{s+\mez}(\xR^d)\times H^{s+\mez}(\xR^d))$ for some $s>1/2+d/2$ (which is the scaling index). It is thus sufficient to solve the 
Craig--Sulem--Zakharov formulation~\eqref{n10} of the water waves equations.

\section{Main result}
\label{sec.1.4}
We shall work with the horizontal and vertical traces of the velocity 
on the free boundary, namely
$$
B= (\partial_y \phi)\arrowvert_{y=\eta},\quad 
V = (\nabla_x \phi)\arrowvert_{y=\eta}.$$ 
They are given in terms of $\eta$ and $\psi$ by means of the formula 
\begin{equation}\label{defi:BV}
B= \frac{\partialx \eta \cdot\partialx \psi+ G(\eta)\psi}{1+|\partialx  \eta|^2},
\qquad
V=\partialx \psi -B \partialx\eta.
\end{equation}
Also, recall that the Taylor coefficient $\ma$ defined 
by
\be\label{n17}
a=-\partial_y P\arrowvert_{y=\eta}
\ee
can be defined in terms of 
$\eta,\psi$ only (see \cite{Bertinoro} or Definition~$1.5$ in \cite{ABZ3}). 

For $\rho= k + \sigma$ with $k\in\xN$ and $\sigma \in (0,1)$, recall that one denotes 
by~$W^{\rho,\infty}(\xR^d)$ 
the space of functions whose derivatives up to order~$k$ 
are bounded and uniformly H\"older continuous with exponent~$\sigma$. Hereafter, we always consider indexes $\rho\not\in\xN$. 

In our previous paper~\cite{ABZ3}, 
we proved that the Cauchy problem is well-posed in Sobolev spaces 
for initial data such that, for some $s>1+d/2$, 
$(\eta_0,\psi_0,V_0,B_0)$ belongs to 
$H^{s+\mez}\times H^{s+\mez}\times (H^{s})^{d}\times H^{s}$. 
Our main result in this paper is a well-posedness result which holds 
for some $s<1+d/2$. In addition, we shall prove 
Strichartz estimates. 
\begin{theo}\label{main}
Let $d\ge 1$ and consider two real numbers $s$ and $r$ satisfying
\begin{equation*}
s>1+\frac{d}{2}-\gain, \quad 1<r<s+\gain-\frac{d}{2}\quad 
 \text{where }\mu = \begin{cases} \frac 1 {24} \text{ if } d=1,\\ \frac 1 {12} \text{ if } d\geq 2. \end{cases}
\end{equation*}
Consider an initial data $(\eta_{0},\psi_{0})$ such that
\begin{enumerate}[(H1)]
\item \label{assu1}
$\eta_0\in H^{s+\mez}(\xR^d),\quad \psi_0\in H^{s+\mez}(\xR^d),\quad V_0\in H^{s}(\xR^d),\quad B_0\in H^{s}(\xR^d)$,
\item there exists $h>0$ such that condition \eqref{n1} holds initially for $t=0$,
\item (Taylor sign condition) there exists $c>0$ such that, for all $x\in \xR^d$, 
$\ma_0(x)\ge c$.
\end{enumerate}
Then there exists $T>0$ such that 
the Cauchy problem for \eqref{n10} 
with initial data  $(\eta_{0},\psi_{0})$ has a unique solution 
such that 
\begin{enumerate}
\item $\eta$ and $\psi$ belong to 
$C^0\big([0,T];H^{s+\mez}(\xR^d)\big)
\cap L^p\big([0,T];W^{r+\mez,\infty}(\xR^d)\big)$ where $p=4$ if $d=1$ and $p=2$ for $d\ge 2$,
\item $V$ and $B$ belong to 
$C^0\big([0,T];H^{s}(\xR^d)\big)\cap L^p\big([0,T];W^{r,\infty}(\xR^d)\big)$ with $p$ as above,
\item the condition \eqref{n1} holds for 
$0\le t\le T$, with $h$ replaced with $h/2$,
\item for all $0\le t\le T$ and for all $x\in \xR^d$, 
$\ma(t,x)\ge c/2$.
\end{enumerate}
\end{theo}

\section{Paradifferential reduction}

The proof relies on a 
paradifferential reduction of the water-waves equations, as in \cite{ABZ1,ABZ3,AM}. 
This is the property that the equations can be reduced to a very simple form
\begin{equation}\label{eqpara}
\Big(  \partial_t + \mez( T_{V } \cdot \nabla + \nabla \cdot T_{V }) +i T_{\gamma } \Big)u  =f 
\end{equation}
where $T_V$ is a paraproduct ($T_{V } \cdot \nabla u+ \nabla \cdot T_{V } u=2T_V\cdot\nabla u 
+T_{\cnx V} u$) and $T_\gamma$ is a paradifferential operator of order 
$\frac{1}{2}$ with symbol 
$$
\gamma=\sqrt{\ma \lambda}$$
where
$$
\lambda=\sqrt{\big(1+ \vert \nabla \eta\vert^2\big)\vert \xi \vert^2 
- \big(\xi \cdot \nabla \eta\big)^2}.
$$
Here $a$ is the Taylor coefficient (see \e{n17}) and 
$\lambda$ is the principal symbol of the Dirichlet-Neumann operator 
(see Appendix~\ref{sec:2} for the definition of paradifferential operators). When 
$d=1$, $\lambda$ simplifies to $\la \xi\ra$ so $T_{\gamma }u=T_{\sqrt{a}}\la D_x\ra^{\mez}u$.

To prove Theorem~\ref{main} 
we need tame estimates which complement the estimates already proved in \cite{ABZ3}. 
We shall study the Dirichlet-Neumann operator.

In the case without bottom ($\Gamma=\emptyset$), 
when $\eta$ is a smooth function, it is known that, modulo a smoothing operator, $G(\eta)$ 
is a pseudo-differential operator whose principal symbol is given by $\lambda$. 

Notice that~$\lambda$ is well-defined for any~$C^1$ function~$\eta$. 
In \cite{ABZ3} we proved several results which 
allow 
to compare~$G(\eta)$ to the paradifferential operator~$T_\lambda$ when~$\eta$ 
has limited regularity. 
In particular we proved that, for any $s>1+d/2$, 
\begin{equation}\label{n010}
\lA G(\eta)f-T_\lambda f\rA_{H^{s-\mez}({\mathbf{R}}^d)}
\le \mathcal{F} \bigl(\| \eta \|_{H^{s+\mez}({\mathbf{R}}^d)}\bigr)\lA f\rA_{H^{s}({\mathbf{R}}^d)}.
\end{equation}

When $\eta$ is a smooth function, 
one expects that $G(\eta)-T_\lambda$ is of order $0$ which means that 
$G(\eta)f-T_\lambda f$ has the same regularity as $f$ 
(this holds true whenever $\eta$ is much smoother than $f$). 
On the other hand, \e{n010} gives only 
that this difference ``is of order'' $1/2$ (it maps $H^s$ to $H^{s-1/2}$). 
This is because we allow $\eta$ to be only $1/2$-derivative more regular than $f$.  
This is tailored to the analysis of gravity waves since, 
for scaling reasons, it is natural to assume that $\eta$ is $1/2$-derivative 
more regular than the trace of the velocity on the free surface. 

We shall improve \eqref{n010} by proving that 
$G(\eta)-T_\lambda$ is of order $1/2$ assuming only that $s>3/4+d/2$ 
together with sharp 
H\"older regularity assumptions on both $\eta$ and $f$ (these H\"older assumptions are the ones 
that hold by Sobolev injections for $s>1+d/2$). 
Since H\"older norms are controlled only in some $L^p$ spaces in time (by Strichartz estimates), we need to precise 
the dependence of the constants. 
We shall prove in Chapter~\ref{C:2} 
the following result that we believe is of independent interest. 

\begin{theo}\label{T3}
Let~$d\ge 1$ and consider real numbers~$s,r$ such that
$$
s>\frac{3}{4}+\frac{d}{2},\qquad r>1.
$$
Consider~$\eta\in H^{s+\mez}({\mathbf{R}}^d)\cap \holdertdm({\mathbf{R}}^d)$ 
and~$f\in H^{s}({\mathbf{R}}^d)\cap \holder{r}({\mathbf{R}}^d)$, then 
$G(\eta)f$ belongs to $H^{s-\mez}({\mathbf{R}}^d)$ and
\begin{equation}\label{Hstrick0}
\lA G(\eta)f-T_\lambda f\rA_{H^{s-\mez}}\le 
\mathcal{F}\bigl( \| \eta \|_{H^{s+\mez}}+\lA f\rA_{H^s}\bigr)\left\{1+\| \eta \|_{\holdertdm}+\lA f\rA_{\holder{r}}\right\},
\end{equation}
for some non-decreasing function~$\mathcal{F}\colon \xR^+\rightarrow \xR^+$ 
depending only on~$s$ and~$r$.
\end{theo}
\begin{rema}
This estimate is tame in the following sense. In the context we will be mostly interested in ($s< 1+d/2$), for oscillating functions, we have ($r>1$)
$$
\lA u\Bigl(\frac{x}{\eps}\Bigr)\rA_{\holdertdm}\sim \Bigl(\frac{1}{\eps}\Bigr)^{r+\mez}
\gg \Bigl(\frac{1}{\eps}\Bigr)^{s+\mez-\frac{d}{2}}\sim \lA u\Bigl(\frac{x}{\eps}\Bigr)\rA_{H^{s+\mez}},
$$
and consequently, the estimate \eqref{Hstrick0} is {\em linear} with respect to the highest norm.
\end{rema}

\section{Strichartz estimates}
Using the previous paradifferential reduction, 
the key point is to obtain estimates in H\"older spaces coming 
from Strichartz ones. Most of the analysis 
is devoted to the proof of the following result.

\begin{theo}\label{T4}
Let $I= [0,T]$, $d\geq 1$. 
Let $\gain$ be such that 
$\gain<\frac{1}{24}$ if $d=1$ and $\gain<\frac{1}{12}$ if $d\ge 2$.

Let  $s\in\xR$ and $f \in L^\infty(I; H^s(\xR^d))$. 
Let $u\in C^0(I; H^s(\xR^d))$ be a solution of \eqref{eqpara}. 
Then one can find $k = k(d)$ such that
\begin{equation*}
\begin{aligned}
\Vert  u  \Vert_{L^p(I;  \zygmund{s-\frac{d}{2}+\gain}(\xR^d))}
\leq    \mathcal{F}\big(\Vert V\Vert_{E_0} +  \mathcal{N}_k(\gamma)\big) 
\Big\{  \Vert  f\Vert_{L^p(I; H^{s }(\xR^d))} 
+ \Vert  u \Vert_{C^0(I; H^s(\xR^d))}\Big\}
\end{aligned}
\end{equation*}
where $\zygmund{r}$ is the Zygmund space of order $r\in\xR$ (see Definition~$\ref{T:Zygmund}$), 
$p=4$ if $ d=1 $ and $p=2$ if $d \geq 2$, 
$E_0 = L^p(I; W^{1,\infty}(\xR^d))^d$ and 
$\mathcal{N}_k(\gamma) = \sum_{\vert \beta  \vert \leq k} \Vert D_\xi^\beta \gamma\Vert_{L^\infty(I\times \xR^d \times \mathcal{C})}$ 
with $\mathcal{C}=\{\frac{1}{10}\le\la\xi\ra\le10\}$.  
\end{theo}

This last theorem is proved in Chapter~\ref{C:4}. 
We conclude this introduction by explaining the strategy 
of its proof.

\subsubsection*{Linearized equation}

To explain our strategy, let us first consider as a simple model the linearized equation at $(\eta =0, \psi =0)$ when $d=2, g=1$, in the case without bottom. Then $G(0) = |D_x|$ and the linearized system reads
$$
\partial_t \eta- |D_x| \psi=0, \qquad \partial_t \psi + g \eta =0,$$ which, with $u= \eta + i|D_x| ^{1/2} \psi $
 can be written under the form
$$
\partial_t u +i \la D_x\ra^{1/2}u=0.
$$
Since the operators $e^{-it\la D_x\ra^{1/2}}$ are unitary on Sobolev spaces, 
the Sobolev embedding $H^{1+\eps}(\xR^2)\subset L^\infty(\xR^2)$ ($\eps>0$) implies that
$$
\lA e^{-it\la D_x\ra^{1/2}} u_0 \rA_{L^\infty_t (]0,1[; L^\infty( \xR^2_x))} \le 
C  \lA u_0 \rA_{H^{1 + \eps}(\xR^2_x)}.
$$
We shall recall the proof of the following Strichartz estimate
$$
\exists C>0,  \forall 2< p\leq + \infty, ~:\quad\lA e^{-it|D_x|^{1/2}} u_0 \rA_{L^p_t(]0,1[; L^{\frac {2p} {p-2}}( \xR^2_x))} 
\le C \lA u_0 \rA_{H^{\frac 3 {2p}}(\xR^2)},
$$
which (taking $p$ close to $2$) allows to gain almost $1/4$ derivative with respect to the Sobolev embedding. 

The strategy of the proof is classical (see Ginibre-Velo~\cite{GV} and Keel-Tao~\cite{KT}). Firstly, by using the Littlewood-Paley decomposition, 
one can reduce the analysis to the case where the spectrum of $u_0$ is in a dyadic shell. 
Namely, it is sufficient to prove that
$$
\big\Vert e^{-it|D_x|^{1/2}} \chi(h |D_x|) u_0 \big\Vert_{L^p_t(]0,1[; L^{\frac {2p} {p-2}})} 
\le C  h^{- \frac 3 {2p}} \lA u_0 \rA_{L^2(\xR^2)},
$$ 
where $C$ is uniform with respect to $h\in ]0,1[$ 
and $\chi\in C^\infty_0(\xR\setminus 0)$ equals $1$ on 
$[1,2]$. 

To prove that $T = e^{-it|D_x|^{1/2}} \chi(h |D_x|)$ is bounded from 
$L^2_x$ to $L^p_t(L^q_x)$ ($q= \frac{2p} {p-2}$) with norm bounded by 
$A\defn C  h^{- \frac 3 {2p}} $, it suffices to prove that 
the operator $TT^* $ is bounded from  
$L^{p'}_t(L^{q'}_x)$ to $L^p_t(L^q_x)$ with norm bounded by  
$A^2=C^2  h^{-\frac 3 p}$. Now, write
$$
TT^* f = e^{-it |D_x|^{1/2}} \int_0^1 e^{is|D_x|^{1/2}} f(s, \cdot) ds.
$$
Using the Hardy-Littlewood-Sobolev inequality, the desired estimate for 
$TT^*$ will be a consequence of the following dispersive estimate :
$$
\| \chi(h |D_x|)e^{-i(t-s) |D_x|^{1/2}}\chi(h |D_x|)\|_{L^1_x \rightarrow L^\infty_x} 
\le \frac {C}{ h^{\frac 3 2}|t-s|}\cdot
$$ 
The proof of this estimate is classical: we have
$$ \chi(h |D_x|)e^{-it|D_x|^{1/2}}\chi(h |D_x|) u = \frac 1 {(2\pi )^2} \int e^{-it |\xi|^{1/2} + i (x-y) \cdot \xi ) }\chi^2 ( h|\xi|) u(y) dy d\xi,$$
and the estimate follows after changing variables ($\eta = h\xi$) from the stationary phase inequality.

\subsubsection*{The nonlinear system}
We now consider the nonlinear equation \eqref{eqpara}, which reads
\begin{equation*}
\Big(  \partial_t + \mez( T_{V } \cdot \nabla + \nabla \cdot T_{V }) +i T_{\gamma } \Big)u  =f.
\end{equation*}
To apply the strategy recalled in the previous paragraph, the main 
difficulties  are the following:
 \begin{itemize}
\item this is a paradifferential equation with non constant coefficients,
\item the coefficients are not smooth. Indeed,  $V$ is in $L^\infty_t(C^1_x)$ and the symbol 
$\gamma=\gamma(t,x,\xi)$, of order $1/2$ in $\xi$, 
is only $L^\infty_t(C^{1/2}_x)$ in the space-time variables,
\item the dispersion is due to the operator 
$T_\gamma$ of order $1/2$, while the equation contains the 
term $T_V\cdot\nabla$ of order $1$.
\end{itemize}

The first step of the proof 
is classical in the context of quasi-linear wave equations (see the works by Lebeau~\cite{Lebeau92}, 
Smith~\cite{Smith}, Bahouri-Chemin~\cite{BaCh}, 
Tataru~\cite{TataruNS} and Blair~\cite{Blair}). 
It consists, after a dyadic decomposition at frequency $h^{-1}$, 
in regularizing the coefficients at scale $h^{-\delta}$, where $ \delta \in (0,1)$ is to be chosen properly. Using the Littlewood-Paley 
decomposition $u=\sum_{j\ge -1} \Delta_j u$, we can write
\begin{equation*}
\Bigl(\partial_t +\mez( S_{j-2}(V)\cdot \nabla + \nabla  \cdot  S_{j-2}(V)) +i T_\gamma\Big)\Delta_j u =f_j,
\end{equation*}
where $S_{j-2} u =\sum_{k=-1}^{j-3} \Delta_k u$ and $f_j$ is easily estimated. 
Then, for some $\delta\in ]0,1[$ (here $\delta = \frac{2}{3}$), one considers instead the 
equation with smoothed coefficients:
$$
\Big(\partial_t +\mez( S_{j\delta}(V)\cdot \nabla + \nabla  \cdot  S_{\delta j}(V)) +i T_{\gamma_\delta}
\Big)\Delta_j u 
= f_j+ g_{j\delta}, \qquad \gamma_\delta = \psi(2^{-j\delta} D_x)\gamma
$$
where
$$
g_{j\delta} = \Big(\mez \big\{ \big(S_{j\delta}(V)  - S_j(V)\big)\cdot\nabla 
 +\nabla \cdot \big(S_{j\delta}(V)  - S_j(V)\big)\big\} + i (T_\gamma -T_{\gamma_\delta})\Big) \Delta_ju.
$$

Since the dispersion is due to the sub-principal term, we chose to straighten the vector field $\partial_t + S_{j\delta}(V)\cdot \nabla$ 
by means of a parachange of variables (following Alinhac~\cite{Alipara}). 
To do so, we solve the system 
$\dot{X}(t) = S_{j\delta}(V)(t,X(t))$ with $X(0)=x$ 
to obtain a mapping $x\mapsto X(t,x)$ which is a small perturbation of the identity in 
small time, satisfying
$$
\Vert \frac{\partial X}{\partial x}(t,\cdot)- Id \Vert_{L^\infty(\xR^d)} \leq C(\Vert V \Vert_{C^1}) \vert t \vert ^\mez.
$$
However, as the vector field $V$ is only Lipschitz, 
we have only the following estimates for the higher order derivatives:
$$
\Vert (\partial^\alpha_x X)(t, \cdot) \Vert_{L^\infty(\xR^d)} \leq C_\alpha(\Vert V \Vert_{E_0})
h^{-{\delta}(\vert \alpha \vert -1)} \vert t \vert ^\mez, \quad \vert \alpha \vert \geq 2, \quad h=2^{-j}.
$$
So one controls $\partial^\alpha_x X$ only on small time intervals whose sizes depend of $h=2^{-j}$ and 
$\alpha$. This is one reason why we will prove a dispersive estimate only in short time intervals 
whose size is tailored to the frequency.

Then, one makes the change of variables
$$
v_h(t,y) = (\Delta_j  u)(t, X(t,y)), \quad h=2^{-j},
$$
to obtain an equation of the form
$$
\partial_t v_h + iA_h(t,y,D_y)v_h = g_h,
$$
for an explicit operator $A$ of order $1/2$. 
For convenience we reduce the equation to a semi classical form by changing variables
$$
z=h^{-\mez}y, \quad \tilde{h} = h^\mez,\quad w_{\tilde{h}}(t,z) = v_h(t, \tilde{h}y)
$$
and multiplying the equation by $\tilde{h}$. We get an equation of the form
$$
(\tilde{h}\partial_t  + i P(t,\tilde{h}z, \tilde{h}D_z,\tilde{h}))w_{\tilde{h}} 
= \tilde{h}F_{\tilde{h}}.
$$

Finally, we are able to write a parametrix for this 
reduced system, which 
allows to prove Strichartz estimates using the classical strategy  outlined above,
on a small time interval $|t| \leq \tilde{h}^\delta=h^{\frac{\delta}{2}}$. 
The key step here is to prove that, on such time intervals one has a parametrix of the form
$$
\mathcal{K} v(t,z) = (2 \pi \tilde{h})^{-d} \iint e^{\frac{i}{\tilde{h}}(\phi(t,z,\xi,\tilde{h}) -z'\cdot \xi)}b(t,z,\xi,\tilde{h}) v(z') dz' d\xi
$$
where $b$ is a symbol and $\phi$ a real-valued phase function, such that
$$
\phi\arrowvert_{t=0} = z\cdot \xi, \quad b\arrowvert_{t=0} = \chi(\xi), \quad \text{supp} \chi \subset \{\xi: \frac{1}{3} \leq \vert \xi \vert \leq 3\}.
$$
Using the  parametrix, the stationary phase estimate and coming back to the original  variable $z \rightarrow y= h^\mez z \rightarrow x= X(t,y)$ we  obtain a dispersive estimate 
(see Theorem~\ref{dispersive}). 
This gives a Strichartz estimate on a time interval of size $h^{\delta/2}$. 
Finally, splitting the time interval $[0,T]$ into $Th^{-\delta /2} $  time intervals of size $h^{\delta/2}$, and gluing together all these estimates, we obtain a Strichartz estimate with loss 
on the time interval $[0,T]$.

\chapter{Tame estimates for the Dirichlet-Neumann operator}\label{C:2}

In this chapter, we prove Theorem~\ref{T3} on 
the paralinearization of the Dirichlet-Neumann operator.

\section{Scheme of the analysis}\label{S:2.1}
We shall revisit the approach given in \cite{ABZ1,ABZ3} using tame estimates at each step. 
In this section, we recall the scheme of the analysis and indicate the points at which the argument must be adapted.

Hereafter, we consider a time-independent fluid domain $\Omega$ 
satisfying the assumptions given in Section~\ref{S:1.1}, which we recall here. 
We assume that
$$
\Omega=\left\{ (x,y)\in \mathcal{O} \, :\, y < \eta(x)\right\},
$$
for some Lipschitz function $\eta$ and a given open domain~$\mathcal{O}$. 
We denote by $\Sigma$ (resp.\ $\Gamma$) the free surface (resp.\ the bottom). They are defined by
$$
\Sigma=\{ (x,y)\in \xR^d\times \xR\,:\, y=\eta(x)\},\qquad \Gamma=\partial\Omega\setminus \Sigma.
$$
We assume that the domain $\mathcal{O}$ 
contains a fixed strip separating the free surface and the bottom. 
This implies that 
there exists~$h>0$ such that
\begin{equation}\label{hypt}
\left\{ (x,y)\in {\mathbf{R}}^d\times \xR \, :\, \eta(x)-h < y < \eta(t,x)\right\} \subset \Omega.
\end{equation}
We also assume that the domain~$\mathcal{O}$ (and hence the domain~$\Omega$) 
is connected.  Without loss of generality we assume below that $h>1$. 

The fact that the Dirichlet-Neumann operator $G(\eta)$ is well-defined in such domains 
is proved in \cite{ABZ1,ABZ3}. 

In the analysis of free boundary problems it is classical to begin 
by reducing the analysis to a domain 
with a fixed boundary. We 
flatten the free surface by using a diffeomorphism introduced in \cite{ABZ3} whose definition is here recalled. Set
\begin{equation}\label{lesomega}
\left\{
\begin{aligned}
\Omega_1 &= \{(x,y): x \in {\mathbf{R}}^d, \eta(x)-h<y<\eta(x)\},\\
\Omega_2 &= \{(x,y)\in \mathcal{O}: y\leq\eta(x)-h\},\\
\Omega &= \Omega_1 \cup \Omega_2,
\end{aligned}
\right.
 \end{equation}
and 
\begin{equation}\label{omega1}
\left\{
\begin{aligned}
&\widetilde{\Omega}_1= \{(x,z): x \in {\mathbf{R}}^d, z \in I\}, \quad I = (-1,0),\\
&\widetilde{\Omega}_2 = \{(x,z)\in \xR^d \times (-\infty, -1]: (x,z+1+\eta(x)-h)\in \Omega_2\},\\
&\widetilde{\Omega}= \widetilde{\Omega}_1  \cup \widetilde{\Omega}_2.  
\end{aligned}
\right.
\end{equation}
Guided by Lannes~(\cite{LannesJAMS}), 
we consider a Lipschitz diffeomorphism from $\widetilde{\Omega}_1$ to $\Omega_1$ 
of the form 
$(x,z) \mapsto (x, \rho(x,z))$ where 
the map $(x,z) \mapsto  \rho(x,z)$ from $\widetilde{\Omega}$ to $\xR$ is defined as follows
\begin{equation}\label{diffeo}
\left\{
\begin{aligned}
&\rho(x,z)=  (1+z)e^{\delta z\langle D_x \rangle }\eta(x) -z \big\{e^{-(1+ z)\delta\langle D_x \rangle }\eta(x) -h\big\}\quad \text{if } (x,z)\in \widetilde{\Omega}_1,\\
&\rho(x,z) = z+1+\eta(x)-h\quad \text{if } (x,z)\in \widetilde{\Omega}_2 
\end{aligned}
\right.
\end{equation}
for some small enough positive constant $\delta$.

\begin{lemm}\label{rho:diffeo}
Assume $\eta \in W^{1,\infty}(\xR^d)$.
\begin{enumerate}
\item There exists $C>0$ such that for every $(x,z) \in \widetilde{\Omega} $ we have
$$ \vert \nabla_{x } \rho (x,z)\vert  \leq C \Vert \eta \Vert_{ W^{1,\infty}(\xR^d)} .$$
\item There exists  $K>0$   such that,  if  $ \delta \Vert \eta \Vert_{ W^{1,\infty}(\xR^d)} \leq \frac{h}{2K}$
 we have
\begin{equation}\label{rhokappa}
  \quad   \text{min }\big(1,\frac{h}{2}\big) \leq \partial_z \rho(x,z) \leq \text{max }(1,  \frac{3h}{2}), \quad \forall (x,z) \in \widetilde{\Omega}.
\end{equation}
\item The map $(x,z) \mapsto (x, \rho(x,z))$ is a Lipschitz diffeomorphism from $\widetilde{\Omega}_1$ to $\Omega_1.$
\item 
Let $I=(-1,0)$ and $s$ be a real number. 
There exists $C>0$ such that for every $\eta \in H^{s+\mez}(\xR^d)$ we have
\begin{equation}\label{eq.rho1}
\bad
& \Vert\partial_z \rho -h \Vert_{C^0_z(I;H^{s-\mez}({\mathbf{R}}^d))\cap L^2_z(I;H^{s}({\mathbf{R}}^d))}
\leq C \sqrt{\delta}  \|\eta\|_{H^{s+ \mez}( {\mathbf{R}}^d)}, \\
&\|\nabla_x \rho \|_{C^0_z(I;H^{s-\mez}({\mathbf{R}}^d))\cap L^2_z(I;H^{s}({\mathbf{R}}^d))}   
\leq \frac{C}{\sqrt{\delta}}   \|\eta\|_{H^{s+ \mez}( {\mathbf{R}}^d)}.
\ead
\end{equation}
\item Assume that $\eta\in W^{r+\mez,\infty}(\xR^d)$ with $r>1/2$. Then, for any $r'$ in $[1/2,r]$, 
\begin{equation}\label{regofrho3}
\lA\nabla_{x,z}\rho\rA_{C^0([-1,0];\holder{r'-\mez})}+\lA\nabla_{x,z}\rho\rA_{L^2([-1,0];W^{r',\infty})}
\le C (1+\| \eta \|_{\holdertdm}). 
\end{equation}
\end{enumerate}
\end{lemm}
\begin{proof}
The first four statements are proved in \cite{ABZ3}. The last one follows from the fact that, for any 
$\mu'>\mu\ge 0$,   
$|z|^{\mu'} \langle D_x\rangle^\mu e^{z\langle D_x\rangle}$ is bounded from $L^\infty$ 
to $L^\infty$, uniformly in $z \in [-1,0]$ and therefore
\begin{align*}
\Vert \Delta_j   e^{z\langle D_x\rangle}u\Vert_{L^\infty} 
&= \vert z \vert^{-\mu'} \Vert |z|^{\mu'} \langle D_x\rangle^\mu e^{z\langle D_x\rangle}\langle D_x\rangle^{-\mu} \Delta_j u\Vert_{L^\infty}\\
&\leq C\vert z \vert^{-\mu'} \Vert \langle D_x\rangle^{-\mu}  \Delta_j u\Vert_{L^\infty}
\end{align*}
which easily implies the desired result in view of Remark \ref{R:Zygmund}.
\end{proof}

\textbf{Flattening the free surface}. In \cite{ABZ1,ABZ3} we proved that the problem 
$$
\Delta_{x,y}\phi=0, \quad \phi\arrowvert_{y=\eta}=f, \quad \partial_n\phi=0 \text{ on }\Gamma,
$$
has a unique variational solution. Then we introduce the following function 
$$
v(x,z)=\phi(x,\rho(x,z))
$$

where $(x,z)$ belongs to  the `flattened' domain $\widetilde{\Omega}$ (notice that we flatten only the free surface). 

The equation satisfied by  $v$ in $\widetilde{\Omega}$ can be written  in three forms. 
Firstly,
\begin{equation}\label{dnint1}
(\partial_z^2  +\alpha\Delta_x  + \beta \cdot\nabla_x\partial_z   - \gamma \partial_z )v =0,
\end{equation}
where 
\begin{equation}\label{alpha}
\alpha\defn \frac{(\partial_z\rho)^2}{1+|\partialx  \rho |^2},\quad 
\beta\defn  -2 \frac{\partial_z \rho \nabla_x \rho}{1+|\nabla_x  \rho |^2} ,\quad 
\gamma \defn \frac{1}{\partial_z\rho}\bigl(  \partial_z^2 \rho 
+\alpha\Delta_x \rho + \beta \cdot \nabla_x \partial_z \rho\bigr).
\end{equation}
Secondly, one has
\begin{equation}\label{dnint2}
(\Lambda_1^2 +\Lambda_2^2)v= 0,
\end{equation}
where
\begin{equation} \label{Lambda}
\Lambda_1 = \frac{1}{\partial_z \rho} \partial_z \quad \Lambda_2 = \nabla_x -  \frac{\nabla_x \rho}{\partial_z \rho} \partial_z.
 \end{equation}
Eventually, 
\begin{equation}\label{P:div}
\widetilde{P}v := \text{div} \big( \partial_z \rho \nabla_xv \big) - \text{div}  \big( \nabla_x \rho \,  \partial_z v\big) - \partial_z \big( \nabla_x \rho \cdot \nabla_x v\big) + \partial_z \big( \frac{1+ \vert \nabla_x \rho \vert^2}{ \partial_z \rho} \partial_z v \big) =0,
\end{equation}
as can be verified starting from \e{dnint2} by a direct calculation. Moreover,
$$
v\arrowvert_{z=0}=\phi\arrowvert_{y=\eta(x)}=f,
$$
and
\begin{equation}\label{DN:forme}
G(\eta) f=  \big(\frac{1 +|\partialx \rho |^2}{\partial_z \rho} \partial_z v  - \nabla_x \rho\cdot\nabla_x v\big)
\big\arrowvert_{z=0} = (\Lambda_1 v - \nabla_x \rho \cdot \Lambda_2 v)\arrowvert_{z=0}.
\end{equation}

The analysis of the Dirichlet-Neumann operator is then divided into three steps.

\textbf{First step}. We paralinearize the equation. That is we write the equation for 
$v=\phi(x,\rho(x,z))$ in the form
\begin{equation}\label{2m.11}
\partial_z^2 v  +T_{\alpha}\Deltax v +T_{\beta}\cdot \partialx \partial_z v =F_1+F_2,
\end{equation}
where  
\be\label{F2e}
\bad
F_1&=\gamma\partial_z v,\\
F_2&= (T_\alpha-\alpha)\Delta v+(T_\beta-\beta)\cdot\partialx\partial_z v.
\ead
\end{equation}
We are going to estimate $F_1$ by product rules in Sobolev spaces and 
$F_2$ by using results recalled in Appendix~\ref{sec:2}.

\textbf{Second step}. We factor out the elliptic equation as the product of  
a forward and a backward parabolic evolution equations. We write, for some  
symbols 
$a,A$ and a remainder~$F_3$, 
\begin{equation}\label{2m2.b}
( \partial_z - T_a) (\partial_z - T_A)v =F_1+F_2+F_3.
\end{equation}
Namely
\begin{equation}\label{aA}
a = \frac{1}{2}\bigl( -i  \beta\cdot \xi
-   \sqrt{ 4\alpha \la \xi \ra^2 -  (\beta \cdot \xi)^2}\bigr), \qquad A  = \frac{1}{2}\bigl( -i  \beta\cdot \xi
+   \sqrt{ 4\alpha \la \xi \ra^2 -  (\beta \cdot \xi)^2}\bigr).
\end{equation}
The term $F_3$ is estimated by means of the symbolic calculus rules recalled 
in \S\ref{sec.2.2}.

\textbf{Third step}. 
Let us view $z$ as a time variable. Then 
$\partial_z u - T_a u=F$ is a parabolic equation (since~$\RE (-a) \ge c\la \xi\ra$). 
On the other hand, $\partial_z u -T_A u =F$ is a {\em backward\/} parabolic evolution equation 
(by definition~$\RE A \ge c\la \xi\ra$). 
We shall use parabolic estimates twice to deduce from the previous step estimates for 
$\nabla_{x,z}v$ and 
$ (\partial_z - T_A)v$. 

\textbf{Previous results}.
Let $I=[-1,0]$. By using the approach explained above, we proved in \cite{ABZ3} that, for any $s>1/2+d/2$,
\begin{equation}\label{oubli}
\lA \nabla_{x,z} v\rA_{C^0_z(I;H^{s-1}({\mathbf{R}}^d))\cap L^2_z(I;H^{s-\mez}({\mathbf{R}}^d))}
\le \mathcal{F}(\| \eta \|_{H^{s+\mez}})\lA f\rA_{H^s}.
\end{equation}
Moreover, for any $0<\eps\leq \mez$ such that $\eps< s-\mez-\frac{d}{2}$, we have
\begin{equation}\label{w-1}
\lA \partial_z v-T_A v \rA_{C^0_z(I;H^{s-1+\eps}({\mathbf{R}}^d))\cap L^2_z(I;H^{s-\mez+\eps}({\mathbf{R}}^d))}
\le \mathcal{F}(\| \eta \|_{H^{s+\mez}})\lA f\rA_{H^s}.
\end{equation}
The key point to prove Theorem~\ref{T3}Êwill be to prove an estimate analogous to \eqref{w-1} 
with $\eps=1/2$ and $s<1+d/2$, assuming an extra control of $\eta$ and $f$ in H\"older spaces.

Actually, concerning elliptic regularity, 
in \cite{ABZ3} we proved more general results than \e{oubli} 
and we record here two statements for later references.

\begin{prop}\label{T:DN-Hs}
Let~$d\ge 1$,~$s>\mez+\frac{d}{2}$ and~$\frac 1 2 \leq \sigma \leq s+ \frac 1 2$. 
Then there exists a non-decreasing function~$\mathcal{F}\colon\xR_+\rightarrow\xR_+$ 
such that, for all 
$\eta\in H^{s+\mez}({\mathbf{R}}^d)$ and all~$f\in H^{\sigma}({\mathbf{R}}^d)$, we have 
$G(\eta)f\in H^{\sigma-1}({\mathbf{R}}^d)$, together with the estimate
\begin{equation}\label{ests+2}
\lA G(\eta)f \rA_{H^{\sigma-1}({\mathbf{R}}^d)}\le 
\mathcal{F}\bigl(\| \eta \|_{H^{s+\mez}({\mathbf{R}}^d)}\bigr)\lA f\rA_{H^{\sigma}({\mathbf{R}}^d)}.
\end{equation}
 \end{prop}
Given~$\mu\in\xR$ we define the spaces
\begin{equation*}
\begin{aligned}
X^\mu(I)&=C^0_z(I;H^{\mu}({\mathbf{R}}^d))\cap L^2_z(I;H^{\mu+\mez}({\mathbf{R}}^d)),\\
Y^\mu(I)&=L^1_z(I;H^{\mu}({\mathbf{R}}^d))+L^2_z(I;H^{\mu-\mez}({\mathbf{R}}^d)) 
\end{aligned}
\end{equation*}
and we consider the problem 
\begin{equation}\label{dnint1F}
\partial_z^2 v +\alpha\Delta v + \beta \cdot\partialx\partial_z v  
- \gamma \partial_z v=F_0 + \partial_z G_0,\quad v\arrowvert_{z=0}=f,
\end{equation}
where~$f=f(x), F_0=F_0(x,z), G_0=G_0(x,z)$ are given functions. Then we have,
\begin{prop}\label{p1}
Let~$d\ge 1$ and 
$$
s> \mez + \frac d 2,\quad - \frac 1 2 \leq \sigma \leq s - \frac 1 2.
$$ 
Consider ~$f\in H^{\sigma+1}({\mathbf{R}}^d)$, 
$F_0\in Y^\sigma([-1,0]),$ $G_0\in Y^{\sigma+1}([-1,0])$ and~$v$  a solution to~\eqref{dnint1F} such that   $ \nabla_{x,z} v \in X^{-\mez}([-1,0]).$
Then for any~$z_0\in (-1,0)$, 
$
\nabla_{x,z}v \in X^{\sigma}([z_0,0])$,
and
\begin{align*}
\lA \nabla_{x,z} v\rA_{X^{\sigma}([z_0,0])}
\le \mathcal{F}(\| \eta \|_{H^{s+\mez}}) \big\{\lA f\rA_{H^{\sigma+1}}+\lA F_0\rA_{Y^\sigma([-1,0])}  + &\lA  G_0\rA_{Y^{\sigma+1}([-1,0])}\\
&+ \lA \nabla_{x,z} v\rA_{X^{-\mez}([-1,0])} \big\} 
\end{align*}
for some non-decreasing function~$\mathcal{F}\colon\xR_+\rightarrow\xR_+$ depending only on~$\sigma$ and $d$. 
\end{prop}

\section{Parabolic evolution equation}\label{s:pe}

As explained above, we need estimates for paradifferential parabolic equations of the form
$$
\partial_z w + T_p w =f,\quad w\arrowvert_{z=z_0}=w_0,
$$
where $p$ is an elliptic symbol and $z\in \xR$ plays the role of a time variable. 

Given~$J\subset \xR$,~$z_0\in J$ and $\varphi=\varphi(x,z)$ defined on~${\mathbf{R}}^d\times J$, 
we denote by~$\varphi(z_0)$ the function~$x\mapsto \varphi(x,z_0)$. When~$a$ and~$u$ are
symbols and functions depending on~$z$, we still denote by~$T_a u$ the function defined by 
$(T_a u)(z)=T_{a(z)}u(z)$ where~$z\in J$ is seen as a parameter. 
$\Gamma^m_\rho({\mathbf{R}}^d\times J)$ denotes 
the space of symbols~$a=a(z;x,\xi)$ such that 
$z\mapsto a(z;\cdot)$ is bounded from~$J$ 
into the space $\Gamma^m_\rho({\mathbf{R}}^d)$ 
introduced in Definition~\ref{T:5}. 
This space is equipped with the semi-norm
\begin{equation}\label{seminorm-bis}
\mathcal{M}^m_\rho(a)=\sup_{z\in J}
\sup_{\la\alpha\ra\le \frac{3d}{2}+\rho+1 ~}\sup_{\la\xi\ra \ge 1/2~}
\lA (1+\la\xi\ra)^{\la\alpha\ra-m}\partial_\xi^\alpha a(z;\cdot,\xi)\rA_{W^{\rho,\infty}({\mathbf{R}}^d)}.
\end{equation}
The next proposition is a parabolic estimate in Zygmund spaces $C^{r}_*({\mathbf{R}}^d)$ with $r\in\xR$ 
(see Definition~\ref{T:Zygmund} for the definition of these spaces). Recall that 
$C^{r}_*({\mathbf{R}}^d)$ is the usual H\"older 
space~$W^{r,\infty}({\mathbf{R}}^d)$ if~$r>0$ is not an integer 
(since we shall need 
to consider various negative indexes, 
we shall often prefer to use the notation $C^{r}_*({\mathbf{R}}^d)$ instead of $W^{r,\infty}({\mathbf{R}}^d)$ 
even when $r>0$ is not an integer).

\begin{prop}\label{prop:maxH}
Let~$\rho\in (0,1)$,~$J=[z_0,z_1]\subset\xR$, 
$p\in \Gamma^{1}_{\rho}({\mathbf{R}}^d\times J)$ 
with the assumption that
$$
\RE p(z;x,\xi) \geq c \la\xi\ra,
$$
for some positive constant~$c$. Assume that ~$w$ solves
\begin{equation*}
\partial_z w + T_p w =F_1 + F_2,\quad w\arrowvert_{z=z_0}=w_0.
\end{equation*}
Then for any $q\in [1,+\infty]$, 
$(r_0,r)\in \xR^2$ with~$r_0<r$, if
$$
w\in L^\infty(J;C^{r_0}_*),~F_1\in L^1(J;C^{r}_*),~F_2 \in L^q (J; C^{r-1+\frac{1}{q}+\delta}_*) 
\text{ with }\delta >0,
$$
and~$w_0\in C^{r}_*({\mathbf{R}}^d)$, we have 
$w\in C^{0}(J;C^{r}_*)$ and 
\begin{equation*}
\lA w \rA_{C^0(J;C^{r}_*)}\le K\left\{\lA w_0\rA_{C^{r}_*}+ \lA F_1\rA_{L^{1}(J;C^{r}_*)}+
\lA F_2\rA_{L^q(J; C^{r-1+\frac{1}{q}+\delta}_*)}+\lA w\rA_{L^{\infty}(J;C^{r_0}_*)}\right\},
\end{equation*}
for some positive constant~$K$ depending only on~$r_0,r,\rho,c,\delta,q$ and~$\mathcal{M}^1_\rho(p)$.
\end{prop}
\begin{proof}We follow a classical strategy (see~\cite{Tougeron,LebeauKH,AM,ABZ1}). 

For this proof, we denote by~$K$ various 
constants which depend only on~$r_0,r,\rho,c$ and~$\mathcal{M}^1_\rho(p)$.
Given~$y\in J$ introduce the symbol~$e=e(y,z;x,\xi)$ defined by
$$
e(y,z;x,\xi)=\exp \bigl( - \int_z^{y} p(s;x,\xi)\, ds\bigr) \quad (z\in [z_0,y]).
$$
This symbol satisfies~$\partial_z e=e p$, so that
$$
\partial_z (T_e w) = (T_{e p }-T_e T_{p})w+T_e F, \quad F= F_1+ F_2.
$$
Integrating on ~$[z_0,y]$ the function~$\frac d {dz} T_{e( y,z,x,\xi)} w(z)$, 
we find
\begin{equation}\label{tv1ef}
T_1 w(y)=T_{e\arrowvert z=z_0} w_0 +\int_{z_0}^y  (T_{e}F)(z)\, dz+\int_{z_0}^y  
(T_{e p}-T_e T_{p})w(z)\, dz.
\end{equation}
(Notice that the paraproduct $T_1$ differs from the identity $I$ only by a smoothing operator.) 
Introduce~$G(y) =T_{e\arrowvert z=z_0} w_0 +\int_{z_0}^y  (T_{e}F)(z)\, dz$ and 
the operator~$R$ defined on functions~$u\colon J\rightarrow C^m_*({\mathbf{R}}^d)$ by 
$$
(R u)(y) = u(y)-T_1u(y)+\int_{z_0}^y  (T_{e p}-T_e T_{p})u (z)\, dz
$$
so that~$w=G +R  w$. 
Now, by a bootstrap argument, to complete 
the proof it is enough to prove that the function 
$G$ belongs to~$L^\infty(J;C^{r}_*)$ and that~$R$ 
is a smoothing operator of order~$-a$ for some~$a>0$, 
which means that 
$R$ maps~$L^\infty(J;C^{t}_*)$ to~$L^\infty(J;C^{t+a}_*)$. 
Indeed, by writing
$$
w=(I+R+\cdots R^N)G-R^{N+1}w,
$$
and choosing~$N$ large enough, we can estimate the second 
term in the right-hand side in~$L^\infty(J;C^{r}_*)$ by means of any 
$L^\infty(J;C^{r_0}_*)$-norm of~$w$. 

In the analysis, we need  to take into account 
how the semi-norms~$M^{-m}_\rho (e(z))$ (see Definition~\ref{T:5}) depend on~$z$. 
Then the key estimates are stated in the following lemma. 
\begin{lemm}\label{lem.2.19}For any~$m\ge 0$ there exists a positive 
constant $K$ depending only on $ \sup_{J} M^1_\rho(p(\cdot;x,\xi))$ such that, for all 
$y\in (0,-z_1]$ and all~$z\in [0,y)$, 
\begin{equation}\label{eez}
M^{-m}_\rho (e(z))\le \frac{K}{(y-z)^m}.
\end{equation}
\end{lemm}

This follows easily from the assumptions~$p \in \Gamma^1_\rho$, 
$
\RE p (s;x,\xi) \geq c\la\xi\ra,
$
and the elementary inequalities (valid for any~$a\ge 0$)
$$
(y- z)^a \la \xi\ra^a \exp\bigr((z-y)\la \xi\ra\bigr) \les 1.
$$

By using the bound~\eqref{eez}, applied with~$m=0$, it follows 
from the operator norm estimate \eqref{esti:quant2} that, for any~$z\le y$ and any function~$f=f(x)$, we have
\begin{equation}\label{tv1b1}
\lA T_{e(y,z)} f\rA_{\zygmund{r}}\les M^0_0(e(y,z)) \lA f\rA_{\zygmund{r}}\le 
K\lA f\rA_{\zygmund{r}}.
\end{equation}
This implies that
\begin{equation*}
\lA T_{e\arrowvert z=z_0} w_0 +\int_{z_0}^y  (T_{e}F_1)(z)\, dz\rA_{L^\infty(J;\zygmund{r})}
\le  K \lA w_0\rA_{\zygmund{r}}+K\lA F_1\rA_{L^1(J;\zygmund{r})}.
\end{equation*}
On the other hand, by using the bound~\eqref{eez}, applied with~$m=1-\frac{1}{q}-\delta$, 
we obtain that
$$
\lA \int_{z_0}^y  (T_{e}F_2)(z)\, dz\rA_{L^\infty(J;\zygmund{r})}
\leq K \int_{z_0}^{y} \frac{ 1}{ |y-z|^{m}}\|F_2(z) \|_{C^{r-m}_*} dz,
$$
which implies by H\"older inequality that 
\begin{equation*}
\lA G\rA_{L^\infty(J;\zygmund{r})}
\le 
K \lA w_0\rA_{\zygmund{r}}
+K\lA F_1\rA_{L^1(J;\zygmund{r})}
+\|F_2 \|_{L^q(J;C^{r-1+\frac{1}{q}+\delta}_*)} .
\end{equation*}

It remains to show that~$R$ is a smoothing operator. To do that, 
we first use the operator norm estimate \eqref{esti:quant2} (applied with 
$(m,m',\rho)$ replaced with~$(-m,1,\rho)$) to obtain
$$
\lA (T_{ ep}-T_e T_{p})(z) \rA _{\zygmund{t}\rightarrow \zygmund{t+m-1+\rho}}\les 
M^{-m}_{\rho}(e(z))M^1_{\rho}(p(z)).
$$
Taking $m=1-\rho/2$, 
it follows from the previous bound and Lemma~\ref{lem.2.19} that
$$
\lA (T_{ep}-T_e T_{p})v (z) \rA _{\zygmund{t+\rho/2}}\le 
\frac{K}{(y-z)^m} \lA v (z) \rA _{\zygmund{t}}.
$$
Since~$0\le m<1$ we have 
$\int_0^y (y-z)^{-m} \, dz <+\infty$ and hence
\begin{equation}\label{tv1b2}
\lA Ru(y)\rA_{\zygmund{t+\rho/2}}\le 
\int_0^y \lA (T_{ep}-T_e T_{p})u(z) \rA _{\zygmund{t+\rho/2}}\le K 
\lA u\rA_{L^\infty (J;\zygmund{t})},
\end{equation}
which completes the proof. \end{proof}

We shall also need the following estimate in Sobolev spaces. 
Given~$\mu\in\xR$, recall that we define the spaces
\begin{equation}\label{XY}\begin{aligned}
X^\mu(I)&=C^0_z(I;H^{\mu}({\mathbf{R}}^d))\cap L^2_z(I;H^{\mu+\mez}({\mathbf{R}}^d)),\\
Y^\mu(I)&=L^1_z(I;H^{\mu}({\mathbf{R}}^d))+L^2_z(I;H^{\mu-\mez}({\mathbf{R}}^d)).
\end{aligned}
\end{equation}

\begin{prop}[from \cite{ABZ3}]\label{prop:max}
Let~$r\in \xR$,~$\rho\in (0,1)$,~$J=[z_0,z_1]\subset\xR$ and let 
$p\in \Gamma^{1}_{\rho}({\mathbf{R}}^d\times J)$ 
satisfying$$
\RE p(z;x,\xi) \geq c \la\xi\ra,
$$
for some positive constant~$c$. Then for any ~$f\in Y^r(J)$ 
and~$w_0\in H^{r}({\mathbf{R}}^d)$, there exists~$w\in X^{r}(J)$ solution of  the parabolic evolution equation
\begin{equation}\label{eqW}
\partial_z w + T_p w =f,\quad w\arrowvert_{z=z_0}=w_0,
\end{equation}
satisfying 
\begin{equation*}
\lA w \rA_{X^{r}(J)}\le K\left\{\lA w_0\rA_{H^{r}}+ \lA f\rA_{Y^{r}(J)}\right\},
\end{equation*}
for some positive constant~$K$ depending only on~$r,\rho,c$ and~$\mathcal{M}^1_\rho(p)$.
 Furthermore, this solution is unique in ~$X^s(J)$ for any~$s\in \xR$. 
 \end{prop}

\section{Paralinearization}

We are now ready to prove Theorem~\ref{T3}. 
Recall that we consider the elliptic equation
\begin{equation}\label{dnint1Fb}
\partial_z^2 v +\alpha\Delta v + \beta \cdot\partialx\partial_z v  - \gamma \partial_z v=0,\quad v\arrowvert_{z=0}=f,
\end{equation}
where~$f=f(x)$ is a given function and the coefficients $\alpha,\beta,\gamma$ are given by \eqref{alpha} 
(these coefficients depend on the variable $\rho$ which is given by \eqref{diffeo}). 
In the sequel we fix indexes~$\delta,s,r,\eps$ in $\xR$ such that
\begin{equation}\label{indices}
0< \delta<\frac{1}{4},\quad s>1+\frac{d}{2}-\delta,\quad r>1,\quad \frac 1 4 < \eps=\mez-\delta<\min (\frac 1 2, s - \frac 1 2 - \frac d 2 ).
\end{equation}
It follows from~\eqref{oubli} and the Sobolev embedding that we have
$$
\lA \nabla_{x,z} v\rA_{C^0([-1,0];\zygmund{s-1-d/2})}
\le \mathcal{F}(\| \eta \|_{H^{s+\mez}})\lA f\rA_{H^s},
$$
for some non-decreasing function~$\mathcal{F}$. 
Since we only assume that~$s>3/4+d/2$, this is not enough to control 
the~$L^\infty$ norm of~$\nabla_{x,z}v$. The purpose of the next  result
is to provide such control under the additional assumption that~$f$ belongs to~$\zygmund{r}$ for some~$r>1$. 
\begin{prop}\label{p1H}
Let~$r>1$ and~$s>3/4+d/2$. 
For any~$-1<z_1< 0$, we have 
$$
\lA \nabla_{x,z} v\rA_{C^0\cap L^\infty (\mathbf{R}^d\times [z_1,0])}
\le \mathcal{F}\big(\| \eta \|_{H^{s+\mez}}\big)\left\{\lA f\rA_{H^s}+\lA f\rA_{\holder{r}}\right\},
$$
for some non-decreasing function~$\mathcal{F}\colon\xR_+\rightarrow \xR_+$.
\end{prop}
\begin{rema}$i)$ 
Since~$v\arrowvert_{z=0}=f\in L^\infty({\mathbf{R}}^d)$, we have also 
\begin{align*}
\lA v\rA_{L^\infty({\mathbf{R}}^d\times [z_1,0])}&\le \Vert f \Vert_{L^\infty({\mathbf{R}}^d)} +\la z_1\ra \lA \partial_z v\rA_{L^\infty({\mathbf{R}}^d\times [z_1,0])}\\
&\le \mathcal{F}(\| \eta \|_{H^{s+\mez}})\left\{\lA f\rA_{H^s}+\lA f\rA_{\zygmund{r}}\right\}.
\end{align*}

$ii)$ Let $\mu>1+d/2$. Since 
$$
\Big\Vert\frac{1+|\nabla\rho|^2}{\partial_z\rho}\Big\Vert_{L^\infty}\les 1+\lA \eta\rA_{H^\mu},\quad 
\lA\partial_z\rho\rA_{L^\infty}\les 1+\lA \eta\rA_{H^\mu},
$$
it follows from the previous proposition and 
the definition of the Dirichlet-Neumann operator that, for any $r>1$ and 
$s>3/4+d/2$, we have
\be\label{n232}
\lA G(\eta)f\rA_{L^\infty}\le \mathcal{F}\big(\lA \eta\rA_{H^{s+\mez}}\big)\big\{ \lA f\rA_{H^s}
+\lA f\rA_{\holder{r}}\big\}.
\ee
Other estimates are known which involve only H\"older norms (see Hu-Nicholls \cite{HN}), but they did not apply directly to our case with arbitrary bottoms. 
The fact that the previous bound involves a Sobolev 
estimate is harmless for our purposes.
\end{rema}
\begin{proof}Recall that the space $X^\mu(I)$ is defined by \eqref{XY}. 
Recall also that (see \eqref{oubli} and \eqref{w-1})
\begin{equation}\label{tv0}
\lA \nabla_{x,z} v\rA_{X^{s-1}([-1,0])}
\le \mathcal{F}(\| \eta \|_{H^{s+\mez}})\lA f\rA_{H^s}
\end{equation}
and
\begin{equation}\label{w-1-bis}
\lA \partial_z v-T_A v \rA_{X^{s-1+\eps}([-1,0])}
\le \mathcal{F}(\| \eta \|_{H^{s+\mez}})\lA f\rA_{H^s}.
\end{equation}

Since~$v\mid_{z=0} = f$, writing $v(z)=v(0)+\int_0^z \partial_z v$, this implies that
\begin{equation}\label{tv00}
\lA  v\rA_{X^{s-1}([-1,0])}
\le \mathcal{F}(\| \eta \|_{H^{s+\mez}})\lA f\rA_{H^s}.
\end{equation}

Introduce a cutoff function~$\chi$ such that 
$$
\chi(-1)=0,\quad \chi(z)=1 \quad\text{for } z\ge z_1,
$$
and set~$ w \defn \chi (z) (\partial_z - T_{A})v$. We shall use the fact that $w$ is already estimated by 
means of \eqref{w-1-bis}Ê
together with the parabolic estimate in H\"older spaces 
established above
to deduce an estimate for $v$.

Since it is convenient to work with forward evolution equation, 
define the function~$\widetilde{v}$ by~$\widetilde v(x,z)=v(x,-z)$, so that
\begin{equation*}
\partial_z \widetilde v  +T_{\widetilde{A}} \widetilde v =- \widetilde w\quad \text{for }z\in \widetilde{I}_1\defn [0,-z_1].
\end{equation*}
We split~$\widetilde v$ as~$\widetilde v=\widetilde v_1+\widetilde v _2$ where~$\widetilde{v}_1$ is the solution to the system
\begin{equation*}
\partial_z \widetilde v_1  +T_{\widetilde{A}} \widetilde v_1 =0\quad \text{for }z\in \widetilde{I}_1,  \qquad  \widetilde v_1\arrowvert_{z=0}=\widetilde v\arrowvert_{z=0}=f
\end{equation*} given by Proposition~\ref{prop:max}, 
while~$\widetilde v_2= \widetilde v - \widetilde v_1$ satisfies 
\begin{equation*}
\partial_z \widetilde v_2  +T_{\widetilde{A}} \widetilde v_2 =- \widetilde w 
\quad \text{for }z\in \widetilde{I}_1,\qquad  \widetilde v_2 \arrowvert_{z=0}=0.
\end{equation*}
According to~\eqref{w-1} we have
\begin{equation*}
\lA \widetilde w\rA_{Y^{s+\eps}(\widetilde I_1)}\le \lA \widetilde w\rA_{L^2(\widetilde I_1;H^{s+\eps-\mez})}
\leq \mathcal{F}(\| \eta \|_{H^{s+\mez}})\lA f\rA_{H^s},
\end{equation*}
which in turn implies, according to Proposition~\ref{prop:max},
\begin{equation}\label{tv000}
\lA \widetilde v_2\rA_{X^{s+\eps}(\widetilde I_1)}\le 
\mathcal{F}(\| \eta \|_{H^{s+\mez}})\lA f\rA_{H^s}.
\end{equation}
Set 
$m=s-1+\eps-\frac{d}{2}$. By the Sobolev embedding we have
$$
\lA \nabla_{x} \widetilde v_2\rA_{L^\infty(\widetilde I_1;\zygmund{m})}\les 
\lA \nabla_{x} \widetilde v_2\rA_{L^\infty(\widetilde I_1;H^{s-1+\eps})}\le 
\lA \widetilde v_2\rA_{X^{s+\eps}(\widetilde I_1)}
\le \mathcal{F}(\| \eta \|_{H^{s+\mez}})\lA f\rA_{H^s}.
$$
Let us prove that $\partial_z \widetilde v_2$ satisfies the same estimate. 
Using the equation for $\widetilde v_2$, we obtain 
$\partial_z \widetilde v_2=-T_{\widetilde{A}} \widetilde v_2 - \widetilde w$. 
Now, $\widetilde{w}$ is estimated by means of the bound \e{w-1}. Moving to the estimate 
of $T_{\widetilde{A}} \widetilde v_2 $, recall that 
$T_{\widetilde{A}}$ is an operator of order $1$ whose operator norm is estimated by means of the first 
inequality in \e{esti:quant1}, to get
\begin{align*}
\lA T_{\widetilde{A}} \widetilde v_2\rA_{L^\infty(\widetilde I_1;\zygmund{m})}&\les 
\lA T_{\widetilde{A}} \widetilde v_2\rA_{X^{s-1+\eps}(\widetilde I_1)}\\
&\le \mathcal{F}(\| \eta \|_{H^{s+\mez}})
\lA \widetilde v_2\rA_{X^{s+\eps}(\widetilde I_1)}\le \mathcal{F}(\| \eta \|_{H^{s+\mez}})\lA f\rA_{H^s}.
\end{align*}
We conclude 
$$
\lA \nabla_{x,z} \widetilde v_2\rA_{L^\infty(\widetilde I_1;\zygmund{m})}
\le \mathcal{F}(\| \eta \|_{H^{s+\mez}})\lA f\rA_{H^s}.
$$
Now, by assumption~$s>1+d/2-\delta$ with~$\delta<1/4$ and $\eps=1/2-\delta$, so that
$$
m=s-1+\eps-\frac{d}{2} = s-1+\mez -\delta -\frac{d}{2}> \mez-2 \delta>0,
$$
and hence
$$
\lA \nabla_{x,z} \widetilde v_2\rA_{L^\infty({\mathbf{R}}^d\times [0,-z_1])}
\le \mathcal{F}(\| \eta \|_{H^{s+\mez}})\lA f\rA_{H^s}.
$$

It remains to estimate $\widetilde{v}_1$. Using Proposition~\ref{prop:maxH}  with~$r=1, r_0 =- 1$, 
we obtain
\be\label{n236}
\|\widetilde{v}_1\|_{C^0(\widetilde{I}_1; \zygmund{r})}
\leq K \bigl(\| f\|_{\zygmund{r}}+  \|\widetilde{v}_1\|_{C^0(\widetilde{I}_1; C^{-1}_*)}\bigr).
\ee
To estimate the last term in the right-hand side above, write, according to~\eqref{tv00} and \eqref{tv000},
$$
\|\widetilde{v}_1\|_{C^0(\widetilde{I}_1; C^{-1}_*)}\leq C \bigl(\|\widetilde{v}\|_{C^0(\widetilde{I}_1; H^{s-1})}
+ \|\widetilde{v}_2\|_{C^0(\widetilde{I}_1; H^{s-1})}\bigr) \leq \mathcal{F}(\|\eta\|_{H^{s+ \mez}}) \| f \|_{H^s}.
$$ 
Since~$\partial_z \widetilde{v} _1= -T_{\widetilde{A}}\widetilde{v}_1$, and since 
$T_{\widetilde{A}}$ is an operator of order $1$ whose operator norm is estimated by means of the second inequality in 
\e{esti:quant1}, the previous inequality~\e{n236} implies also
$$
\|\nabla_{x,z} \widetilde{v}_1\|_{C^0(\widetilde{I}_1; C^{r-1}_*)}
\leq \mathcal{F}(\|\eta\|_{H^{s+ \mez}}) (\| f \|_{H^s} +\| f\|_{\zygmund{r}}).
$$
This completes the proof of Proposition~\ref{p1H}. 
\end{proof}

Gathering \eqref{oubli} and the previous estimate, for any $z_0$ in $(-1,0]$, 
we have
\begin{equation}\label{4.10'}
\begin{aligned}
&\lA \nabla_{x,z}v\rA_{C^0([z_0,0];H^{s-1})\cap L^2((z_0,0);H^{s-\mez})}
\le \mathcal{F}(\| \eta \|_{H^{s+\mez}})\lA f\rA_{H^s},\\
&\lA \nabla_{x,z} v\rA_{C^0\cap L^\infty({\mathbf{R}}^d\times [z_0,0])}
\le \mathcal{F}(\| \eta \|_{H^{s+\mez}})\left\{\lA f\rA_{H^s}+\lA f\rA_{\zygmund{r}}\right\}.
\end{aligned}
\end{equation}
The end of the proof of Theorem~\ref{T3} is in four steps. 
\paragraph{Step 1. Tame estimates in Zygmund spaces.} 

(See Definition~\ref{T:Zygmund} for the definition of Zygmund spaces.) 

Recall the following bounds 
for the coefficients~$\alpha,\beta,\gamma$ defined in~\eqref{alpha} (see \cite[Lemma~$3.25$]{ABZ3}):  for any 
$s> \mez+ \frac d 2$, we have
\begin{equation}\label{esti:abc}
\lA \alpha - h^2\rA_{X^{s-\mez}([-1,0])}+
\lA \beta\rA_{X^{s-\mez}([-1,0])}+\lA \gamma\rA_{X^{s-\tdm}([-1,0])}\le \mathcal{F}(\| \eta \|_{H^{s+\mez}}).
\end{equation}
We need also estimates in Zygmund spaces.

\begin{lemm}\label{L6} There holds
\begin{equation}\label{regv-2b}
\lA (\alpha,  \beta)\rA_{C^0([-1,0];\holdermez)}+\lA \gamma\rA_{L^2([-1,0];L^\infty)}\\
\le 
\mathcal{F}(\| \eta \|_{H^{s+1/2}}) \big\{1+\| \eta \|_{\holdertdm}\big\}.
\end{equation}
\end{lemm}
\begin{proof}Recall that, according to Lemma~\ref{rho:diffeo},
\begin{align}
&\lA\nabla_{x}\rho\rA_{C^0([-1,0];\holdermez)}+\lA\nabla_{x,z}\rho\rA_{L^2([-1,0];W^{1,\infty})}
\les \| \eta \|_{\holdertdm},\label{regofrho3-bis}\\
&\lA\partial_{z}\rho-h\rA_{C^0([-1,0];\holdermez)}
\les 1+\| \eta \|_{\holdertdm}.\notag
\end{align}
and (since $s-1/2>d/2$), by Sobolev embedding, 
\be\label{regofrho10}
\lA \nabla_{x,z}\rho \rA_{L^\infty(\mathbf{R}^d\times (-1,0))} \les 
1+  \lA \eta \rA_{H^{s+\mez}}.
\ee
We deduce the estimates for~$\alpha-1$ and~$\beta$ from 
the composition rule \eqref{esti:F(u)bis} and the equality $W^{1/2,\infty}=C^{1/2}_*$. 
The estimate for~$\gamma$ follows from~\eqref{regofrho10} and 
the estimate
\begin{equation*}
\lA\nabla_{x,z}^2\rho\rA_{L^2([-1,0];L^{\infty})}
\les \| \eta \|_{\holdertdm}. 
\end{equation*}
which follows from \e{regofrho3-bis} 
and the equation satisfied by $\rho$ (to estimate 
$\partial_z^2\rho$).
\end{proof}

\paragraph{Step 2. Estimates for the source terms.}

We now estimate the source terms $F_1$, $F_2$ and $F_3$ which appear in \eqref{F2e} and \eqref{2m2.b}.
\begin{lemm}\label{L4.5}
For any~$z_0 \in (-1,0)$, and any~$j=1,2,3$ we have, 
\begin{equation}\label{F123}
\lA F_j\rA_{L^2_z((z_0, 0);H^{s-1})}\le \mathcal{F}\bigl( \| \eta \|_{H^{s+\mez}},\lA f\rA_{H^{s}}\bigr)
\left\{1+\| \eta \|_{\holdertdm}+\lA f\rA_{\holder{r}}\right\}.
\end{equation}
\end{lemm}
\begin{proof}
By using the tame product rule (see~\eqref{prS2})
\begin{equation*}
\lA u_1 u_2 \rA_{H^{s-1}}\les  \lA u_1\rA_{L^\infty}\lA u_2\rA_{H^{s-1}}
+\lA u_2\rA_{L^{\infty}} \lA u_1\rA_{H^{s-1}}
\end{equation*}
we find that~$F_1=\gamma \partial_z v$ satisfies
\begin{equation*}
\begin{aligned}
\lA F_1\rA_{L^2((z_0, 0);H^{s-1})}
&\les  \lA \partial_z v\rA_{C^0([z_0,0];L^\infty)}\lA \gamma \rA_{L^2((z_0,0);H^{s-1})}\\
&\quad + 
\lA \gamma\rA_{L^2((z_0,0);L^\infty)}
\lA \partial_z v\rA_{L^\infty((z_0,0);H^{s-1})}.
\end{aligned}
\end{equation*}
The desired estimate for~$F_1$ follows from Lemma~\ref{L6}, 
\eqref{4.10'} and \eqref{esti:abc}.

Let us now study
\begin{equation*}
F_2 =(T_\alpha-\alpha)\Delta v+(T_\beta-\beta)\cdot\partialx\partial_z v 
=-\bigl( T_{\Delta v}\alpha +R(\alpha,\Delta v)+T_{\partialx\partial_z v}\cdot \beta+R(\beta,\partialx\partial_z v)\bigr).
\end{equation*}
According to~\eqref{niS}, we obtain
\begin{align*}
\lA T_{\Delta v(z)}\alpha(z)\rA_{H^{s-1}}&\les \lA \Delta v(z)\rA_{C^{-1}_*}\lA \alpha(z)\rA_{H^{s}},\\
\lA T_{\partialx\partial_z v(z)}\cdot \beta(z)\rA_{H^{s-1}}
&\les \lA \partialx\partial_z v(z)\rA_{C^{-1}_*}\lA \beta(z)\rA_{H^{s}}.
\end{align*}
On the other hand, since~$s-1>0$ we can apply~\eqref{Bony3} to obtain
\begin{align*}
\lA R(\alpha,\Delta v)(z)\rA_{H^{s-1}}&\les \lA \Delta v(z)\rA_{C^{-1}_*}\lA \alpha(z)\rA_{H^{s}},\\
\lA R(\beta,\partialx\partial_z v)(z)\rA_{H^{s-1}}
&\les \lA \partialx\partial_z v(z)\rA_{C^{-1}_*}\lA \beta(z)\rA_{H^{s}}.
\end{align*}
Consequently we have proved
\begin{multline}
\lA F_2\rA_{L^2([z_0,0]; H^{s-1})}\les 
\lA \Delta v\rA_{C^0([z_0,0]; C^{-1}_*)}\lA \alpha\rA_{L^2([z_0,0];H^{s})}\\
+\lA \partialx\partial_z v\rA_{C^0([z_0,0];C^{-1}_*)}\lA \beta\rA_{L^2([z_0,0];H^{s})}.
\end{multline}
Notice that 
$$
\lA \Delta v\rA_{C^{-1}_*}\les 
\lA \nabla v \rA_{C^0_*} \les \lA \nabla v\rA_{L^\infty},
\qquad  \lA \nabla \partial_z v\rA_{C^{-1}_*}\les  \lA \partial_z v \rA_{C^0_*}\les \lA \partial_z v\rA_{L^\infty} 
$$
and consequently, according to~\eqref{4.10'} and \eqref{esti:abc} 
we conclude the proof of the claim \eqref{F123} for~$j=2$. 

It remains to estimate~$F_3$. 
In light of \eqref{4.10'} it is enough to prove that
\begin{equation}\label{F3-1}
\lA F_3\rA_{L^2(I;H^{s-1})}\le \mathcal{F}\bigl(\| \eta \|_{H^{s+\mez}}\bigr)\| \eta \|_{\holdertdm}\lA 
\nabla_{x,z} v\rA_{L^2(I;H^{s-\mez})},
\end{equation}
for some non-decreasing function. 
Directly from the definition of~$\alpha$ and~$\beta$, 
by using the tame estimates in H\"older spaces~\eqref{prZ1}, we verify 
that the symbols~$a,A$ (given by \e{aA}) 
belong to~$\Gamma^{1}_{1/2}({\mathbf{R}}^d\times I)$ 
and that they satisfy 
\begin{equation}\label{boundsaAbis}
\mathcal{M}^1_{1/2}(a)+ \mathcal{M}^1_{1/2}(A)\le 
\mathcal{F}\bigl(\| \eta \|_{H^{s+\mez}}\bigr)\| \eta \|_{\holdertdm}.
\end{equation}
(For later purpose, notice that we used here only~$s> \mez + \frac d 2$.) 
Moreover, 
\begin{equation*}
 \mathcal{M}^1_{-1/2}(\partial_z A)\le 
\mathcal{F}\bigl(\| \eta \|_{H^{s+\mez}}\bigr)\| \eta \|_{\holdertdm}.
\end{equation*}
By definition \e{aA}, $a+A=-i\beta\cdot\xi$ so $T_a+T_A=-T_{\beta}\cdot\nabla$. It follows that
$$
F_3=\big(T_aT_A-T_\alpha\Delta\big)v-T_{\partial_z A}v.
$$
Set
\begin{equation*}
R_0(z)\defn  T_{a(z)} T_{A(z)}   - T_\alpha\Delta  ,\qquad R_1(z)\defn -T_{\partial_z A}.
\end{equation*}
Since $aA=-\alpha |\xi|^2$, we deduce, using Theorem~\ref{theo:sc0} (see (ii)
applied with~$\rho=1/2$), that for any $\mu\in \xR$,
\begin{equation}\label{eq.3.96bis}
\sup_{z\in [-1,0]}\lA R_0(z)\rA_{H^{\mu+\tdm}\rightarrow H^{\mu}}
\le 
\mathcal{F}\bigl(\| \eta \|_{H^{s+\mez}}\bigr)\| \eta \|_{\holdertdm}.
\end{equation}
On the other hand, 
Proposition~\ref{prop.niSbis} (applied with~$\rho= -1/2$) implies that
\begin{equation}\label{eq.3.96ter}
\sup_{z\in [-1,0]}\lA R_1(z)\rA_{H^{\mu+\tdm}\rightarrow H^{\mu}}
\le 
\mathcal{F}\bigl(\| \eta \|_{H^{s+\mez}}\bigr)\| \eta \|_{\holdertdm}.
\end{equation}
Using these inequalities for $\mu=s-1$, we obtain the desired result \eqref{F3-1}. This completes the proof of Lemma~\ref{L4.5}.
\end{proof}

\paragraph{Step 3 : elliptic estimates} Introduce a 
cutoff function~$\kappa=\kappa(z), z\in [-1,0]$, such that~$\kappa(z)=1$ near~$z=0$ and such that 
$\kappa(z_1)=0$ (recall that~$I_1=[z_1,0]$ for some $z_1\in (-1,0)$). Set 
\begin{equation}\label{nomdeW}
W \defn \kappa (z) (\partial_z - T_{A})v.
\end{equation} 
Now it follows from the paradifferential equation~\eqref{2m2.b} for~$v$ that
\begin{equation*}
\partial_z W  -T_{a} W = F',
\end{equation*}
where
$$
F'=\kappa(z)(F_1+F_2+F_3)+\kappa'(z)(\partial_z - T_{A})v.
$$
Our goal is to prove that 
\begin{equation}\label{estiW}
\lA W\rA_{L^\infty(I_1;H^{s-\mez})}\le \mathcal{F}\bigl( \| \eta \|_{H^{s+\mez}},\lA f\rA_{H^{s}}\bigr)
\left\{1+\| \eta \|_{\holdertdm}+\lA f\rA_{\holder{r}}\right\}.
\end{equation}
We have already proved that
$$
\lA F_1+F_2+F_3\rA_{L^2_z(I_1;H^{s-1})}\le \mathcal{F}\bigl( \| \eta \|_{H^{s+\mez}},\lA f\rA_{H^{s}}\bigr)
\left\{1+\| \eta \|_{\holdertdm}+\lA f\rA_{\holder{r}}\right\}.
$$
We now turn to an estimate for~$(\partial_z-T_A)v$. 
To do that we estimate separately~$\partial_z v$ and~$T_A v$. 
Clearly, by definition of the space~$X^{s-1}$, noting that~$I_1\subset I_0=[-1,0]$, we have
$$
\lA \partial_z v\rA_{L^2(I_1;H^{s-\mez})}\le \lA \partial_z v\rA_{L^2(I_0;H^{s-\mez})}\le
\lA \nabla_{x,z}v\rA_{X^{s-1}(I_0)}.
$$
On the other hand, as in the previous step, 
since 
$\mathcal{M}^1_0(A)\le C\big(\| \eta \|_{H^{s+\mez}}\big)$, 
we have
\begin{align*}
\lA T_{A}v\rA_{L^2(I_1;H^{s-\mez})}&\le \mathcal{F}(\| \eta \|_{H^{s+\mez}}) 
\lA \nabla_x v\rA_{L^2(I_1;H^{s-\mez})}\\
&\le \mathcal{F}(\| \eta \|_{H^{s+\mez}})\lA \nabla_{x,z}v\rA_{X^{s-1}(I_0)}.
\end{align*}
Now recall from \eqref{oubli} that 
$\lA \nabla_{x,z} v\rA_{X^{s-1}(I_0)}\le \mathcal{F}(\| \eta \|_{H^{s+\mez}})\lA f\rA_{H^s}$. 
Therefore, 
$$
\lA (\partial_z-T_A)v \rA_{L^2(I_0;H^{s-\mez})} \le \mathcal{F}(\| \eta \|_{H^{s+\mez}})\lA f\rA_{H^s},
$$
and we end up with
$$
\lA F'\rA_{L^2(I_1;H^{s-1})}\le \mathcal{F}\bigl( \| \eta \|_{H^{s+\mez}},\lA f\rA_{H^{s}}\bigr)
\left\{1+\| \eta \|_{\holdertdm}+\lA f\rA_{\holder{r}}\right\}.
$$
Since $\partial_z W  -T_{a} W = F'$ and $W(x,z_1)=0$ (by definition of the cutoff function~$\kappa$) 
and since~$a\in \Gamma^1_{\eps}$ 
satisfies~$\RE (-a(x,\xi))\ge c \la \xi\ra$, by using Proposition~\ref{prop:max} applied with~$J=I_1$, 
$\rho=\eps$ and~$r=s-1/2$, we have 
\begin{equation*}
\lA W\rA_{X^{s-\mez}(I_1)}\le \mathcal{F}(\| \eta \|_{H^{s+\mez}})
\lA F'\rA_{Y^{s-\mez}(I_1)}.
\end{equation*}
Now, by definition, 
$\lA F'\rA_{Y^{s-\mez}(I_1)}\le \lA F'\rA_{L^2(I_1;H^{s-1})}$, so 
we conclude that~$W$ satisfies the desired estimate~\eqref{estiW}.

\paragraph{Step 4 : paralinearization of the Dirichlet-Neumann.} 
We shall only use the following obvious 
consequence of~\eqref{estiW}: 
$\lA W\arrowvert_{z=0}\rA_{H^{s-\mez}}$ 
is estimated by the right-hand side of ~\eqref{estiW} (we can take the trace on $z=0$ 
since $W$ belongs to $X^{s-\mez}\subset 
C^0_z(H^{s-\mez})$ and not only to 
$L^\infty_z(H^{s-\mez})$, as follows from Proposition~\ref{prop:max}). 
Since~$W\arrowvert_{z=0}=\partial_z v-T_{A}v \arrowvert_{z=0}$, 
we thus have proved that
\begin{equation}\label{prop:total}
\lA \partial_z v-T_{A}v \arrowvert_{z=0} \rA_{H^{s-\mez}}\le 
\mathcal{F}\bigl( \| \eta \|_{H^{s+\mez}},\lA f\rA_{H^{s}}\bigr)
\left\{1+\| \eta \|_{\holdertdm}+\lA f\rA_{\holder{r}}\right\}.
\end{equation}

Now, recall that
\begin{equation*}
G(\eta) f= \zeta_1 \partial_z v  -  \zeta_2 \cdot\partialx v
\big\arrowvert_{z=0}
\end{equation*}
with
$$
\zeta_1 \defn  \frac{1 +|\partialx \rho |^2}{\partial_z \rho},\quad \zeta_2 \defn  \partialx \rho .
$$
As for the coefficients~$\alpha,\beta$ (see Lemma~\ref{L6}), we have
\begin{align}
&\lA \zeta_1-\frac 1 {h}\rA_{L^\infty([-1,0];\holdermez)}+
\lA \zeta_2\rA_{L^\infty([-1,0];\holdermez)}\le \mathcal{F}(\| \eta \|_{H^{s+\mez}}) 
\| \eta \|_{\holdertdm},
\label{zetaC12}\\
&\lA \zeta_1- \frac 1 {h}\rA_{L^\infty([-1,0];H^{s-1/2})}+
\lA \zeta_2\rA_{L^\infty([-1,0];H^{s-1/2})}\le \mathcal{F}(\| \eta \|_{H^{s+\mez}}).\label{zetaHs12}
\end{align}
Write
\begin{align*}
\zeta_1\partial_z v
-\zeta_2 \cdot\partialx v = T_{\zeta_1}  \partial_z v  
- T_{\zeta_2}\partialx v+R',
\end{align*}
with
\be\label{n50RHS}
R'=T_{\partial_z v} \zeta_1 -T_{\nabla v}\cdot \zeta_2 + R(\zeta_1,\partial_z v)- R(\zeta_2,\nabla v).
\ee
Since a paraproduct by an $L^\infty$ function acts on any Sobolev spaces, 
according to Proposition~\ref{p1H} and~\eqref{zetaHs12}, we obtain 
$$
\lA T_{\partial_z v} \zeta_1 -T_{\nabla v}\cdot \zeta_2 \rA_{L^\infty(I_1;H^{s-\mez})}
\le \mathcal{F}\bigl( \| \eta \|_{H^{s+\mez}},\lA f\rA_{H^{s}}\bigr)
\left\{1+\lA f\rA_{\holder{r}}\right\}.
$$
We estimate similarly the last two terms in the right-hand side of \e{n50RHS}, so
$$
\lA R'\rA_{L^\infty(I_1;H^{s-\mez})}\le 
\mathcal{F}\bigl( \| \eta \|_{H^{s+\mez}},\lA f\rA_{H^{s}}\bigr)
\left\{1+\lA f\rA_{\holder{r}}\right\}.
$$
Furthermore, \eqref{prop:total} implies that
\begin{equation*}
T_{\zeta_1}  \partial_z v  
- T_{\zeta_2}\partialx v\big\arrowvert_{z=0}
= T_{\zeta_1}   T_A  v -T_{i\zeta_2\cdot \xi}  v \big\arrowvert_{z=0} +R'',
\end{equation*}
where~$\lA R''\rA_{H^{s-\mez}}$ satisfies 
$$
\lA R''\rA_{H^{s-\mez}}\le \mathcal{F}\bigl( \| \eta \|_{H^{s+\mez}},\lA f\rA_{H^{s}}\bigr)
\left\{1+\| \eta \|_{\holdertdm}+\lA f\rA_{\holder{r}}\right\}.
$$

Thanks to \eqref{esti:quant2} we have
\begin{align*}
\lA 
T_{\zeta_1}   T_A -T_{\zeta_1 A}\rA_{ H^s\rightarrow H^{s-\mez}}
&\les \lA \zeta_1\rA_{L^\infty} \mathcal{M}^1_{1/2}(A) +\lA \zeta_1\rA_{\holdermez} 
\mathcal{M}^1_{0}(A) \\
&\le \mathcal{F}(\| \eta \|_{H^{s+\mez}}) \| \eta \|_{\holdertdm},
\end{align*}
where we used \e{boundsaAbis} and $\mathcal{M}^1_0(A)\le C\big(\| \eta \|_{H^{s+\mez}}\big)$. Therefore,
$$
G(\eta)f=T_{\zeta_1 A}  v -T_{i\zeta_2\cdot \xi}  v \big\arrowvert_{z=0} +R(\eta)f
$$
where 
$$
\lA R(\eta)f \rA_{H^{s-\mez}}
\le \mathcal{F}\bigl( \| \eta \|_{H^{s+\mez}},\lA f\rA_{H^{s}}\bigr)
\left\{1+\| \eta \|_{\holdertdm}+\lA f\rA_{\holder{r}}\right\}.
$$
Now by definition of $A$ (see \e{aA}) one has
$$
\zeta_1 A  v -i\zeta_2\cdot \xi =\sqrt{(1+|\nabla\rho|^2)|\xi|^2-(\nabla\rho\cdot\xi)^2}
$$
so $T_{\zeta_1 A}  v -T_{i\zeta_2\cdot \xi}  v \big\arrowvert_{z=0} =T_{\lambda} f $ since, by definition of $\lambda$, 
$$
\lambda=\sqrt{(1+|\nabla\eta|^2)|\xi|^2-(\nabla\eta\cdot\xi)^2}.
$$ 
This proves 
that~$G(\eta)f=T_{\lambda} f +R(\eta)f$ which concludes the  proof of Theorem~\ref{T3}. 

\chapter{Sobolev estimates}\label{C:3}

In this chapter, we prove sharp {\em a priori } estimates in Sobolev spaces.

\section{Introduction}

We begin by recalling a formulation of the water waves system 
which involves the unknowns
\begin{equation}\label{defi:unknowns}
\zeta = \partialx \eta, \quad 
\B=\partial_y\phi \eval ,  \quad 
V=\nabla_x \phi\eval , \quad \ma=-\partial_y P \eval ,
\end{equation}
where recall that~$\phi$ is the velocity potential,  $P=P(t,x,y)$ is the pressure 
given by
\begin{equation}\label{defi:P}
-P=\partial_t \phi+\mez \la \nabla_{x,y}\phi\ra^2+gy,
\end{equation}
and $\ma$ is the Taylor coefficient.
We consider smooth solutions  
$(\eta,\psi)$ of \eqref{n10} defined on the time interval~$[0,T_0]$ and 
satisfying the following assumptions on that time interval. 
\begin{assu}\label{A:2}
 We consider smooth solutions of the water waves equations such that
\begin{enumerate}[i)]
\item $(\eta,\psi)$ belongs to $C^1([0,T_0]; H^{s_0}(\xR^d)\times H^{s_0}(\xR^d))$ for some 
$s_0$ large enough and $0 <T_0$;
\item there exists $h>0$ such that \eqref{n1} holds for any $t$ in $[0,T_0]$ (this is the assumption 
that there exists a curved strip of width $h$ separating the free surface from the bottom);
\item there exists $c>0$ such that the Taylor coefficient $a(t,x)=-\partial_y P\arrowvert_{y=\eta(t,x)}$ is bounded from below by $c$ 
for any $(t,x)$ in $[0,T_0]\times \xR^d$.
\end{enumerate}
\end{assu}

We begin by recalling two results from \cite{ABZ3}.
\begin{prop}[from \cite{ABZ3}]\label{prop:GBV}
For any ~$s> \mez + \frac d 2$ one has
\be\label{n32}
G(\eta){B}=-\cnx V +\widetilde{\gamma}
\ee
where 
\be\label{n340a}
\|\widetilde{\gamma}\|_{H^{s- \mez}} \leq \mathcal{F}( \|( \eta,  V, B)\|_{H^{s+ \mez} \times H^{\mez}\times H^\mez}).
\ee
\end{prop}

\begin{prop}[from \cite{ABZ3}]\label{prop:newS}We have 
\begin{align}
(\partial_{t}+V\cdot\partialx)\B&=\ma-g,\label{eq:B}\\
(\partial_t+V\cdot\partialx)V+\ma\zeta&=0,\label{eq:V}\\
(\partial_{t}+V\cdot\partialx)\zeta&=G(\eta)V+ \zeta G(\eta)\B +\gamma,\label{eq:zeta}
\end{align}
where the remainder term~$\gamma=\gamma(\eta,\psi,V)$ satisfies the following estimate : 
if~$s>\mez +\frac{d}{2}$ then 
\begin{equation}\label{p42.1}
\lA \gamma\rA_{H^{s-\mez}}\le \mathcal{F}(\lA (\eta,\psi,V)\rA_{H^{s+\mez}\times H^s\times H^s}).
\end{equation}
\end{prop}
\begin{rema}With $\gamma$ and $\widetilde{\gamma}$ as above, 
there holds $\gamma=-\zeta \widetilde{\gamma}$. 
In particular, it follows from \e{n32} and \e{eq:zeta} that
\begin{equation}\label{n38}
(\partial_{t}+V\cdot\partialx)\zeta
=G(\eta)V-(\cnx V)\zeta.
\end{equation}
Moreover, in the case without bottom ($\Gamma = \emptyset$), 
one can see that $\widetilde{\gamma}=0$ and hence $\gamma=0$. 
\end{rema}

Our goal is to estimate, for $T$ in $(0,T_0]$, the norm
$$
M_s(T)\defn 
\lA (\psi,\eta,B,V)\rA_{C^0([0,T];H^{s+\mez}\times H^{s+\mez}\times H^s\times H^s)},
$$
in terms of the norm of the initial data
$$
M_{s,0}\defn \lA (\psi(0),\eta(0),B(0),V(0))\rA_{H^{s+\mez}\times H^{s+\mez}\times H^s\times H^s}
$$
and in terms of a quantity which involves H\"older norms, that will be later estimated by means of a Strichartz estimate, defined by
$$
Z_r(T)\defn \lA \eta\rA_{L^p([0,T];W^{r+\mez,\infty})}
+\lA (B,V)\rA_{L^p([0,T];W^{r,\infty}\times W^{r,\infty})},
$$
where $p=4$ if $d=1$ and $p=2$ for $d\ge 2$.

Our goal in this chapter is to prove the following result.
\begin{theo}\label{T1b}
Let~$T_0>0$, $d\ge 1$ and consider~$s,r\in ]1,+\infty[$ such that
\begin{equation*}
s>\frac{3}{4}+\frac{d}{2},\qquad s+\frac{1}{4}-\frac{d}{2} > r>1.
\end{equation*}
There exists 
a non-decreasing 
function~$\mathcal{F}\colon \xR^+\rightarrow\xR^+$ such that, 
for all smooth solution 
$(\eta,\psi)$ of \eqref{n10} defined on the time interval~$[0,T_0]$ and satisfying 
Assumption~\ref{A:2} on that time interval, for any $T$ in $[0,T_0]$, there holds
\begin{equation}\label{apriori}
M_s(T)\le \mathcal{F}\bigl(\mathcal{F}(M_{s,0})+T\mathcal{F}\bigl(M_s(T)+Z_r(T)\bigr)\bigr).
\end{equation}
\end{theo}

If~$s>1+d/2$ then one can apply the previous inequality with $r=s-d/2$. Then 
$Z_r(T)\les M_s(T)$ and one deduces from \eqref{apriori} an 
estimate which involves only $M_s(T)$. 
Thus we recover the {\em a priori\/} estimate in Sobolev spaces proved in \cite{ABZ3} for $s>1+d/2$. 
The proof of Theorem~\ref{T1b} follows closely the proof 
of \cite[Prop.\ $4.1$]{ABZ3}. For the reader convenience we shall recall the scheme of the analysis, but 
we shall only prove the points which must be adapted. 
The main new difficulty is to prove sharp H\"older estimates for the Taylor coefficient.

Hereafter,~$\mathcal{F}$ always refers to a non-decreasing function~$\mathcal{F}\colon \xR^+\rightarrow \xR^+$ depending only 
on~$s$,~$r$,~$d$ and~$h,c_0$ (being of course independent of the time~$T$ and 
the unknowns).

\section{Estimates for the Taylor coefficient}\label{S:tameP}
Here we prove several estimates for the Taylor 
coefficient.
\begin{prop}\label{prop:ma}
Let~$d\ge 1$ and consider~$s,r$ such that
\begin{equation*}
s>\frac{3}{4}+\frac{d}{2},\quad r>1.
\end{equation*}
For any $0<\eps <\min(r-1, s- \frac 3 4 - \frac d 2)$, there exists a non-decreasing function~$\mathcal{F}$ such that, for all~$t\in [0,T]$,
\begin{equation}
\lA \ma(t)-g\rA_{H^{s-\mez}}\le 
\mathcal{F}\Bigl( \lA  (\eta, \psi,V,B)(t)\rA_{H^{s+\mez}\times H^{s+\mez}\times 
H^s\times H^s}\Bigr),\label{esti:a1}
\end{equation}
and
\begin{multline}
\lA  \ma(t) \rA_{\holder{\mez+\eps}}+\lA (\partial_t \ma +V\cdot \partialx \ma)(t)\rA_{\holder{\eps}}\\
\le \mathcal{F}\bigl( \lA (\eta,\psi)(t)\rA_{H^{s+\mez}},\lA (V,B)(t)\rA_{H^s}\bigr) 
\left\{1+\lA \eta(t)\rA_{\holder{r+\mez}}+\lA (V,B)(t)\rA_{\holder{r}}\right\}.
\end{multline}
\end{prop}
The estimate of~$\lA \ma - g\rA_{H^{s-\mez}}$ follows directly from the arguments in \cite{ABZ3}. 
The estimate of the~$\holder{\eps}$-norm of 
$\partial_t \ma +V\cdot \nabla \ma~$ is the easiest one. 
The main new difficulty here is to prove a tame estimate for~$\lA  \ma \rA_{\holder{\mez+\eps}}$. 
Indeed, there are several further complications which appear in the analysis in H\"older spaces. 

Hereafter, since the time variable is fixed, we shall skip it. 
To prove the above estimates on~$\ma$, we form an elliptic 
equation for $P$. As in Chapter~\ref{C:2}Ê
we flatten the free surface by 
using the change of variables~$(x,z)\mapsto (x,\rho(x,z))$ (see \eqref{diffeo} and 
Lemma~\ref{rho:diffeo}). Set
$$
v(x,z)=\phi(x,\rho(x,z)),\quad \wp(x,z)=P(x,\rho(x,z))+g\rho(x,z),
$$
and notice that
$$
\ma-g = - \frac 1 { \partial_z \rho} \partial_z  \wp\mid_{z=0}.
$$
The first elementary step is to compute the equation satisfied
by the new unknown as well as the boundary conditions
on~$\{z=0\}$. 
As in \cite{ABZ3}, one computes that
\begin{equation}
\begin{aligned}
&\partial_z^2 \wp +\alpha\Delta \wp + \beta \cdot\partialx\partial_z \wp  - \gamma \partial_z \wp=F_0(x,z)&& \text{for }z<0,\\
&\wp=g\eta && \text{on }z=0,
\end{aligned}
\end{equation}
where~$\alpha,\beta,\gamma$ are as above (see~\eqref{alpha}) and where
\begin{equation}\label{defi:F0Lambda}
F_0=-\alpha \la \Lambda^2 v\ra^2, \quad 
\Lambda=(\Lambda_1,\Lambda_2),\quad 
\Lambda_1=\frac{1}{\partial_z\rho}\partial_z, \quad
\Lambda_2 = \nabla -\frac{\partialx \rho}{\partial_z \rho}\partial_z.
\end{equation}
Our first task is to estimate the source term~$F_0$. 

\begin{lemm}\label{es:F0p}
Let~$d\ge 1$ and consider~$s\in ]1,+\infty[$ such that
$$
s>\frac{3}{4}+\frac{d}{2}.
$$
Then there exists~$z_0 <0$ such that 
\begin{equation*}
\lA F_0\rA_{L^1([z_0, 0];H^{s-\mez})}
+\lA F_0\rA_{L^2([z_0, 0]; \zygmund{s+ \frac 1 4 - \frac d 2})}
\le \mathcal{F}\bigl(\lA (\eta,\psi,V,B)\rA_{H^{s+\mez}\times H^{s+\mez}\times H^s\times H^s}\bigr).
\end{equation*}
\end{lemm}
\begin{proof}
The first part of this result follows from the proof of \cite[Lemma~$4.7$]{ABZ3} (although 
this lemma is proved under the assumption that $s>1+d/2$, its proof 
shows that the results \e{D2varbis}--\e{D2var} we quote below 
hold for any $s>1/2+d/2$). We proved in \cite{ABZ3} that
\be\label{D2varbis}
\lA F_0\rA_{L^1([z_0,0];H^{s-\mez})}
\le \mathcal{F}\bigl(\lA (\eta,\psi,V,B)\rA_{H^{s+\mez}\times H^{s+\mez}\times H^s\times H^s}\bigr),
\ee
together with
\begin{multline}\label{D2var}
\lA \Lambda_j \Lambda_kv\rA_{C^0([z_0, 0]; H^{s-1})}
+\lA \Lambda_j \Lambda_kv\rA_{L^2([z_0,0];H^{s-\mez})}\\
\le \mathcal{F}\bigl(\lA (\eta,\psi,V,B)\rA_{H^{s+\mez}\times H^{s+\mez}\times H^s\times H^s}\bigr).
\end{multline}
By interpolation, \eqref{D2var} also implies that
$$
\lA \Lambda_j \Lambda_kv\rA_{L^4([z_0,0];H^{s-\frac{3}{4}})}
\le \mathcal{F}\bigl(\lA (\eta,\psi,V,B)\rA_{H^{s+\mez}\times H^{s+\mez}\times H^s\times H^s}\bigr).
$$
Since~$s>3/4+d/2$ by assumption, the Sobolev space~$H^{s-\frac{3}{4}}(\xR^d)$ 
is an algebra and hence, according to~\e{esti:abc},
$$
\lA \alpha \la \Lambda_j \Lambda_kv\ra^2 \rA_{L^2([z_0,0];H^{s-\frac{3}{4}})}\les 
\big(1+\big\| \alpha-h^2\big\|_{L^\infty([z_0,0];H^{s-\frac{3}{4}})}\big)
\lA  \Lambda_j \Lambda_kv \rA_{L^4([z_0,0];H^{s-\frac{3}{4}})}^2.
$$
The Sobolev embedding 
$H^{s-\tq}\subset C^{s- \tq - \frac d 2}_*$ then yields 
\begin{equation}\label{esti:F0Sob}
\| F_0 \|_{L^2([z_0, 0]; C^{s- \frac 3 4 - \frac d 2}_*)}
\leq \mathcal{F}\bigl(\lA (\eta,\psi,V,B)\rA_{H^{s+\mez}\times H^{s+\mez}\times H^s\times H^s}\bigr).
\end{equation}
This completes the proof.
\end{proof}

The proof of the fact that 
~$\lA \ma-g\rA_{H^{s-\mez}}$ is bounded by a constant 
depending only on 
$\lA(\eta, \psi)\rA_{H^{s+\mez}}$ and~$\lA (V,B)\rA_{H^s}$ then follows from elliptic regularity 
which implies that
\begin{equation}\label{inwpCSob}
\lA \nabla_{x,z} \wp\rA_{X^{s-\mez}([z_0,0])}
\le \mathcal{F}\bigl(\lA (\eta,\psi,V,B)\rA_{H^{s+\mez}\times H^{s+\mez}\times H^s\times H^s}\bigr),
\end{equation} 
as can be proved using the arguments given above the statement of Proposition~$4.8$ in 
\cite{ABZ3}. 
We shall now prove the estimate of~$\lA \ma\rA_{\holder{1/2+\epsilon}}$. 
This estimate will follow directly from 
the following result.
\begin{prop}\label{p1:H}
Let~$d\ge 1$ and consider $(s,r,r')\in \xR^3$ such that
$$
s>\tq+\frac{d}{2},\quad s+\uq-\frac{d}{2} > r>r'>1.
$$
Then there exists~$z_0<0$ such that 
\begin{equation}\label{est:pH32}
\lA \nabla_{x,z} \wp\rA_{C^0([z_0,0];\holder{r'-\mez})}
\le \mathcal{F}\bigl(\lA (\eta,\psi,V,B)\rA_{H^{s+\mez}\times H^{s+\mez}\times H^s\times H^s}\bigr)
\left\{1+\| \eta \|_{\holder{r+\mez}}\right\}
\end{equation}
for some non-decreasing function~$\mathcal{F}$ depending only on~$s,r,r'$.
\end{prop}
\begin{proof}
It follows from \eqref{inwpCSob} and 
the Sobolev embedding $H^{s-\mez}(\xR^d)\subset \holder{r-\tq}(\xR^d)$ that 
\begin{equation}\label{inwpCrho}
\lA \nabla_{x,z} \wp\rA_{L^\infty([z_0,0];\holder{r-\tq})}
\le
\mathcal{F}\bigl(\lA (\eta,\psi,V,B)\rA_{H^{s+\mez}\times H^{s+\mez}\times H^s\times H^s}\bigr).
\end{equation}
The key point is that, since $r-3/4\ge 0$, we now have an $L^\infty$-estimate for~$\nabla_{x,z}\wp$ 
which does not depend on the h\"older norms, which are the highest order norms for $s<1+d/2$ 
(compare with Proposition~\ref{p1H}). 

To prove \eqref{est:pH32}, let us revisit the proof of Theorem~\ref{T3}. 
With the notations of Chapter~\ref{C:2} (see~\eqref{nomdeW}), 
$W=\kappa(z)(\partial_z - T_A)  \wp$ 
satisfies a parabolic evolution equation of the form
\be\label{n321}
\partial_z W -T_a W =F_0+ F_1+ F_2+ F_3+F_4,
\ee
where the symbols~$a$ and~$A$ are as defined in \e{aA}, 
$F_0$ is given by~\eqref{defi:F0Lambda} and
\begin{align*}
F_1&=\gamma \partial_z \wp,\\
F_2&=-\bigl( T_{\Delta \wp}\alpha +R(\alpha,\Delta \wp)
+T_{\partialx\partial_z \wp}\cdot \beta+R(\beta,\partialx\partial_z \wp)\bigr),\\
F_3&= (T_a T_A -T_\alpha \Delta)\wp
-T_{\partial_z A} \wp,\\
F_4&=\kappa'(z)(\partial_z - T_{A})\wp.
\end{align*} 
Since $(\partial_z - T_A)  \wp=W$ for $z$ small enough (by definition of $W$), 
in light of \eqref{inwpCrho} and Proposition~\ref{prop:maxH} (applied with $r_0=1/4$, $q=+\infty$), 
one can reduce the proof of \eqref{est:pH32} to proving that, for some $\delta>0$, 
the $L^\infty([z_0,0];\zygmund{r'-\mez+\delta})$ norm of $W$ 
is bounded by the right-hand side of~\eqref{est:pH32}. 
Again, since $W\arrowvert_{z=z_1}=0$, 
by using Proposition~\ref{prop:maxH} (with $q=2$), the former estimate for $W$ will be deduced 
from the equation~\e{n321} and the following lemma. 

\begin{lemm} \label{lem.5.10}
There exists~$z_0<0$ and $\epsilon>0$ such that, for $i\in\{0,\ldots,4\}$,
$$
\| F_i\|_{L^2([z_0, 0]; \zygmund{r'-1+\epsilon})}
\le \mathcal{F}\bigl(\lA (\eta,\psi,V,B)\rA_{H^{s+\mez}\times H^{s+\mez}\times H^s\times H^s}\bigr)
\left\{1+\| \eta \|_{\holdertdm}\right\}.
$$
\end{lemm}

To prove Lemma~\ref{lem.5.10} 
we begin by recording the following easy refinement of previous bounds 
on the coefficients~$\alpha,\beta,\gamma$. We 
shall use the following variant of Lemma~\ref{L6}:
\begin{multline}\label{eq.ep3}
\lA \alpha\rA_{L^2([z_0,0];C^{r}_{*})}+
\lA \beta\rA_{L^2([z_0,0];C^{r}_{*})}+\lA \gamma\rA_{L^2([z_0,0];C^{r-1}_{*})}\\
\le 
\mathcal{F}(\| \eta \|_{H^{s+\mez}}) (\| \eta \|_{\zygmund{r+\mez}}+1).
\end{multline}
As in the proof of Lemma~\ref{L6}, 
such estimates follow from the definition of~$\rho$ (by means of the Poisson kernel), 
the Sobolev embedding~$H^{s+\mez}\subset W^{1,\infty}$ and 
tame estimates in H\"older spaces~\eqref{esti:F(u)bis}.
We are now ready to conclude the proof of Lemma~\ref{lem.5.10}.

\noindent {\em Estimate of }$F_0$. 
Since $C^{s-3/4-d/2}_*\subset C^{r-1}_*$, \eqref{esti:F0Sob} implies that
$$
\| F_0\|_{L^2([z_0, 0]; C^{r-1}_*)}
\le \mathcal{F}\bigl(\lA (\eta,\psi,V,B)\rA_{H^{s+\mez}\times H^{s+\mez}\times H^s\times H^s}\bigr).
$$

\noindent {\em Estimate of }$F_1$. 
Using \e{prZ1}, the term~$F_1=\gamma \partial_z \wp$ 
is estimated by 
\begin{equation*}
\lA F_1\rA_{L^2_z([z_0,0];C^{r-1}_*)}
\le K \lA \partial_z \wp\rA_{L^\infty_z([z_0,0];C^{r-1}_*)}\lA \gamma \rA_{L^2_z([z_0,0];C^{r-1}_*)}.
\end{equation*}
The desired estimate then follows from~\eqref{inwpCrho} and~\eqref{eq.ep3}.

\noindent {\em Estimate of }$F_2$. 
According to~\eqref{niZ} with~$m=1, s=r$, 
we obtain
\begin{align*}
\lA T_{\Delta \wp(z)}\alpha(z)\rA_{C^{r-1}_*}&\les 
\lA \Delta \wp(z)\rA_{C^{-1}_*}\lA \alpha(z)\rA_{\zygmund{r}},\\
\lA T_{\partialx\partial_z \wp(z)}\cdot \beta(z)\rA_{C^{r-1}_*}
&\les \lA \partialx\partial_z \wp(z)\rA_{C^{-1}_*}\lA \beta(z)\rA_{\zygmund{r}}.
\end{align*}
On the other hand, since~$r>1$ we can apply~\eqref{Bony2} to obtain
\begin{align*}
\lA R(\alpha,\Delta \wp)(z)\rA_{C^{r-1}_*}&
\les \lA \Delta\wp(z)\rA_{C^{-1}_*}\lA \alpha(z)\rA_{\zygmund{r}}\les 
\lA \partialx\wp(z)\rA_{L^{\infty}}\lA \alpha(z)\rA_{\zygmund{r}},\\
\lA R(\beta,\partialx\partial_z \wp)(z)\rA_{C^{r-1}_*}
&\les \lA \partialx\partial_z \wp(z)\rA_{C^{-1}_*}\lA\beta(z)\rA_{\zygmund{r}}
\les \lA \partial_z \wp(z)\rA_{L^{\infty}}\lA\beta(z)\rA_{\zygmund{r}}
\end{align*}
By using~\eqref{inwpCrho} and~\eqref{eq.ep3}, we conclude the proof of the claim in Lemma~\ref{lem.5.10} for~$i=2$.

\noindent {\em Estimate of }$F_3$. 
Using~\eqref{eq.3.96bis}--\e{eq.3.96ter}  with~$\mu=s-1/2$ we find
$$
\lA F_3(z)\rA_{C^{r-1}_*}\les \lA F_3(z)\rA_{H^{s-\mez}}\le 
\mathcal{F}\bigl(\| \eta \|_{H^{s+\mez}}\bigr)\| \eta \|_{\holdertdm} \lA \nabla_{x,z} \wp(z)\rA_{H^{s}},
$$
and hence
$$
\lA F_3\rA_{L^2([z_0,0];C^{r-1}_*)}
\le 
\mathcal{F}\bigl(\| \eta \|_{H^{s+\mez}}\bigr)\| \eta \|_{\holdertdm} \lA \nabla_{x,z} \wp\rA_{X^{s-\mez}([z_0,0])},
$$
by definition of $X^{s-1/2}([z_0,0])$. The desired estimate follows from~\eqref{inwpCSob}. 

\noindent {\em Estimate of }$F_4$. This follows from \eqref{inwpCrho} and \e{esti:quant1}.  

This completes the proof of Lemma~\ref{lem.5.10} and hence the proof of Proposition~\ref{p1:H}.
\end{proof}

\section{Sobolev estimates}

The proof of Sobolev estimates uses in an essential way the introduction of a new unknown 
(following Alinhac, see~\cite{AM,ABZ1,Alipara,Ali}) 
which allows us to circumvent the classical issue that there is a loss of $1/2$ derivative 
when one works with the Craig-Sulem-Zakharov system. 
By working with the unknowns $(\eta,V,B)$, the introduction of this good unknown 
amounts to work with $U=V+T_\zeta B$ where recall that $\zeta=\nabla \eta$ 
(the~$i$th component ($i=1,\ldots,d$) 
of this vector valued unknown is 
$U_i=V_i+T_{\partial_i\eta}B$). 

To prove Sobolev estimates, it is convenient to work with
\begin{equation}\label{defiUszetas}
\begin{aligned}
&U_s \defn \lDx{s} V+T_\zeta \lDx{s}\B,\\[0.5ex]
&\zeta_s\defn \lDx{s}\zeta.
\end{aligned}
\end{equation}
Now we have the following result which complements \cite[Prop.\ 4.8]{ABZ3}.

\begin{prop}\label{prop:csystem2}
There holds
\begin{align}
&(\partial_t+T_V\cdot\partialx)U_s+T_\ma\zeta_s=f_1,\label{premiere}\\
&(\partial_{t}+T_V\cdot\partialx)\zeta_s=T_\lambda U_s +f_2,\label{seconde}
\end{align}
where recall that~$\lambda$ is the symbol
$$
\lambda(t;x,\xi)\defn\sqrt{(1+\la\partialx\eta(t,x)\ra^2)\la\xi\ra^2-(\partialx\eta(t,x)\cdot\xi)^2},
$$
and where, for each time~$t\in [0,T]$, 
\begin{multline}\label{p49:estf}
\lA (f_1(t),f_2(t))\rA_{L^2\times H^{-\mez}}\\
\le \mathcal{F}\bigl( \lA (\eta,\psi)(t)\rA_{H^{s+\mez}},\lA (V,B)(t)\rA_{H^s}\bigr) 
\left\{1+\lA \eta(t)\rA_{\holder{r+\mez}}+\lA (V,B)(t)\rA_{\holder{r}}\right\}.
\end{multline}
\end{prop}
By using the tame estimates for the paralinearization of the Dirichlet-Neumann operator proved in 
Chapter~\ref{C:2}, 
there is nothing new in the proof of Proposition~\ref{prop:csystem2} compared 
to the proof of Prop.\ $4.8$ in \cite{ABZ3}. Indeed, 
the proof of Prop.\ $4.8$ in \cite{ABZ3}Ê
applies verbatim (up to replacing
$\mathcal{F}\bigl(\lA (\eta,\psi,V,B)\rA_{H^{s+\mez}\times H^{s+\mez}\times H^s\times H^s}\bigr)$ 
by the right-hand side of \e{p49:estf}). 

The next step is a symmetrization of the non-diagonal part of the equations. 
We have the following result, whose proof follows directly from the proof of Proposition~$4.10$ in \cite{ABZ3} and the estimates 
on the Taylor coefficient proved above (see Proposition~\ref{prop:ma}).

\begin{prop}\label{psym}
Introduce the symbols
$$
\gamma=\sqrt{\ma \lambda},\quad q= \sqrt{\frac{\ma}{\lambda}},
$$
and set~$\theta_s=T_{q} \zeta_s$. Then 

\begin{align}
&\partial_{t} U_s+T_{V}\cdot\partialx  U_s  + T_\gamma \theta_s = F_1,\label{systreduit1}
\\
&\partial_{t}\theta_s+T_{V}\cdot\partialx \theta_s - T_\gamma U_s =F_2,
\label{systreduit2}
\end{align}
for some source terms~$F_1,F_2$ satisfying 
\begin{multline*}
\lA (F_{1}(t),F_{2}(t))\rA_{L^{2}\times L^{2}}\\
\le \mathcal{F}\bigl( \lA  (\eta,\psi)(t)\rA_{H^{s+\mez}},\lA (V,B)(t)\rA_{H^s}\bigr) 
\left\{1+\lA \eta(t)\rA_{\holder{r+\mez}}+\lA (V,B)(t)\rA_{\holder{r}}\right\}.
\end{multline*}
\end{prop}
Notice that, by definition of $M_s(T)$ and $Z_r(T)$,
\begin{align*}
&\lA (F_{1},F_{2})\rA_{L^1([0,T];L^{2}\times L^{2})}\\
&\quad\le 
\mathcal{F}(M_s(T))\left\{T+\lA \eta\rA_{L^1([0,T];\holder{r+\mez})}
+\lA (V,B)\rA_{L^1([0,T];\holder{r})}\right\}\\
&\quad\le \sqrt{T}\mathcal{F}(M_s(T))Z_r(T),
\end{align*}
for $T\le 1$. Then it follows from the previous proposition and classical results (see~\S4.4 in \cite{ABZ3}) 
that we have the following~$L^\infty_t(L^2_x)$ estimate for~$(U_s,\theta_s)$.
\begin{lemm}\label{estimationdeUs}
There exists a non-decreasing function~$\mathcal{F}$ such that
\begin{equation}\label{esti:4.46}
\lA U_s\rA_{L^\infty([0,T];L^{2})}+
\lA \theta_s\rA_{L^\infty([0,T];L^{2})}
\le \mathcal{F}(M_{s,0})+\sqrt{T}\mathcal{F}(M_s(T))(1+Z_r(T)).
\end{equation}
\end{lemm}
It remains to deduce from this lemma estimates for the Sobolev norms 
of~$\eta,\psi,V,B$. 
Recall that the functions~$U_s$ and~$\theta_s$ are obtained 
from~$(\eta,V,\B)$ through:
\begin{align*}
&U_s \defn \lDx{s} V+T_\zeta \lDx{s}\B,\\[0.5ex]
&\theta_s\defn T_{\sqrt{\ma/\lambda}}\lDx{s}\nabla \eta.
\end{align*}

We begin with the following result. 
It gives the desired estimate for $\eta$ but only 
weaker estimates for~$(V,B)$ 
and the Taylor coefficient~$\ma$.

\begin{lemm}\label{ecBV}
There exists a non-decreasing function~$\mathcal{F}$ such that for any~$r>0$,
\begin{multline}\label{ecBVBV}
\| \eta \|_{L^\infty([0,T];H^{s+\mez})}+
\lA (B,V)\rA_{L^\infty([0,T];H^{s-\mez})}\\
\le \mathcal{F}(M_{s,0})+\sqrt{T}\mathcal{F}(M_s(T))(Z_r(T)+1),
\end{multline}
and, for any~$1<r'<r$, 
\begin{equation}\label{ecBVma}
\lA \ma\rA_{L^\infty([0,T];\zygmund{r'-1})}\le \mathcal{F}(M_{s,0})+\sqrt{T}\mathcal{F}(M_s(T))(Z_r(T)+1).
\end{equation}
\end{lemm}
We omit the proof since this lemma follows directly from the 
proof of Lemma~$4.13$ and Lemma~$4.14$ in \cite{ABZ3}.

Once~$\eta$ is estimated in~$L^\infty([0,T];H^{s+\mez})$, by using the 
estimate for~$U_s$, we are going to estimate~$(B,V)$ in~$L^\infty([0,T];H^s)$. 
Here we shall make an essential use of the following result 
about the paralinearization of the Dirichlet-Neumann operator for domains 
whose boundary is in~$H^{\mu}$ for some~$\mu>1+d/2$. 

\begin{prop}[from \cite{ABZ3}]\label{coro:paraDN1}
Let~$d\ge 1$ and~$\mu>1+\frac{d}{2}$. For any~$\frac{1}{2}\leq \sigma \leq \mu-\mez$ and any
$$
0<\eps\leq \mez, \qquad \eps< \mu-1-\frac{d}{2},
$$
there exists a non-decreasing function~$\mathcal{F}\colon\xR^+\rightarrow\xR^+$ such that 
$R(\eta)f\defn G(\eta)f-T_\lambda f$ satisfies
\begin{equation*}
\lA R(\eta)f\rA_{H^{\sigma-1+\eps}({\mathbf{R}}^d)}\le \mathcal{F} \bigl(\| \eta \|_{H^{\mu}({\mathbf{R}}^d)}\bigr)\lA f\rA_{H^{\sigma}({\mathbf{R}}^d)}.
\end{equation*}
\end{prop}

\begin{lemm}\label{ecBVs}
There exists a non-decreasing function~$\mathcal{F}$ such that
\begin{equation}\label{n333}
\lA (V,\B)\rA_{L^\infty([0,T];H^{s})}
\le \mathcal{F}\bigl(\mathcal{F}(M_{s,0})+T\mathcal{F}\bigl(M_s(T)+Z_r(T)\bigr)\bigr).
\end{equation}
\end{lemm}
\begin{proof}

\noindent{\sc Step} 1. Recall that $U=V+T_\zeta B$. 
We begin by proving that there exists a non-decreasing function~$\mathcal{F}$ such that
\begin{equation}\label{p50:1}
\lA U \rA_{L^\infty([0,T];H^{s-\uq})}
\le \mathcal{F}\bigl(\mathcal{F}(M_{s,0})+T\mathcal{F}\bigl(M_s(T)+Z_r(T)\bigr)\bigr).
\end{equation}
Since $U_s = \lDx{s} V+T_\zeta \lDx{s}\B$ by definition, we have
$$
\lDx{s-\uq} U =\lDx{-\uq} \bigl\{ U_s +\bigl[ \lDx{s},T_\zeta\bigr] \B\bigr\}.
$$
Theorem~\ref{theo:sc0} implies that
\be\label{n336}
\lA \bigl[ \lDx{s},T_\zeta\bigr] \rA_{H^{\mu}\rightarrow H^{\mu-s+\uq}}\les 
\lA \zeta\rA_{\holder{\uq}}
\ee
so
$$
\lA \bigl[ \lDx{s},T_\zeta\bigr] \B\rA_{H^{-\uq}}\les \lA \zeta\rA_{\holder{\uq}}\lA \B\rA_{H^{s-\mez}}.
$$
Since $s>3/4+d/2$, we have
$$
\lA \zeta\rA_{\holder{\uq}}\les \lA \zeta\rA_{H^{s-\mez}}\le \| \eta \|_{H^{s+\mez}}
$$
and hence
$$
\lA U\rA_{H^{s-\uq}}\les \lA U_s\rA_{H^{-\uq}}+ \| \eta \|_{H^{s+\mez}} \lA \B\rA_{H^{s-\mez}}.
$$
The three terms in the right-hand side of the above inequalities have been already estimated; indeed, Lemma~\ref{estimationdeUs} gives an estimate for the $L^\infty_t(L^2_x)$-norm 
of $U_s$ and {\em a fortiori}Ê
for its $L^\infty_t(H^{-\uq}_x)$-norm, 
see also Lemma~\ref{ecBV} for~$B$ and 
Lemma~\ref{ecBV} for~$\eta$. This proves~\eqref{p50:1}.

\noindent{\sc Step} 2. Recall that we have already estimated 
the $L^\infty_t(H^{s-\mez}_x)$-norms of $B,V$ and that we want to estimate 
their  $L^\infty_t(H^{s}_x)$-norms. 
As an intermediate step, we begin by proving that
\begin{equation}\label{n337}
\lA B \rA_{L^\infty([0,T];H^{s-\uq})}
\le \mathcal{F}\bigl(\mathcal{F}(M_{s,0})+T\mathcal{F}\bigl(M_s(T)+Z_r(T)\bigr)\bigr).
\end{equation}
To do so, the two key points are the paralinearization estimate for 
$R(\eta)\defn G(\eta)-T_\lambda$ (see Proposition~\ref{coro:paraDN1}) and 
the relation \e{n340a} 
between~$V$ and~$B$: for any ~$s> \mez + \frac d 2$ one has $G(\eta){B}=-\cnx V +\widetilde{\gamma}$ where 
\be\label{n340}
\|\widetilde{\gamma}\|_{H^{s- \mez}} \leq \mathcal{F}( \|( \eta,  V, B)\|_{H^{s+ \mez} \times H^{\mez}\times H^\mez}).
\ee

Taking the divergence in $U=V+T_{\zeta}B$, 
we get according to Lemma~\ref{ecBV}  and the previous identity $G(\eta){B}=-\cnx V +\widetilde{\gamma}$,
\begin{align*}
\cn U &= \cn V + \cn T_\zeta B=\cn V +T_{\cn \zeta} B+T_\zeta\cdot \nabla B\\
&= -G(\eta) B+T_{i \zeta \cdot \xi+ \cn \zeta}B + \widetilde{\gamma}\\
&=-T_\lambda B -R(\eta)B+T_{i \zeta \cdot \xi+ \cn \zeta}B + \widetilde{\gamma}\\
&=T_q B -R(\eta)B+T_{\cn \zeta}B + \widetilde{\gamma}
\end{align*}
where, by notation,
$$
q\defn -\lambda+i\zeta\cdot\xi.
$$

Now write
\begin{equation*}
T_{q} B = \cn U - T_{\cn \zeta} B + R(\eta) B-\widetilde{\gamma}
\end{equation*}
and
$$
B=T_{\frac{1}{q}}T_q B +\Bigl(I-T_{\frac{1}{q}}T_q\Bigr) B
$$
to obtain
\be\label{n338}
B=T_{\frac{1}{q}}  \cn U -T_{\frac{1}{q}} \widetilde{\gamma}+\mathcal{R} B
\ee
where
\begin{equation}\label{defi:Repsilon}
\mathcal{R}\defn T_{\frac{1}{q}} \Bigl( -T_{\cn \zeta}+R(\eta)\Bigr)+\Bigl(I-T_{\frac{1}{q}}T_q\Bigr).
\end{equation}
We now claim that $\mathcal{R}$ is of order $-1/4$ 
together with the following estimate: for any $1/2\le \sigma\le s$, we have  
\be\label{n339}
\lA \mathcal{R}B\rA_{H^{\sigma+\uq}}\le
\mathcal{F}\big(\lA \eta\rA_{H^{s+\mez}}\big)\lA B\rA_{H^{\sigma}}.
\ee
Fix $\sigma\in [1/2,s]$. 
We begin by estimating $R(\eta)$. 
According to Proposition~\ref{coro:paraDN1} 
(applied with~$\mu= s +\mez$ and $\eps=\uq$), we have  
\begin{equation*}
\lA R(\eta)B\rA_{H^{\sigma-\tq}}\le \mathcal{F} \bigl(\| \eta \|_{H^{s+\mez}}\bigr)
\lA B\rA_{H^{\sigma}}.
\end{equation*}
On the other hand, since $\cnx \zeta=\Delta\eta$, the rule \e{niS}Ê
implies that
$$
\lA T_{\cnx \zeta}\rA_{H^\sigma\rightarrow H^{\sigma-\tq}}\les 
\lA \cnx \zeta\rA_{\zygmund{-\tq}}\les \lA \eta\rA_{H^{s+\mez}}.
$$
Finally,
$q=- \lambda + i \zeta \cdot \xi \in \Gamma^1_{1/4}$ with 
$M^1_{1/4}(q)\leq C( \|\eta\|_{H^{s+ \mez}})$ 
since~$s> \tq + \frac d 2$. Moreover, $q^{-1}$ is of order $-1$ and we have
$$
M^{-1}_{1/4}\bigl(q^{-1}\bigr)\leq \mathcal{F}\big( \|\eta\|_{H^{s+ \mez}}\big).
$$
Consequently, according to~\eqref{esti:quant1} and \eqref{esti:quant2}, we have
\be\label{n341}
\Big\Vert T_{\frac{1}{q}}\Big\Vert_{H^{\sigma-\tq}\rightarrow H^{\sigma+\uq}}
+
\Big\Vert I-T_{\frac{1}{q}}T_q\Big\Vert_{H^{\sigma}\rightarrow H^{\sigma+\uq}}
\le \mathcal{F}\big( \|\eta\|_{H^{s+ \mez}}\big).
\ee
By combining the previous bound, we obtain the desired estimate \e{n339}. 
Now, \e{n339}Ê
applied with $\sigma=s-1/2$ implies the following estimate for the last term in the right-hand side of \e{n338},
$$
\lA \mathcal{R}B\rA_{H^{s-\uq}}\le  \mathcal{F}\big( \|\eta\|_{H^{s+ \mez}}\big)\lA B\rA_{H^{s-\mez}}.
$$
To estimate the two other terms in the right-hand side of \e{n338}, we use 
the operator norm estimate \e{n341} for $T_{1/q}$, to obtain 
$$
\lA T_{\frac{1}{q}}  \cn U -T_{\frac{1}{q}} \widetilde{\gamma}\rA_{H^{s-\uq}}
\le \mathcal{F}\big( \|\eta\|_{H^{s+ \mez}}\big) \big\{ \lA U\rA_{H^{s-\uq}}
+\lA \widetilde{\gamma}\rA_{H^{s-\frac{5}{4}}}\big\}.
$$
By combining the two previous estimates, it follows from \e{n338}Ê
that
\be\label{n344}
\lA B\rA_{H^{s-\uq}}
\le \mathcal{F}\big( \|\eta\|_{H^{s+ \mez}}\big) 
\big\{ \lA U\rA_{H^{s-\uq}}
+\lA \widetilde{\gamma}\rA_{H^{s-\frac{5}{4}}}+\lA B\rA_{H^{s-\mez}}\big\}.
\ee
Taking the $L^\infty$-norm in time, one obtains the claim \e{n337} from 
the previous estimates, see 
\e{p50:1}, \e{n340} and \e{ecBVBV}.

\noindent{\sc Step} 3: Bootstrap. We now use the previous bound \e{n337} 
for $B$ to improve the estimate \e{p50:1}Ê
for $U$, namely to prove that
\begin{equation}\label{n346}
\lA U \rA_{L^\infty([0,T];H^{s})}
\le \mathcal{F}\bigl(\mathcal{F}(M_{s,0})+T\mathcal{F}\bigl(M_s(T)+Z_r(T)\bigr)\bigr).
\end{equation}
Firstly, writing 
$\lDx{s} U =U_s +\bigl[ \lDx{s},T_\zeta\bigr] \B$ and using \e{n336}, one has
$$
\lA U\rA_{H^{s}}\les \lA U_s\rA_{L^{2}}+ \| \eta \|_{H^{s+\mez}} \lA \B\rA_{H^{s-\uq}}.
$$
As above, the three terms in the right-hand side of the above inequalities have been already estimated; indeed, Lemma~\ref{estimationdeUs} 
gives an estimate for the $L^\infty_t(L^2_x)$-norm 
of $U_s$, $\eta$ is estimated by means of Lemma~\ref{ecBV} and we can 
now use \e{n338} to estimate $\lA \B\rA_{H^{s-\uq}}$. This proves \e{n346}.

We next use \e{n346} to improve the estimate \e{n338} for $B$. Firstly, 
by using the estimate \e{n339} with $\sigma=s-1/4$ instead of $s-1/2$, we obtain as 
above
\be\label{n348}
\lA B\rA_{H^{s}}
\le \mathcal{F}\big( \|\eta\|_{H^{s+ \mez}}\big) 
\big\{ \lA U\rA_{H^{s}}
+\lA \widetilde{\gamma}\rA_{H^{s-\mez}}+\lA B\rA_{H^{s-\uq}}\big\}.
\ee
Taking the $L^\infty$-norm in time, it 
follows from the previous estimates (see 
\e{n346}, \e{n340} and \e{n337}) that 
\begin{equation}\label{n350}
\lA B \rA_{L^\infty([0,T];H^{s})}
\le \mathcal{F}\bigl(\mathcal{F}(M_{s,0})+T\mathcal{F}\bigl(M_s(T)+Z_r(T)\bigr)\bigr).
\end{equation}

\noindent{\sc Step} 4: Estimate for $V$. Writing $V=U-T_\zeta B$, it easily follows from 
\e{n346} and \e{n350} that $\lA V \rA_{L^\infty([0,T];H^{s})}$ is bounded by the right-hand side of \e{n333}. 

This completes the proof of the lemma.
\end{proof}

It remains to estimate the $L^\infty([0,T];H^{s+\mez})$-norm of 
$\psi$. This estimate is obtained in two steps. Firstly, 
since $\nabla\psi=V+\B \nabla \eta$ and since the~$L^\infty([0,T];H^{s-\mez})$-norm of 
$(\nabla\eta,V,\B)$ has been previously estimated, we obtain 
the desired estimate for the $L^\infty([0,T];H^{s-\mez})$-norm of $\nabla\psi$. It 
remains to estimate $\lA \psi\rA_{L^\infty([0,T];L^{2})}$. This in turn follows from the identity 
\begin{equation*}
\partial_t \psi+V\cdot \nabla \psi
=-g\eta+\mez V^2+\mez B^2
\end{equation*}
and classical $L^2$ estimate for hyperbolic equations (see \e{eq.ii}). 

\chapter{Strichartz estimates}\label{C:4}

In this chapter we shall prove Strichartz estimates for rough solutions of the gravity water waves equations. 
This is the main new point in this paper. 
Namely, we shall prove Theorem~\ref{T4} stated in the introduction.

\section{Symmetrization of the equations}\label{s:2}

Consider a smooth solution of \eqref{n10} defined on some time interval $[0,T_0]$ and 
satisfying Assumption~\ref{A:2}. 
We begin by proving 
that the water-waves equations can be reduced to a very simple form
\begin{equation*}
\Big(  \partial_t + \mez( T_{V } \cdot \nabla + \nabla \cdot T_{V }) +i T_{\gamma } \Big)u  =f 
\end{equation*}
where $T_V$ is a paraproduct and $T_\gamma$ is a para-differential operator of 
$\frac{1}{2}$ with symbol 
$$
\gamma=\sqrt{\ma \lambda},\quad 
\lambda=\sqrt{\big(1+ \vert \nabla \eta\vert^2\big)\vert \xi \vert^2 
- \big(\xi \cdot \nabla \eta\big)^2},
$$
where $a$ is the Taylor coefficient (see \e{n17}). 
Recall that we introduced the following notations
$$
\zeta=\nabla\eta,\quad 
U_s \defn \lDx{s} V+T_\zeta \lDx{s}\B,\quad \zeta_s\defn \lDx{s}\zeta,\quad 
q= \sqrt{\frac{\ma}{\lambda}},\quad\theta_s=T_{q} \zeta_s.
$$
We already proved in the previous chapter that one can symmetrize 
the water-waves equations and obtain a system of two equations for 
$U_s$ and $\theta_s$ of the form
\begin{align}
&\partial_{t} U_s+T_{V}\cdot\partialx  U_s  + T_\gamma \theta_s = F_1,\label{c41}
\\
&\partial_{t}\theta_s+T_{V}\cdot\partialx \theta_s - T_\gamma U_s =F_2,
\label{c42}
\end{align}
for some source terms $F_1,F_2$ satisfying 
\begin{multline*}
\lA (F_{1}(t),F_{2}(t))\rA_{L^{2}\times L^{2}}\\
\le C\left( \lA (\eta,\psi)(t)\rA_{H^{s+\mez}},\lA (V,B)(t)\rA_{H^s}\right)
\left\{1+\lA \eta(t)\rA_{W^{r+\mez,\infty}}+\lA (V,B)(t)\rA_{W^{r,\infty}}\right\},
\end{multline*}
for any 
real numbers $s$ and $r$ such that
$s>\frac{3}{4}+\frac{d}{2}$, $r>1$.

One has the following corollary, which is the starting point of the proof of the 
Strichartz estimate.

\begin{coro}\label{T10}
With the above notations, set
\begin{equation}\label{eq:r0}
u=\lDx{-s}(U_s-i \theta_s)=\lDx{-s}(U_s-iT_{\sqrt{a/\lambda}}\nabla \eta_s).
\end{equation}
Then $u$ satisfies the complex-valued equation
\begin{equation}\label{eq:r1}
\partial_t u +\mez \bigl( T_V\cdot \partialx +\partialx\cdot T_V\bigr)u + i T_\gamma u =f ,
\end{equation}
where $f$ satisfies
for each time $t\in [0,T]$,
\begin{equation}\label{eq:r2}
\begin{aligned}
&\lA f(t)\rA_{H^s}\\
&\le \mathcal{F}\big( \lA (\eta,\psi)(t)\rA_{H^{s+\mez}},\lA (V,B)(t)\rA_{H^s}\big) 
\left\{1+\lA \eta(t)\rA_{W^{r+\mez,\infty}}
+\lA (V,B)(t)\rA_{W^{r,\infty}}\right\}.
\end{aligned}
\end{equation}
\end{coro}
\begin{proof}
It immediately follows from \e{c41}--\e{c42} that $U_s-i \theta_s$ 
satisfies
\begin{equation}\label{troisieme}
(\partial_t +T_V\cdot \partialx + i T_\gamma ) (U_s-i \theta_s)=h ,
\end{equation}
where, for each time $t\in [0,T]$, the $L^{2}(\xR^d)$-norm of $h(t)$ 
is bounded by the right-hand side of \eqref{eq:r2}. 
By commuting $\lDx{-s}$ to~\eqref{troisieme} we thus obtain
$$
\partial_t u +T_V\cdot \partialx u + i T_\gamma u =f' ,
$$
where the $H^s$-norm of $f'(t)$ is bounded by the right-hand side of \eqref{eq:r2}. 
This gives \eqref{eq:r1} with $f=f'+\mez T_{\cnx V} u$ and 
it remains to prove 
that the $H^s$-norm of $T_{\cnx V} u$ is bounded by the right-hand side of \eqref{eq:r2}. 
Since a paraproduct by an $L^\infty$-function acts on any Sobolev spaces, we have 
$\lA T_{\cnx V} u\rA_{H^s}\les \lA \cnx V\rA_{L^\infty}\lA u\rA_{H^s}$. Therefore, 
to complete the proof, it is sufficient to prove that
\begin{equation}\label{n201}
\lA u\rA_{H^s}\le 
\mathcal{F}\big( \lA (\eta,\psi)\rA_{H^{s+\mez}},\lA (V,B)\rA_{H^s}\big) .
\end{equation}
To do so, write, by definition of $u$,
\begin{align*}
\lA u\rA_{H^s}&\le \lA U_s\rA_{L^2}+\lA \theta_s\rA_{L^2}\\
&\le \lA \langle D_x\rangle^s V\rA_{L^2}+\lA T_\zeta \langle D_x\rangle^s B\rA_{L^2} 
+\lA T_q \zeta_s\rA_{L^2}\\
&\le \lA V\rA_{H^s}+\lA \zeta\rA_{L^\infty}\lA B\rA_{H^s} 
+\mathcal{M}^{-\mez}_0(q) \lA \zeta\rA_{H^{s-\mez}}
\end{align*}
where we used that $\lA T_p v\rA_{H^{\mu-m}}\les \mathcal{M}^m_0(p)\lA v\rA_{H^\mu}$ 
for any paradifferential operator with symbol $p$ in $\Gamma^m_0$ (see \eqref{defi:norms} 
for the definition of $\mathcal{M}^m_0(p)$). Now we have, using the Sobolev embeding, 
$$
\lA \zeta\rA_{L^\infty}+\mathcal{M}^{-\mez}_0(q)
\le \mathcal{F}\big( \lA \eta\rA_{W^{1,\infty}},\lA a\rA_{L^\infty}\big)
\le \mathcal{F}\big( \lA \eta\rA_{H^{s+\mez}},\lA a\rA_{H^{s-\mez}}\big)
$$
so \eqref{n201} follows from the Sobolev estimates for $a$ 
(see Proposition~\ref{prop:ma}).
\end{proof}

\section{Smoothing the paradifferential symbols}

>From now on we assume  $s>1 + \frac{d}{2} - \sigma_d$ 
where $\sigma_1= \frac{1}{24},\, \sigma_d =  \frac{1}{12}$ if $d \geq 2$ 
and we set $s_0 :=\mez - \sigma_d >0.$ Then  $s-\mez >\frac{d}{2} +s_0$. 
Now, with $I = [0,T], $ we introduce the spaces
\begin{equation}\label{EFG}
\left\{
\begin{aligned}
&E:=C^0( I; H^{s - \mez}(\xR^d)) \cap L^2(I; W^{\mez, \infty}(\xR^d)),\\
& F:= C^0( I; H^{s}(\xR^d)) \cap L^2(I; W^{1, \infty}(\xR^d)),\\
&G:= C^0( I; W^{s_0, \infty}(\xR^d)) \cap  L^2(I; W^{\mez, \infty}(\xR^d)) 
\end{aligned}
\right.
\end{equation}
endowed with their natural norms.
 
We shall assume that 
\begin{equation}
\begin{aligned}
&(i) \quad \ma \in E, \quad  \nabla \eta \in E, \quad V \in  F,\\
&(ii) \quad  \exists c>0 : \ma(t,x) \geq c,\quad  \forall (t,x) \in I \times \xR^d.
\end{aligned}
\end{equation}
Let us recall that 
\begin{equation}\label{U}
\begin{aligned}
&\gamma (t,x,\xi) =   \big(\ma ^2\mU(t,x,\xi)\big)^{\frac{1}{4}}\\
&\mU(t,x,\xi) \defn (1+ \vert \nabla \eta\vert^2(t,x))\vert \xi \vert^2 - (\xi \cdot \nabla \eta(t,x))^2.
\end{aligned}
\end{equation}
   
Now we have,  for $\xi \in \mathcal{C}_0 := \{\xi : \mez \leq \vert \xi \vert \leq2\}$ 
considered as a parameter,
$$
\ma^2\mU\in G 
$$
uniformly in $\xi$.

\begin{lemm}\label{estgamma}
There exists $\mathcal{F}:\xR^+ \to \xR^+$ 
such that $\Vert \gamma \Vert_G \leq \mathcal{F}( \Vert \nabla \eta \Vert_E) $ for all $\xi \in \mathcal{C}_0.$
\end{lemm}
\begin{proof}
By the Cauchy-Schwartz inequality we have $\mathcal{U}(t,x,\xi) \geq \vert \xi \vert^2$ from which we deduce that 
\begin{equation}\label{gamma}
\gamma(t,x,\xi) \geq c_0>0  \quad  \forall    (t,x,\xi) \in I\times \xR^d \times \mathcal{C}_0. 
\end{equation}
Moreover $\gamma \in C^0(I; L^\infty(\xR^d \times \mathcal{C}_0)).$ On the other hand, since
$$
\gamma^4(t,x,\xi) - \gamma^4(t,y,\xi) = (\gamma(t,x,\xi) -\gamma(t,y,\xi))\sum_{j=0}^3 
( \gamma(t,x,\xi))^{3-j}(\gamma(t,y,\xi))^j
$$
we have, using \eqref{gamma},
$$
\frac{\vert \gamma(t,x,\xi) 
-\gamma(t,y,\xi)\vert}{\vert x-y\vert^\sigma} 
\leq \frac{1}{ 4c_0^3} \frac{\vert (\ma^2\mU)(t,x,\xi) -(\ma^2\mU)(t,y,\xi)\vert}{\vert x-y\vert^\sigma}\cdot
$$
Taking $\sigma = s_0$ and $\sigma = 1/2$ we deduce the lemma.
\end{proof}

Guided by works by Lebeau~\cite{Lebeau92}, Smith \cite{Smith}, 
Bahouri-Chemin~\cite{BaCh}, Tataru~\cite{TataruNS} and Blair~\cite{Blair}, 
we smooth out the symbols of the operator. 

\begin{defi}
Now let $\psi \in C_0^\infty(\xR^d), \psi(\xi) 
= 1 \text{ if } \vert \xi \vert \leq \mez,  \psi(\xi) = 0 \text{ if } \vert \xi \vert \geq 1.$ With $ h= 2^{-j}$ 
and $\delta >0,$ which will be chosen later on, we set
\begin{equation}
\gamma_\delta (t,x,\xi) = \psi(h^\delta D_x)\gamma(t,x,\xi).
\end{equation}
\end{defi}

\begin{lemm}\label{estgamma-delta}
$(i)\quad  \forall \alpha, \beta \in \xN^d \quad \exists C_{\alpha, \beta} >0 : \forall t \in I, \forall \xi \in \xR^d$
$$\vert D_x^\alpha D^\beta_\xi \gamma_\delta (t,x,\xi) \vert \leq C_{\alpha, \beta} h^{-\delta \vert \alpha \vert} 
\Vert D^\beta_\xi  \gamma(t, \cdot, \xi) \Vert_{L^\infty(\xR^d)}.$$
$(ii) \quad  \forall \alpha, \beta \in \xN^d, \vert \alpha \vert \geq 1 \quad \exists C_\alpha >0 : \forall t \in I, \forall \xi \in \xR^d$
$$ \vert D_x^\alpha D^\beta_\xi  \gamma_\delta (t,x,\xi) \vert \leq C_{\alpha, \beta} h^{-\delta (\vert \alpha \vert - \mez)} 
\Vert D^\beta_\xi  \gamma(t, \cdot, \xi) \Vert_{W^{\mez,\infty}(\xR^d)}.$$
\end{lemm}

\begin{proof}
$(i)$ follows from the fact that
$$\gamma_\delta (t,x,\xi) = h^{-\delta d}\widehat{\psi}\Big(\frac{\cdot}{h^\delta}\Big) \ast \gamma (t,\cdot, \xi).$$
$(ii)$ We write
$$
D_x^\alpha D^\beta_\xi  \gamma_\delta(t,x,\xi)
= \sum_{k =-1}^{+\infty}\Delta_k D_x^\alpha \psi(h^\delta D_x) D^\beta_\xi \gamma(t,x,\xi):= \sum_{k =-1}^{+\infty}v_k 
$$
where $\Delta_k$ denotes the usual Littlewood-Paley  frequency localization.

If $\mez 2^{k} \geq h^{-\delta} = 2^{j\delta}$ we have $\Delta_k \psi(h^\delta D_x) =0$. 
Therefore
$$
D_x^\alpha D^\beta_\xi \gamma_\delta(t,x,\xi) = \sum_{k =-1}^{2+ [j\delta]} v_k.
$$
Now 
$$
v_k = 2^{k \vert \alpha \vert}  \varphi_1(2^{-k}D_x) \psi(h^\delta D_x) \Delta_k D^\beta_\xi  \gamma(t,x,\xi) 
 $$
where $\varphi_1(\xi)$ is supported in $\{\frac{1}{3} \leq \xi \vert \leq 3\}.$    Therefore,
$$
\Vert v_k\Vert_{L^\infty(\xR^d)} 
\leq 2^{k \vert \alpha \vert} \Vert  {\Delta}_k  D^\beta_\xi  \gamma(t,\cdot,\xi)\Vert_{L^\infty(\xR^d)} 
\leq C2^{k \vert \alpha \vert} 2^{- \frac{k}{2}} 
\Vert D^\beta_\xi \gamma(t, \cdot, \xi)\Vert_{W^{\mez, \infty}(\xR^d)}.
$$
It follows that
$$
\Vert D_x^\alpha D^\beta_\xi \gamma_\delta(t,\cdot,\xi)\Vert _{L^\infty(\xR^d)} 
\leq C  \sum_{k =-1}^{2+ [j\delta]} 2^{k(\vert \alpha \vert - \mez)} \Vert D^\beta_\xi  \gamma (t,\cdot, \xi) \Vert_{W^{\mez, \infty}(\xR^d)}.
$$
Since $\vert \alpha\vert - \mez >0$ we deduce that
$$
\Vert D_x^\alpha D^\beta_\xi  \gamma_\delta(t,\cdot,\xi)\Vert _{L^\infty(\xR^d)} 
\leq C  2^{j\delta(\vert \alpha \vert - \mez)}\Vert D^\beta_\xi \gamma (t,\cdot, \xi) \Vert_{W^{\mez, \infty}(\xR^d)}.
$$
This completes the proof.
\end{proof}

We introduce now the Hessian matrix of $\gamma_\delta,$
\begin{equation}\label{Hess}
\Hess_\xi (\gamma_{\delta}) (t,x,\xi) = \Big( \frac{\partial^2 \gamma_\delta}{\partial \xi_j \partial \xi_k}(t,x,\xi)\Big).
\end{equation}

Our purpose is to prove the following result.
\begin{prop}\label{hess}
There exist $c_0>0, h_0 >0$ such that 
$$
\vert \det   \Hess_\xi (\gamma_{\delta}) (t,x,\xi)\vert \geq c_0 $$
for all $t \in I, x\in \xR^d, \xi \in \mathcal{C}_0, 0< h \leq h_0.$
\end{prop}

With the notation in \eqref{U} we can write
$$
\mU(t,x,\xi) = \langle A(t,x) \xi, \xi \rangle,
$$
where $A(t,x)$ is a symmetric matrix. 
Since we have
\begin{equation}\label{inf}
\vert \xi \vert^2 \leq \mU(t,x,\xi) \leq C(1+ \Vert \nabla \eta \Vert^2_{L^\infty(I \times \xR^d)}) \vert \xi \vert ^2 
\end{equation}
we see that the eigenvalues of $A$ are greater than one;  therefore we have
\begin{equation}\label{detA}
\det A(t,x) \geq 1 \quad \forall (t,x) \in I \times \xR^d.
\end{equation}
We shall need the following lemma.
\begin{lemm}
With $\alpha = \frac{1}{4}$ we have
$$
\vert \det  \Hess_\xi (\gamma)(t,x,\xi) \vert = a^{\frac{d}{2}}(2\alpha)^d\vert 2\alpha -1\vert \det A(t,x) 
\big(\mU(t,x,\xi)\big)^{(\alpha-1)d}.
$$
\end{lemm}

\begin{proof}
Here $t$ and $x$ are fixed parameters.
The matrix $A$ being symmetric one can find an orthogonal matrix $B$ such that 
$B^{-1}AB = D = \text{diag}(\mu_j)$ where the $\mu_j's$ are the eigenvalues of $A.$ Setting 
$C=\text{diag}(\sqrt{\mu_j})$ and $M= CB^{-1}$ we see that $\mU(t,x,\xi) = \vert M\xi \vert^2$ 
which implies that
\begin{equation}
\gamma(t,x,\xi) = a^\mez g(M\xi) \quad \text{where} \quad g(\zeta) = \vert \zeta \vert^{2\alpha},
\end{equation}
so that $\text{Hess}_\xi (\gamma)(t,x,\xi) = a^{\frac{1}{2}}  \,  {}^t\!M \big( \text{Hess}_\zeta( g) (M\xi)\big) M$. 
Since $\vert \det M \vert^2 = \vert \det C \vert^2 = \det A$ we obtain
\begin{equation}\label{Hgamma}
\vert \det  \text{Hess}_\xi (\gamma)(t,x,\xi) \vert  = a^{\frac{d}{2}} \det A(t,x) \, \vert \det \text{Hess}_\zeta (g) (M(t,x)\xi) \vert.
\end{equation}
Now we have,
\begin{equation}\label{Hg}
\frac{\partial^2 g}{\partial \zeta_j \partial \zeta_k} (\zeta) 
= 2\alpha \vert \zeta \vert^{2\alpha -2}\big(\delta_{jk} + 2(\alpha -1)\omega_j \omega_k\big), \quad \omega_j = \frac{\zeta_j}{\vert \zeta \vert}.
\end{equation}
Let us consider the function $F: \xR \to \xR$ defined by
$$F(\lambda) = \det \big( \delta_{jk} + \lambda \omega_j \omega_k\big).$$
It is a polynomial in $\lambda$ and we have
\begin{equation}\label{F}
F(0) = 1.
\end{equation}
Let us denote by $C_j(\lambda)$ the $j^{th}$ column of this determinant. Then
$$
F'(\lambda) = \sum_{k=1}^d \det \big(C_1(\lambda), \ldots, C'_k(\lambda),\ldots C_d(\lambda) \big).
$$
We see easily that
$$
\det \big(C_1(0), \ldots, C'_k(0),\ldots C_d(0) \big) = \omega_j^2
$$
which ensures that
\begin{equation}\label{F'}
F'(0) = 1.
\end{equation}
Now since $C_j(\lambda)$ is linear with respect to $\lambda$ 
we have $C''_j(\lambda) = 0$ therefore
$$
F''(\lambda) = \sum_{j=1}^d \sum_{k=1,k\neq j}^d 
\det \big(C_1(\lambda),\ldots, C'_j(\lambda),\ldots,C'_k(\lambda),\ldots, C_d(\lambda) \big).
$$
Now $C'_j(\lambda) = \omega_j(\omega_1,\ldots,\omega_d)$ 
and $C'_k(\lambda) = \omega_k(\omega_1,\ldots,\omega_d)$. 
It follows that $F''(\lambda) = 0$ for all $\lambda \in \xR.$ 
We deduce from \eqref{F}, \eqref{F'} that $F(\lambda) = 1 + \lambda$ and from \eqref{Hg} that
$$
\det \text{Hess}_\zeta(g)(\zeta) = \big(2\alpha \vert \zeta \vert^{2\alpha -2}\big)^d (2\alpha -1).
$$
The lemma follows then from \eqref{Hgamma} since $\mU(t,x,\xi) = \vert M(t,x)\xi \vert^2.$
\end{proof}

\begin{coro}\label{hess1}
One can find $c_0 >0$ such that  
$$
\vert \det  \Hess_\xi (\gamma)(t,x,\xi) \vert \geq c_0,
$$
for all $t\in I, x \in \xR^d, \xi \in \mathcal{C}_0.$ 
\end{coro}
\begin{proof}
This follows from the preceding lemma and from \eqref{inf}, \eqref{detA}.
\end{proof}
\begin{proof}[Proof of Proposition \ref{hess}]
 
Recall that  we have   for all $\alpha \in \xN^d$
$$\sup_{t\in I} \sup _{\vert \xi \vert \leq 2} \Vert D^\alpha_\xi \gamma (t,\cdot,\xi)\Vert_{W^{s_0, \infty}(\xR^d)} < +\infty.$$
For fixed $j,k\in \{1,\ldots,d\}$ we write
\begin{equation}\label{perturb}
\frac{\partial^2 \gamma_\delta}{\partial \xi_j \partial \xi_k}(t,x,\xi) 
= \frac{\partial^2 \gamma}{\partial \xi_j \partial \xi_k}(t,x,\xi) 
- (I-\psi(h^\delta D_x))\frac{\partial^2 \gamma}{\partial \xi_j \partial \xi_k}(t,x,\xi).
\end{equation}
Setting $\gamma_{jk}= \frac{\partial^2 \gamma}{\partial \xi_j \partial \xi_k}$ 
we  have, since $\psi(0) =1$, 
$$
(I-\psi(h^\delta D_x))\gamma_{jk}(t,x,\xi) 
= h^{-\delta d}\int_{\xR^d} \overline{\mathcal{F}}\psi \Big(\frac{y}{h^\delta}\Big)
\big[\gamma_{jk}(t,x,\xi) - \gamma_{jk}(t,x-y,\xi)\big] dy,
$$
where $\overline{\mathcal{F}}$ denotes the inverse Fourier transform with respect to $x$. 
Then 
\begin{equation*}
\begin{aligned}
&\Vert (I-\psi(h^\delta D_x))\gamma_{jk}(t,\cdot,\xi)  \Vert_{L^\infty (\xR^d)} \\
&\qquad \leq h^{-\delta d}\Big( \int_{\xR^d} \vert \overline{\mathcal{F}}\psi 
\big(\frac{y}{h^\delta}\big)\vert  \vert y \vert^{s_0} dy\Big) \Vert  \gamma_{jk}(t,\cdot,\xi) \Vert_{W^{s_0,\infty}(\xR^d)}\\
&\qquad \leq C h^{\delta s_0} \sup_{t\in I} \sup_{\vert \xi \vert \leq 2}
\Vert \gamma_{jk}(t,\cdot,\xi)\Vert_{W^{s_0,\infty}(\xR^d)}. 
\end{aligned}
\end{equation*}
Then Proposition \ref{hess} follows from Corollary \ref{hess1} and \eqref{perturb} 
if $h_0$ is small enough.
\end{proof}

\section{The pseudo-differential symbol}

In this paragraph, we study the pseudo-differential symbol of $T_{\gamma_\delta}$

Let $\chi \in C_0^\infty(\xR^d \times \xR^d) $ be such that
\begin{equation}\label{chi}
\begin{aligned}
&(i)\,  &&\chi(-\zeta, \eta) = \chi(\zeta,\eta),\\
&(ii)\,  &&\chi \text{ is homogeneous of order zero }, \\
&(iii)\, && \chi(\zeta, \eta) =1 \text{ if } \vert \zeta \vert \leq \eps_1 \vert \eta \vert,  
\quad\chi(\zeta, \eta) =0 \text{ if } \vert \zeta \vert \geq \eps_2 \vert \eta \vert 
\end{aligned}
\end{equation}
where $0<\eps_1 <\eps_2$ are small constants. 
Then modulo an operator of order zero we have
\begin{equation}\label{pseudo}
T_{\gamma_\delta}u  = \sigma_{\gamma_\delta}(t,x, D_x)u 
\end{equation}
with
$$
\sigma_{\gamma_\delta}(t,x, \eta) = \Big( \int_{\xR^d} 
e^{i x\cdot \zeta} \chi(\zeta, \eta) \widehat{\gamma}_\delta(t, \zeta, \eta) \,d\zeta \Big) \psi_0(\eta)
$$
where $\widehat{\gamma}_\delta$ denotes the Fourier transform of $\gamma_\delta$ 
with respect to  the variable $x$ and $\psi_0$ is a cut-off function such that $\psi_0(\eta) = 0$ 
for $\vert \eta \vert \leq \frac{1}{4}, \psi_0(\eta) = 1$  for $\vert \eta \vert \geq \frac{1}{3}.$ (Notice that $\chi$ can be so chosen that 
Remark $2.2$ in \cite{ABZ3} remains true.)

We shall use in the sequel the following result.
\begin{lemm}\label{l0.6}
For all $N\in \xN, M>0$ and all $\alpha, \beta \in \xN^d$ there exists $C  >0$ 
such that
$$\sup_{\vert \eta \vert \leq M} \vert D_\mu^\alpha 
D_\eta^\beta \widehat{\chi}(\mu, \eta)\vert \leq \frac{C}{\langle \mu \rangle^N}$$
where the Fourier transform is taken with respect to the first variable of $\chi$.
\end{lemm}
\begin{proof}
Recall that we have $\supp \chi \subset \{(\zeta, \eta): \vert \zeta \vert \leq \eps_2 \vert \eta \vert\}$. 
Then for $\vert \eta \vert \leq M$ we see easily that for any $k\in \xN$ we can write
$$
\vert \mu\vert^{2k} D_\mu^\alpha D_\eta^\beta \widehat{\chi}(\mu, \eta) 
= \int_{\xR^d} e^{-i \mu\cdot \zeta} 
(-\Delta_\zeta)^k\big [(-\zeta)^\alpha D_\eta^\beta \chi(\zeta, \eta)\big] \,d\zeta.
$$
The result follows then from the fact that when $\vert \eta \vert \leq M$ 
the domain of integration is contained in the set $\{\zeta : \vert \zeta \vert\leq \eps_2 M\}.$
\end{proof}
\section{Frequency localization}
Let us recall the Littlewood-Paley decomposition. 
There exists a function $\psi\in C^\infty_0(\xR^d)$ such that 
$\psi(\xi)=1$ for $\la \xi\ra \le 1/2$ and $\psi(\xi)=0$ for $\la \xi\ra\ge 1$, and a function 
$\varphi\in C^\infty_0(\xR^d)$ whose support is contained 
in the shell $\mathcal{C}_0 := \{\xi : \mez \leq \vert \xi \vert \leq2\}$ such that
$$
\psi(\xi)+\sum_{k=0}^{N-1}\varphi(2^{-k}\xi)=\psi(2^{-N}\xi).
$$
We set
$$
\Delta_{-1}u=\psi(D)u,\quad 
\Delta_j u = \varphi(2^{-j}D)u,\quad S_j u =\sum_{k=-1}^{j-1} \Delta_k u=\psi(2^{-j}D)u.
$$

Let us consider the operator
\begin{equation*}
L = \partial_t +\mez(T_V\cdot \nabla + \nabla \cdot T_V)+i T_\gamma
\end{equation*}
where
$$
V = (V_1,\ldots,V_d), \quad V_j \in C^0(I; H^s(\xR^d))\cap L^2(I; W^{1,\infty}(\xR^d)).
$$
  
Let $u \in L^\infty(I; H^s(\xR^d))$ be such that
$$
Lu = f \in L^2(I;H^s(\xR^d)).
$$
Then we have
\begin{equation}\label{commuta}
 L \Delta_j u = \Delta_j f + \mez \big( [T_V, \Delta_j]\cdot \nabla u +\nabla\cdot  [T_V, \Delta_j] u\big) +i [T_\gamma, \Delta_j]u.
\end{equation}
Now we have the following lemma.
\begin{lemm}\label{TV-S}There holds
\begin{align*}
 T_V \cdot \nabla \Delta_j u&= S_j(V)\cdot\nabla \Delta_ju + R_ju\\
 \nabla \cdot T_V \Delta_j u&= \nabla \cdot S_j(V) \Delta_j u + R'_j u
 \end{align*}
where $R_ju, R'_ju$ have their  spectrum contained in a ball $B(0, C2^j)$ and satisfies the estimate
$$
\Vert R_ju \Vert_{H^s(\xR^d)} + \Vert R'_ju \Vert_{H^s(\xR^d)} \leq   C\Vert V \Vert_{W^{1,\infty}(\xR^d)} \Vert u \Vert_{H^s(\xR^d)}.
$$
\end{lemm}
\begin{proof}
It is sufficient to  prove the first line. Indeed we have
$ \nabla \cdot T_V \Delta_j u = T_V \cdot \nabla \Delta_j u + T_{\text{div}V} \Delta_j u$ 
and $\Vert T_{\text{div}V} \Delta_j u\Vert_{H^s(\xR^d)} \leq C \Vert V \Vert_{W^{1,\infty}(\xR^d)} \Vert u \Vert_{H^s(\xR^d)}$.

Since $\Delta_k \Delta_j = 0$ if $\vert k-j\vert \geq 2$ we can write
\begin{align*}
T_V \cdot \Delta_j \nabla u  &= \sum_{\vert j-k \vert \leq 1} S_{k-3}(V) \cdot \Delta_k \Delta_j \nabla u = S_j(V)\cdot\sum_{\vert j-k \vert \leq 1}   \Delta_k \Delta_j \nabla u + R_ju\\
&   = S_j(V)\Delta_j \nabla u +R_j u, 
\end{align*}
where $R_j u = \sum_{\vert j-k \vert \leq 1} \big( S_{k-3}(V)-S_j(V) \big)  \cdot \Delta_k \Delta_j \nabla u$. 
We have three terms in the sum defining $R_ju$.  Each of them is a finite sum of terms of the form $A_j = \Delta_{j+\mu}(V)\cdot   \Delta_{j+ \nu} \Delta_j \nabla u.$ Since the spectrum of $A_j$ is contained in a ball of radius $C2^j$ we can write
  \begin{align*}
  \Vert A_j \Vert_{H^s(\xR^d)} &\leq C_1 2^{js}\Vert F_j \Vert_{L^2(\xR^d)} \leq C_1 2^{js}\Vert \Delta_{j+\mu}V\Vert_{L^\infty(\xR^d)} \Vert \Delta_{j }\nabla u\Vert_{L^2(\xR^d)} \\
  &\leq C_2 2^{js} 2^{-j} \Vert V \Vert_{W^{1,\infty}(\xR^d)} 2^{-js} 2^j \Vert u \Vert_{H^s(\xR^d)} \leq C_2\Vert V \Vert_{W^{1,\infty}(\xR^d)}  \Vert u \Vert_{H^s(\xR^d)}
  \end{align*}
  which completes the proof of the lemma.
\end{proof}
It follows from the lemma that we have 
\begin{multline}
 \big(\partial_t +\mez( S_j(V)\cdot \nabla + \nabla  \cdot  S_j(V))+i T_\gamma\big) \Delta_j u =\\
   \Delta_j f+\mez \big( [T_V, \Delta_j]\cdot \nabla u +\nabla\cdot  [T_V, \Delta_j] u\big) +i [T_\gamma, \Delta_j]u +R_j u + R'_ju
 \end{multline}
 As already mentioned, we need to smooth out the symbols. To do so, 
we replace 
$S_j(V)$ by $S_{\delta j}(V)=\psi(2^{-j\delta}D)V$ and $\gamma$ by $\gamma_\delta = \psi(2^{-j\delta}D)\gamma.$  We set
\begin{equation}\label{Ldelta}
L_\delta = \partial_t +\mez( S_{j\delta}(V)\cdot \nabla + \nabla  \cdot  S_{\delta j}(V)) +i T_{\gamma_\delta}.
\end{equation}
Then we have
\begin{equation}\label{Ldeltau}
L_\delta \Delta_ju = F_j
\end{equation}
where 
  \begin{multline}\label{Fj}
F_j = \Delta_jf +[T_V, \Delta_j]\cdot \nabla u + i [T_\gamma, \Delta_j]u +R_j u + R'_ju + \\ \mez \big\{ \big(S_{j\delta}(V)  - S_j(V)\big)\cdot\nabla \Delta_ju +\nabla \cdot \big(S_{j\delta}(V)  - S_j(V)\big) \Delta_ju\big\}
  +i\big(T_{\gamma_\delta} - T_\gamma \big)\Delta_j u.
\end{multline}
  Then, for all fixed $t,$ the function $F_j(t,\cdot)$ 
has its spectrum contained in a ball $B(0, C2^j)$.  

\section{Straightening the vector field}

We want to straighten the vector field $\partial_t + S_{j\delta}(V)\cdot \nabla.$
Consider the system of differential equations
\begin{equation}\label{eqdif}
\left \{
\begin{aligned}
&\dot{X}_k(s) =  S_{j\delta}(V_k)(s,X(s)), \quad 1 \leq k \leq d,\quad X=(X_1,\ldots,X_d) \\
&X_k(0) =x_k.
\end{aligned}
\right.
\end{equation}
For $k=1,\ldots,d$ we have $S_{j\delta}(V_k) \in L^\infty(I;H^\infty(\xR^d))$ and
$$
\vert S_{j\delta}(V_k)(s,x) \vert \leq C\Vert V_k \Vert_{L^\infty(I\times \xR^d)} \quad   \forall (s,x) \in I \times \xR^d.
$$
Therefore System \eqref{eqdif} has a unique solution defined on $I$ 
which will be denoted $X(s; x,h) (h=2^{-j})$ or sometimes simply $X(s).$

 We shall set 
 \begin{equation}\label{E0}
 E_0 = L^p(I; W^{1,\infty}(\xR^d))^d
 \end{equation}
  where $p= 4$ if $d=1$, $p=2$ if $d \geq 2$, endowed with its natural norm.
\begin{prop}\label{estX}
For fixed $(s,h)$ the map $x \mapsto X(s;x,h)$ 
belongs to $C^\infty (\xR^d, \xR^d).$ Moreover there exist functions $\mathcal{F}, \mathcal{F}_\alpha: \xR^+ \to \xR^+$ such that  
\begin{equation*}
\begin{aligned}
&(i) \quad &&\Big\Vert \frac{\partial X}{\partial x}(s; \cdot, h) - Id 
\Big\Vert_{L^\infty(\xR^d)} \leq \mathcal{F}(\Vert V \Vert_{E_0})\vert s \vert^\mez, \\  
&(ii)\quad  &&\Vert (\partial_x^\alpha X) (s; \cdot, h)  \Vert_{L^\infty(\xR^d)} 
\leq \mathcal{F}_\alpha(\Vert V \Vert_{E_0}) h^{-\delta(\vert \alpha \vert -1)} \vert s \vert^\mez, \quad    \vert \alpha \vert \geq 2 \\
\end{aligned}
\end{equation*}
for   all $(s,h) \in I\times (0, h_0]$.
\end{prop}
\begin{proof}
To prove $(i)$ we differentiate the system \label{eqdiff} with respect to $x_l$ and we obtain
\begin{equation*}
\left\{
\begin{aligned}
&\dot{\frac{\partial X_k}{\partial x_l}}(s) = 
 \sum_{q = 1}^d S_{j\delta}\Big (\frac{\partial V_k}{\partial x_q}\Big)(s,X(s))\frac{\partial X_q}{\partial x_l}(s)\\
&\frac{\partial X_k}{\partial x_l}(0) = \delta_{kl}
\end{aligned}
\right. 
\end{equation*}
from which we deduce 
\begin{equation}\label{dX}
\frac{\partial X_k}{\partial x_l}(s) = \delta_{kl} 
+ \int_0^s  \sum_{q = 1}^d S_{j\delta}
\Big (\frac{\partial V_k}{\partial x_q}\Big)(\sigma,X(\sigma))\frac{\partial X_q}{\partial x_l}(\sigma) \, d \sigma.
\end{equation}
Setting 
$\vert \nabla X\vert = \sum_{k,l =1}^d \vert \frac{\partial X_k}{\partial x_l} \vert$ 
we obtain from \eqref{dX}
$$
\vert \nabla X(s) \vert \leq C_d + \int_0^s \vert \nabla V(\sigma, X(\sigma)) \vert  \, \vert \nabla X(\sigma) \vert \, d \sigma.
$$
The Gronwall inequality implies that
\begin{equation}\label{Gron}
 \vert \nabla X(s) \vert \leq \mathcal{F} (\Vert V \Vert_{E_0}) \quad \forall s \in I. 
 \end{equation}
Coming back to \eqref{dX} and using \eqref{Gron} we can write
$$
\la \frac{\partial X }{\partial x}(s) - Id \ra \leq  \mathcal{F} (\Vert V \Vert_{E_0}) \int_0^s \Vert \nabla V(\sigma, \cdot)  \Vert_{L^\infty(\xR^d)} 
\, d\sigma \leq \mathcal{F}_1 (\Vert V \Vert_{E_0})\vert s \vert^\mez.
$$
Notice that in dimension one we have used the inequality $\vert s \vert^{\frac{3}{4}} \leq C\vert s \vert^\mez $ when $s\in I.$ 

To prove $(ii)$ we shall show by induction on $\vert \alpha \vert$ that the estimate
$$
\Vert(\partial_x^\alpha X)(s;\cdot, h)\Vert_{L^\infty(\xR^d)} \leq \mathcal{F}_\alpha(\Vert V \Vert_{E_0}) h^{-\delta(\vert \alpha \vert -1)}
$$
for  $1 \leq \vert \alpha \vert \leq k$ implies $(ii)$ for  $\vert \alpha \vert = k+1.$ 
The above estimate is true for  $\vert \alpha \vert =1$ by $(i)$. 
Let us differentiate  $\vert \alpha \vert$ times the system \eqref{eqdif}. 
We obtain
\begin{equation}\label{rec1}
 \frac{d}{ds}\big(\partial_x^\alpha X\big)(s) = S_{j\delta}(\nabla V)(s,X(s)) \partial_x^\alpha X +(1) 
 \end{equation}
where the term $(1)$ is a finite linear combination of terms of the form
$$
A_\beta (s,x) = \partial_x^\beta \big(S_{j\delta}(V)\big)(s,X(s)) \prod_{i=1}^q \big(\partial_x^{L_i} X(s)\big)^{K_i}
$$
where 
$$
2 \leq \vert \beta \vert \leq \vert \alpha \vert, \quad  1 
\leq q \leq \vert \alpha \vert, \quad \sum_{i=1}^q \vert K_i \vert L_i = \alpha, \quad \sum_{i=1}^q K_i = \beta.
$$
Then  $1 \leq \vert L_i \vert < \vert \alpha \vert$ 
which allows us to use the induction. Therefore we can write
\begin{equation*}
\begin{aligned}
\Vert A_\beta (s,\cdot) \Vert_{L^\infty(\xR^d)} 
&\leq \big\Vert \partial_x^\beta \big(S_{j\delta}(V)\big)(s,\cdot)\big\Vert_{L^\infty(\xR^d)} \prod_{i=1}^q 
\Big\Vert \partial_x^{L_i} X(s, \cdot)\Big\Vert _{L^\infty(\xR^d)}^{\vert K_i \vert}\\[1ex]
&\leq C h^{-\delta (\vert \beta \vert -1)}\Vert V(s,\cdot) \Vert_{W^{1, \infty}(\xR^d)} 
h^{-\delta \sum_{i=1}^q \vert K_i\vert(\vert L_i \vert -1)}\mathcal{F} (\Vert V \Vert_{E_0}) \\[1ex]
&\leq \mathcal{F} (\Vert V \Vert_{E_0})h^{-\delta(\vert \alpha \vert -1)}\Vert V(s,\cdot) \Vert_{W^{1, \infty}(\xR^d)}.
\end{aligned}
\end{equation*}
It follows then from \eqref{rec1} that
\begin{equation}
\begin{aligned}
\vert \partial_x^\alpha X(s) \vert &\leq   \mathcal{F} (\Vert V \Vert_{E_0})h^{-\delta(\vert \alpha \vert -1)} 
\int_0^s\Vert V(\sigma,\cdot) \Vert_{W^{1, \infty}(\xR^d)} \,d\sigma\\
&\quad + C \int_0^s\Vert V(\sigma,\cdot) \Vert_{W^{1, \infty} (\xR^d)}\vert \partial_x^\alpha X(\sigma) \vert \, d\sigma.
\end{aligned}
\end{equation}
The H\" older and Gronwall inequalities imply immediately $(ii)$.
\end{proof}

\begin{coro}
There exist  $T_0>0,h_0>0 $ such that for $t\in [0,T_0]$ and $0<h \leq h_0$ the map $x\mapsto X(t; x,h)$ 
from $\xR^d$ to $\xR^d$ is a $C^\infty$ diffeomorphism.
\end{coro}
\begin{proof}
This follows from a result by Hadamard (see~\cite{Be}). 
Indeed if $T_0$ is small enough, Proposition~\ref{estX} 
shows that the matrix $\big( \frac{\partial X_k}{\partial x_j}(t; x,h) \big)$ 
is invertible. On the other hand since
$$
\vert X(t; x,h) - x \vert \leq
\int_0^t \vert S_{j\delta}(V)(\sigma, X(\sigma))\vert \, d \sigma 
\leq T_0 \Vert V \Vert_{L^\infty([0,T_0]\times \xR^d)}
$$
we see that the map  $x\mapsto X(t; x,h)$ is proper.
\end{proof}
  \section{Reduction to a semi-classical form}
Now according to \eqref{pseudo} and \eqref{Ldeltau} we see that the function 
$ {U}\defn \Delta_ju$ is a solution of the equation 
$$ \big( \partial_t +\mez \big\{S_{j\delta}(V)\cdot \nabla + \nabla  \cdot S_{j\delta}(V)\big\} +i  \sigma_{\gamma_\delta}(t,x, D_x)\varphi_1(hD_x)\big)
 {U}(t,x) =F_j(t,x), \quad h=2^{-j}
$$
where 
\begin{equation}\label{phi1}
\varphi_1 \in C^\infty(\xR^d), \quad   
 \supp \varphi_1 \subset \{ \xi: \frac{1}{4} \leq \vert \xi \vert \leq 4\}, \quad  \varphi_1 = 1 \text{ on }  \{ \xi: \frac{1}{3} \leq \vert \xi \vert \leq 3\}  
\end{equation}
and $F_j$ has been defined in \eqref{Fj}. Notice that 
$$\mez \big(S_{j\delta}(V)\cdot \nabla + \nabla  \cdot S_{j\delta}(V)\big\} = S_{j\delta}(V)\cdot \nabla + \mez S_{j\delta}(\text{div}V).$$
We shall set 
\begin{equation}\label{a}
a(t,x,h,\xi) =  \sigma_{\gamma_\delta}(t,x, \xi)\varphi_1(h\xi).
\end{equation}
We make now the change of variable $x = X(t;y,h).$ Let us set
\begin{equation}
v_h(t,y) = {U}(t,X(t;y,h)) \quad t \in [0,T_0].
\end{equation}
Then it follows from \eqref{eqdif} that
\begin{equation}\label{Dv}
\partial_t v_h(t,y) = -i \big( a(t,x,h,D_x) {U} \big) (t, X(t;y,h)) + F_j(t,X(t;y,h)).
\end{equation}
Our next purpose is to give another expression to the quantity
\begin{equation}\label{A}
A= \big(a(t,x,h,D_x) {U} \big) (t, X(t;y,h)).
\end{equation}
In  the computation below, $t\in [0,T_0]$ and $h$ being  fixed, we will omit them. We have 
$$
A = (2\pi)^{-d} \iint e^{i(X(y)-x')\cdot \eta } a(X(y), \eta) {U}(x') \,dx'\, d\eta.
$$
{\bf Notations}: we set
\begin{equation}\label{notation}
\begin{aligned}
H(y,y') &= \int_0^1 \frac{\partial X}{\partial x}(\lambda y +(1-\lambda) y') \, d\lambda, 
\quad M(y,y') =\big(\, {}\!^tH(y,y') \big)^{-1} \\
M_0(y)  &= \Big(\, {}\!^t  \Big(\frac{\partial X}{\partial x}(y)\Big)  \Big)^{-1}, 
\quad  J(y,y') = \Big\vert \det \Big(\frac{\partial X}{\partial x}(y')\Big)\Big\vert \vert \det M(y,y') \vert.
\end{aligned}
\end{equation}
Let us remark that $M,M_0$ are well defined by Proposition~\ref{estX}. 
Moreover $M_0(y) = M(y,y)$ and $J(y,y) =1.$

Now, in the integral defining $A,$ we make the change of variables $x' =X(y')$.  
Then using the equality $X(y) - X(y') = H(y,y')(y-y')$ and setting $\eta = M(y,y')\zeta$ we obtain,
$$
A = (2\pi)^{-d} \iint e^{i (y-y')\cdot \zeta } a\big(X(y), M(y,y')\zeta\big) J(y,y') v_h(y') \,dy' \,d\zeta.
$$
Now we set
\begin{equation}\label{htilde}
z = h^{-\mez} y, \quad w_h(z) = v_h(h^\mez z), \quad \htilde  = h^\mez.
\end{equation}
Then
$$A = (2\pi)^{-d} \iint e^{i (\htilde z-y')\cdot \zeta } 
a\big(X(\htilde z), M\big(\htilde z,y'\big)\zeta\big) J\big(\htilde z,y'\big) v_h(y') \, dy' d\zeta.$$
Then setting $y' = \htilde z'$ and $\htilde \zeta = \zeta'$ 
we obtain
\begin{equation}\label{A1}
 A = (2\pi)^{-d} \iint e^{i (z-z')\cdot \zeta'  } 
 a\big(X(\htilde x), M\big(\htilde z,\htilde z'\big)\htilde ^{-1}\zeta'\big)
 J\big(\htilde z,\htilde z'\big) w_h(z') \,dz' \,d\zeta'.
 \end{equation}
Our aim is to reduce ourselves to a semi-classical form, after multiplying 
the equation by $\htilde .$ However this not straightforward since the symbol $a$ 
is not homogeneous in $\xi$ although $\gamma$ is homogeneous of order $\mez.$ We proceed as follows.
 
First of all on the support of the function $\varphi_1$ 
in the definition of $a$ (see \eqref{a}) the function $\psi_0$ 
appearing in the definition of $\sigma_{\gamma_\delta}$ (see \eqref{pseudo}) is equal to one. 
 
Therefore we can write for $X\in \xR^d, \rho\in \xR^d$,  (skipping the variable $t$),
\begin{equation*} 
\begin{aligned}
\sigma_{\gamma_\delta}(X, \htilde ^{-1}\rho) 
&=\int e^{iX\cdot \zeta}\chi\big(\zeta, \htilde ^{-1}\rho\big)
\widehat{\gamma}_\delta(\zeta, \htilde ^{-1}\rho) \, d \zeta\\
&= \iint e^{i(X-y)\cdot\zeta} \chi
\big(\zeta, \htilde ^{-1}\rho\big) {\gamma}_\delta(y, \htilde ^{-1}\rho)\,dy \,d\zeta\\
 &=\int \widehat{\chi}\big(\mu, \htilde ^{-1}\rho\big){\gamma}_\delta
 \big(X-\mu, \htilde ^{-1}\rho\big)\,d\mu.
\end{aligned}
\end{equation*}
Now since $\chi$ is homogeneous of degree zero we have,
$$\widehat{\chi}(\mu, \lambda \eta) = \lambda^d \widehat \chi(\lambda \mu, \eta), $$
which follows from the fact that $\chi(\zeta, \lambda \eta) = \chi(\lambda\lambda^{-1}\zeta, \lambda \eta) = \chi(\lambda^{-1}\zeta,   \eta).$

Applying this equality with $\lambda = \htilde ^{-2}$ and $\eta = \htilde  \rho$ we obtain,
\begin{equation*} 
\begin{aligned}
 \sigma_{\gamma_\delta}(X, \htilde ^{-1}\rho)
 &= \htilde ^{-2d}\int \widehat{\chi}(\htilde ^{-2} \mu, \htilde  \rho){\gamma}_\delta(X-\mu, \htilde ^{-1}\rho)\,d\mu\\
 &= \int \widehat{\chi}(\mu' , \htilde  \rho){\gamma}_\delta(X-\htilde ^{2}\mu'  , \htilde ^{-1}\rho)
 \,d\mu'. 
 \end{aligned}
\end{equation*}
Using the fact that $\gamma$ and $\gamma_\delta$ are  homogeneous of order $\mez$ in $\xi$  we obtain 
$$ \htilde  \sigma_{\gamma_\delta}(X, \htilde ^{-1}\rho) 
= \int \widehat{\chi}(\mu  , \htilde  \rho){\gamma}_\delta(X-\htilde ^{2}\mu   , \htilde  \rho)\, d\mu
$$ 
and since $\htilde ^{-1}h=\htilde $, we obtain
$$
\htilde  a(X,\htilde ^{-1}\rho) =  \Big(\int \widehat{\chi}(\mu  , \htilde  \rho){\gamma}_\delta(X-\htilde ^{2}\mu,\htilde\rho)\,d\mu\Big) 
\varphi_1(\htilde \rho).
$$
It follows then from \eqref{Dv},\eqref{A},\eqref{A1} that the function $w_h$ defined in \eqref{htilde} 
is solution of the equation
\begin{equation}\label{Eqfin}
(\htilde \partial_t + \htilde c+iP)w_h(t,z) = \htilde F_j(t,X(t,\htilde z,h))
\end{equation}
where $c(t,z, \htilde) =  \mez S_{j\delta}(\text{div}V)(t, X(t,\htilde z))$ and 
\begin{equation}\label{operateur}
Pw(t,z) = (2\pi \htilde )^{-d}\iint e^{i{\htilde}^{-1}(z-z')\cdot\zeta}
\widetilde{p}(t,z,z',\zeta,\htilde )w(t,z') \, dz' \,d\zeta
\end{equation}
with
\begin{equation}\label{ptilde}
\begin{aligned}
 \widetilde{p}(t,z,z',\zeta,\htilde ) =    \int \widehat{\chi}\big(\mu  , &M(t,\htilde z,\htilde z')\zeta\big)
 {\gamma}_\delta\big (t, X(t,\htilde z)-\htilde ^{2}\mu, M(t,\htilde z,\htilde z')\zeta\big)d\mu \\
 & \times \varphi_1( M(t,\htilde z,\htilde z')\zeta)J(t, \htilde z,\htilde z').
 \end{aligned}
\end{equation}
We shall set in what follows
\begin{equation*}
p(t,z,\zeta,\htilde ) = \widetilde{p}(t,z,z,\zeta,\htilde ).
\end{equation*}
Since $M(t,\htilde z,\htilde z)= M_0(t,\htilde z)$ and $J(t,\htilde z,\htilde z) = 1$ we obtain
\begin{equation}\label{encorep}
p(t,z,\zeta,\htilde ) =    \int \widehat{\chi}\big(\mu  , M_0(t, \htilde z)\zeta\big)
{\gamma}_\delta\big (t, X(t,\htilde z)-\htilde ^{2}\mu, M_0(t,\htilde z) \zeta\big)\,d\mu \\
\cdot\varphi_1( M_0(t,\htilde z)\zeta).
\end{equation}

Notice that since the function $\chi$ is even with respect to its first variable the symbol $p$ is real.

Summing up we have proved that
\begin{equation}\label{semiclass}
   \htilde \big\{ \partial_t +\mez \big( S_{j\delta}(V)\cdot \nabla  + \nabla  \cdot S_{j\delta}(V)\big) +i  T_{\gamma_\delta} \big\}{U}(t,x) = \big(\htilde \partial_t + \htilde c + iP\big)w_h(t,z)
  \end{equation}
where 
\begin{equation}\label{c=}
x = X(t, \htilde z), \quad c(t,z, \htilde) =  \mez S_{j\delta}(\text{div}V)(t, X(t,\htilde z)), \quad w_h(t,z) = U(t, X(t, \htilde z)).
\end{equation}

 \subsection{ Estimates on the pseudo-differential symbol} 
 Let $I_{\htilde }:=[0,\htilde ^{ \delta}]$. We introduce   norms on the para-differential symbol $\gamma.$ For $k \in \xN$ we set
  \begin{equation}\label{norme:gamma}
\mathcal{N}_k(\gamma) = \sum_{\vert \beta  \vert \leq k} \Vert D_\xi^\beta \gamma\Vert_{L^\infty(I_{\htilde}\times \xR^d \times \mathcal{C}_3)}.  
\end{equation}

We   estimate now the derivatives of the  symbol of the operator  appearing in the right hand side of \eqref{semiclass}. Recall that $E_0$ is defined in \eqref{E0}.
\begin{lemm}\label{est:c}
For any $\alpha \in \xN^d$ there exists $\mathcal{F}_\alpha: \xR^+ \to \xR^+$ such that for $t\in I_{\htilde}$
$$\Vert (D_z^\alpha c)(t, \cdot)\Vert_{L^\infty(\xR^d)}  \leq \mathcal{F}_\alpha (\Vert V \Vert_{E_0}) \htilde^{\vert \alpha \vert (1-2 \delta)}\Vert V(t, \cdot) \Vert_{W^{1, \infty}(\xR^d)}.$$
\end{lemm}
\begin{proof}
By the Faa-di-Bruno formula $D_z^\alpha c$ is a finite linear combination of terms of the form 
\begin{equation}\label{(1)=}
 (1) = \htilde^{\vert \alpha \vert} D_x^a \big[S_{j\delta}(\text{div}V)\big](t, X(t,\htilde z)) \prod_{j=1}^r \Big(\big(D_x^{l_j}X\big) (t,\htilde z)\Big)^{p_j} 
 \end{equation}
where $1\leq \vert a \vert \leq \vert \alpha \vert, \quad \vert l_j \vert \geq 1, \quad \sum_{j=1}^r \vert p_j\vert l_j = \alpha, \quad \sum_{j=1}^r p_j = a.$ Now for fixed $t$ we have 
\begin{equation}\label{est:Sjdelta}
 \vert D_x^a \big[S_{j\delta}(\text{div}V)\big](t, \cdot)\vert \leq C_{\alpha} \htilde^{-2 \delta \vert a \vert}\Vert V(t, \cdot) \Vert_{W^{1, \infty}(\xR^d)}
 \end{equation}
and by Proposition \ref{estX}
$$\vert (D_x^{l_j}X\big) (t, \cdot)\vert \leq   \mathcal{F}_\alpha (\Vert V \Vert_{E_0}) \htilde^{-2\delta (\vert l_j \vert -1)}.$$
Then the product appearing in the term $(1)$ is bounded by $ \mathcal{F}_\alpha (\Vert V \Vert_{E_0}) \htilde ^M$  where $ M= -2 \delta\sum_{j=1}^r \vert p_j \vert  (\vert l_j \vert -1)  = -2 \delta \vert \alpha \vert + 2 \delta \vert a \vert$. Then the lemma follows from \eqref{(1)=} and \eqref{est:Sjdelta}.

\end{proof}

\begin{lemm}\label{derp}
For every $k \in \xN$   there exist $ \mathcal{F}_{k}: \xR^+ \to \xR^+$ such that,  
 $$\vert  D_z ^\alpha D_\zeta^\beta p(t,z,\zeta,\htilde )\vert \leq  \mathcal{F}_{ k}(\Vert V \Vert_{E_0})\sum_{\vert a \vert \leq   k}  \sup_{\xi \in \mathcal{C}_3} \Vert D^a_\xi \gamma(t,\cdot,\xi)\Vert_{L^{\infty}(\xR^d)} \htilde^{\vert \alpha \vert(1-2 \delta)} 
 $$
 for   all $\vert \alpha\vert + \vert \beta \vert \leq k$  and all $(t,z, \zeta,\htilde) \in I_{\htilde }\times \xR^d \times  \mathcal{C}_1\times (0, \htilde _0]$.
 \end{lemm}
 \begin{coro}\label{derp'}
For every $k \in \xN$   there exist $ \mathcal{F}_{k}: \xR^+ \to \xR^+$ such that, 
$$\int_0^s \vert  D_z ^\alpha D_\zeta^\beta p(t,z,\zeta,\htilde )\vert \, dt \leq   \mathcal{F}_{ k}(\Vert V \Vert_{E_0})\,  \mathcal{N}_k(\gamma) \, \htilde^{\vert \alpha \vert(1-2 \delta) + \delta} 
 $$
  for all $\vert \alpha\vert + \vert \beta \vert \leq k$  and all $(s,z, \zeta,\htilde) \in I_{\htilde }\times \xR^d \times  \mathcal{C}_1\times (0, \htilde _0]$.
 \end{coro}

\begin{proof}[Proof of Lemma \ref{derp}]
Here $t$ is considered as a parameter 
which will be skipped, keeping in 
mind that the estimates should be uniform with respect to $t \in [0, \htilde^\delta].$
On the other hand we recall that, by Proposition \ref{estX} and Lemma \ref{estgamma-delta}, we have (since $h= \widetilde{h}^2$)
  \begin{align}
 & \vert D_x^\alpha X(x) \vert \leq  \mathcal{F}_\alpha(\Vert V \Vert_{E_0})\, \widetilde{h}^{-2 \delta(\vert \alpha \vert -1)}, \quad \vert \alpha \vert \geq 1, \beta \in \xN^d \label{est2} \\
 & \vert  D_x^\alpha D_\xi^\beta \gamma_\delta( x,\xi)\vert \leq C_{\alpha, \beta}\,\htilde^{-2\delta \vert \alpha \vert}  \Vert D_\xi^\beta \gamma( \cdot, \xi) \Vert_{L^\infty(\xR^d)}, \quad \alpha, \beta \in \xN^d.\label{est3} 
 \end{align}
Set
$$
F(\mu, z, \zeta, \htilde) = \widehat{\chi}\big(\mu  , M_0(z)\zeta\big) \varphi_1( M_0(z)\zeta){\gamma}_\delta\big ( X(z)-\htilde ^{2}\mu, M_0(z) \zeta\big),  
$$
the lemma will follow immediately from the fact that for every $N \in \xN$ we have
\begin{multline}\label{est:F}
  \vert  D_z ^\alpha D_\zeta^\beta F(\mu, z, \zeta, \htilde) \vert \\ \leq \mathcal{F}_{ \alpha,\beta}(\Vert V \Vert_{E_0})\sum_{\vert a \vert \leq   \vert \alpha\vert+ \vert \beta \vert}  \sup_{\xi \in \mathcal{C}_3} \Vert D^a_\xi \gamma(\cdot,\xi)\Vert_{L^{\infty}(\xR^d)} \htilde^{  -2\delta \vert\alpha \vert }C_N \langle \mu \rangle^{-N}.
\end{multline}
If we call $ m_{ij}(z) $ the entries of the matrix $M_0(z)$ we see easily that $ D_\zeta^\beta F$ is a finite linear combination of terms of the form 
\begin{equation}\label{DbetaF}
  (D_\xi^{\beta_1}(\widehat{\chi} \varphi_1)) (\mu, M_0(z)\zeta) \cdot (D_\xi^{\beta_2}\gamma_\delta)(X(z)-\htilde ^{2}\mu, M_0(z) \zeta)\cdot P_{\vert \beta\vert}(m_{ij}(z)) := G_1 \cdot G_2 \cdot G_3 
  \end{equation}
where $P_{\vert \beta \vert}$ is a polynomial of order $\vert \beta \vert.$

The estimate \eqref{est:F} will follow from the following ones.
\begin{align} 
&\vert D^\alpha_z G_1  \vert \leq \mathcal{F}_{\alpha, \beta} (\Vert V \Vert_{E_0})\widetilde{h}^{- 2\delta \vert \alpha \vert} C_N \langle \mu \rangle^{-N},\label{est:G1}\\ 
 &\vert D^\alpha_z G_2  \vert \leq \mathcal{F}_{\alpha, \beta} (\Vert V \Vert_{E_0}))\sum_{\vert a \vert \leq   \vert \alpha\vert+ \vert \beta \vert}  \sup_{\xi \in \mathcal{C}_3} \Vert D^a_\xi \gamma(\cdot,\xi)\Vert_{L^{\infty}(\xR^d)} \htilde^{-2\delta \vert\alpha \vert } \label{est:G2}\\
 &\vert D^\alpha_z G_3  \vert \leq \mathcal{F}_{\alpha, \beta} (\Vert V \Vert_{E_0})\widetilde{h}^{- 2\delta \vert \alpha \vert}. \label{est:G3}
\end{align}
 Using the  equality $^t{}\! \big(\frac{\partial X}{\partial z}\big)( z) M_0(  z) \zeta = \zeta,$  Proposition \ref{estX} and an induction we see that
\begin{equation}\label{est1}
\vert D^\alpha_z  m_{ij}(z) \vert \leq \mathcal{F}_\alpha (\Vert V \Vert_{E_0})\widetilde{h}^{- 2\delta \vert \alpha \vert    + \frac{\delta}{2}} 
\end{equation}
from which  \eqref{est:G3} follows since $G_3$ is  polynomial.
Now according to the Faa-di-Bruno formula $D_z^\alpha G_1$ is a finite linear combination of terms of the form 
\begin{align*}
&D_\xi^{\beta_1 + b}(\widehat{\chi} \varphi_1)) (\mu, M_0(z)\zeta)\prod_{j=1}^r \Big(D_z^{l_j}M_0(z)\zeta)\Big)^{p_j}, \\
&1\leq \vert b \vert \leq \vert \alpha \vert, \vert l_j \vert \geq 1, \sum_{j=1}^r \vert p_j \vert l_j = \alpha, \sum _{j=1}^r p_j =b. 
\end{align*}
 Then \eqref{est:G1} follows immediately from Lemma \ref{l0.6} and \eqref{est1}. By the same formula we see that $D_z^\alpha G_2$ is a linear combination of terms of the form
 $$(D^a_z D_\xi^{\beta_2+b}\gamma_\delta)(X(z)-\htilde ^{2}\mu, M_0(z) \zeta)\prod_{j=1}^r \Big(D_z^{l_j}X(z)\Big)^{p_j}\Big( D_z^{l_j}M_0(z)\zeta\Big)^{q_j}$$
 where $1 \leq \vert a \vert + \vert b \vert \leq \vert \alpha \vert, \quad \sum_{j=1}^r (\vert p_j \vert + \vert q_j \vert) l_j = \alpha, \quad \sum_{j=1}^r p_j = a,  \quad\sum_{j=1}^r q_j = b.$
 Then \eqref{est:G2} follows  \eqref{est2}, \eqref{est3} and \eqref{est1}. The proof is complete.
      \end{proof}
\begin{rema}\label{rem:est:ptilde}
By exactly the same method we can show that we have the estimate
\begin{equation}\label{est:ptilde}
 \vert  D_z ^{\alpha_1}   D_{z'} ^{\alpha_2}  D_\zeta^\beta \widetilde{p}(t,z,z',\zeta,\htilde )\vert \leq  \mathcal{F}_{ k}(\Vert V \Vert_{E_0} \, \mathcal{N}_k(\gamma)) \, \htilde^{ (\vert \alpha_1 \vert + \vert \alpha_2\vert) (1-2 \delta)} 
\end{equation} 
 for   all $\vert \alpha_1\vert + \vert \alpha_2\vert +\vert \beta \vert \leq k$  and all $(t,z,z', \zeta,\htilde) \in I_{\htilde }\times \xR^d  \times \xR^d \times  \mathcal{C}_1\times (0, \htilde _0]$.
  \end{rema}
 
\begin{prop}\label{Hessp}
There exist $T_0>0, c_0>0, \htilde _0>0$ such that
$$
\la \det \Big( \frac{\partial^{2}p}{\partial \zeta_j \partial \zeta_k} (t,z,\zeta,\htilde ) \Big) \ra \geq c_0
$$
for any $t\in [0,T_0], z\in \xR^d, \zeta \in \mathcal{C}_0= \{\mez \leq \vert \zeta \vert \leq 2\}, 0<\htilde  \leq \htilde _0.$  
\end{prop}
\begin{proof}
By Proposition \ref{estX}, \eqref{notation} and  \eqref{phi1} we have $\varphi_1( M_0(t,\htilde z)\zeta) =1.$ 
Let us set
$$
M_0(t,\htilde z) =(m_{ij}) \text{ and }
M_0(t,\htilde z) \zeta= \rho.
$$
Then we have,
$$
\frac{\partial^{2}p}{\partial \zeta_j \partial \zeta_k} (t,z,\zeta,\htilde ) = A_1 +A_2+A_3,
$$
where
\begin{equation}
\begin{aligned}
A_1&=\sum_{l,r=1}^d\int  \frac{\partial^{2}\widehat{\chi}}{\partial \zeta_l \partial\zeta_r}(\mu,\rho) 
m_{lj}m_{rk} {\gamma}_\delta\big (t, X(t,\htilde z)-\htilde ^{2}\mu, \rho\big)\,d\mu,\\
A_2&= 2\sum_{l,r=1}^d\int  \frac{\partial\widehat{\chi}}{\partial \zeta_l }(\mu,\rho)  
\frac{\partial {\gamma_\delta}}{\partial \zeta_r}
\big (t, X(t,\htilde z)-\htilde ^{2}\mu, \rho\big)m_{lj}m_{rk}\,d\mu,\\
A_3&= \sum_{l,r=1}^d\int \widehat{\chi}(\mu, \rho) \frac{\partial^{2}\gamma_\delta}{\partial \zeta_l \partial\zeta_r}
\big (t, X(t,\htilde z)-\htilde ^{2}\mu, \rho\big)m_{lj}m_{rk} \, d\mu.
\end{aligned}
\end{equation}
Now we notice that by \eqref{chi} we have 
$$\int (\partial_\zeta^\alpha \widehat{\chi})(\mu, \rho) d\mu = (2\pi)^d  (\partial_\zeta^\alpha  {\chi})(0,\rho)= \left \{ \begin{array}{ll} 0, & \alpha \neq 0\\   (2\pi)^d, & \alpha =0. \end{array} \right.$$
Using this remark we can write
$$
 A_1 =\sum_{l,r=1}^d\int  \frac{\partial^{2}\widehat{\chi}}{\partial \zeta_l \partial \zeta_r}(\mu,\rho) 
 m_{lj}m_{rk}
 \Big[{\gamma}_\delta\big (t, X(t,\htilde z)-\htilde ^{2}\mu, \rho\big) - {\gamma}_\delta\big (t, X(t,\htilde z) , \rho\big)\Big]\, d\mu.
$$
Now recall (see \eqref{EFG} and Lemma \ref{estgamma}) that for bounded $\vert \zeta\vert $ (considered as a parameter) 
we have for all $\alpha\in \xN^d$
$$
\partial_\zeta ^\alpha \gamma_\delta \in L^\infty(I; H^{s-\mez}(\xR^d))
\subset L^\infty (I; W^{s_0,\infty}(\xR^d)), \quad s_0>0,
$$
uniformly in~$\zeta$. 
Since, by Proposition~\ref{estX},  $\Vert M_0(t, \htilde z)\Vert$ is uniformly bounded we can write
$$
\vert A_1 \vert \leq C \htilde ^{2s_0}\sum_{l,r=1}^d\int \vert \mu \vert^{s_0}
\la  \frac{\partial^{2}\widehat{\chi}}{\partial \zeta_l \partial \zeta_r}(\mu,\rho) \ra d\mu,
$$
the integral in the right hand side being bounded by Lemma~\ref{l0.6}.
 
By exactly the same argument we see that we have the following   inequality
$$
\vert A_2 \vert \leq C\htilde ^{2s_0}.
$$
Moreover one can write
$$
A_3 =  \sum_{l,r=1}^d\frac{\partial^{2}\gamma_\delta}{\partial \zeta_l \partial \zeta_r}\big (t, X(t,\htilde z) , \rho\big)
m_{lj}m_{rk} \Big(\int \widehat{\chi}(\mu, \rho) \, d\mu\Big) + \mathcal{O}(\htilde ^{2s_0}).
$$
Gathering the estimates we see that
$$
(2\pi)^{-d} \Big(\frac{\partial^{2}p}{\partial \zeta_j \partial \zeta_k} (t,z,\zeta,\htilde )\Big) 
=   {}\!^t\!M_0(t,\htilde z)\text{Hess}_\zeta (\gamma_\delta)(t,z,\zeta,\htilde )M_0(t, \htilde z) + \mathcal{O}(\htilde ^{2s_0}).
$$
Then our claim follows from Proposition~\ref{hess} 
and Proposition~\ref{estX} if $\htilde _0 $ is small enough.
\end{proof}

\section{The parametrix}

Our aim is to construct a parametrix for the operator 
$L = \htilde \partial_t +\htilde c+ iP$ on a time interval of size $\htilde ^\delta$  where $\delta = \frac{2}{3}.$  This parametrix will be of the following form
$$
\mathcal{K}v(t,z) = (2\pi\htilde )^{-d} \iint e^{i{\htilde}^{-1}(\phi(t,z,\xi,\htilde ) -y\cdot \xi)} \widetilde{b}(t,z,y,\xi,\htilde )v(y)\, dyd\xi.
$$
Here $\phi$  is a real valued phase such that 
$\phi\arrowvert _{t=0} = z\cdot \xi,$ $\widetilde{b} $ is  of the form
\begin{equation}\label{tildeb}
\widetilde{b}(t,z,y,\xi,\htilde) = b(t,z,\xi,\htilde ) \Psi_0\Big( \frac{\partial \phi}{\partial \xi}(t,z,\xi,\htilde ) -y\Big ) 
\end{equation}
where 
$b\arrowvert_{t=0} = \chi(\xi), \chi \in C_0^\infty(\xR^d\setminus\{0\})$ and 
$\Psi_0\in C_0^\infty(\xR^d)$ is such that $ \Psi_0(t) = 1$ if $\vert t \vert \leq 1$.
\subsection{Preliminaries}

An important step in this construction is to compute the expression
\begin{equation}\label{J}
J (t,z,y,\xi,\htilde) = e^{-i{\htilde}^{-1}\phi(t,z,\xi,\htilde )}P(t,z,D_z)\big(e^{i{\htilde}^{-1}\phi(t,\cdot,\xi,\htilde )}\widetilde{b}(t,z,y,\xi,\htilde) \big).
\end{equation}
In this computation since $(t,y, \xi,\htilde )$ are fixed we shall 
skip them and   write $\phi = \phi(z), \widetilde{b} = \widetilde{b} (z).$

Using \eqref{operateur} we obtain
$$
J= (2\pi\htilde )^{-d}\iint e^{i{\htilde}^{-1}(\phi(z') - \phi(z) 
+(z-z')\cdot \zeta)} \widetilde{p} (z,z',\zeta)\widetilde{b} (z') \, dz\, 'd\zeta.
$$
Then we write
\begin{equation}\label{teta}
\phi(z') - \phi(z) = \theta (z,z')\cdot(z'-z), \quad  \theta (z,z') 
= \int_0^1 \frac{\partial \phi}{\partial z}(\lambda z + (1-\lambda) z') \, d \lambda.
\end{equation}
Using this equality and setting $\zeta - \theta (z,z') = \eta$ in the integral we obtain
$$
J=(2\pi\htilde )^{-d}\iint e^{i{\htilde}^{-1}(z-z')\cdot\eta}\widetilde{p}(z,z',\eta + \theta(z,z')) \widetilde{b} (z') \, dz' d\eta.
$$
The phase that we will obtain will be uniformly bounded, 
say $\vert \frac{\partial \phi}{\partial z}\vert \leq C_0.$
It also can be seen that, due to the cut-off $\varphi_1$ in the expression of $\widetilde{p}$ and to Proposition~\ref{estX},  
we also have $\vert \eta +\theta(z,z') \vert \leq C_0$.  
Therefore $\vert \eta \vert \leq 2C_0.$ Let $\kappa \in C_0^\infty(\xR^d)$ be such that 
$\kappa(\eta) = 1$ if $\vert \eta \vert \leq 2C_0$. Then we can write
$$
J =(2\pi\htilde )^{-d}\iint e^{i{\htilde}^{-1}(z-z')\cdot\eta}\kappa(\eta)
\widetilde{p}(z,z',\eta + \theta(z,z')) \widetilde{b} (z') \, dz' \, d\eta.
$$
By the Taylor formula we can write
\begin{equation*}
\begin{aligned}
 &\widetilde{p}(z,z',\eta +\theta(z,z')) = 
 \sum_{\vert \alpha \vert \leq {N-1}} \frac{1}{\alpha!}(\partial^\alpha_\eta \widetilde{p})(z,z',\theta(z,z'))\eta^\alpha +r_N \\
 &r_N = \sum_{\vert \alpha \vert =N} \frac{N}{\alpha!}
 \int_0^1 (1-\lambda)^{N-1}(\partial^\alpha_\eta \widetilde{p})(z,z',\eta+ \lambda\theta(z,z'))\eta^\alpha \, d\lambda.
\end{aligned}
\end{equation*}
It follows that
\begin{equation}\label{J=}
\left\{
\begin{aligned}
&J =J_N + R_N\\
&J_N =  \sum_{\vert \alpha \vert \leq {N-1}} \frac{(2\pi\htilde )^{-d}}{\alpha!}
\iint e^{i{\htilde}^{-1}(z-z')\cdot\eta}\kappa(\eta)(\partial^\alpha_\eta \widetilde{p})(z,z',\theta(z,z'))\eta^\alpha \widetilde{b} (z') \, dz' d\eta \\
&R_N = (2\pi \htilde )^{-d} \iint e^{i{\htilde}^{-1}(z-z')\cdot\eta} 
\kappa(\eta) r_N(z,z',\eta) \widetilde{b} (z') \, dz' d\eta.
 \end{aligned}
 \right.
\end{equation}
Using the fact that $\eta^\alpha e^{i{\htilde}^{-1}(z-z')\cdot\eta}= (-\htilde D_{z'})^\alpha e^{i{\htilde}^{-1} (z-z')\cdot\eta}$ and integrating by parts in the integral with respect to $z$ we get
$$
J_N =  (2\pi\htilde )^{-d}\sum_{\vert \alpha \vert \leq {N-1}}\frac{\htilde ^{\vert \alpha \vert}}{\alpha!}
\iint e^{i{\htilde}^{-1}(z-z')\cdot\eta}\kappa(\eta)D_{z'}^\alpha
\big[(\partial^\alpha_\eta \widetilde{p})(z,z',\theta(z,z')) \widetilde{b} (z')\big] \, dz' d\eta.
$$
Therefore we can write
$$
J_N =  (2\pi\htilde )^{-d}\sum_{\vert \alpha \vert 
\leq {N-1}}\frac{\htilde ^{\vert \alpha \vert}}{\alpha!}
\int \widehat{\kappa}\big(\frac{z'-z}{\htilde }\big)
D_{z'}^\alpha\big[(\partial^\alpha_\eta \widetilde{p})(z,z',\theta(z,z')) \widetilde{b} (z')\big] \, dz'.
$$ 
Let us set
\begin{equation}\label{falpha}
f_\alpha(z,z',\htilde ) = D_{z'}^\alpha\big[(\partial^\alpha_\eta \widetilde{p})(z,z',\theta(z,z')) \widetilde{b} (z')\big] 
\end{equation}
and then, $z'-z = \htilde \mu$ in the integral. We obtain
$$
J_N = (2\pi)^{-d}\sum_{\vert \alpha \vert 
\leq {N-1}}\frac{\htilde ^{\vert \alpha \vert}}{\alpha!}
\int \widehat{\kappa}(\mu)f_\alpha(z,z+ \htilde \mu, \htilde )\, d\mu.
$$
By the Taylor formula we can write
$$
J_N = (2\pi)^{-d}\sum_{\vert \alpha \vert \leq {N-1}}
\frac{\htilde ^{\vert \alpha \vert}}{\alpha!}\sum_{\vert \beta \vert \leq {N-1}}
\frac{\htilde ^{\vert \beta \vert}}{\beta!}
\Big(\int \mu^\beta \widehat{\kappa}(\mu) \, d\mu\Big)\big(\partial_{z'}^\beta f_\alpha\big)(z,z,\htilde ) +S_N,
$$
with
\begin{equation}\label{SN}
S_N = (2\pi)^{-d}\sum_{\substack{\vert \alpha \vert \leq {N-1} \\ \vert \beta \vert=N}} 
N \frac{\htilde ^{\vert \alpha \vert+\vert \beta \vert}}{\alpha! \beta!} 
 \int \int_0^1 
 (1-\lambda)^{N-1}\mu^\beta \widehat{\kappa}(\mu) \big(\partial_{z'}^\beta f_\alpha\big) 
 (z,z+\lambda\htilde \mu,\htilde )  \, d\lambda \, d\mu.
\end{equation}
Noticing that
$$
\int \mu^\beta \widehat{\kappa}(\mu) \, d\mu = (2\pi)^d (D^\beta \kappa)(0) 
= \left \{ \begin{array}{ll} 0 & \text{if } \beta \neq 0 \\(2\pi)^d & \text{if } \beta =0 \end{array} \right.
$$
we conclude that
$$
J_N = \sum_{\vert \alpha \vert \leq {N-1}}\frac{\htilde ^{\vert \alpha \vert}}{\alpha!} f_\alpha(z,z,\htilde ) +S_N.
$$
It follows from \eqref{J=}, \eqref{falpha} and \eqref{SN} that 
\begin{equation}\label{Jfinal}
J =  \sum_{\vert \alpha \vert \leq {N-1}}\frac{\htilde ^{\vert \alpha \vert}}{\alpha!} 
D_{z'}^\alpha\big[(\partial^\alpha_\eta \widetilde{p})(z,z',\theta(z,z')) \widetilde{b} (z')\big]\arrowvert_{z'=z} +R_N +S_N
\end{equation}
where $R_N$ and $S_N$ are defined in \eqref{J=} and \eqref{SN}.

Reintroducing the variable $(t,y,\xi,\htilde)$ we conclude from \eqref{J} that 
\begin{equation}\label{conjugaison}
e^{-i{\htilde}^{-1}\phi(t,z,\xi,\htilde )}(\htilde  \partial_t +\htilde c+iP)\big(e^{i{\htilde}^{-1}\phi(t,z,\xi,\htilde )}\widetilde{b} \big)
=\Big[i\frac{\partial \phi}{\partial t} \widetilde{b}  + iJ + \htilde \frac{\partial \widetilde{b} }{\partial t} +\htilde c \widetilde{b}\Big]
\big(t,z,y,\xi,\htilde\big).
\end{equation}
We shall gather the terms the right hand side of \eqref{conjugaison} according to the power of $\htilde $. 
The term corresponding to  $\htilde ^0$ leads to the eikonal equation.

\subsection{The eikonal equation}
It is the equation
\begin{equation}\label{eikonale}
\frac{\partial \phi}{\partial t} + p\Big(t,z, \frac{\partial \phi}{\partial z}, \htilde \Big) =0 \quad \phi(0,z,\xi,\htilde ) = z \cdot \xi
\end{equation}
where $p$ is defined by the formula
\begin{equation}\label{p}
  p(t,z,\zeta,\htilde ) =    \int \widehat{\chi}\big(\mu  , M_0(t, \htilde z)\zeta\big)
{\gamma}_\delta\big (t, X(t,\htilde z)-\htilde ^{2}\mu, M_0(t,\htilde z) \zeta\big)  
\cdot\varphi_1( M_0(t,\htilde z)\zeta)\,d\mu.
\end{equation}
We set
 $$
q(t,z,\tau,\zeta,\htilde ) = \tau + p(t,z,\zeta,\htilde ) 
$$
and for $j\geq 1$ we denote by $\mathcal{C}_j$  the ring
$$
\mathcal{C}_j = \{\xi \in \xR^d: 2^{-j} \leq \vert \xi \vert \leq 2^j\}.
$$
Moreover in all what follows we shall have
\begin{equation}\label{delta}
  \delta = \frac{2}{3}.
  \end{equation} 
 
\subsubsection{The solution of the eikonal equation}
Recall that $I_{\htilde} = [0, \htilde^\delta] $ is the time interval, 
(where $\delta = \frac{2}{3} $) and $\mathcal{C}_j$ 
the ring $\{2^{-j} \leq \vert \xi \vert \leq 2^j\}$. 
Consider the null-bicharacteristic flow of $q$. 
It is defined by the system 
\begin{equation}\label{bicar}
\left\{
\begin{aligned}
\dot{t}(s) &=1,\quad t(0) =0,\\
\dot{z}(s) &= \frac{\partial p}{\partial \zeta}\big(t(s),z(s), \zeta(s), \htilde \big), \quad z(0) =z_0,\\
\dot{\tau}(s) &=- \frac{\partial p}{\partial t}\big(t(s),z(s),\zeta(s), \htilde \big), \quad \tau(0) = - p(0,z_0,\xi,\htilde ),\\
\dot{\zeta}(s) &=- \frac{\partial p}{\partial z}\big(t(s),z(s), \zeta(s), \htilde \big), \quad \zeta(0) =\xi.
\end{aligned}
\right.
\end{equation}
Then $t(s) =s$  and this system has a unique solution defined on $ I_{\htilde},$ depending on $(s,z_0,\xi,\tilde{h}).$

We claim that for all fixed $ s \in I_{\htilde }$ and $ \xi \in \xR^d,$ the map
\begin{equation*}
z_0 \mapsto z(s; z_0,\xi ,\htilde )
\end{equation*}
is a global diffeomorphism from $\xR^d$ to $\xR^d$. 
This will follow from the facts that this map is proper and the 
matrix $\big( \frac{\partial z}{\partial z_0}(s;z_0,\xi,\htilde )\big)$ is invertible. 
Let us begin by the second point. 

Let us set $m(s) =(s, z(s), \zeta(s),\htilde )$. 
Differentiating System \eqref{bicar} with respect to $z_0$, we get
\begin{equation}\label{Dbicar}
\begin{aligned}
\dot{\frac{\partial z}{\partial z_0}}(s) 
&= p''_{z \zeta}(m(s)) \frac{\partial z}{\partial z_0}(s) 
+ p''_{\zeta \zeta} (m(s))\frac{\partial \zeta}{\partial z_0}(s), \quad \frac{\partial z}{\partial z_0}(0)= Id\\
\dot{\frac{\partial \zeta}{\partial z_0}}(s) 
&= -p''_{zz}(m(s))\frac{\partial z}{\partial z_0}(s) -p''_{z\zeta}(m(s))\frac{\partial \zeta}{\partial z_0}(s), 
\quad \frac{\partial \zeta}{\partial z_0}(0)=0.
\end{aligned}
\end{equation}
Setting $U(s) =(\frac{\partial z}{\partial z_0}(s),\frac{\partial \zeta}{\partial z_0}(s))$ and 
\begin{equation}\label{matrice}
\mathcal{A}(s)=\begin{pmatrix} 
 \hphantom{-}p''_{z\zeta}(m(s))&    \hphantom{-} p''_{\zeta \zeta}(m(s))\\[1ex] -p''_{zz}(m(s)) &- p''_{z\zeta}(m(s))   
\end{pmatrix}
. 
\end{equation}
The system \eqref{Dbicar} can be written as $\dot{U}(s) = \mathcal{A}(s)U(s), U(0) = (Id,0).$ 
 Lemma \ref{derp} gives
\begin{multline}
  \vert  p''_{\zeta \zeta}(m(s))\vert  +    \vert  p''_{z \zeta}(m(s))\vert   +  \vert  p''_{zz}(m(s))\vert
  \\  \leq   \mathcal{F} (\Vert V \Vert_{E_0})
   \sum_{\vert \beta \vert \leq  2}  \sup_{\xi \in \mathcal{C}_3}\Vert D^\beta_\xi \gamma(t,\cdot,\xi)\Vert_{L^\infty (\xR^d)}  \htilde^{2(1-2\delta)}, \end{multline}
therefore
$$
\Vert \mathcal{A}(s) \Vert 
\leq \mathcal{F} (\Vert V \Vert_{E_0})
\sum_{\vert \beta \vert \leq  2}  \sup_{\xi \in \mathcal{C}_3}\Vert D^\beta_\xi \gamma(t,\cdot,\xi)\Vert_{L^ \infty (\xR^d)} \htilde^{2(1-2\delta)}.
$$ 
Using the equality $2(1-2\delta) + \delta =0$, we deduce  that for $ s \in I_{\htilde }= (0, \widetilde{h}^\delta) $ we have
\begin{equation}\label{calA}
\int_0^s \left\Vert \mathcal{A}(\sigma)  \right\Vert d \sigma 
\leq  \mathcal{F} (\Vert V \Vert_{E_0}) \mathcal{N}_2(\gamma).
 \end{equation}
The Gronwall inequality shows that $\Vert U(s)\Vert$ 
is uniformly bounded on $I_{\htilde }.$ 
Coming back to   \eqref{Dbicar} we see that we have
\begin{equation}\label{estDx}
  \la \frac{\partial \zeta}{\partial z_0} (s) \ra \leq  \mathcal{F} \big(\Vert V \Vert_{E_0} +   \mathcal{N}_2(\gamma)\big) , \quad 
  \la \frac{\partial z}{\partial z_0}(s) - Id \ra \leq
 \mathcal{F} \big(\Vert V \Vert_{E_0} +   \mathcal{N}_2(\gamma)\big)   \htilde^{\frac{\delta}{2}}.
  \end{equation}
Taking $\htilde $ small enough we obtain the invertibility of the 
matrix  $\big( \frac{\partial z}{\partial z_0}(s;z_0,\xi,\htilde )\big)$.

Now we have
$$
\vert z(s; z_0,\xi,\htilde ) -z_0 \vert \leq \int_0^s \vert \dot{z}(\sigma,x_0,\xi,\htilde ) \vert \, d\sigma.
$$
Since the right hand side is uniformly bounded for $s \in \big[0, \htilde ^{\delta}\big]$, 
we see that our map is proper. 
Therefore we can write
\begin{equation}\label{diffeo1}
z(s;z_0,\xi,\htilde ) = z \Longleftrightarrow z_0=\kappa(s;z,\xi,\htilde ). 
\end{equation}
Let us set for $t\in\big[0, \htilde ^{\delta}\big]$
\begin{equation}\label{phi}
\phi(t,z,\xi,\htilde ) = z\cdot\xi - 
\int_0^t p\big(\sigma,z,\zeta(\sigma;\kappa(\sigma;z,\xi,\htilde ),\xi,\htilde ), \tilde{h} \big)\, d\sigma.
\end{equation}
\begin{prop}\label{eqeiko}
The function $\phi$ defined in \eqref{phi} is the solution of the eikonal equation \eqref{eikonale}.
\end{prop}
\begin{proof}
The initial condition is trivially satisfied. Moreover we have
$$
\frac{\partial \phi}{\partial t}\big(t,z,\xi,\htilde\big) 
= - p\big(t,z,\zeta(t;\kappa\big(t;z,\xi,\htilde\big),\xi,\htilde), \tilde{h}\big).
$$
Therefore it is sufficient to prove that
\begin{equation}\label{eiksuite}
\frac{\partial \phi}{\partial z}\big(t,z,\xi,\htilde\big) 
= \zeta\big(t;\kappa\big(t;z,\xi,\htilde\big),\xi,\htilde\big).
\end{equation}
Let us consider the Lagrangean manifold
\begin{equation}\label{Lagrange}
\Sigma =\Big\{\big(t,z(t; z_0,\xi ,\htilde ),\tau(t; z_0,\xi ,\htilde ),\zeta(t; z_0,\xi ,\htilde )\big):   
t \in I_{\htilde}  , (z_0,\xi ) \in \xR^{2d} \Big\}.
\end{equation}
According to \eqref{diffeo1} we can write
$$
\Sigma = \Big\{(t,z, \tau(t;\kappa(t;z,\xi,\htilde ),\xi,\htilde ), \zeta(t;\kappa(t;z,\xi,\htilde ),\xi,\htilde ))
: t\in I_{\htilde }, (z,\xi)\in \xR^{2d} \Big\}.
$$
Let us set 
\begin{equation*}
\begin{aligned}
F_0\big(t,z,\xi,\htilde\big) &= \tau\big(t;\kappa(t;z,\xi,\htilde ),\xi,\htilde\big),\\
F_j\big(t,z,\xi,\htilde\big)  &= \zeta_j\big(t;\kappa(t;z,\xi,\htilde ),\xi,\htilde\big).
\end{aligned}
\end{equation*}
Since the symbol $q$ is constant along its bicharacteristic   and $q(0, z(0), \tau(0), \zeta(0), \htilde)= 0$ we have
$$
F_0\big(t,z,\xi,\htilde\big) = - p\big(t,z,\zeta(t;\kappa(\sigma;z,\xi,\htilde ),\xi,\htilde\big), \htilde \big).
$$
Now $\Sigma$ being Lagrangean we have
$$
\di t \wedge \di F_0 + \di z \wedge \di F =0.
$$
Thus $\partial_{z_j} F_0 - \partial_t F_j =0$ since it is the coefficient of $\di t\wedge \di z_j$ in the above expression. 
Therefore using \eqref{phi} we can write
\begin{equation*}
\begin{aligned}
\frac{\partial \phi}{\partial z_j}(t,z,\xi,\htilde ) &=\xi_j - \int_0^t \frac{\partial}{\partial z_j} 
\big[p\big(\sigma,z,\zeta(\sigma;\kappa(\sigma;z,\xi,\htilde ),\xi,\htilde )\big)\big]\, d \sigma\\
&= \xi_j + \int_0^t \frac{\partial}{\partial \sigma}\big[ \zeta_j\big(\sigma;\kappa(\sigma;z,\xi,\htilde\big),\xi,\htilde )\big] \, d \sigma\\
&= \zeta_j\big(t;\kappa(t;z,\xi,\htilde ),\xi,\htilde\big).
\end{aligned}
\end{equation*}
    \end{proof}
    
\subsubsection{The Hessian of the phase}

Let us recall that the phase $\phi$ is the solution of the problem
\begin{equation}\label{phase}
\left\{
\begin{aligned}
&\frac{\partial \phi}{\partial t} (t,z,\xi,\htilde ) + p\Big(t,z,\frac{\partial \phi}{\partial z}(t,z,\xi,\htilde ), \htilde \Big) =0\\
&\phi \arrowvert_{t=0} = z \cdot \xi.
\end{aligned}
\right.
\end{equation}
On the other hand the map $(t,z,\xi) \mapsto \phi(t,z,\xi,\htilde )$ is $C^1$ in time and $C^\infty$ in $(x,\xi)$. 
Differentiating twice, with respect to $\xi$, the above equation we obtain
\begin{equation*}
\begin{aligned}
\frac{\partial }{\partial t} \Big( \frac{\partial^2 \phi }{\partial  \xi_i \partial \xi_j}\Big) 
&= - \sum_{k,l = 1}^d  
\frac{\partial^2 p}{\partial  \zeta_k \partial \zeta_l}\Big(t,z,\frac{\partial \phi}{\partial z},\htilde \Big)  
\frac{\partial^2 \phi }{\partial  z_k \partial \xi_i}   \frac{\partial^2 \phi }{\partial z_l \partial \xi_j} \\
& \quad- \sum_{k=1}^d \frac{\partial p}{ \partial \zeta_k}\Big(t,z,\frac{\partial \phi}{\partial z},\htilde \Big)
\frac{\partial^3 \phi}{\partial z_k \partial \xi_i \partial \xi_j}\cdot
\end{aligned}
\end{equation*}
By the initial condition in \eqref{phase} we have
$$
\frac{\partial^2 \phi}{\partial z_k \partial \xi_i}\Big  \arrowvert_{t=0} 
= \delta_{k i}\quad   \frac{\partial^2 \phi}{\partial z_l \partial \xi_j} \Big  \arrowvert_{t=0} 
= \delta_{lj}, \quad \frac{\partial^3 \phi}{\partial z_k \partial \xi_i \partial \xi_j} \Big   \arrowvert_{t=0} = 0, 
\quad  \frac{\partial^2 \phi }{\partial  \xi_i \partial \xi_j} \Big   \arrowvert_{t=0} =0.
$$
It follows that 
$$
\frac{\partial }{\partial t} \Big( \frac{\partial^2 \phi }{\partial  \xi_i \partial \xi_j}\Big)\Big\arrowvert_{t=0} 
= -\frac{\partial^2 p}{\partial  \xi_i \partial \xi_j}\big(0,z,\xi,\htilde \big)
$$
from which we deduce that 
$$
\frac{\partial^2 \phi }{\partial  \xi_i \partial \xi_j}(t,z,\xi,\htilde ) 
= -t \frac{\partial^2 p}{\partial  \xi_i \partial \xi_j}(0,z,\xi,\htilde ) +o(t).
$$
It follows from Proposition \ref{Hessp} that one can find $M_0>0$ such that
\begin{equation}\label{HessphiND}
\la \det \Big( \frac{\partial^2 \phi}{\partial  \xi_i \partial \xi_j}(t,z,\xi,\htilde )\Big) \ra \geq M_0 t^d,
\end{equation}
for $t\in I_{\htilde }, z \in \xR^d, \xi \in \mathcal{C}_0, 0<\htilde  \leq \htilde _0$.

Our goal now is to prove estimates of higher order on the phase (see Corollary \ref{estkappa} below.)
 \subsubsection{Classes of symbol and symbolic calculus}
Recall here that $\delta = \frac{2}{3}$ and that $ \mathcal{N}_{k}(\gamma)$ 
has been defined in \eqref{norme:gamma}.
\begin{defi}
Let $m\in \xR, \mu_0 \in \xR^+ $ and $a =  a(t, z,\xi, \htilde)$ 
be a smooth function defined on $ \Omega =[0, \htilde^\delta]\times \xR^d \times \mathcal{C}_0\times (0, \htilde_0]$. We shall say that
 
$(i)$ \quad  $a \in S^m_{\mu_0}$ if for every $k\in \xN$ one can find $\mathcal{F}_k:\xR^+ \to \xR^+$ such that for   all $(t,z,\xi,\htilde) \in \Omega $ \begin{equation}\label{est:symb}
\vert D_z^\alpha D_\xi^\beta a(t,z,\xi,\htilde)\vert \leq \mathcal{F}_k( \Vert V\Vert_{E_0} + \mathcal{N}_{k+1}(\gamma)) \, \htilde^{m - \vert \alpha \vert \mu_0},\quad \vert \alpha \vert + \vert \beta \vert =k,
\end{equation}
\\
$(ii)$ \quad   $a \in \dot{S}^m_{\mu_0}$  if \eqref{est:symb} holds for every $k\geq1$.
\end{defi}
\begin{rema}\label{remarque1}
\begin{enumerate}
\item If $m\geq m'$ then $ S^m_{\mu_0} \subset  S^{m'}_{\mu_0} $ and  $ \dot{S}^m_{\mu_0} \subset  \dot{S}^{m'}_{\mu_0}$.
\item Let $a(t,z, \xi, \htilde) = z$ and $b(t,z, \xi, \htilde) = \xi$. Then   $a \in \dot{S}^{\delta/2}_{2\delta -1}$,    $b \in \dot{S}^0_{2\delta -1}$.
\item If $a \in S^m_{\mu_0}$ with $m \geq 0$ then $b=e^a \in S^0_{\mu_0}.$
\end{enumerate}
\end{rema}
We study now the composition of such symbols.
 \begin{prop}\label{est:compo}
Let $m \in \xR , f \in S^{m}_{2\delta-1} (\text{resp.}\dot{S}^{m}_{2\delta-1}),  Ê U \in \dot{S}^{\delta/2}_{2\delta-1}, V\in \dot{S}^{0}_{2\delta-1} $ and assume that $V\in \mathcal{C}_0.$  Set 
$$F(t,z,\xi, \htilde) = f(t, U(t,z,\xi,\htilde), V(t,z,\xi,\htilde), \htilde).$$ Then $F\in S^{m}_{2\delta-1} (\text{resp.\ }\dot{S}^{m}_{2\delta-1}).$
\end{prop}
\begin{proof}
Let $\Lambda = ( \alpha, \beta)\in \xN^d \times \xN^d, \vert \Lambda\vert =k$. 
If $k=0$ the estimate of $F$ follows easily from the hypothesis on $f$. Assume 
$k  \geq 1$. Then $D^\Lambda F$ is a finite linear combination of terms of the form
$$(1)= (D^Af)( \cdots) \prod_{j=1}^r (D^{L_j} U)^{p_j} (D^{L_j} V)^{q_j} $$
where $A=(a,b), \quad 1\leq \vert A \vert \leq  \vert \Lambda \vert, \quad L_j= (l_j,m_j)$ and
$$\sum_{j=1}^r p_j = a, \quad \sum_{j=1}^r q_j = b, \quad \sum_{j=1}^r (\vert p_j\vert + \vert q_j \vert)L_j = \Lambda$$
By the hypothesis on $f$   we have 
\begin{equation}\label{est:DAf3}
 \vert D^A f(\cdots)\vert \leq \mathcal{F}_k( \Vert V \Vert_{E_0} + \mathcal{N} _{k+1}(\gamma))\, \htilde^{m-\vert a \vert (2 \delta -1)}.
\end{equation}
By the hypotheses on $U,V$, the product occuring in the definition of $(1)$ is bounded by $  \mathcal{F}_k( \Vert V \Vert_{E_0} + \mathcal{N}_{k+1}(\gamma))\, \htilde^M$ where 
$$M= \sum_{j=1}^r \vert p_j \vert  \big(\frac{\delta}{2} -\vert l_j \vert (2 \delta -1) \big)  - \sum_{j=1}^r \vert q_j \vert ( \vert l_j \vert (2 \delta -1)) = -\vert \alpha \vert (2 \delta -1) + \frac{\delta}{2} \vert a \vert.$$
 Using \eqref{est:DAf3} and the fact that $1-2\delta +Ê\frac{\delta}{2} = 0$ we obtain the desired conclusion.
\end{proof}
 
  \subsubsection{Further estimates on the flow}
   We shall denote by  $z(s) = z(s; z, \xi, \htilde), \zeta(s) = \zeta(s; z, \xi, \htilde)$ 
   the solution of \eqref{bicar} with $z(0) = z, \zeta(0) = \xi.$  Recall  that $\delta = \frac{2}{3} $. 
        \begin{prop}\label{estflow}
There exists $\mathcal{F}: \xR^+ \to \xR^+$ non decreasing such that
\begin{alignat*}{2}
&(i) \quad &&\la\frac{\partial z}{\partial {z }}(s) - Id \ra  \leq   \mathcal{F} \big(\Vert V \Vert_{E_0} +   \mathcal{N}_2(\gamma)\big)   \htilde^{\frac{\delta}{2}}, \quad \la\frac{\partial \zeta}{\partial {z }}(s)\ra\leq   \mathcal{F} \big(\Vert V \Vert_{E_0} +   \mathcal{N}_2(\gamma)\big),  \\
  &(ii) &&\la\frac{\partial z}{\partial \xi}(s) \ra   + \la\frac{\partial \zeta}{\partial \xi}(s) - Id \ra \leq \mathcal{F} \big(\Vert V \Vert_{E_0} +   \mathcal{N}_2(\gamma)\big) \, \htilde^{\frac{\delta}{2}}.
  \end{alignat*}
for all $s\in I_{\htilde } = \big[0, \htilde ^\delta\big], z\in \xR^d, \xi \in \mathcal{C}_0.$

For any $k\geq 1$ there exists $\mathcal{F}_k: \xR^+ \to \xR^+$ non decreasing such that for $\alpha, \beta \in \xN^ d$ with $ \vert \alpha \vert + \vert \beta \vert =k $ 
\begin{equation}\label{est:z:zeta}
\left\{ 
\begin{aligned}
&  \la D_z^\alpha D_\xi^\beta  z(s) \ra \leq \mathcal{F}_k\big(\Vert V \Vert_{E_0} +   \mathcal{N}_{k+1}(\gamma)\big) \htilde ^{ \vert \alpha \vert  (1-2 \delta) + \frac{\delta}{2}},\\
& \la D_z^\alpha D_\xi^\beta  \zeta(s)  \ra\leq  \mathcal{F}_k\big(\Vert V \Vert_{E_0} +   \mathcal{N}_{k+1}(\gamma)\big) \htilde ^{ \vert \alpha \vert  (1-2 \delta)}.
\end{aligned}
\right.
\end{equation} 
\end{prop}
\begin{proof}
The estimates of the first terms in  $(i)$ and $(ii)$ have been proved in \eqref{estDx}.
By exactly the same argument one deduces the estimates  on the second terms. 

We shall prove \eqref{est:z:zeta} by induction on $k$. According to $(i)$ and $(ii)$ it is  true for $k =1.$
   Assume it is true up to the order $k$ and let $\vert \alpha \vert +\vert \beta \vert =k+1\geq 2$. 
Let us set $\Lambda =(\alpha, \beta),$ $D^\Lambda = D_z^\alpha D_\zeta^\beta$ and  $m(s) =(s, z(s),\zeta(s), \htilde ).$  
By the Faa-di-Bruno formula we have
\begin{equation}\label{F1-F2}
\left\{
\begin{aligned}
 &D^\Lambda[p'_\zeta(m(s))] = p''_{z\zeta}(m(s)) D^\Lambda z(s) + p''_{\zeta\zeta}(m(s)) D^\Lambda \zeta(s) + F_1(s),\\
 &D^\Lambda[p'_z(m(s))] = p''_{zz}(m(s)) D^\Lambda z(s) + p''_{z\zeta}(m(s)) D^\Lambda \zeta(s) + F_2(s).  
 \end{aligned}
\right.
\end{equation}
It follows that $U(s) = (D^\Lambda z(s),D^\Lambda \zeta(s)$ is the solution of the problem 
$$\dot{U}(s) = \mathcal{A}(s) U(s) + F(s), \quad U(0) = 0$$
where $\mathcal{A}(s)$ has been defined in \eqref{matrice} and $F(s) = (F_1(s), F_2(s)).$

According to the estimates of the symbol $p$ given in Lemma \ref{derp} the worse term is $F_2.$ By the   formula mentionned above we see that  $F_1$ is a finite linear combination of terms of the form
$$
\big(D^A p'_z \big)(m(s)) \prod_{i=1}^r \big(D^{L_i}z(s) \big)^{p_i} \prod_{i=1}^{r} \big(D^{L_i} \zeta(s) \big)^{q_i},
$$
where $ A=(a,b),\quad  2\leq \vert A \vert \leq \vert \Lambda \vert$ and
\begin{equation*}
  L_i=(l_i,l'_i), \quad 1 \leq \vert L_i \vert \leq \vert \Lambda \vert -1, \quad \sum_{i=1}^r p_i = a,\quad   \sum_{i=1}^{r} q_i = b,   \quad \sum_{i=1}^k (\vert p_i \vert +\vert q_i \vert )L_i   = \Lambda. 
  \end{equation*}
It follows from Corollary \ref{derp'} that for $s$ in $[0, \htilde^\delta]$ we have, 
\begin{equation}\label{estDqxi}
\int_0^s \la \big(D^Ap'_z \big)(m(\sigma)) \ra d \sigma \leq  \mathcal{N}_{\vert A \vert +1}(\gamma) \htilde^{(\vert a \vert +1)(1-2 \delta) + \delta}  \leq  \mathcal{N}_{\vert A \vert +1}(\gamma) \htilde^{ \vert a \vert (1-2 \delta) + \frac{\delta}{2}}.
 \end{equation}
 since $1-2 \delta +\delta =  \frac{\delta}{2}.$ Now since $1 \leq \vert L_i \vert \leq \vert \Lambda \vert -1= k$ we have, by the induction,
 \begin{align*}
 & \vert D^{L_i}z(s)\vert \leq \htilde^{\vert l_i\vert (1-2 \delta) + \frac{\delta}{2}} \mathcal{F}_k\big(\Vert V \Vert_{E_0} +   \mathcal{N}_{k+1}(\gamma)\big),\\
 &\vert D^{L_i}\zeta(s)\vert \leq \htilde^{\vert l_i\vert (1-2 \delta)}  \mathcal{F}_k\big(\Vert V \Vert_{E_0} +   \mathcal{N}_{k+1}(\gamma)\big). 
  \end{align*}
 It follows that
  $$
\int_0^s \vert F_2(\sigma) \vert d \sigma 
\leq   \Big(\int_0^s \la \big(D^Ap'_z \big)(m(\sigma)) \ra d \sigma \Big) \mathcal{F}_{k+1}\big(\Vert V \Vert_{E_0} +   \mathcal{N}_{k+1}(\gamma)\big) \htilde^{M} 
     $$
     where $M = \sum_{i=1}^r \big( (\vert p_i \vert + \vert q_i \vert) \vert l_i\vert (1- 2 \delta) \vert  + \vert p_i \vert \frac{\delta}{2}\big) = \vert \alpha \vert (1-2 \delta) + \vert a \vert \frac{\delta}{2}.$ It follows from \eqref{estDqxi} and the fact that $1 - \frac{3 \delta}{2} = 0$ that 
     $$\int_0^s \vert F_2(\sigma) \vert d \sigma  \leq \mathcal{F}_{k+1}\big(\Vert V \Vert_{E_0} +   \mathcal{N}_{k+1}(\gamma)\big)\htilde^{ \vert \alpha \vert (1-2 \delta) + \frac{\delta}{2}}.
 $$
 Since $D^Ap'_\zeta$ has even a better estimate, the same computation shows that 
 $$\int_0^s \vert F_1(\sigma) \vert d \sigma  \leq \mathcal{F}_{k+1}\big(\Vert V \Vert_{E_0} +   \mathcal{N}_{k+1}(\gamma)\big)\htilde^{ \vert \alpha \vert (1-2 \delta) + \frac{\delta}{2}}.
 $$ 
Then we write
 $$
U(s) = \int_0^s F(\sigma) d \sigma + \int_0^s \mathcal{A}(\sigma)U(\sigma)d \sigma 
$$    
 and we use the above estimates on $F_1, F_2$, \eqref{calA} and   the Gronwall lemma to see that the step $k+1$ of the induction is achieved. This completes the proof of Proposition \ref{estflow}.
  \end{proof}
\begin{coro}\label{estkappa}
For every $k \geq 1$ there exists $\mathcal{F}_k: \xR^+ \to \xR^+$ non decreasing such that for every   $(\alpha, \beta) \in \xN^d \times \xN^d$ with  $ \vert \alpha \vert +  \vert \beta \vert =k$     we have
 \begin{equation*}
\begin{aligned}
& (i) \quad &&\vert D_z^\alpha D_\xi^\beta \kappa(s,z,\xi,\htilde) \vert \leq \mathcal{F}_k\big(\Vert V \Vert_{E_0} +   \mathcal{N}_{k+1}(\gamma)\big) \htilde ^{ \vert \alpha \vert  (1-2 \delta) + \frac{\delta}{2}},\\
&(ii) \quad &&\Big\vert D_z^\alpha D_\xi^\beta \big(\frac{\partial \phi}{\partial z}\big)(s,z,\xi,\htilde )\Big\vert
 \leq \mathcal{F}_k\big(\Vert V \Vert_{E_0} +   \mathcal{N} _{k+1}(\gamma)\big) \htilde ^{ \vert \alpha \vert  (1-2 \delta)},\\
 &(iii)\quad && \vert D_\xi^\beta \phi (s,z,\xi,\htilde) \vert \leq \mathcal{F}_k\big(\Vert V \Vert_{E_0} +   \mathcal{N} _{k+1}(\gamma)\big)\vert s \vert, \quad \vert \beta \vert \geq 2,
  \end{aligned}
\end{equation*}
for all $s \in I_{\htilde }, z \in \xR^d, \xi \in \mathcal{C}_0.$   This implies that $\kappa \in \dot{S}_{2\delta-1}^{\delta/2}$ and $  \frac{\partial \phi}{\partial z}  \in  {S}_{2\delta-1}^0.$
\end{coro}
\begin{proof}
We first show $(ii)$ and  $(iii)$. Recall that 
$$
\frac{\partial \phi}{\partial z}(s,z,\xi,\htilde ) = \zeta(s; \kappa(s; z,\xi,\htilde ), \xi, \htilde ).
$$
By   Proposition \ref{estflow} (since $\zeta$ is bounded) we have $\zeta\in S^0_{2\delta-1}.$ By $(i)$ we have $\kappa \in \dot{S}^{\delta/2}_{2\delta-1}$ and by Remark \ref{remarque1} we have $\xi \in \dot{S}^0_{2\delta-1}.$  Then Proposition \ref{est:compo} implies that $\frac{\partial \phi}{\partial z} \in \dot{S}_{2\delta-1}^0.$ Moreover $\frac{\partial \phi}{\partial z} $ is bounded since $\vert \zeta(s) - \xi \vert \leq \int_0^s \vert\frac{ \partial p}{\partial z}(t, \ldots) \vert dt \leq \mathcal{F} \big(\Vert V \Vert_{E_0} +   \mathcal{N}_2(\gamma)\big) \widetilde{h}^\frac{\delta}{2}$ and $\xi \in \mathcal{C}_0.$ Now $(iii)$ follows from the definition \eqref{phi} of the phase, the facts that $p \in S^0_{2\delta-1}, z\in \dot{S}^\mez_{2\delta-1}, \zeta(s; \kappa(s; z,\xi,\htilde ), \xi, \htilde )\in \dot{S}^0_{2\delta-1}$ and Proposition \ref{est:compo}.
 
  We are left with the proof of $(i)$. We proceed by induction on   $\vert \alpha \vert + \vert \beta \vert = k \geq1.$ Recall that  by definition of $\kappa$ we have the equality $z(s; \kappa(s;z,\xi,\htilde ), \xi, \htilde ) = z$. It follows that  
 $$\frac{\partial z}{\partial z} \cdot \frac{\partial \kappa}{\partial z} = Id, \quad  \frac{\partial z}{\partial z} \cdot\frac{\partial \kappa}{\partial \xi} =  -\frac{\partial z}{\partial \xi} .$$
 Then the estimate for $k =1$ follows from $(i)$ in Proposition \ref{estflow}. Assume the estimate true up to the order $k$ and let $\Lambda =(\alpha, \beta), \vert \Lambda \vert = k+1 \geq 2.$ Then differentiating  $\vert \Lambda \vert $ times the first above equality we see that $\frac{\partial z}{\partial z} \cdot D^\Lambda \kappa$ is a finite linear combination of terms of the form
 $$ (2)=  D^A z(\cdots) \prod_{j=1}^r \big(D^{L_j}\kappa\big)^{p_j} \prod_{j=1}^r \big(D^{L_j} \xi)^{q_j}$$
where $A=(a,b) \quad 2 \leq \vert A \vert \leq \vert \Lambda \vert, \quad L_j =(l_j,m_j),\quad  1 \leq \vert L_j \vert \leq k$ and 
$$ \sum_{j=1}^r p_j = \alpha,\quad \sum_{j=1}^r q_j = \beta, \quad \sum_{j=1}^r (\vert p_j \vert +\vert q_j \vert )L_j = (\alpha, \beta).$$
 We use the estimate (given by Proposition \ref{estflow})
$$ \vert D^A z(\cdots) \vert \leq  \mathcal{F}_{k+1}\big(\Vert V \Vert_{E_0} +   \mathcal{N}^0_{k+2}(\gamma)\big)\htilde ^{ \vert a\vert  (1-2 \delta) + \frac{\delta}{2}},$$ 
 the induction, the fact that $\xi \in \dot{S}^0_{2\delta-1}$  and the equality $1-2\delta + \frac{\delta}{2} =0$ to see that  
$$ \vert (2) \vert \leq \mathcal{F}_{k+1}\big(\Vert V \Vert_{E_0} +   \mathcal{N}^0_{k+2}(\gamma)\big)\htilde ^{ \vert \alpha\vert  (1-2 \delta) + \frac{\delta}{2}}.$$
Then we use Proposition \ref{estflow} $(i)$ to conclude the induction.
\end{proof}
\begin{rema}\label{rem:est:theta}
Since $\theta(t,z,z',\xi,\htilde) =  
\int_0^1 \frac{\partial \phi}{\partial z}(t, \lambda z + (1-\lambda) z' , \xi,\htilde) \, d \lambda $ we have also the estimate
\begin{equation}\label{est:theta}
\Big\vert D_z^{\alpha_1} D_{z'}^{\alpha_2}D_\xi^\beta \theta(s,z,z',\xi,\htilde )\Big\vert
 \leq \mathcal{F}_k\big(\Vert V \Vert_{E_0} +   \mathcal{N} _{k+1}(\gamma)\big) \htilde ^{ (\vert \alpha_1 \vert + \vert \alpha_2 \vert) (1-2 \delta)}.
\end{equation}
for $ \vert \alpha_1 \vert + \vert \alpha_2 \vert + \vert \beta \vert =k$
\end{rema}

\subsection{The transport equations}

According to \eqref{Jfinal} and \eqref{conjugaison} if $\phi$ satisfies the eikonal equation we have
\begin{equation}\label{RS}
\begin{aligned}
&e^{-i{\htilde}^{-1}\phi }(\htilde  \partial_t + \htilde c+iP)\big(e^{i{\htilde}^{-1}\phi }\widetilde{b} \big) \\
&\quad=\htilde  \partial_t \widetilde{b}  +\htilde c \widetilde{b} +   i\sum_{\vert \alpha \vert=1}^{N-1}\frac{\htilde ^{\vert \alpha \vert}}{\alpha!}
D_{z'}^\alpha\Big[(\partial^\alpha_\eta \widetilde{p})\big(t,z,z',\theta(t,z,z',\htilde )),\htilde \big) \widetilde{b} (z')\Big]\Big\arrowvert_{z'=z} \\
&\quad\quad+R_N +S_N
\end{aligned}
\end{equation}

Recall (see \eqref{tildeb}) that $\widetilde{b}  = b \Psi_0$. Let us set 
\begin{equation} \label{CT1}
T_N=   \partial_t b + c +i \sum_{1\leq \vert \alpha \vert \leq {N-1}}\frac{\htilde ^{\vert \alpha \vert - 1}}{\alpha!} D_{z'}^\alpha\Big[(\partial^\alpha_\eta \widetilde{p})\big(t,z,z',\theta(t,z,z',\htilde ), \htilde \big) b(z')\Big]\Big\arrowvert_{z'=z}. 
\end{equation}
Then 
\begin{equation}\label{UN}
e^{-i{\htilde}^{-1}\phi }(\htilde  \partial_t + \htilde c +iP)\big(e^{i{\htilde}^{-1}\phi }\widetilde{b} \big) = \htilde T_N \Psi_0  +U_N + R_N+S_N.
\end{equation}

Our purpose is to show that one can find a symbol $b$ such that, in a sense to be explained,
\begin{equation}\label{OhN}
T_N = \mathcal{O}({\htilde }^M), \quad \forall M \in \xN.
\end{equation}
 We set
\begin{equation*}
\mathcal{L}b = \partial_t b + c + \sum_{i=1}^n \frac{\partial}{\partial {z'}_i} 
\Big[ \frac{\partial \widetilde{p}}{\partial \zeta_i}\big(t,z,z',\theta(t,z,z',\htilde ), \htilde \big) b(t,z',\xi,\htilde )\Big]\Big\arrowvert_{z'=z}
\end{equation*}
Then  we can write
\begin{equation}\label{mathcalL}
\mathcal{L} = \frac{\partial}{\partial t} + \sum_{i=1}^d a_j(t,z,\xi,\htilde ) \frac{\partial}{\partial {z}_i} + c_0(t,z,\xi,\htilde )
\end{equation}
where
\begin{equation}\label{coeffL}
\left\{
\begin{aligned}
a_i(t,z,\xi,\htilde ) &= \frac{\partial p}{\partial \zeta_i}\Big(t,z, \frac{\partial \phi}{\partial z}(t,z,\xi,\htilde ), \htilde \Big),\\
c_0(t,z,\xi,\htilde )&= \sum_{i=1}^d\frac{\partial}{\partial z'_i} 
\Big[ \frac{\partial \widetilde{p}}{\partial \zeta_i}\big(t,z,z',\theta(t,z,z',\xi, \htilde ), \htilde \big)  \Big]\Big\arrowvert_{z'=z} + c(t,z,\htilde),\\
\theta(t,z,z',\xi, \htilde )&= \int_0^1 \frac{\partial \phi}{\partial z}\big(t,\lambda z+(1-\lambda) z',\xi,\htilde \big)\, d\lambda  
\end{aligned}
\right.
\end{equation}
and $c$ has been defined in \eqref{c=}.

Notice that, with $m=(t,z,\xi,\htilde)$ we have 
\begin{equation}\label{c0}
c_0(m) = \sum_{i=1}^d \frac{\partial^2 \widetilde{p}}{\partial\zeta_i \partial z'_i}(t,z,z,\frac{\partial \phi}{\partial z}(m), \htilde) + \mez \sum_{i,j=1}^d \frac{\partial^2  {p}}{\partial\zeta_i \partial \zeta_j}(t,z,\frac{\partial \phi}{\partial z}(m),\htilde))\frac{\partial^2\phi}{\partial z_i \partial z_j}(m) +c.
\end{equation}
Then we can write
\begin{equation}\label{K=L}
T_N = \mathcal{L}b  + i \sum_{2\leq \vert \alpha \vert 
\leq {N-1}}\frac{\htilde ^{\vert \alpha \vert - 1}}{\alpha!} 
D_{z'}^\alpha\Big[(\partial^\alpha_\zeta \widetilde{p})\big(t,z,z', \theta(t,z,z',\htilde), \htilde \big) b(t,z',\xi, \htilde )\Big]\Big\arrowvert_{z'=z}. 
\end{equation}

We shall seek $b$ on the form
\begin{equation}\label{formedeb}
b = \sum_{j=0}^N \htilde ^j \, b_j. 
\end{equation}
Including this expression of $b$ in \eqref{K=L} after a change of indices we obtain
$$
T_N= \sum_{k=0}^N \htilde ^k \mathcal{L}b_k 
+ i  \sum_{k=1}^{N+1} \htilde ^k\sum_{2\leq \vert \alpha \vert \leq N-1}\frac{\htilde ^{\vert \alpha \vert - 2}}{\alpha!}D_{z'}^\alpha \big[(\partial^\alpha_\zeta \widetilde{p})(\cdots)b_{k-1}\big]\big\arrowvert_{z'=z}.
$$
We will take $b_j$ for  $ j=0, \ldots, N,$ as solutions of the following problems
\begin{equation}\label{bk}
\left\{
\begin{aligned}
\mathcal{L}b_0 &= 0, \quad b_0 \arrowvert_{t=0} = \chi, \quad \chi \in C_0^\infty(\xR^d),\\
\mathcal{L}b_j & = F_{j-1} := -i  \sum_{2\leq \vert \alpha \vert \leq N-1}
\frac{\htilde ^{\vert \alpha \vert - 2}}{\alpha!}D_{z'}^\alpha 
\big[(\partial^\alpha_\zeta \widetilde{p})(\cdots)b_{j-1}\big]\big\arrowvert_{z'=z}, \quad   b_j \arrowvert_{t=0} =0.
\end{aligned}
\right.
\end{equation}
This choice will imply that
\begin{equation}\label{n200}
T_N = i\htilde ^{N+1}\sum_{2\leq \vert \alpha \vert \leq N-1}
\frac{\htilde ^{\vert \alpha \vert - 2}}{\alpha!}
D_{z'}^\alpha \big[(\partial^\alpha_\zeta \widetilde{p})(\cdots)b_{N}\big]\big\arrowvert_{z'=z}.
\end{equation}

 \begin{prop}\label{b:existe}
 The system \eqref{bk} has a unique solution   with $ b_j \in S^0_{2\delta-1}.$
     \end{prop}
 We  prove this result by induction.   To solve these equations   we use the method of characteristics and we begin by preliminaries. 
\begin{lemm}\label{aj}
 We have $a_i \in S^0_{2\delta-1} $ for $i=1,\ldots,d.$
  \end{lemm}
\begin{proof}
This follows from Lemma \ref{derp}, Proposition \ref{est:compo} with $f = \frac{\partial p}{\partial \zeta_i}, U(t,z,\xi,\htilde) = z, V(\cdots) =  \frac{\partial \phi}{\partial z}$ and Corollary \ref{estkappa}.
  \end{proof}

Consider now the system of differential equations
$$
\dot{Z}_j(s) = a_j(s,Z(s),\xi,\htilde ), \quad Z_j(0) = z_{j}, \quad  1 \leq j \leq d.
$$
By Lemma \ref{aj}  $a_j   $   is bounded. Therefore this system has a unique solution 
defined on $I_{\htilde }$. Differentiating with respect to $z$   we obtain 
$$
\la \frac{\partial Z}{\partial z}(s) \ra \leq  C+   \mathcal{F}_{ 2 }(\Vert V \Vert_{E_0}+ \mathcal{N}_{2}(\gamma))\int_0^s \htilde ^{1- 2 \delta}\la \frac{\partial Z}{\partial z }(\sigma) \ra d\sigma, \quad 0< s \leq \htilde ^{\delta},
$$
since $ \vert s \vert \, \htilde ^{1- 2 \delta} \leq \htilde ^{1- \delta} = \htilde ^{ \frac{\delta}{2}},$  the Gronwall inequality  shows that $ \la \frac{\partial Z}{\partial z }(s) \ra$ is uniformly bounded.
Using again the equation satisfied by $\frac{\partial Z}{\partial z }(s)$ we deduce that
\begin{equation}\label{est:dZ}
\la \frac{\partial Z}{\partial z }(s) - Id \ra 
\leq     \mathcal{F}_{ 2 }(\Vert V \Vert_{E_0}+ \mathcal{N}_{2}(\gamma)) \, {\htilde }^{ \frac{\delta}{2}}, \quad 0< s \leq \htilde ^{\delta}.
\end{equation}
This shows that the map $z  \mapsto Z(s;z , \xi, \htilde)$ is a global diffeomorphism from $\xR^d$ to itself so
\begin{equation}\label{diffeo3}
Z(s;z , \xi, \htilde ) = z \Longleftrightarrow z  = \omega(s; Z;\xi,\htilde ).
\end{equation}
An analogue computation shows that 
\begin{equation}\label{est:dxi}
\la \frac{\partial Z}{\partial \xi}(s) \ra 
\leq     \mathcal{F}_{ 2 }(\Vert V \Vert_{E_0}+ \mathcal{N}_{2}(\gamma)) \, {\htilde }^{ \delta}, \quad 0< s \leq \htilde ^{\delta}.
\end{equation}
 \begin{lemm}\label{est:derZ}
 
 The function $(s,z,\xi,\htilde) \mapsto Z(s,z,\xi,\htilde)$ belongs to $\dot{S}^{\delta/2}_{2\delta-1}.$
 \end{lemm}
 \begin{proof}
We have to prove that for   $  \vert \alpha \vert + \vert \beta \vert =  k \geq 1 $ we have the estimate
\begin{equation}\label{est:dalphaZ}
\vert D^\alpha_z D^\beta_\xi Z(s;z,\xi,\htilde)  \vert \leq \mathcal{F}_{ k}(\Vert V \Vert_{E_0}+ \mathcal{N}_{k+1} (\gamma))\, \htilde^{\vert \alpha \vert (1-2\delta) + \frac{\delta}{2}}.
\end{equation}
  Indeed this is true for $k=1$ by \eqref{est:dZ}, \eqref{est:dxi}. Assume this is true up to the order $k$ and let $\vert \alpha \vert + \beta \vert =k+1\geq 2.$ Set $U(s) =D^\Lambda  Z(s;z,\xi,\htilde)$ where $\Lambda =(\alpha, \beta).$  It satifies the system $\dot{U}(s) = \frac{\partial a}{\partial z} (s;Z(s),\xi,\htilde)U(s) + F(s),    U(0)=0 $ where $F(s)$ is a finite linear combination of terms of the form 
$$(1) =(D^A a) (\cdots) \prod_{j=1}^r (\partial^{L_j}Z(s))^{p_j}  (\partial^{L_j}\xi)^{q_j} $$
where $A =(a,b),   2 \leq \vert A \vert \leq \vert \Lambda \vert, L_j = (l_j,m_j), 1 \leq \vert L_j \vert \leq k$ and
$$\sum_{j=1}^r (\vert p_j \vert + \vert q_j \vert)  L_j = (\alpha, \beta), \quad \sum_{j=1}^r p_j = \alpha, \quad \sum_{j=1}^r q_j = \beta.$$
 First of all, by Lemma \ref{aj} we can write
 $$ \vert (D^A a) (\cdots) \vert \leq \mathcal{F}_k\big(\Vert V \Vert_{E_0} +   \mathcal{N}_{k+1}(\gamma)\big)\htilde^{\vert a \vert (1-2 \delta)}.$$
 Using the induction and the fact that $\xi \in \dot{S}^0_{2\delta-1}$ we can estimate the product occuring in $(1)$ by $\mathcal{F}_k\big(\Vert V \Vert_{E_0} +   \mathcal{N}_{k+1}(\gamma)\big) \htilde^{M} $ where $$M = \sum_{j=1}^r \Big\{\vert p_j\vert  \big (\vert l_j \vert (1-2 \delta) + \frac{\delta}{2}\big ) + \vert q_j \vert \vert l_j \vert (1-2 \delta)\Big\} = \vert \alpha \vert   (1-2 \delta) + \vert a \vert \frac{\delta}{2}.$$
 It follows that $  \int_0^s \vert F(t)\vert dt  \leq \mathcal{F}_k\big(\Vert V \Vert_{E_0} +   \mathcal{N}_{k+1}(\gamma)\big)\htilde^{\vert \alpha \vert (1-2 \delta) + \delta}$ and we conclude by the Gronwall inequality.
 \end{proof}
 \begin{coro}\label{est:omega}
 The function $\omega$ defined in \eqref{diffeo3} belongs to $ \dot{S}^{\delta/2}_{2\delta-1}.$
\end{coro}
\begin{proof}
The proof is the same as that of Corollary \ref{estkappa}.
\end{proof}

 \begin{proof}[Proof of Proposition \ref{b:existe}]
Now, with the notations in \eqref{coeffL} and \eqref{bk} we have 
$$
\frac{d}{ds} \big[b_j(s,Z(s)) \big] = \big( \frac{\partial u}{\partial t} + a \cdot \nabla u \big)(s, Z(s) )= -(c_0 u)(s,Z(s)) +F_{j-1}(s,z(s)),\quad  j \geq 0
$$
with $F_{-1} =0.$ It follows that 
$$
\frac{d}{ds} \Big[e^ {  \int_0^s c_0(\sigma,Z(\sigma)) d \sigma } b_j(s,Z(s)) \Big]
= e^{  \int_0^s c_0(\sigma,Z(\sigma)) \, d \sigma  }F_{j-1}(s,Z(s)),
$$
Using \eqref{diffeo3} we see that the unique solution of \eqref{bk} is given by 
\begin{equation}\label{bj=}
\left\{
\begin{aligned}
b_0(s,z,\xi,\htilde ) &= \chi(\xi)  \,  \text{exp} \Big({ \int_0^s   c_0(t, Z(t; \omega(s,z,\xi,\htilde ), \xi, \htilde))  \,dt  }\Big), 
 \\
   b_j(s,z,\xi,\htilde ) &= \int_0^s 
e^{ \int_s^\sigma   c_0(t, Z(t; \omega(s,z,\xi,\htilde ), \xi, \htilde)  \,dt  } 
F_{j-1}( \sigma, Z(\sigma; \omega(s,z,\xi, \htilde ),\xi, \htilde)  \,d \sigma.
\end{aligned}
\right.
\end{equation}

 The last step in the proof of Proposition \ref{b:existe} is contained in the following lemma.
 \begin{lemm}\label{derb}
We have $b_j \in S^0_{2\delta-1}.$
\end{lemm}
\begin{proof}
  Step1: we show that 
  \begin{equation}\label{est:e^c0}   
 e^{ \int_s^\sigma   c_0(t, Z(t; \omega(s,z,\xi,\htilde ), \xi, \htilde)  \,dt  } \in S^0_{2\delta-1}.
 \end{equation}
 According to Remark \ref{remarque1} this will be implied by $ \int_s^\sigma c_0(t, Z(t; \omega(s,z,\xi,\htilde ), \xi, \htilde)  \,dt   \in S^{\delta/2}_{2\delta-1}.$  By Lemma \ref{est:derZ} we have $Z\in \dot{S}^{\delta/2}_{2\delta-1}$ and $\omega\in \dot{S}^{\delta/2}_{2\delta-1}.$ Moreover $\xi \in \dot{S}^{0}_{2\delta-1}$.  By Proposition \ref{est:compo} the function $Z(t; \omega(s; z,\xi,\htilde),\xi,\htilde)$ belongs to $\dot{S}^{\delta/2}_{2\delta-1}$. Now by Corollary \ref{estkappa} we have $\frac{\partial \phi}{\partial z} \in \dot{S}^0_{2\delta-1}$ and $\frac{\partial^2 \phi}{\partial z^2} \in  {S}^{-\delta/2}_{2\delta-1} $ (since $1-2\delta= -\delta/2.$)  It follows from Proposition \ref{est:compo} that  for $s\in[0, \htilde^\delta]$ 
\begin{equation}\label{U=}
\left\{
\begin{aligned}
 &U_1(t;  z,\xi,\htilde )=  \frac{\partial \phi}{\partial z}(t,Z(t; \omega(s,z,\xi,\htilde ), \xi, \htilde), \xi,\htilde)\in  \dot{S}^0_{2\delta-1},\\
  &U_2(t;  z,\xi,\htilde)=  \frac{\partial^2 \phi}{\partial z^2}(t,Z(t; \omega(s,z,\xi,\htilde ), \xi, \htilde), \xi,\htilde)\in   {S}^{-\delta/2}_{2\delta-1}.
 \end{aligned}
 \right.
 \end{equation}
Now by Lemma \ref{derp}   the functions 
$\frac{\partial^2 \widetilde{p}}{\partial\zeta \partial z' }(t,z,z,\zeta, \htilde) $ 
(resp.\ $\frac {\partial^2  {p}}{\partial\zeta  \partial \zeta }(t,z,\zeta,\htilde) $) 
satisfy the condition of Proposition \ref{est:compo} with $m = 1-2\delta$ 
(resp.\ $m=0.$) Using \eqref{U=} and the fact that 
$z\in  \dot{S}^{\delta/2}_{2\delta-1}$ 
we deduce that 
\begin{align*}
&\int_s^\sigma \frac{\partial^2 \widetilde{p}}{\partial\zeta \partial z' }\big(t,z,z,U_1\big(t,z,\xi,\htilde\big), \htilde\big)\, dt 
\in {S}^{ \frac{\delta}{2}}_{2\delta-1},\\
& \int_s^\sigma \frac{\partial^2  {p}}{\partial\zeta^2   }
\big(t,z, U_1\big(t,z,\xi,\htilde\big), \htilde\big)U_2\big(t;  z,\xi,\htilde\big)\, dt \in {S}^{\frac {\delta}{2}}_{2\delta-1}.
\end{align*}
This shows that $\int_s^\sigma   c_0\big(t, Z(t; \omega\big(s,z,\xi,\htilde \big), \xi, \htilde\big)  \,dt \in S^ {\frac{\delta}{2}}_{2\delta-1}$ as claimed.

Step 2: we show that for $\vert a \vert +\vert b \vert = k \geq 0$ we have, with $\Lambda =(a,b)\in \xN^d \times \xN^d$
\begin{equation}\label{est:Fj}
\int_0^s \la D^\Lambda \big [G_{j-1}\big( \sigma;Z\big(\sigma; \omega\big(s;z,\xi, \htilde\big ),\xi, \htilde\big)\big)\big]\ra \,d \sigma 
\leq \mathcal{F}_{ k}\big(\Vert V \Vert_{E_0}+ \mathcal{N}_{k+1} (\gamma)\big) \htilde^{\vert a \vert(1-2\delta)}
 \end{equation}
 where for $\vert \rho \vert \geq 2,$
\begin{equation}\label{FjGj}
G_{j-1}(\sigma, z, \xi, \htilde)=  
\htilde ^{\vert \rho \vert - 2} D_{z'}^\rho
\big[(\partial^\rho_\zeta \widetilde{p})\big(\sigma;z,z', \theta\big(\sigma;z,z',\xi, \htilde\big), \htilde \big)b_{j-1}(\sigma; z', \xi, \htilde)\big]\big\arrowvert_{z'=z}. 
  \end{equation}
 We claim that for $\Lambda = (\alpha, \beta), \vert \alpha \vert + \vert \beta\vert = k \geq 0,$
 \begin{equation}\label{est:Gj1}
\la D^\Lambda G_{j-1}(\sigma, z, \xi, \htilde)\ra \leq \htilde^{-\delta + \vert \alpha \vert (1-2\delta)} 
\mathcal{F}_{ k}\big(\Vert V \Vert_{E_0}+ \mathcal{N}_{k+1} (\gamma)\big) .\end{equation}
 Indeed $ D^\Lambda G_{j-1}$ is a finite sum of terms of the form $H_1 \times H_2$ with
\begin{align*}
H_1  &=  \htilde ^{\vert \rho \vert - 2} D^{\Lambda_1}D_{z'}^{\rho_1}  \big[(\partial^\rho_\zeta \widetilde{p})(\sigma,,z,z', \theta(\sigma;z,z',\xi, \htilde), \htilde \big)) \big] \arrowvert_{z'=z}, \\
H_2   &= D^{\Lambda_2}D_{z}^{\rho_2}b_{j-1}(\sigma, z, \xi, \htilde),
\end{align*}
where $ \Lambda_i =(\alpha_i, \beta_i) \quad \vert \Lambda_1 \vert + \vert \Lambda_2 \vert = \vert \Lambda  \vert, \quad \vert \rho_1  \vert +\vert \rho_2 \vert = \vert \rho  \vert.$

By the induction we have Ä
$$
\vert H_2 \vert \leq 
\mathcal{F}_{ k}\big(\Vert V \Vert_{E_0}+ \mathcal{N}_{k+1} (\gamma)\big)
\, \htilde^{(\vert \alpha_2 \vert + \vert \rho_2\vert)(1-2\delta)}
$$
and since $z'\in \dot{S}_{2\delta-1}, \theta \in \dot{S}^0_{2\delta-1}$ 
using Proposition \ref{derp} we see that 
$$
\vert H_1 \vert \leq  
\mathcal{F}_{ k}\big(\Vert V \Vert_{E_0}+ \mathcal{N}_{k+1} (\gamma)\big) 
\, \htilde^{\vert \rho \vert -2 + (\vert \alpha_1 \vert + \vert \rho_1\vert)(1-2\delta)}.
$$
Now since $ \vert \rho \vert\geq 2$ and $\delta = \frac{2}{3},$ we have 
$\vert \rho \vert - 2 + \vert \alpha \vert (1-2\delta) + \vert \rho \vert (1-2\delta) \geq \vert \alpha \vert (1-2\delta) -\delta$ which proves \eqref{est:Gj1}.
   
Eventually since the function $Z\big(t; \omega(s; z,\xi,\htilde),\xi,\htilde\big)$ 
belongs to $\dot{S}^{\delta/2}_{2\delta-1}$ (see  Step 1), 
we deduce from  Proposition \ref{est:compo}, 
with $m=-\delta,$ that \eqref{est:Fj} holds. 
Then Step 1 and Step 2 prove the lemma. Notice that $b_j$ can be written  $ \chi(\xi)b_j^0.$
\end{proof}
Thus the proof of Proposition \ref{b:existe} is complete.
\end{proof}

Summing up we have proved that with the choice of $\phi$ and $b$ 
made in Proposition \ref{eqeiko} and in \eqref{formedeb} we have 
\begin{equation}\label{UN1}
e^{-i{\htilde}^{-1}\phi }(\htilde  \partial_t + \htilde c +iP)\big(e^{i{\htilde}^{-1}\phi }\widetilde{b} \big) = \htilde T_N \Psi_0  +U_N + R_N+S_N 
\end{equation}
 where $R_N$ is defined in \eqref{J=}, $S_N$ in \eqref{SN},   $T_N$ in \eqref{n200} and $U_N$ in \eqref{UN}.

\section{The dispersion estimate}

The purpose of this section is to prove the following result. Recall that $\delta = \frac{2}{3}.$
\begin{theo}\label{dispersive}
Let $\chi \in C_0^\infty(\xR^d)$ be such that $\supp \chi \subset \{\xi: \mez \leq \vert \xi \vert \leq 2\} $. Let $t_0 \in \xR$, 
$u_0 \in L^1(\xR^d)$ and set $u_{0,h} = \chi(hD_x)u_0$. 
Denote by $S(t,t_0)u_{0,h}$ the solution of the problem
\begin{equation*}
\Big(  \partial_t + \mez( T_{V_\delta} \cdot \nabla + \nabla \cdot T_{V_\delta}) 
+i T_{\gamma_\delta} \Big)U_{h}(t,x) =0,  \quad
U_{h}(t_0,x) = u_{0,h}(x).
\end{equation*}
Then there exist  $\mathcal{F}: \xR^+ \to \xR^+$ $ k= k(d) \in \xN$ and $h_0 >0$ such that  
$$
\Vert S(t,t_0)u_{0,h} \Vert_{L^\infty(\xR^d)} 
\leq  \mathcal{F}\big(\Vert V\Vert_{E_0} + \mathcal{N}_k(\gamma)\big) \, h^{-\frac{3d}{4}}\vert t- t_0 \vert^{-\frac{d}{2}} \Vert u_{0,h} \Vert_{L^1(\xR^d)},
$$
for all $0<\vert t-t_0 \vert \leq h^{\frac{\delta}{2}}$  and all $0<  h  \leq h_0.$
\end{theo}

This result will be a consequence of the following one in the variables $(t,z)$. 
\begin{theo}\label{dispersive1}
Let $\chi \in C_0^\infty(\xR^d)$ be such that $\supp \chi \subset \{\xi: \mez \leq \vert \xi \vert \leq 2\} $. 
Let $t_0 \in \xR$, 
$w_0 \in L^1(\xR^d)$ and set $w_{0,\htilde } = \chi(\htilde D_z)w_0$. 
Denote by $\widetilde{S}(t,t_0)w_{0,\htilde }$ the solution of the problem
\begin{equation*}
\big( \htilde  \partial_t + \htilde c+iP\big) \widetilde{U}(t,z) =0, \quad \widetilde{U}(t_0,z) = w_{0,\htilde }(z).
\end{equation*}
Then there exist  $\mathcal{F}: \xR^+ \to \xR^+$, $ k= k(d) \in \xN$ and $h_0 >0$ such that  
$$
\Vert \widetilde{S}(t, t_0)w_{0,\htilde } \Vert_{L^\infty(\xR^d)} 
\leq  \mathcal{F}\big(\Vert V\Vert_{E_0} + \mathcal{N}_k(\gamma)\big)\,   \htilde ^{-\frac{d}{2}}\vert t-t_0 \vert^{-\frac{d}{2}} \Vert w_{0,\htilde } \Vert_{L^1(\xR^d)}
$$
for all $0<\vert t -t_0 \vert\leq {\htilde }^{ \delta}$ and $0<{\htilde }  \leq \htilde _0$.
\end{theo}
Indeed suppose that Theorem \ref{dispersive1} is proved.  We can assume $t_0 =0$. 
According to the two change of variables $x =X(t,y)$ and $z= \widetilde {h}^{-1}y$ we have for any smooth function $U$ (see \eqref{semiclass} and \eqref{c=})
$$
\big( \htilde  \partial_t +\htilde c+ iP\big) \big [U (t,X(t, \htilde z)\big] 
= \htilde   \Big( \big( \partial_t + \mez( T_{V_\delta} \cdot \nabla + \nabla \cdot T_{V_\delta}) 
+i T_{\gamma_\delta}\big) U \Big)\big(t, X(t, \htilde z)\big).
$$
It follows that 
$$
\big( \widetilde{S}(t, 0)w_{0,\htilde }\big) (t,z) = \big( S(t,0)u_{0,h}\big) (t,X(t,\htilde z)).
$$
Moreover since $w_0(z) = u_0(\htilde z)$ we have
$$
w_{0,\htilde }(z)=(\chi(\htilde D_z)w_0) (z) = ( \chi(hD_x)u_0)( \htilde z) =u_{0,h}(\htilde z).
$$
 
Therefore using Theorem \ref{dispersive1} we obtain
\begin{equation*}
\begin{aligned}
\Vert  S(t, 0)u_{0,h} \Vert_{L^\infty(\xR^d)}  &= \Vert  \tilde{S}(t, 0)w_{0,\htilde }\Vert _{L^\infty(\xR^d)}  
\leq \mathcal{F}(\cdots)\,   \htilde ^{-\frac{d}{2}}t^{-\frac{d}{2}}\Vert  w_{0,\htilde } \Vert_{L^1(\xR^d)}   \\
& \leq \mathcal{F}(\cdots)\, \htilde ^{-\frac{d}{2}}t^{-\frac{d}{2}} \htilde ^{-d}\Vert u_{0,h} \Vert_{L^1(\xR^d)}\\
&\le \mathcal{F}(\cdots)\,  h^{-\frac{3d}{4}}t^{-\frac{d}{2}} \Vert u_{0,h} \Vert_{L^1(\xR^d)}
\end{aligned}
\end{equation*}
since $\htilde  = h^\mez.$ Thus Theorem \ref{dispersive} is proved.
 \begin{proof}[Proof of Theorem \ref{dispersive1}]
We set 
\begin{equation}\label{K1}
\mathcal{K}w_{0,\htilde }(t,z) 
= (2 \pi \htilde )^{-d} \iint e^{i{\htilde}^{-1}(\phi(t,z,\xi,\htilde ) -y\cdot \xi)}\widetilde{b} \big(t,z,y,\xi,\htilde\big)
\chi_1(\xi)w_{0,\htilde}(y) \,dy \, d\xi
\end{equation}
where $\chi_1$ belongs to $C^\infty_0(\xR^d)$ with $\chi_1 \equiv 1$ 
on the support of $\chi$ and $\tilde{b}$ is defined in \eqref{tildeb}. We can write
\begin{equation} \label{K2}
\begin{aligned}
\mathcal{K}w_{0,\htilde }(t,z) &= \int K\big(t,z,y,\htilde \big) w_{0,\htilde }(y) dy \quad \text{with}\\
K(t,z,y,\htilde ) 
&=  (2 \pi \htilde )^{-d} \int e^{i\htilde ^{-1} (\phi (t,z,\xi,\htilde )-y\cdot \xi )}\widetilde{b}  (t,z,y,\xi,\htilde\big)
\chi_1(  \xi) \, d\xi.
\end{aligned}
\end{equation}
It follows from the \eqref{HessphiND} for the hessian, \eqref{formedeb}, 
Proposition~\ref{b:existe} and the stationary 
phase theorem that 
$$
\big\vert  K(t,z,y,\htilde ) \big\vert 
\leq \mathcal{F}(\cdots)\,  \htilde ^{-d} \htilde ^{\frac{d}{2}} t^{- \frac{d}{2}} \leq  \mathcal{F}(\cdots)  \htilde ^{-\frac{d}{2}} t^{- \frac{d}{2}}
$$
for all $0<t \leq h^{ \frac{\delta}{2}}$, $z,y \in \xR^d$ and $0< \htilde  \leq \htilde _0$.

Therefore we obtain
\begin{equation}\label{estimK}
\big\Vert  \mathcal{K}w_{0,\htilde }(t,\cdot) \big\Vert_{L^\infty(\xR^d)}
\leq \mathcal{F}(\cdots)\,   \htilde ^{-\frac{d}{2}} t^{- \frac{d}{2}} \big\Vert w_{0,\htilde } \big\Vert_{L^1(\xR^d)}
\end{equation}
for all $0<t \leq \htilde^{\delta} $ and $0< \htilde  \leq \htilde _0.$

We can state now the following result.
\begin{prop}\label{EQK}
Let $\sigma_0 $ be an integer such that $\sigma_0> \frac{d}{2}.$ Set
$$
\big( \htilde  \partial_t + \htilde c+ iP\big)\Big( \mathcal{K}w_{0,\htilde } \Big)(t,z) = F_{\htilde }(t,z).
$$
Then there exists  $k= k(d) \in \xN$ and  for any $N \in \xN,$  $\mathcal{F}_N:\xR^+ \to \xR^+$  such that
$$
\sup_{0< t \leq \htilde ^{\delta}} \big\Vert F_{\htilde }(t, \cdot ) \big\Vert_{H^{\sigma_0}(\xR^d)}
\leq  \mathcal{F}_N\big(\Vert V\Vert_{E_0} + \mathcal{N}_k(\gamma)\big)\,  \htilde ^N \Vert w_{0,\htilde } \Vert_{L^1(\xR^d)}.
$$
\end{prop}
 We shall use the following result.
\begin{lemm}\label{estint}
Let $k_0>\frac{d}{2}$. Let us set $m(t,z,y,\xi,\htilde ) = \frac{\partial \phi}{\partial \xi} (t,z ,\xi, \htilde ) -y$ 
and 
$$
\Sigma = \left\{(t,y,\xi,\htilde ):  0< \htilde \leq {\tilde {h}}_0, 0
\leq t \leq \htilde ^{\delta}, y \in \xR^d, \vert \xi \vert \leq C \right\}.
$$
Then
$$
\sup_ {(t,y,\xi,\htilde ) \in \Sigma} \int_{\xR^d} \frac{dz}{\langle m(t,z,y,\xi, \htilde )\rangle^{2k_0}} \leq  \mathcal{F}\big(\Vert V\Vert_{E_0} + \mathcal{N}_2(\gamma)\big).
$$
\end{lemm}
\begin{proof}
By \eqref{eiksuite} we have $\frac{\partial \phi}{\partial z}\big(t,z,\xi,\htilde\big) 
= \zeta\big(t;\kappa\big(t;z,\xi,\htilde\big),\xi,\htilde\big)$ so 
$\frac{\partial^2 \phi}{\partial\xi \partial z} = \frac{\partial \zeta}{\partial z}  \frac{\partial \kappa}{\partial \xi}  +  \frac{\partial \zeta}{\partial \xi} .$ We deduce from Proposition \ref{estflow} and Corollary \ref{estkappa} that  $\vert \frac{\partial^2 \phi}{\partial\xi \partial z} - Id\vert \leq  \mathcal{F} \big(\Vert V \Vert_{E_0} +   \mathcal{N}_2(\gamma)\big) \, \htilde^{\frac{\delta}{2}}. $ It follows  that the  map $z \mapsto  \frac{\partial \phi}{\partial \xi} (t,z,\xi, \htilde )$  is proper and therefore is a global diffeomorphism from $\xR^d $ to itself. Consequently, 
we can perform the change of variable $X=  \frac{\partial \phi}{\partial \xi} (t,z,\xi, \htilde )$ and the lemma follows.
\end{proof}
\begin{proof}[Proof of Proposition \ref{EQK}]
 
According to \eqref{K2} Proposition \ref{EQK} will be proved if we have
\begin{equation}\label{estK}
\sup_{(t,y,\tilde{h}) \in \Sigma}  \big\Vert \big(\tilde{h} \partial_t +\htilde c+  iP\big)K(t,\cdot,y,\tilde{h}\big)\big\Vert_{H^{\sigma_0}}
\leq \mathcal{F}_N(\cdots) \, \tilde{h}^N.
\end{equation}
Now, setting $L= \tilde{h} \partial_t + \htilde c +  iP(t,z,D_z)$, we have
\begin{equation}\label{egalite1}
LK\big(t,z,y,\tilde{h}\big)= \big(2\pi \tilde{h}\big)^{-d}\int  e^{-i{\htilde}^{-1}  y\cdot \xi}
L\Big(e^{i{\htilde}^{-1} \phi(t,z,\xi,\tilde{h})}\widetilde{b} \big(t,z,y,\xi,\tilde{h}\big)\Big) \chi_1(\xi)\, d\xi 
\end{equation}
and according to \eqref{RS}, \eqref{CT1}, \eqref{n200} we have,
\begin{equation}\label{egalite2}
L\Big(e^{i{\htilde}^{-1}\phi(t,z,\xi,\tilde{h})}\widetilde{b} \big(t,z,y,\xi,\tilde{h}\big)\Big)
= e^{i{\htilde}^{-1}\phi(t,z,\xi,\tilde{h})}\big(R_N +S_N + \tilde{h}T_N \Psi_0 +U_N\big)(t,z,y,\xi,\htilde) 
\end{equation}
where $R_N$, $S_N$, $T_N$, $U_N$ are defined in \eqref{J=}, \eqref{SN}, \eqref{CT1}, \eqref{UN}.

\begin{lemm}\label{estRNSN}
Let $\sigma_0,k_0 $ be   integers, $\sigma_0 > \frac{d}{2},$  $ k_0 >\frac{d}{2}$. There exists  a fixed integer  $N_0(d)$   such that for any $N \in \xN$  there exists $C_N >0$ such that, if we set $\Xi = (t,z,y,\xi,\htilde),$ then
\begin{equation}\label{estRN}
  \big\langle m\big(\Xi\big)\big\rangle^{k_0} 
\left\{\vert D^\beta_z R_N(\Xi) \vert 
+ \vert D^\beta_z S_N (\Xi)\vert + \vert D_z^\beta (\tilde{h}(T_N \Psi_0) (\Xi)\big)\vert  \right\} 
 \leq\mathcal{F}_N(\cdots) \, {\tilde{h}}^{\delta N-N_0}, 
\end{equation}
 for all $\vert \beta \vert \leq \sigma_0$,  all $(t,y,\xi, \tilde{h}) \in \Sigma $ and all $z\in \xR^d$. 
  \end{lemm}
\begin{proof}
According to \eqref{J=}, $  D^\beta_z R_N(\Xi)$ is a finite linear combination 
of terms of the form \begin{equation*}
\begin{aligned}
R_{N, \beta}(\Xi)  = \tilde{h} ^{-d-\vert \beta_1\vert} 
\iint e^{i{\htilde}^{-1}(z-z')\cdot\eta}& \eta^{\beta_1}\kappa(\eta) D_z^{\beta_2}r_N(t,z, z',\eta, \xi,\htilde)\\
&{b}(t,z', \xi,\tilde{h}) \Psi_0\Big(\frac{\partial \phi}{\partial \xi} \big(t,z',\xi, \tilde{h}\big) -y\Big)\,dz' \,d\eta\\
\end{aligned}
\end{equation*}
where  $\beta_1 + \beta_2 = \beta$  and 
\begin{equation*}
D_z^{\beta_2}r_N (\cdots) = \sum_{\vert \alpha \vert =N} 
\frac{N}{\alpha!}\int_0^1 (1-\lambda)^{N-1} D_z^{\beta_2}
\Big[\big(\partial^\alpha_\eta \widetilde{p}\big)(t,z,z',\eta+ \lambda\theta(t,z,z', \xi,\htilde)) \Big ]\eta^\alpha \,d\lambda.
\end{equation*}

Using the equality $\eta^{\alpha + \beta_1}  e^{i{\htilde}^{-1}(z-z')\cdot\eta} 
= (-\tilde{h}D_{z'})^{\alpha+ \beta_1}   e^{i{\htilde}^{-1}(z-z')\cdot\eta}$ 
we see that $R_{N, \beta}$ is a linear combination of  terms of the form
\begin{align*}
R'_{N, \beta}(\Xi)= \tilde{h}^{N   -d}\int_0^1 
\iint  &e^{i{\htilde}^{-1}(z-z')\cdot\eta}\kappa(\eta)\times \\ 
& D_{z'}^{\alpha + \beta_1} 
D_z^{\beta_2}\Big[(\partial_\eta^\alpha  \widetilde{p})(t,z,z',\eta+ \lambda\theta(t,z,z', \xi,\htilde))\\
&b(t,z',\xi,\tilde{h}) \Psi_0\big(\frac{\partial \phi}{\partial \xi} (t,z'\xi, \tilde{h}) -y\big)\Big] 
\, d\lambda \, dz' \, d\eta.
\end{align*}
Then we insert in the integral  the quantity 
$  \big (\frac{\partial \phi}{\partial {\xi }} (t,z ,\xi, \tilde{h}) -y \big)^{\gamma }$ where $  \vert \gamma\vert =  k_0. $  It is a finite linear combination of terms of the form
$$
\Big(\frac{\partial \phi}{\partial {\xi }} (t,z ,\xi, \tilde{h}) 
- \frac{\partial \phi}{\partial{\xi }} (t,z',\xi, \tilde{h})\Big)^{\gamma_1 } 
\Big(\frac{\partial \phi}{\partial {\xi }} (t,z',\xi, \tilde{h}) -y\Big)^{\gamma_2 }.
$$
Using the Taylor formula we see that $ \langle m(\Xi) \rangle^{k_0} R'_{N, \beta}(\Xi)$ 
is a finite linear combination of terms of the form
\begin{equation*}
\begin{aligned}
& {\tilde{h}}^{N -d} \int_0^1 \iint  (z-z')^{\nu}  e^{i{\htilde}^{-1}(z-z')\cdot\eta} F(t,z,z',\xi,\tilde{h}) 
\kappa(\eta)\Big(\frac{\partial \phi}{\partial {\xi }} (t,z',\xi, \tilde{h}) -y\Big)^{l}\\ 
&\qquad\qquad D_{z'}^{\alpha+ \beta_1} D_z^{\beta_2} \Big[
(\partial_\eta^\alpha  \widetilde{p})(\cdots)
b(t,z',\xi,\tilde{h}) \Psi_0\big(\frac{\partial \phi}{\partial \xi} (t,z',\xi, \tilde{h}) -y\big)\Big] \, d\lambda \, dy \, d\eta,  
\end{aligned}
\end{equation*}
where, by Corollary \ref{estkappa} (ii), $F$  is a bounded function.

Eventually we use  the identity 
$(z-z')^{\nu}  e^{i{\htilde}^{-1}(z-z')\cdot\eta} = (\htilde D_\eta)^\nu  e^{i{\htilde}^{-1}(z-z')\cdot\eta},$ we 
integrate by parts in the integral with respect to $\eta$ and we use Remark \ref{rem:est:ptilde}, Remark \ref{rem:est:theta}, the estimate $(ii)$ in Corollary \ref{estkappa}, the fact that $b\in S^0_{2\delta-1}$  and the fact that $N+ N(1-2\delta) = \delta N$ to deduce that   
\begin{equation}\label{est:RN}
  \big\langle m\big(\Xi\big)\big\rangle^{k_0}  \vert D_z^\beta R_N(\Xi) \vert \leq \mathcal{F}_N(\cdots) \, {\htilde }^{\delta N -N_d}
\end{equation}
where    $N_d$ is a fixed number depending only on the dimension.
 
Let us consider the term $S_N.$ Recall that $D_z^\beta S_N$ 
is a finite linear combination for $\vert \alpha \vert \leq N-1$ and $ \vert \gamma \vert = N$ of terms of the form
\begin{multline*}
  S_{N,\alpha, \beta,Ê\gamma} = {\htilde }^{N+\vert \alpha \vert}\int_0^1 \int (1-\lambda)^{N-1}\mu^\gamma \widehat{\kappa}(\mu) D_{z'}^{\alpha + \beta + \gamma}
\Big[\big(\partial^\alpha_\zeta \widetilde{p}\big)\big(t,z,z',\theta (t,z,z',\xi,\htilde),\htilde\big)\\
b\big(t,z',\xi,\htilde\big) \Psi_0\Big(\frac{\partial \phi}{\partial \xi }\big(t,z',\xi,\htilde\big)
-y\Big)\Big] \Big\arrowvert_{z'=z+\lambda \htilde  \mu}\, d\lambda \, d\mu.
  \end{multline*}
Then we multiply  $S_{N,\alpha, \beta, \gamma}$ by $ \langle m(\Xi) \rangle^{ k_0}$ 
and we write
$$
\frac{\partial \phi}{\partial \xi }\big(t,z,\xi,\htilde\big) -y
= \frac{\partial \phi}{\partial \xi }\big(t,z,\xi,\htilde\big)-
\frac{\partial \phi}{\partial \xi }\big(t,z+\lambda \htilde \mu,\xi,\htilde\big) 
+ \frac{\partial \phi}{\partial \xi }\big(t,z+\lambda \htilde \mu,\xi,\htilde\big)-y.
$$
By the Taylor formula the first term will give rise to a power of $\lambda \htilde  \mu$ 
which will be absorbed by $\hat{\kappa}(\mu)$ and the second term will be absorbed by $\Psi_0.$ Then we use again Remark \ref{rem:est:ptilde}, Remark \ref{rem:est:theta} to conclude that 
\begin{equation}\label{est:SN}
 \big\langle m\big(\Xi\big)\big\rangle^{k_0}  \vert D_z^\beta S_N(\Xi) \vert \leq \mathcal{F}_N(\cdots) \,{\htilde }^{\delta N -N'_d}.
\end{equation}
Let us look now to the term $\htilde D_z^\beta T_N \Psi_0$. 
According to \eqref{n200} this expression is a linear combination for $2 \leq \vert \alpha \vert  \leq N$ of terms of the form
\begin{align*}
{\htilde }^{\vert \alpha\vert +N }
D_z^\beta \Big\{ &D_{z'}^ \alpha \Big[(\partial_\zeta^\alpha \widetilde{p}(t,x,z ,z',\theta(t,z,z',\xi,\htilde),\htilde )
b_N(t,z',\xi,\htilde )\Big]\Big\arrowvert_{z'=z} \\
&\Psi_0\Big(\frac{\partial \phi}{\partial \xi }(t,z',\xi,\htilde ) -y\Big)\Big\}.
\end{align*}
It follows from Remark \ref{rem:est:ptilde}, Remark \ref{rem:est:theta} and Corollary \ref{estkappa} that we have 
\begin{equation}\label{est:TN}
\big \langle m (\Xi ) \big\rangle^{ k_0}
\big\vert \htilde D_z^\betaÊ\big[T_N \Psi_0 \big] \big\vert \leq \mathcal{F}_N(\cdots) \, {\htilde }^{\delta N}.
\end{equation}
Lemma \ref{estRNSN} follows from \eqref{est:RN}, \eqref{est:SN},  \eqref{est:TN}.
\end{proof}
>From   Lemma \ref{estRNSN} we can write
\begin{equation}\label{estint2}
{\htilde }^{-d}\int  \big\Vert e^{i{\htilde}^{-1}( \phi(t,z,\xi,\htilde )- y\cdot \xi)} 
\big(R_N + S_N + \htilde  T_N \Psi_0\big)\big\Vert_{H^{\sigma_0}_z} \vert \chi_1(\xi) \vert \, d \xi 
\leq \mathcal{F}_N(\cdots) \, {\htilde }^{\delta N-N_1(d)}
\end{equation}
where $N_1(d)$ is a fixed number depending only on the dimension.

To conclude the proof of Proposition \ref{EQK} we have to estimate the integrals
$$
I_{N, \beta} =   (2\pi \htilde )^{-d}\int D_z^\beta \big[ e^{i{\htilde}^{-1}( \phi(t,z,\xi,\htilde )- y\cdot \xi)}
U_N\big(t,z,y,\xi,\htilde\big)\big] \chi_1(\xi) \,d \xi.
$$
 
Now according to \eqref{RS}, \eqref{CT1} and \eqref{UN} on the support of $U_N$ 
the function $\Psi_0$ 
is differentiated at least one time. Thus on this support 
one has $\big\vert \frac{\partial \phi}{\partial \xi} (t,z ,\xi, \htilde ) -y\big\vert \geq 1$.  
Then we can use the vector field
$$
X= \frac{\htilde }{\big\vert \frac{\partial \phi}{\partial \xi} (t,z ,\xi, \htilde ) -y\big\vert^2} 
\sum_{j=1}^d \Big(\frac{\partial \phi}{\partial {\xi_j}} (t,z ,\xi, \htilde ) -y_j\Big) D_{\xi_j}
$$
to integrate by parts in $I_{N,\beta}$   to obtain
$$
\vert I_{N,\beta} \vert \leq \mathcal{F}_N(\cdots) \,{\htilde }^{\delta N}.
$$
The proof of Proposition \ref{EQK} is complete.
\end{proof}

We show now that
\begin{equation}\label{rh}
\mathcal{K}w_{0,\htilde }(0,z)=  w_{0,\htilde }(z) + r_{\htilde }(z)
\end{equation}
with
\begin{equation}\label{estrh}
\big\Vert r_{\htilde }\big\Vert_{H^{\sigma_0}(\xR^d)}
\leq \mathcal{F}_N(\cdots) \, {\htilde }^N \big\Vert w_{0,\htilde }\big\Vert_{L^1(\xR^d)} \quad \forall N \in \xN.
\end{equation}
 
It follows from the initial condition on $\phi$ given in \eqref{eikonale} and the initial condition on $b$ that \eqref{rh} is true with
$$
r_{\htilde }(z)=  (2 \pi \htilde )^{-d} \iint e^{i \htilde ^{-1}(z -y) \cdot \xi} \chi_1(\xi)(1-\Psi_0(z-y))w_{0,\htilde }(y) \, dy \, d\xi.
$$
We see easily that for $\vert \beta \vert \leq \sigma_0, D_z^\beta r_{\htilde }(z)$ is a finite linear combination of terms of the form
$$
r_{\htilde ,\beta}(z) = {\htilde }^{-d-\vert \beta_1\vert}  
\iint e^{i \htilde ^{-1}(z -y) \cdot \xi} \xi^{\beta_1}\chi_1(\xi)\Psi_{\beta}(z-y)
w_{0,\htilde }(y) \, dy \,d\xi, \quad \vert \beta_1  \vert \leq \vert \beta \vert
$$
where $\Psi_\beta \in C^\infty_b(\xR^d)$ and $\vert z-y \vert \geq 1$ on the support of $\Psi_\beta $. Then one can write
$$
r_{\htilde ,\beta}(z) = \big(F_{\htilde }\ast w_{0, \htilde }\big)(z)
$$
where 
$$
F_{\htilde }(X) = {\htilde }^{-d-\vert \beta_1\vert}  \int e^{i \htilde ^{-1}X \cdot \xi} \xi^{\beta_1}\chi_1(\xi)\Psi_{\beta}(X)\, d\xi
$$
and $\vert X \vert \geq 1$ on the support of $\Psi_\beta(X).$
Then we remark that if we set $L = \frac{1}{\vert X \vert^2} \sum_{j=1}^d X_j \frac{\partial}{\partial \xi_j}$ we have \, $\htilde L e^{i \htilde ^{-1}X \cdot \xi} =e^{i \htilde ^{-1}X \cdot \xi}$. 
Therefore one can write
$$
F_{\htilde }(X) = {\htilde }^{M-d-\vert \beta_1\vert}  
\int e^{i \htilde ^{-1}X \cdot \xi} (-L)^M\big[\xi^{\beta_1}\chi_1(\xi)\big]\Psi_{\beta}(X) \, d\xi
$$ 
from which we deduce
$$
\vert F_{\htilde }(X) \vert \leq \mathcal{F}_M(\cdots) \, {\htilde }^{M-d-\vert \beta_1 \vert} 
\frac{\vert \tilde{\Psi} (X)\vert}{\vert X \vert^M}, \quad \forall M\in \xN
$$
where $\tilde{\Psi} \in C^\infty_b(\xR^d)$ is equal to $1$ on the support of $\Psi_\beta.$

It follows then that $ \Vert F_{\htilde }\Vert_{L^2(\xR^d)} \leq \mathcal{F}_M(\cdots) \, {\htilde }^{M-d-\vert \beta_1 \vert}$ 
from which we deduce that for $\vert \beta \vert \leq \sigma_0$, 
\begin{align*}
\big\Vert D^\beta r_{\htilde }\big\Vert_{L^2(\xR^d)}
&\leq \mathcal{F}_N(\cdots) \, \big\Vert F_{\htilde }\big\Vert_{L^2(\xR^d)}
\big\Vert w_{0,\htilde }\big\Vert_{L^1(\xR^d)}\\
&\leq \mathcal{F}_N(\cdots) \, {\htilde }^{M-d-\vert \beta_1 \vert}\big\Vert w_{0,\htilde }\big\Vert_{L^1(\xR^d)}
\end{align*}
which proves \eqref{estrh}.

Using Proposition \ref{EQK} and the Duhamel formula one can write
$$
\tilde{S}(t,0)w_{0,\htilde }(z)
=\mathcal{K}w_{0,\htilde }(t,z)  - \tilde{S}(t,0) r_{\htilde }(z) - \int_0^t  \tilde{S}(t,s)[F_{\htilde }(s,z)] \,ds.
$$
Now we can write
\begin{equation*}
\begin{aligned}
 \lA \int_0^t  \tilde{S}(t,s)\big[F_{\htilde }(s,z)\big] \,ds\rA_{L^\infty(\xR^d)}
 &\leq \int_0^t  \lA \tilde{S}(t,s)\big[F_{\htilde }(s,z)\big]\rA_{H^{\sigma_0}(\xR^d)} \, ds \\
 &\leq C \int_0^t\lA  F_{\htilde }(s,z) \rA_{H^{\sigma_0}(\xR^d)} \, ds \\
 &\leq \mathcal{F}_N(\cdots) \, \htilde ^N \big\Vert w_{0,\htilde }\big\Vert_{L^1(\xR^d)}
\end{aligned}
\end{equation*}
and, for every $N \in \xN$,
\begin{align*}
\big\Vert \tilde{S}(t,0) r_{\htilde}\big\Vert_{L^\infty(\xR^d)}
&\leq C \big\Vert \tilde{S}(t,0) r_{\htilde}\big\Vert_{H^{\sigma_0}(\xR^d)}
\leq C'  \big\Vert r_{\htilde}\big\Vert_{H^{\sigma_0}(\xR^d)}\\
&\leq \mathcal{F}_N(\cdots) \,{\htilde }^N \Vert w_{0,\htilde }\big\Vert_{L^1(\xR^d)}.
\end{align*}
Then Theorem \ref{dispersive1} follows from these estimates and \eqref{estimK}. 
\end{proof}

  \section{The Strichartz estimates}

\begin{theo}\label{semicl}
Consider the problem
\begin{equation*}
\Big(  \partial_t + \mez( T_{V_\delta} \cdot \nabla + \nabla \cdot T_{V_\delta}) +i T_{\gamma_\delta} \Big)u_{h}(t,x) =f_h (t,x),  \quad
u_{h}(t_0,x) = u_{0,h}(x).
\end{equation*}
where $u_h, u_{0,h} $ and $f_h$ 
have spectrum  in  $\big\{ \xi : c_1h^{-1} \leq \vert \xi \vert \leq c_2 h^{-1}\big\}$. Let  $I_h = (0, h^{ \frac{\delta}{2}}).$
 
Then there exists  $k = k(d), h_0>0$ such that for any $s \in \xR$ and $ \eps>0$ there exists $\mathcal{F}, \mathcal{F}_\eps :\xR^+ \to \xR^+$,  such that, with $N\defn   \Vert V\Vert_{E_0} + \mathcal{N}_k(\gamma)$, 
\begin{align*}
& (i) \quad \text{if }  d=1: \\
  &  \quad \Vert u_h \Vert_{L^{4}(I_h, W^{s - \frac{3}{8} , \infty}(\xR ))} 
\leq \mathcal{F}\big(N\big) \, \Big( \Vert u_{0,h}\Vert_{H^s(\xR)} + \Vert f_h \Vert_{L^1(I_h, H^s (\xR))} \Big), \\
& (ii) \quad \text{if }  d\geq 2:\\
&   \quad \Vert u_h \Vert_{L^{2+\eps}(I_h, W^{s - \frac{d}{2} + \frac{1}{4} - \eps, \infty}(\xR^d))} 
\leq \mathcal{F}_\eps\big(N\big) \,  \Big( \Vert u_{0,h}\Vert_{H^s(\xR^d)} + \Vert f_h \Vert_{L^1(I_h, H^s (\xR^d))} \Big),    
  \end{align*}
  for any  $0<h \leq h_0$  
 \end{theo}
\begin{proof}
If $d =1,$   by the $TT^*$ argument  we deduce from the dispersive estimate   given in Theorem \ref{dispersive} that 
\begin{equation*}
\Vert u_h \Vert_{L^4(I_h, L^\infty(\xR))} \leq \mathcal{F}(\cdots)\, h^{-\frac{3}{8}}
\Big(  \Vert u_{0,h}\Vert_{L^2(\xR)} + \Vert f_h \Vert_{L^1(I_h,L^2 (\xR))} \Big).
\end{equation*}
Then multiplying this estimate by $h^s$   and using the fact that $u_h,u_{0,h},f_h$ 
are spectrally supported in     $ \{ \xi : c_1h^{-1} \leq \vert \xi \vert \leq c_2 h^{-1} \}$  we deduce
$(i).$  
 
If $d\geq 2$ we use the same argument. Then   if $(q,r)\in \xR^2$ is such that $ q>2$ and 
$\frac{2}{q} = \frac{d}{2} - \frac{d}{r}$  we obtain
\begin{equation}\label{strich1}
\Vert u_h \Vert_{L^q(I_h, L^r(\xR^d))} \leq \mathcal{F}(\cdots)\, h^{-\frac{3}{2q}}
\Big(  \Vert u_{0,h}\Vert_{L^2(\xR^d)} + \Vert f_h \Vert_{L^1(I_h,L^2 (\xR^d))} \Big).
\end{equation}
Taking $q = 2+ \eps$ we find $r= 2+ \frac{8}{(2+\eps)d-4}.$  Moreover $ h^{-\frac{3}{2q}} \leq h^{-\frac{3}{4}}.$ Then multiplying both members of \eqref{strich1} by $h^s$ we obtain 
$$
\Vert u_h \Vert_{L^{2+ \eps}(I_h,W^{s - \frac{3}{4}, r}(\xR^d))} 
\leq \mathcal{F}(\cdots)\,  \Big(  \Vert u_{0,h}\Vert_{H^s(\xR^d)} + \Vert f_h \Vert_{L^1(I_h,H^s (\xR^d))} \Big).
$$
On the other hand the Sobolev embedding shows that $W^{a+b,r}(\xR^d) \subset W^{b,\infty} (\xR^d)$ 
provided that $a >  \frac{d}{r} =  \frac{d}{2} -1 + \frac{\eps}{2+ \eps}$. 
In particular we can take $a = \frac{d}{2} -1 + \eps$. Taking $ \frac{d}{2} -1+ \eps +b =s - \frac{3}{4}$
we obtain the conclusion of the Theorem.
\end{proof}

\begin{coro}\label{strich2}
With  the notations in Theorem \ref{semicl} and  $\delta = \frac{2}{3}, I=[0,T]$ we have

  $(i)  \quad   \text{if  }   d=1$: 
  \begin{equation*}
  \begin{aligned}
     \quad  \Vert  u_{h} \Vert_{L^4(I; W^{s - \frac{3}{8}-\frac{\delta}{8}, \infty}(\xR))} \leq 
     \mathcal{F}\big(N\big) \,  
     &\Big(
     \Vert  f_h\Vert_{L^4(I; H^{s-\frac{\delta}{2}}(\xR))} 
     +  \Vert u_h\Vert_{C^0(I; H^{s }(\xR))}\Big),
  \end{aligned}
  \end{equation*} 
     $(ii)  \quad     \text{if  }  d\geq 2$: 
     \begin{equation*}
  \begin{aligned}
     \quad   \Vert  u_{h} \Vert_{L^2(I; W^{s - \frac{d}{2} +  \frac{1}{4} -\frac{\delta}{4} - \eps, \infty}(\xR^d))}  
  \leq \mathcal{F}_\eps\big(N\big) \,  &\Big(\Vert  f_h\Vert_{L^2(I; H^{s-\frac{\delta}{2}}(\xR^d))} 
  +  \Vert u_h\Vert_{C^0(I; H^s(\xR^d))}\Big)
    \end{aligned}
 \end{equation*}
   for any $\eps>0.$ 
\end{coro}
\begin{proof}
Let $T>0$ and  $\chi \in C_0^\infty(0,2)$ equal to one on $[ \mez, \frac{3}{2}].$ For $0 \leq k \leq [Th^{-\frac{\delta}{2}}]-2$ define
$$
I_{h,k} = [kh^\frac{\delta}{2}, (k+2)h^\frac{\delta}{2}], \quad \chi_{h,k}(t) = \chi\Big( \frac{t-kh^\frac{\delta}{2}}{h^\frac{\delta}{2}}\Big), \quad  u_{h,k} = \chi_{h,k}(t)u_h.
$$
Then
\begin{equation*} 
\Big(  \partial_t + \mez( T_{V_\delta} \cdot \nabla + \nabla \cdot T_{V_\delta}) 
+i T_{\gamma_\delta} \Big)u_{h,k} 
= \chi_{h,k}f_h   +h^{-\frac{\delta}{2}}\chi'\Big( \frac{t-kh^\frac{\delta}{2}}{h^\frac{\delta}{2}}\Big)u_h
\end{equation*}
and $u_{h,k}(kh^\frac{\delta}{2},\cdot)=0$. 

Consider first the case $d=1$. Applying Theorem \ref{semicl}, $(i)$ to each $u_{h,k}$ on the interval $I_{h,k}$   we obtain, since $\chi_{h,k}(t) = 1$ for $(k+\mez)h^\frac{\delta}{2}\leq t \leq (k+\frac{3}{2})h^\frac{\delta}{2}$, 
\begin{equation*} 
\begin{aligned}
\Vert &u_{h } \Vert_{L^4(  (k+\mez)h^\frac{\delta}{2},(k+\frac{3}{2})h^\frac{\delta}{2});W^{s - \frac{3}{8} , \infty}(\xR))}  \\
& \leq \mathcal{F}(\cdots) \Big( \Vert   f_h \Vert_{L^1(  (kh^\frac{\delta}{2},(k+2)h^\frac{\delta}{2} );H^s(\xR))} 
+ h^{-\frac{\delta}{2}} \Vert \chi'\Big( \frac{t-kh^\frac{\delta}{2}}{h^\frac{\delta}{2}}\Big)u_h \Vert_{L^1(\xR;H^s(\xR))}\Big)\\
&\leq \mathcal{F}(\cdots)  \Big(h^{ \frac{3\delta}{8}} 
\Vert   f_h \Vert_{L^4(   (kh^\frac{\delta}{2},(k+2)h^\frac{\delta}{2} );H^s(\xR))} 
+ \Vert u_h\Vert_{L^\infty(I; H^s(\xR))}\Big).
\end{aligned}
\end{equation*}
    Multiplying both members of the above inequality by $h^\frac{\delta}{8}$ 
and taking into account that $u_h$ and $f_h$ 
are spectrally supported in a ring of size $h^{-1}$ we obtain 
\begin{equation}\label{global3}
\begin{aligned}
\Vert &u_{h } \Vert_{L^4((k+\mez)h^\frac{\delta}{2} ,(k+\frac{3}{2})h^\frac{\delta}{2}); W^{s - \frac{3}{8}-\frac{\delta}{8}, \infty}(\xR))} \\
&\leq    \mathcal{F}(\cdots)   \Big(\Vert  f_h\Vert_{L^4((kh^\frac{\delta}{2},(k+2)h^\frac{\delta}{2}  ); H^{s-\frac{\delta}{2}}(\xR))} 
+ h^\frac{\delta}{8}  \Vert u_h\Vert_{L^\infty(I; H^s(\xR))}\Big).
\end{aligned}
\end{equation}
Taking the power $4$ of \eqref{global3}, summing in $k$ from $0$ to  $[Th^{-\frac{\delta}{2}}]-2$ 
  we obtain (since there are $\approx T h^{-\frac{\delta}{2}}$ intervals)
\begin{equation*}\label{global4}
 \Vert  u_{h} \Vert_{L^4(I; W^{s - \frac{3}{8}-\frac{\delta}{8}, \infty}(\xR))}\\
  \leq   \mathcal{F}(\cdots)  \Big(
  \Vert  f_h\Vert_{L^4(I; H^{s-\frac{\delta}{2}}(\xR))} 
+  \Vert u_h\Vert_{C^0(I; H^s(\xR))}\Big).
  \end{equation*}
This completes the proof of $(i)$. 

The proof of $(ii)$ follows exactly the same path.  We apply Theorem \ref{semicl}, $(ii)$ to each $u_{h,k}$ on the interval $I_{h,k}.$ The only difference with the  case $d=1$ is that,   passing from the $L^1$ norm in $t$ of $f_h$  to the  $L^2$ norm, it  gives rise to a $h^\frac{\delta}{4}$ factor. Therefore we multiply the inequality by  $h^\frac{\delta}{4},$ we take the square of the new inequality and we sum in $k$.
\end{proof}

\begin{coro}\label{strich3}
Consider the problem
\begin{equation*}
\Big(  \partial_t + \mez( T_{V } \cdot \nabla + \nabla \cdot T_{V }) 
+i T_{\gamma } \Big)u_{h}(t,x) =F_h (t,x),  \quad
u_{h}(t_0,x) = u_{0,h}(x).
\end{equation*}
where $u_h, u_{0,h} $ and $F_h$ is spectrally supported in 
$ \{ \xi : c_1h^{-1} \leq \vert \xi \vert \leq c_2 h^{-1} \}.$  

  Then there exists  $k = k(d), h_0>0$ such that for any $s \in \xR$, for any $T>0$, $\eps>0$  one can find $\mathcal{F}, \mathcal{F}_\eps: \xR^+ \to\xR^+ $ such that with $I =[0,T]$ 
  and $N\defn \Vert V\Vert_{E_0} +  \mathcal{N}_k(\gamma)$,
  
  $(i)$ if $d=1$: 
\begin{equation*}
\begin{aligned}
  \Vert  u_{h} \Vert_{L^4(I; W^{s - \frac{3}{8}-\frac{\delta}{8}, \infty}(\xR))} 
  \leq \mathcal{F}\big(N\big) \,  
  \Big( \Vert  F_h\Vert_{L^4(I; H^{s }(\xR))}
+  \Vert u_h\Vert_{C^0(I; H^{s }(\xR))}\Big), 
\end{aligned} 
\end{equation*}
   $(ii)$ if $d\geq 2$: 
 \begin{equation*}
 \begin{aligned}
 \Vert  u_{h} \Vert_{L^2(I; W^{s - \frac{d}{2} +  \frac{1}{4} -\frac{\delta}{4} - \eps, \infty}(\xR^d))}  
  \leq \mathcal{F}\big(N\big) \,  
  \Big( \Vert  F_h\Vert_{L^2(I; H^{s }(\xR^d))}
+  \Vert u_h\Vert_{C^0(I; H^s(\xR^d))}\Big) .
\end{aligned}
   \end{equation*}
    \end{coro}
 \begin{proof}
Applying Corollary \ref{strich2} we see that Corollary \ref{strich3} will follow from the following estimates with $ p =  4$ if $d=1$ and $p=2$ if $d\geq 2.$
\begin{equation}\label{diffV} 
\left\{ 
\begin{aligned}
&\Vert  \big( T_V - T_{V_\delta}\big).\nabla u_h \Vert_{L^p(I; H^{s-\frac{\delta}{2}}(\xR^d))} \leq \mathcal{F}(\cdots)  \Vert u_h\Vert_{L^\infty(I; H^s(\xR^d))} \\
&\Vert  \big( T_{\nabla V} - T_{{\nabla V}_\delta}\big)  u_h \Vert_{L^p(I; H^{s-\frac{\delta}{2}}(\xR^d))} \leq \mathcal{F}(\cdots)  \Vert u_h\Vert_{L^\infty(I; H^s(\xR^d))} \\
&\Vert  \big( T_\gamma - T_{\gamma _\delta}\big) u_h \Vert_{L^p(I; H^{s-\frac{\delta}{2}}(\xR^d))} \leq \mathcal{F}(\cdots)  \Vert u_h\Vert_{L^\infty(I; H^s(\xR^d))}.  
\end{aligned}
\right.
\end{equation}
  These three estimates are proved along the same lines. Let us prove the first one.

We have 
$$
\big(T_V - T_{V_\delta}\big).\nabla u_h 
= \sum_{k=-1}^{+\infty} \big( (S_k(V) - S_{k \delta}(V))\cdot \nabla\Delta_k u_h \big).
$$
Since $u_h = \Delta_j u$ where $h=2^{-j}$ we see easily that the above sum is reduced 
to a finite number of terms where $\vert k-j \vert \leq 3.$ All the terms beeing estimated in the same way 
it is sufficient to consider for simplicity the term where $k=j.$ 
Let us set 
$$
A_j(t,x) = \big((S_j(V_i) - S_{j \delta}(V_i))\cdot \partial_i\Delta_ju_h\big)(t,x) \quad i=1,\ldots, d 
$$
where $V=(V_1,\ldots,  V_d)$.

Since the spectrum of $A_j $ is contained 
in a ball of radius $C\, 2^j$, we can write for fixed~$t$,
\begin{equation*}
\begin{aligned}
&\Vert A_j(t,\cdot) \Vert_{H^{s-\frac{\delta}{2}}(\xR^d)} \\
&\qquad\leq C\, 2^{j(s-\frac{\delta}{2})} 
\Vert  (S_j(V_i) - S_{j \delta}(V_i))\cdot \partial_i\Delta_ju_h)(t,\cdot) \Vert_{L^2(\xR^d)}\\
&\qquad \leq C\, 2^{j(s-\frac{\delta}{2})}\Vert  (S_j(V_i) - S_{j \delta}(V_i))(t,\cdot) \Vert_{L^\infty(\xR^d)} 
2^{j(1-s)}  \Vert u_h(t,\cdot) \Vert_{H^s(\xR^d)}. 
\end{aligned}
\end{equation*}
  Now we can write  
$$
(S_j(V_i) - S_{j \delta}(V_i))(t,x) 
= \int_{\xR^d} \hat{\psi}(z) \big(V_i(t,x-2^{-j }z) - V_i(t,x-2^{-j\delta} z) \big) \,dz.
$$
where $\psi \in C_0^\infty(\xR^d)$ has its support contained in a ball of radius $1$. 
It follows easily, since $0<\delta<1,$ that 
$$
\Vert (S_j(V_i) - S_{j \delta}(V_i))(t,\cdot) \Vert_{L^\infty(\xR^d)} 
\leq \mathcal{F}(\cdots)\, 2^{-j \delta} \Vert V_i(t,\cdot) \Vert_{W^{1,\infty}(\xR^d)}.
$$
Therefore we obtain with $p= 2,4$ and $I=[0,T],$
$$
\Vert A_j\Vert_{L^p(I; H^{s-\frac{\delta}{2}}(\xR^d))} \leq \mathcal{F}(\cdots)\,  2^{j(1-\frac{3\delta}{2})}  
\Vert V_i\Vert_{L^p(I; W^{1,\infty}(\xR^d))}  \Vert u_h \Vert_{L^\infty(I; H^s(\xR^d))},
$$
and the estimate \eqref{diffV} for the first term follows from the fact that $ \frac{3\delta}{2} = 1$.
\end{proof}
The theorem below has been stated in Theorem \ref{T4} but for the convenience of the reader we state it again.
\begin{theo}\label{T20}
Let $I= [0,T]$, $d\geq 1$. 
Let $\mu$ be such that 
$\mu<\frac{1}{24}$ if $d=1$ and $\mu<\frac{1}{12}$ if $d\ge 2$.

Let  $s\in\xR$ 
and $f \in L^\infty(I; H^s(\xR^d))$. 
Let $u\in C^0(I; H^s(\xR^d))$ be a solution of the problem 
$$
\Big(  \partial_t + \mez( T_{V } \cdot \nabla + \nabla \cdot T_{V }) +i T_{\gamma } \Big)u  =f.
$$
Then one can find $k = k(d)$ such that
\begin{equation*}
\begin{aligned}
\Vert  u  \Vert_{L^p(I;  \zygmund{s-\frac{d}{2}+\mu}(\xR^d))}
\leq    \mathcal{F}\big(\Vert V\Vert_{E_0} +  \mathcal{N}_k(\gamma)\big) 
\Big\{  \Vert  f\Vert_{L^p(I; H^{s }(\xR^d))} 
+ \Vert  u \Vert_{C^0(I; H^s(\xR^d))}\Big\}
\end{aligned}
\end{equation*}
where   $p=4$ if $ d=1 $ and $p=2$ if $d \geq 2.$  
\end{theo}
\begin{proof}
The function  $  \Delta_j u  $ is a solution of 
$$
\Big(  \partial_t  + \mez( T_{V } \cdot \nabla + \nabla \cdot T_{V })  
+i T_{\gamma }  \Big)(\Delta_ju)  =F_j ,  \quad (\Delta_ju)\arrowvert_{t=0} = \Delta_ju_0, 
$$
where
$$
F_j = \Delta_j f +\mez [ T_{V } \cdot \nabla + \nabla \cdot T_{V }, \Delta_j]u +i[T_\gamma, \Delta_j]u.
$$
Then $F_j$ is spectrally supported in a ring 
$\big\{\xi: c_1 2^{-j} \leq \vert \xi \vert \leq c_2 2^{-j}\big\}$ 
and since
$$
V_i \in L^p(I; W^{1 , \infty}(\xR^d)),\quad \gamma \in L^2\big(I; \Gamma_{1/2}^{1/2}(\xR^d)\big)\cap  L^\infty\big(I; \Gamma_{0}^{1/2}(\xR^d)\big)
$$ 
it follows that, with constants independent of $j$ we 
have 
\begin{equation*}
\begin{aligned}
&\Vert [ T_{V } \cdot \nabla + \nabla \cdot T_{V }, \Delta_j]u\Vert_{L^p(I; H^s(\xR^d))} 
\leq C\Vert V \Vert_{L^p(I; W^{1, \infty}(\xR^d))} \Vert \widetilde \Delta_j u \Vert_{L^\infty(I; H^s(\xR^d))}\\
& \Vert [T_\gamma, \Delta_j]u \Vert_{L^p(I; H^s(\xR^d))} \leq C\Vert \gamma \Vert_{E_1} 
\Vert \widetilde \Delta_j u  \Vert_{L^\infty(I; H^s(\xR^d))} 
\end{aligned}
\end{equation*}
with $\widetilde \Delta_j u = \widetilde{\varphi}(2^{-j}D)u$  where $\widetilde{\varphi}$ 
is supported in a ring $\{\xi: 0< c'_1  \leq \vert \xi \vert \leq c'_2 \}$ and $\widetilde{\varphi} = 1$ on the support of $\varphi$.

It follows then from Corollary \ref{strich3} that, with constants independent of $j$ we have 
\begin{equation}\label{finale}
\begin{aligned}
\Vert  \Delta_j u  \Vert_{L^p(I;  W^{s -\sigma_d, \infty}(\xR^d))}
&\leq    \mathcal{F}(\cdots) \big\{\Vert \Delta_j f\Vert_{L^p(I; H^{s }(\xR^d))} 
+   \Vert \Delta_j u \Vert_{C^0(I; H^s(\xR^d))} \big\}
\end{aligned}
 \end{equation}
 where $\sigma_1 = \frac{3}{8} + \frac{\delta}{8}= \mez - \frac{1}{24}, \,  \sigma_d = \frac{d}{2} - \frac{1}{4} +\frac{\delta}{4}= \frac{d}{2}- \frac{1}{12}$ if $ d\geq 2.$
 
The right hand side of \eqref{finale} is bounded by 
\begin{equation}\label{finale1}
A :=  \mathcal{F} (\cdots) \Big\{ 
 \Vert f\Vert_{L^p(I; H^{s }(\xR^d))} 
+   \Vert  u \Vert_{C^0(I; H^s(\xR^d))} 
\Big\}.
\end{equation}

On the other hand we have
$$
\Vert  \Delta_j u  \Vert_{L^p(I;  W^{s - \sigma_d - \eps, \infty}(\xR^d))} 
\leq 2^{-j\eps} \Vert  \Delta_j u  \Vert_{L^p(I;  W^{s - \sigma_d, \infty}(\xR^d))}\leq 2^{-j\eps}A.
$$
Since $\sum_{j=-1}^\infty 2^{-j\eps}<+\infty$, summing on $j$, this completes the proof.
\end{proof}

\chapter{Cauchy problem}\label{C:5}

In this chapter we complete the proof of our main result, which is Theorem~\ref{main} stated in the introduction. 
We begin in Section~\ref{S:51} by combining the Sobolev and Strichartz estimates proved previously to obtain 
{\em a priori} estimates. Then, in Section \ref{S:52}, we obtain an estimate for the difference of two solutions. This estimate will be 
used to prove the uniqueness of solutions as well as to prove that a family of approximate solutions is a Cauchy sequence, in some larger space, and then converges strongly. In Section~\ref{S:53} 
we prove that one can pass to the limit in the equations under weak assumptions. 
In Section~\ref{S:54} we briefly recall how to complete the proof from these three technical ingredients.

\section{A priori estimates}\label{S:51} 

\subsection{Notations}
For the sake of clarity we recall here our assumptions and notations. 
We work with the Craig-Sulem-Zakharov formulation of the water-waves equations:
\begin{equation}\label{n400}
\left\{
\begin{aligned}
&\partial_{t}\eta-G(\eta)\psi=0,\\[1ex]
&\partial_{t}\psi+g \eta
+ \smash{\frac{1}{2}\la\partialx \psi\ra^2  -\frac{1}{2}
\frac{\bigl(\partialx  \eta\cdot\partialx \psi +G(\eta) \psi \bigr)^2}{1+|\partialx  \eta|^2}}
= 0.
\end{aligned}
\right.
\end{equation}

\begin{assu}\label{T:22}
We consider smooth solutions of \eqref{n400} such that
\begin{enumerate}[i)]
\item $(\eta,\psi)$ belongs to $C^1([0,T_0]; H^{s_0}(\xR^d)\times H^{s_0}(\xR^d))$ for some 
$T_0$ in $(0,1]$ and some 
$s_0$ large enough;
\item there exists $h>0$ such that \eqref{n1} holds for any $t$ in $[0,T_0]$ (this is the assumption 
that there exists a curved strip of width $h$ separating the free surface from the bottom);
\item there exists $c>0$ such that the Taylor coefficient $a(t,x)=-\partial_y P\arrowvert_{y=\eta(t,x)}$ is bounded from below by $c$ 
for any $(t,x)$ in $[0,T_0]\times \xR^d$.
\end{enumerate}
\end{assu}
We work with the horizontal and vertical traces of the velocity 
on the free boundary, namely 
$B= (\partial_y \phi)\arrowvert_{y=\eta}$ and 
$V = (\nabla_x \phi)\arrowvert_{y=\eta}$, 
which can be defined in terms of $\eta$ and $\psi$ by means of
\begin{equation}\label{n405}
B\defn \frac{\partialx \eta \cdot\partialx \psi+ G(\eta)\psi}{1+|\partialx  \eta|^2},
\qquad
V\defn \partialx \psi -B \partialx\eta.
\end{equation}
Let $s$ and $r$ be two positive real numbers such that
\begin{equation}\label{n406}
s>\frac{3}{4}+\frac{d}{2}, \quad s+\frac{1}{4}-\frac{d}{2} > r>1,\quad 
r\not\in\mez \xN. 
\end{equation}
Define, for $T$ in $(0,T_0]$, the norms
\begin{equation}\label{n410}
\begin{aligned}
M_s(T)&\defn 
\lA (\psi,\eta,B,V)\rA_{C^0([0,T];H^{s+\mez}\times H^{s+\mez}\times H^s\times H^s)},\\
Z_r(T)&\defn \lA \eta\rA_{L^p([0,T];W^{r+\mez,\infty})}
+\lA (B,V)\rA_{L^p([0,T];W^{r,\infty}\times W^{r,\infty})},
\end{aligned}
\end{equation}
where $p=4$ if $d=1$ and $p=2$ for $d\ge 2$. 

Our goal is to estimate $M_s(T)+Z_r(T)$ in terms of
\begin{equation}\label{n411}
M_{s,0}\defn \lA (\psi(0),\eta(0),B(0),V(0))\rA_{H^{s+\mez}\times H^{s+\mez}\times H^s\times H^s}.
\end{equation}

In Chapter~\ref{C:3} 
we proved that, for any $s$ and $r$ satisfying \eqref{n406}, there exists 
a continuous non-decreasing 
function~$\mathcal{F}\colon \xR^+\rightarrow\xR^+$ such that, for 
all smooth solution 
$(\eta,\psi)$ of \eqref{n400} defined on the time interval~$[0,T_0]$ and satisfying 
Assumption~\ref{T:22} on this time interval, for any~$T\in (0,T_0]$,
\begin{equation}\label{n415}
M_s(T)\le \mathcal{F}\bigl(\mathcal{F}(M_{s,0})+T\mathcal{F}\bigl(M_s(T)+Z_r(T)\bigr)\bigr).
\end{equation}
If $s>1+d/2$, then one can apply the previous inequality with $r=s-d/2$. Then $Z_r(T)\les M_s(T)$ by Sobolev embedding and one deduces from \e{n415} an estimate which involves only $M_s(T)$. Thus we recover the {\em a priori} estimate in Sobolev spaces proved in \cite{ABZ3} under the assumption that $s>1+d/2$. 
Using 
classical inequalities, this implies that 
for any $A>0$ there exist $B>0$ and $T_1>0$  such that
$$
M_{s,0}\le A \Rightarrow M_s(T_1)\le B.
$$
We shall prove that a stronger {\em a priori} estimate holds. We extend the previous estimate in 
two directions. Firstly, we prove that one can control Sobolev norms 
for some $s<1+d/2$. Secondly, we prove that one can control Strichartz norms even for rough solutions.

\begin{prop}\label{T2}
Let  $\gain$ be such that $\gain<\frac{1}{24}$ if $d=1$ and 
$\gain<\frac{1}{12}$ for $d\ge 2$. Consider two real numbers $s$ and $r$ satisfying
\begin{equation}\label{n416}
s>1+\frac{d}{2}-\gain, \quad 1<r<s+\gain-\frac{d}{2},\quad 
r\not\in\mez \xN. 
\end{equation}
For any $A>0$ there exist $B>0$ and $T_1>0$  such that, 
for all $0<T_0\leq T_1$ and all smooth solution 
$(\eta,\psi)$ of \eqref{n10} defined on the time interval~$[0,T_0]$  satisfying 
Assumption~\ref{A:2} on this time interval, then the solution satisfies the {\em a priori} bound 
$$M_{s,0}\le A \Rightarrow M_s(T_0)+Z_r(T_0)\le B.$$
\end{prop}

\subsection{Reduction}

In this section we show that one can reduce the proof of Proposition~\ref{T2} to 
the proof of an {\em a priori}Ê
estimate for $Z_r(T)$.

To prove Proposition~\ref{T2}, the key point is 
to prove that there exists a continuous non-decreasing 
function~$\mathcal{F}\colon \xR^+\rightarrow\xR^+$ such that
\begin{equation}\label{n417}
M_s(T)+Z_r(T)\le \mathcal{F}\bigl(\mathcal{F}(M_{s,0})+T\mathcal{F}\bigl(M_s(T)+Z_r(T)\bigr)\bigr).
\end{equation}
Since $\mu<1/4$ and since 
the estimate \eqref{n415} is proved under the general assumption~\eqref{n406}, 
it remains only to prove that $Z_r(T)$ is bounded by the right-hand side of \eqref{n417}. 
\begin{prop}\label{T25}
Let~$d\ge 1$ and consider~$s,r,\mu$ satisfying \eqref{n416}. 
There exists 
a continuous non-decreasing 
function~$\mathcal{F}\colon \xR^+\rightarrow\xR^+$ such that, for all~$T_0\in (0,1]$ 
and all smooth solution 
$(\eta,\psi)$ of \eqref{n400} defined on the time interval~$[0,T_0]$ and satisfying 
Assumption~\ref{T:22} on this time interval, there holds
\begin{equation}\label{n425}
Z_r(T)\le \mathcal{F}\bigl(T\mathcal{F}\bigl(M_s(T)+Z_r(T)\bigr)\bigr),
\end{equation}
for any $T$ in $[0,T_0]$.
\end{prop}

Let us admit this result and deduce Proposition~\ref{T2}. 

\begin{proof}[Proof of Proposition~\ref{T2} given Proposition~\ref{T25}]
Introduce for $T$ in $[0,T_0]$, 
$f(T)=M_s(T)+Z_r(T)$. It follows from \eqref{n415}Ê
and \eqref{n425} that \eqref{n417}Ê
holds. This means that there exists a continuous non-decreasing 
function~$\mathcal{F}\colon \xR^+\rightarrow\xR^+$ such that, for all~$T\in (0,T_0]$, 
\begin{equation}\label{n430}
f(T)\le \mathcal{F}\bigl(\mathcal{F}(A)+T\mathcal{F}\bigl(f(T)\bigr)\bigr).
\end{equation}
Now fix $B$ such that $B>\max\big\{ A,\mathcal{F}\bigl(\mathcal{F}(A)\bigr)\big\}$
and then chose $T_1\in (0,T_0]$ such that
$$
\mathcal{F}\bigl(\mathcal{F}(A)+T_1\mathcal{F}(B)\bigr)<B.
$$
We claim that $f(T)<B$ for any $T$ in $[0,T_1]$. Indeed, since $f(0)=M_{s,0}\le A <B$, 
assume that there exists $T'\in (0,T_1]$ such that $f(T')>B$. Since $f$ is continuous, this implies that 
there is $T''\in (0,T_1)$ such that $f(T'')=B$. Now it follows from~\eqref{n430}, the fact that 
$\mathcal{F}$ is increasing, and the definition of $T_1$ that
$$
B=f(T'')\le  \mathcal{F}\bigl(\mathcal{F}(A)+T''
\mathcal{F}\bigl(f(T'')\bigr)\bigr)
\le \mathcal{F}\bigl(\mathcal{F}(A)+T_1
\mathcal{F}\bigl(B\bigr)\bigr)<B.
$$
Hence the contradiction which proves that $f(T)\le B$ 
for any $T$ in $[0,T_1]$. \end{proof}

It remains to prove Proposition~\ref{T25}. This will be the purpose of the end of this chapter.

We begin by using an interpolation 
inequality to reduce the proof of Proposition~\ref{T25} to the proof of the following proposition.

\begin{prop}\label{T26}
Let~$d\ge 1$ and consider~$\mu,s,r$ as in \eqref{n416}. 
Consider in addition $r'$ such that
$$
r<r'<s+\mu-\frac{d}{2},\quad r'\not\in\mez \xN
$$
and set
$$
Z_{r'}(T)\defn \lA \eta\rA_{L^p([0,T];W^{r'+\mez,\infty})}
+\lA (B,V)\rA_{L^p([0,T];W^{r',\infty}\times W^{r',\infty})}
$$
where $p=4$ if $d=1$ and $p=2$ for $d\ge 2$. 
There exists 
a continuous non-decreasing 
function~$\mathcal{F}\colon \xR^+\rightarrow\xR^+$ such that, for all~$T_0\in (0,1]$ 
and all smooth solution 
$(\eta,\psi)$ of \eqref{n400} defined on the time interval~$[0,T_0]$ and satisfying 
Assumption~\ref{T:22} on this time interval, there holds
\begin{equation}\label{n435}
Z_{r'}(T)\le \mathcal{F}\bigl(M_s(T)+Z_r(T)\bigr),
\end{equation}
for any $T$ in $[0,T_0]$.
\end{prop}
We prove in this paragraph that Proposition~\ref{T26} implies 
Proposition~\ref{T25}. Proposition~\ref{T26} will be proved in the next 
paragraph.

\begin{proof}[Proof of Proposition~\ref{T25} given Proposition~\ref{T26}] 
Consider a function $v=v(t,x)$. By interpolation, since $1-\mu<1<r<r'$, 
there exists a real number $\theta\in (0,1)$ such that
$$
\lA v(t,\cdot)\rA_{W^{r,\infty}}\les \lA v(t,\cdot)\rA_{W^{1-\mu,\infty}}^\theta
\lA v(t,\cdot)\rA_{W^{r',\infty}}^{1-\theta}.
$$
This implies that
$$
\int_0^T \lA v(t,\cdot)\rA_{W^{r,\infty}}^p\, dt
\les \lA v\rA_{C^0([0,T];W^{1-\mu,\infty})}^{p\theta}
\int_0^T \lA v(t,\cdot)\rA_{W^{r',\infty}}^{p(1-\theta)}\, dt.
$$
The H\"older inequality then implies that
$$
\lA v\rA_{L^p([0,T];W^{r,\infty})}
\les T^{\frac{\theta}{p}}
\lA v\rA_{C^0([0,T];W^{1-\mu,\infty})}^{\theta}
\lA v\rA_{L^p([0,T];W^{r',\infty})}^{1-\theta}.
$$
By the same way, there holds
$$
\lA v\rA_{L^p([0,T];W^{r+\mez,\infty})}
\les T^{\frac{\theta'}{p}}
\lA v\rA_{C^0([0,T];W^{1-\mu+\mez,\infty})}^{\theta'}
\lA v\rA_{L^p([0,T];W^{r'+\mez,\infty})}^{1-\theta'}.
$$
Since $s>(1-\mu)+d/2$, the Sobolev embedding implies that
$$
\lA v\rA_{C^0([0,T];W^{1-\mu,\infty})}\les \lA v\rA_{C^0([0,T];H^{s})},
\quad \lA v\rA_{C^0([0,T];W^{1-\mu+\mez,\infty})}\les \lA v\rA_{C^0([0,T];H^{s+\mez})}.
$$
This proves that
$$
Z_{r}(T)\le T^{\frac{\theta}{p}} M_s(T)^\theta (Z_{r'}(T))^{1-\theta}
$$
for some $\theta>0$. This in turn proves that \eqref{n435} implies \eqref{n425}.
\end{proof} 

\subsection{Proof of Proposition~\ref{T26}}

Recall that the positive real number 
$\mu$ has been chosen such that $\mu<1/24$ if $d=1$ and $\mu<1/12$ 
for $d\ge 2$, and $s,r,r'$ are such that
\begin{equation*}
s>1+\frac{d}{2}-\mu, \quad 1<r<r'<s+\mu-\frac{d}{2},\quad 
r\not\in\mez \xN. 
\end{equation*}
Let $T>0$ and set $I=[0,T]$.

The proof of Proposition~\ref{T26} is based on Corollary~\ref{T10} and Theorem~\ref{T20}. 
By combining these two results we shall deduce in the first step of the proof that 
\begin{equation}\label{n440}
\lA u\rA_{L^p(I;W^{r',\infty})}\le \mathcal{F}\bigl(M_s(T)+Z_r(T)\bigr)
\end{equation}
where $u$ is defined in terms of $(\eta,V,B)$ by (see \eqref{eq:r0})
\begin{equation}\label{n440.5}
\begin{aligned}
&u=\lDx{-s}(U_s-i \theta_s),\\[0.5ex]
&U_s \defn \lDx{s} V+T_\zeta \lDx{s}\B \qquad (\zeta=\partialx\eta),\\[0.5ex]
&\theta_s\defn T_{\sqrt{\ma/\lambda}}\lDx{s}\nabla \eta.
\end{aligned}
\end{equation}

In the next steps of the proof we show 
how to recover estimates for the original unknowns 
$(\eta,V,\B)$ in $L^p([0,T];W^{r'+\mez}\times W^{r',\infty}\times W^{r',\infty})$. 
 
\paragraph{Step 1: proof of \eqref{n440}.}

It follows from Theorem~\ref{T20} that 
\begin{equation*}
\Vert  u  \Vert_{L^p(I;W^{r', \infty}(\xR^d))}
\leq    \mathcal{F}\big(\Vert V\Vert_{E_0} +  \mathcal{N}_k(\gamma)\big) 
\Big\{\Vert  f\Vert_{L^p(I; H^{s }(\xR^d))} 
+ \Vert  u \Vert_{C^0(I; H^s(\xR^d))}\Big\}.
\end{equation*}
Clearly we have
$$
\Vert V\Vert_{E_0}\le Z_1(T)\le Z_r(T),\quad \mathcal{N}_k(\gamma)
\les M_s(T).
$$
Moreover, \eqref{eq:r2} and \eqref{n201} imply that 
$$
\Vert  u \Vert_{C^0(I; H^s(\xR^d))}\le \mathcal{F}(M_s(T)),\quad 
\Vert  f\Vert_{L^p(I; H^{s }(\xR^d))} \le \mathcal{F}\bigl(M_s(T)+Z_r(T)\bigr).
$$
By combining the previous estimates we deduce the desired estimate \eqref{n440}.

\paragraph{Step 2: estimate for $\eta$.} 

Separating real and imaginary parts, directly from the definition \eqref{n440.5} 
of $u$, we get
$$
\big\lVert \lDx{-s}T_{\sqrt{a/\lambda}}\lDx{s}\partialx\eta \big\rVert_{W^{r',\infty}}
\le \lA u\rA_{W^{r',\infty}}.
$$

We shall make repeated uses of the following elementary result.
\begin{lemm}\label{theo:scZ}
Consider $m\in \xR$ and $\rho$ in $[0,1]$. 
Let $(r,r_1,r_2)\in [0,+\infty)^3$ be such that $r\le \min (r_1+\rho,r_2+m)$, $r\not\in\xN$. 
Consider the equation $T_\tau v=f$ where $\tau=\tau(x,\xi)$ is a symbol such that 
$\tau$ (resp.\ $1/\tau$) belongs to $\Gamma^m_\rho$ (resp.\ $\Gamma^{-m}_\rho$). 
Then
$$
\lA v\rA_{W^{r,\infty}}\le K \lA v\rA_{W^{r_1,\infty}}+K \lA f\rA_{W^{r_2,\infty}}
$$
for some constant $K$ depending only on 
$\mathcal{M}^m_\rho(\tau)+\mathcal{M}^{-m}_\rho(1/\tau)$.
\end{lemm}
\begin{proof}
Write
$$
v=\big(I -T_{1/\tau}T_\tau\big)v+T_{1/\tau}f
$$
and use \eqref{esti:quant1} (resp.\ \e{esti:quant2}) 
to estimate the first (resp.\ second) term.
\end{proof}
Now write
$$
\lDx{-s}T_{\sqrt{a/\lambda}}\lDx{s}\partialx=
T_{\sqrt{a/\lambda}}\partialx+R
$$
where $R=\big[ \lDx{-s},T_{\sqrt{a/\lambda}}\lDx{s}\partialx\bigr]$. 
Since $\sqrt{a/\lambda}$ is a symbol of order $-1/2$ in $\xi$, it follows from \e{esti:quant2} 
that, for any $\rho\in (0,1)$, 
$$
\lA R\eta \rA_{W^{r',\infty}}\le K
\mathcal{M}^{-\mez}_\rho\Big(\sqrt{\frac{a}{\lambda}}\Big)\lA \eta\rA_{W^{r'+\mez-\rho,\infty}}
$$
and hence
$$
\lA R\eta \rA_{W^{r',\infty}}\le \mathcal{F}\bigl(\lA\partialx\eta\rA_{W^{\rho,\infty}},
\lA a\rA_{W^{\rho,\infty}}\big)
\lA \eta\rA_{W^{r'+\mez-\rho,\infty}}.
$$
Now by assumption on $s$ and $r'$ we can chose $\rho$ (say $\rho=1/4$) so that
$$
\rho<s-\mez-\frac{d}{2},\quad r'+\mez-\rho<s+\mez-\frac{d}{2}
$$
and hence 
$$
\lA R\eta \rA_{W^{r',\infty}}\le\mathcal{F}\bigl(\lA\eta\rA_{H^{s+\mez}},
\lA a-g\rA_{H^{s-\mez}}\big).
$$
Recalling (see \e{esti:a1}) that 
$\lA a-g\rA_{H^{s-\mez}}\le \mathcal{F}(M_s)$ for any $s>3/4+d/2$, we obtain 
$\lA R\eta \rA_{W^{r',\infty}}\le \mathcal{F}(M_s)$. 

We thus deduce that
$$
\lA T_{\sqrt{a/\lambda}}\partialx\eta\rA_{W^{r',\infty}}\le 
\lA u\rA_{W^{r',\infty}}+\mathcal{F}(M_s).
$$
Now, Lemma~\ref{theo:scZ}, applied with 
$m=-1/2$ and $\rho=1/4$, yields an estimate for the $W^{r'-\mez,\infty}$-norm of $\nabla\eta$ 
which implies that
$$
\lA \eta\rA_{W^{r'+\mez,\infty}}\le 
K\lA u\rA_{W^{r',\infty}}+K\lA \eta\rA_{W^{r'+\mez-\rho,\infty}}+K \mathcal{F}(M_s)
$$
for some constant $K$ depending only on $\mathcal{M}^{-\mez}_\rho\Big(\sqrt{\frac{a}{\lambda}}\Big)$. As already seen, $K\le \mathcal{F}(M_s)$ and 
$\lA \eta\rA_{W^{r'+\mez-\rho,\infty}}\le M_s$ (using the Sobolev embedding). 
We conclude that
$$
\lA \eta\rA_{W^{r'+\mez,\infty}}\le 
\mathcal{F}(M_s)\lA u\rA_{W^{r',\infty}}+\mathcal{F}(M_s).
$$
Therefore \eqref{n440}Ê
implies that
\begin{equation}\label{n445}
\lA \eta\rA_{L^p(I;W^{r'+\mez,\infty})}\le \mathcal{F}\bigl(M_s(T)+Z_r(T)\bigr).
\end{equation}

\paragraph{Step 3: estimate for $V+T_{\zeta}B$.} 

We proceed as above: starting from \eqref{n440} one deduces an estimate for 
the $W^{r',\infty}$-norm of $\lDx{-s}\bigl(\lDx{s}V+T_{\zeta}\lDx{s}B\bigr)$. One 
rewrite this term as $V+T_\zeta B$ plus a commutator which is estimated by 
means of \eqref{esti:quant2} 
and the Sobolev embedding. 
It is find that
$$
\lA V+T_\zeta B\rA_{W^{r',\infty}}\le \mathcal{F}(M_s)\lA u\rA_{W^{r',\infty}}+\mathcal{F}(M_s)
$$
so that \eqref{n440}Ê
implies that
\begin{equation}\label{n447}
\lA V+T_\zeta B\rA_{L^p(I;W^{r',\infty})}\le \mathcal{F}\bigl(M_s(T)+Z_r(T)\bigr).
\end{equation}

\paragraph{Step 4: estimate for $V$ and $B$.} We shall estimate the 
$L^p(I;W^{r',\infty})$-norm of $B$. The estimate for 
the $L^p(I;W^{r',\infty})$-norm of $V$ will follow from 
$V=(V+T_\zeta B)-T_\zeta B$ since the first term 
$V+T_\zeta B$ is already estimated (see \eqref{n447}) and since, 
for the second term, 
one can use the rule \eqref{esti:quant1} to obtain 
$\lA T_\zeta B\rA_{W^{r',\infty}}\les 
\lA \zeta\rA_{L^\infty}\lA B\rA_{W^{r',\infty}}\les \lA \eta\rA_{H^{s+\mez}}\lA B\rA_{W^{r',\infty}}$.

To estimate the $W^{r',\infty}$-norm of $B$, as above, we use the identity 
$G(\eta)B=-\cnx V+\widetilde{\gamma}$ where (see \e{n340a}) 
$$
\lA \widetilde{\gamma}\rA_{W^{r'-1,\infty}}
\le \lA \widetilde{\gamma}\rA_{H^{s-\mez}}\le \mathcal{F}
\big(\lA (\eta,V,B)\rA_{H^{s+\mez}\times H^\mez\times H^\mez}\big).
$$
In order to use this identity, write
\begin{equation}\label{n450}
\begin{aligned}
\cnx \big( V+T_\zeta B\big)&=\cnx V+T_{\cnx \zeta}B+T_{\zeta}\cdot\partialx B\\
&=-G(\eta)B+T_{\cnx \zeta}B+T_{\zeta}\cdot\partialx B
+\widetilde{\gamma}\\
&=T_{-\lambda+i\zeta\cdot\xi}B+r
\end{aligned}
\end{equation}
where
$$
r=T_{\cnx \zeta} B+\widetilde{\gamma}+(T_\lambda-G(\eta))B.
$$
The first term in the right-hand side is estimated by means of 
\begin{alignat*}{2}
\lA T_{\cnx \zeta} B\rA_{W^{r'-1,\infty}}
&\les \lA T_{\cnx \zeta} B\rA_{H^{s-1+\mu}} \qquad &&\text{since }
r'<s+\mu-\frac{d}{2}\\
&\les \lA \cnx \zeta\rA_{C^{\mu-1}_*}\lA B\rA_{H^s} && \text{(see \e{niS})}\\
&\les \lA \eta\rA_{W^{1+\mu,\infty}}\lA B\rA_{H^s}&&\text{since } \cnx \zeta=\Delta\eta \\
&\les \lA \eta\rA_{H^{s+\mez}}\lA B\rA_{H^s}&&\text{since }1+\mu<1+\frac{1}{4}
<s+\mez-\frac{d}{2}\cdot
\end{alignat*}
The key point is to estimate $(T_\lambda-G(\eta))B$. 
To do so we use Proposition~\ref{coro:paraDN1} 
with $(\mu,\sigma,\eps)$ replaced by $(s+\mez,s,\uq)$ and 
the Sobolev embedding $H^{s-\tq}(\xR^d)\subset W^{r'-1,\infty}(\xR^d)$. This implies that
$$
\lA (T_\lambda-G(\eta))B\rA_{W^{r'-1,\infty}}
\les \lA (T_\lambda-G(\eta))B\rA_{H^{s-\tq}}
\le \mathcal{F}
\big(\lA \eta\rA_{H^{s+\mez}}\big)\lA B\rA_{H^s}.
$$
We end up with
$$
\lA r\rA_{W^{r'-1,\infty}}\le 
\mathcal{F}
\big(\lA (\eta,V,B)\rA_{H^{s+\mez}\times H^s\times H^s}\big).
$$
Writing (see \eqref{n450}) 
$$
T_{-\lambda+i\zeta\cdot\xi}B=\cnx \big( V+T_\zeta B\big)-r,
$$ 
it follows from \eqref{n447} and Lemma~\ref{theo:scZ} 
that 
$$
\lA B\rA_{L^p(I;W^{r',\infty})}\le \mathcal{F}\bigl(M_s(T)+Z_r(T)\bigr).
$$
This completes the proof of Proposition~\ref{T26} and hence 
the proof of Proposition~\ref{T2}.

\section{Contraction estimates}\label{S:52}

Our goal in this section is to prove the following estimate for the difference of two solutions. 

\begin{theo}\label{th.lipschitz}
Let  $\gain$ be such that $\gain<\frac{1}{24}$ if $d=1$ and 
$\gain<\frac{1}{12}$ for $d\ge 2$. Consider two real numbers $s$ and $r$ satisfying
\begin{equation*}
s>1+\frac{d}{2}-\gain, \quad 1<r<s+\gain-\frac{d}{2},\quad 
r\not\in\mez \xN. 
\end{equation*}
Let~$(\eta_j,\psi_j)$,~$j=1,2$, be two solutions such that 
\begin{align*}
&(\eta_j,\psi_j,V_j,B_j)\in 
C^0([0,T];H^{s+\mez}\times H^{s+\mez}\times H^s\times H^s),\\
&(\eta_j,V_j,B_j)\in L^p([0,T];\holdertdm\times \holder{r}\times \holder{r}),
\end{align*}
for some fixed~$T>0$,~$d\ge 1$ with $p=4$ if $d=1$ and $p=2$ otherwise. We also assume 
that the condition \eqref{hypt} holds for  
$0\leq t\le T$ and that there exists $c>0$ 
such that for all~$0\le t\le T$ and  $x\in {\mathbf{R}}^d$, we have 
$\ma_j(t,x)\ge c$ for~$j=1,2$.
Set
\begin{align*}
M_j&\defn \sup_{t\in [0,T]} \lA (\eta_j,\psi_j,V_j,B_j)(t)\rA_{H^{s+\mez}\times H^{s+\mez} \times H^s\times H^s},\\
&\quad + \lA (\eta_j,V_j,B_j)\rA_{L^p([0,T];\holdertdm\times \holder{r}\times \holder{r})}.
\end{align*}
Set
$$
\deta\defn \eta_1-\eta_2, \quad \psi \defn \psi_1-\psi_2,\quad \dV\defn V_1-V_2, \quad \dB \defn\B_1-\B_2,
$$
and
\begin{align*}
N(T) &\defn  \sup_{t\in [0,T]} \lA (\eta,\psi,V,B)(t)\rA_{H^{s-\mez}\times H^{s-\mez} \times H^{s-1}\times H^{s-1}}\\ 
&\quad+ \lA (\eta,V,B)\rA_{L^p([0,T];\holder{r-\mez}\times \holder{r-1}\times \holder{r-1})}.
\end{align*}
Then we have 
\be\label{eq.lipschitz}
N(T)\leq \mathcal{K}   \|(\eta, \psi, V, B) 
\mid_{t=0}\| _{H^{s-\mez}\times H^{s-\mez}\times H^{s-1}\times H^{s-1}},
\ee
where $\mathcal{K}  $ is a positive constant depending only on $T, M_1,M_2,r,  s, d, c.$  
\end{theo} 
\begin{rema}\label{T522}
To prove this theorem, we shall follow closely the analysis in \cite{ABZ3}.
 However, there are quite difficulties which appear for $s<1+d/2$, in particular for $d=1$ and with a general bottom. For instance, one has to estimate the $H^{s-\tdm}$-norm of various products of the form $uv$ with $u\in H^{s-1}$ and $v\in H^{s-\tdm}$. For $s<1+d/2$, the product is no longer 
bounded from $H^{s-1}\times H^{s-\tdm}$ to $H^{s-\tdm}$ and, clearly, one has to further assume some control in H\"older or Zygmund norms. Namely, we assume that $u\in H^{s-1}\cap L^\infty$ 
and $v\in H^{s-\tdm}\cap \zygmund{-\mez}$. 
Then, paralinearizing the product
$uv=T_u v+T_vu+R(u,v)$
and using the usual estimate for paraproducts (see \e{esti:quant1} and \e{niS}), one obtains
$$
\lA T_u v\rA_{H^{s-\tdm}}\les \lA u\rA_{L^\infty}\lA v\rA_{H^{s-\tdm}},\quad 
\lA T_v u\rA_{H^{s-\tdm}}\les \lA v\rA_{\zygmund{-\mez}}\lA u\rA_{H^{s-1}}
$$
so the only difficulty is to estimate the $H^{s-\tdm}$-norm of the remainder $R(u,v)$. However, the 
estimate \e{Bony3} requires that $s-\tdm>0$, which does not hold in general under the assumption 
\e{n416}. To circumvent this problem, each times that we shall need to estimate 
the $H^{s-\tdm}$-norm of such remainders, we shall prove that one can factor out some derivative exploiting the structure of the water waves equations. This means that one can replace $R(u,v)$ by $\partial_x R(\tilde{u},\tilde{v})$ for some functions, say, 
$\tilde{u}\in L^\infty$ 
and $\tilde{v}\in H^{s-\mez}$. Now one can estimate 
the $H^{s-\mez}$-norm of $R(\tilde{u},\tilde{v})$ by means of \e{Bony3} which immediately implies the desired estimate for the $H^{s-\tdm}$-norm of 
$\partial_x R(\tilde{u},\tilde{v})$.
\end{rema}

\subsection{Contraction for the Dirichlet-Neumann}
A key step in the proof of Theorem~\ref{th.lipschitz} is to prove a Lipschitz 
property for the Dirichlet-Neumann operator. 
 
\begin{prop}\label{L:p60}Assume $d\ge 1, \quad s>\tq+\frac{d}{2}, \quad s+\frac{1}{4}-\frac{d}{2}>r>1.$ Then there exists a non decreasing function $\mathcal{F}: \xR^+ \to \xR^+$ such that
\begin{equation}\label{contr:DN}
\begin{aligned}
&\lA G(\eta_1)f-G(\eta_2)f\rA_{H^{s-\tdm}} \\
&\quad\leq \mathcal{F} ( \Vert  (\eta_1,  \eta_2 ) 
\Vert_{H^{s+\mez} }\big)\Big\{
\lA \eta_1-\eta_2\rA_{\holder{r-\mez}}   \lA f\rA_{H^s}\\
&\quad \phantom{\leq \mathcal{F} ( \Vert  (\eta_1,  \eta_2 ) 
\Vert_{H^{s+\mez} }\big)\Big\{ } 
+\lA \eta_1-\eta_2\rA_{H^{s-\mez}}\big( \lA f\rA_{H^s}+\lA f\rA_{\holder{r}}\big)\Big\}.
\end{aligned} 
\end{equation}
\end{prop}
 In the proof of Proposition \ref{L:p60} we shall use the following classical lemma.
 \begin{lemm}\label{lions}
Let $I = (-1,0)$ and $\sigma \in \xR$. Let   $u \in L_z^2(I, H^{\sigma+ \mez}(\xR^d))$ such that 
$\partial_z u \in L_z^2(I, H^{\sigma-\mez}(\xR^d)).$ Then $u \in C^0([-1,0], H^{\sigma}(\xR^d))$ 
and there exists an absolute constant $C>0$ such that
$$
\sup_{z \in [-1,0]} \| u(z, \cdot) \|_{ H^{\sigma}({\mathbf{R}}^d)} 
\leq C \bigl(\|u\|_{L^2(I,H^{\sigma+ \mez}(\xR^d))}+  \|\partial_z u\|_{L^2(I,H^{\sigma- \mez}(\xR^d))} \bigr).
$$
\end{lemm}

\begin{proof}[Proof of Proposition \ref{L:p60}]
In the sequel we shall denote by $\text{RHS}$    the right hand side of \eqref{contr:DN} where $\mathcal{F}$ may vary from line to line. 

We want to apply changes of variable as in ~\eqref{diffeo}. However, here, we have an additional constraint. Indeed, after this change of variables, we want to get the same domain for $\eta_1$ and $\eta_2$ to be able to compare the variational solutions. For this purpose, we need to modify slightly the change of variables in \eqref{diffeo}. 

To prove the theorem, we can assume without loss of generality that 
$\|\eta_1 - \eta_2\|_{L^\infty}$ is small enough. 
Then there exists  $\widetilde {\eta} \in C^\infty_b$ such that for $j=1;2$,
$$ \eta_j - \frac {3h} 4 \leq \widetilde{\eta} \leq \eta_j - \frac{2h} 3.$$
Let
\begin{equation}\label{lesomegabis}
\left\{
\begin{aligned}
\Omega_{1,j} &= \{(x,y): x \in {\mathbf{R}}^d, \eta_j(x)- \frac{h}{2}<y<\eta_j(x)\},\\
\Omega_{2,j} &= \{(x,y)\in \mathcal{O}:  \widetilde{\eta}(x) < y\leq\eta_j(x)-\frac{h}{2}\},\\
\Omega_3 &= \{(x,y)\in \mathcal{O}: y< \widetilde{\eta}(x)\}\\
\Omega_j &= \Omega_{1,j} \cup \Omega_{2,j} \cup \Omega_3
\end{aligned}
\right.
 \end{equation}
and 
\begin{equation}\label{omega1bis}
\left\{
\begin{aligned}
&\widetilde{\Omega}_1= \mathbf{R}^d_x\times [-1,0)_z\\
&\widetilde{\Omega}_2= \mathbf{R}^d_x\times [-2,-1)_z\\
&\widetilde{\Omega}_3 = \{(x,z)\in \xR^d \times (-\infty, -2): (x,z+2+\widetilde \eta(x))\in \Omega_3\},\\
&\widetilde{\Omega}= \widetilde{\Omega}_1  \cup \widetilde{\Omega}_2\cup \widetilde{\Omega}_3.  
\end{aligned}
\right.
\end{equation}
We define  Lipschitz diffeomorphisms from $\widetilde{\Omega}$ to $\Omega_j$ 
of the form 
$(x,z) \mapsto (x, \rho_j(x,z))$ where 
the map $(x,z) \mapsto  \rho_j(x,z)$ from $\widetilde{\Omega}$ to $\xR$ is defined as follows
\begin{equation}\label{diffeobis}
\left\{
\begin{aligned}
&\rho_j(x,z)=  (1+z)e^{\delta z\langle D_x \rangle }\eta_j(x) -z e^{-(1+ z)\delta\langle D_x \rangle }\big(\eta_j(x) -\frac{h}{2}\big)\quad \text{if } (x,z)\in \widetilde{\Omega}_1,\\
&\rho_j(x,z)=  (2+z)e^{\delta z\langle D_x \rangle }\big(\eta_j(x)-\frac{h}{2} \big)  -(1+z)\tilde{\eta} \quad \text{if } (x,z)\in \widetilde{\Omega}_2,\\
&\rho_j(x,z) = z+2+\widetilde{\eta}(x)\quad \text{if } (x,z)\in \widetilde{\Omega}_3 
\end{aligned}
\right.
\end{equation}
for some small enough positive constant $\delta$. Notice that, since for $z\in I=(-2,0)$ we kept essentially for $\rho_j$ the same expression as in~\eqref{diffeo}, we get the same estimates as in Lemma~\ref{rho:diffeo}.

 Recall that according to \eqref{DN:forme} we have for $j=1,2$
$$
G(\eta_j)f = U_j\arrowvert_{z=0}, \quad  U_j =  \frac{1 +|\nabla_x \rho_j |^2}{\partial_z \rho_j} \partial_z \widetilde{\phi}_j  - \nabla_x \rho_j\cdot\nabla_x \widetilde{\phi}_j,
$$
where $\widetilde{\phi}_j$ is the variational solution of the problem
$$
\widetilde{P}_j \widetilde{\phi}_j = 0, \quad \widetilde{\phi}_j  \arrowvert_{z=0} = f.
$$
 
Set $U = U_1 - U_2.$ According to Lemma \ref{lions} with $\sigma = s-\frac{3}{2}$, 
the theorem will be a consequence of the following estimate
\begin{equation}\label{est:Uj}
\Vert U \Vert_{L^2(I, H^{s-1})} 
+  \Vert \partial_z U \Vert_{L^2(I, H^{s-2})} \leq  \text{RHS}, \quad I =(-1,0).   
\end{equation}
Set $\widetilde{\phi} =\widetilde{\phi}_1 - \widetilde{\phi}_2$. 
We claim that \eqref{est:Uj} is a consequence of the following estimate
\begin{equation}\label{est:phi}
\Vert \nabla_{x,z} \widetilde{\phi} \Vert_{L^2(I, H^{s-1})} \leq \text{RHS}.
\end{equation}
Indeed assume that \eqref{est:phi}  is proved. One term in $U$ can be written as
$$
\nabla_x \rho_1 \cdot \nabla_x \widetilde{\phi}_1 
- \nabla_x \rho_2 \cdot \nabla_x \widetilde{\phi}_2 
= \nabla_x \rho_1 \cdot \nabla_x \widetilde{\phi}  
+ \nabla_x(  \rho_1 - \rho_2)\cdot \nabla_x \widetilde{\phi}_2.
$$
Now, since $s> \frac{3}{4} + \frac{d}{2},$ we can apply \e{pr} 
with $s_0 = s-1, s_1 = s- \mez, s_2 = s-1.$ It follows that 
\begin{align*}
\Vert \nabla_x \rho_1 \cdot \nabla_x \widetilde{\phi} \Vert_{L^2(I, H^{s-1})} &\les \Vert \nabla_x \rho_1   \Vert_{L^\infty(I, H^{s-\mez})} \Vert  \nabla_x \widetilde{\phi} \Vert_{L^2(I, H^{s-1})}\\
& \les \Vert \eta_1 \Vert_{H^{s+ \mez}} \text{RHS}.
\end{align*}
For the second term, since $s> \frac{3}{4} + \frac{d}{2} >1$,  we can write
\begin{align*}
\Vert \nabla_x(  \rho_1 - \rho_2)\cdot \nabla_x \widetilde{\phi}_2\Vert_{L^2(I, H^{s-1})}
\les &\Vert  \nabla_x(  \rho_1 - \rho_2) \Vert_{L^2(I, H^{s-1})}\Vert   \nabla_x \widetilde{\phi}_2\Vert_{L^\infty(I, L^\infty)} \\
&+  \Vert \nabla_x(  \rho_1 - \rho_2) \Vert_{L^2(I,L^\infty)} 
\Vert   \nabla_x \widetilde{\phi}_2\Vert_{L^\infty(I, H^{s-1})}.
\end{align*}
Now we have (as in \e{eq.rho1} and \e{regofrho3})
\begin{align*}
&\Vert \nabla_x(  \rho_1 - \rho_2)\Vert_{L^2(I, H^{s-1})}\les 
\lA \eta_1-\eta_2\rA_{H^{s-\mez}},\\
& \Vert \nabla_x(  \rho_1 - \rho_2) \Vert_{L^2(I,L^\infty)}\les \lA \eta_1-\eta_2\rA_{W^{r-\mez,\infty}}.
\end{align*}
Also it follows from \e{oubli} that 
$$
\lA \nabla_x\widetilde{\phi}_2\rA_{L^\infty(I,H^{s-1})}\le \mathcal{F}\big(\lA \eta_2\rA_{H^{s+\mez}}\big)
\lA f\rA_{H^s}
$$
and it follows from Proposition \ref{p1H} that
$$
\Vert   \nabla_x \widetilde{\phi}_2\Vert_{L^\infty(I, L^\infty)}\le 
\mathcal{F}\big(\lA \eta_2\rA_{H^{s+\mez}}\big)\big\{ \lA f\rA_{H^s}+\lA f\rA_{\holder{r}}\big\}.
$$
Therefore
\begin{equation}\label{estimation1}
\Vert \nabla_x(  \rho_1 - \rho_2)\cdot \nabla_x \widetilde{\phi}_2\Vert_{L^2(I, H^{s-1})} 
\leq \text{RHS}.
\end{equation}
We have thus completed the estimate of the first term in $U$. For the second term, one checks that, 
similarly, the $L^2(I,H^{s-1})$-norm of 
$$
\frac{1+|\nabla_x\rho_1|^2}{\partial_z\rho_1}\partial_z \widetilde{\phi}_1-
\frac{1+|\nabla_x\rho_2|^2}{\partial_z\rho_2}\partial_z \widetilde{\phi}_2
$$
is bounded by $\text{RHS}$. This completes the proof of the fact 
$\Vert U \Vert_{L^2(I, H^{s-1})} \leq  \text{RHS}$ provided that \e{est:phi} is granted.  
  
Let us now prove that, similarly, \e{est:phi} implies that 
$\Vert \partial_z U \Vert_{L^2(I, H^{s-2})} \leq  \text{RHS}$.  
We begin by claiming that we have, for $j=1,2$,
\begin{equation}\label{dzU}
\partial_z  {U}_j =   \text{div} (\nabla_x \rho_j  \partial_z  \widetilde{\phi}_j ) - \text{div} ( \partial_z  \rho_j\nabla_x \widetilde{\phi}_j) .
\end{equation}
This follows from the fact (see   \eqref{P:div}) that we can write 
\begin{equation*}
0 = \widetilde{P}_j \widetilde{\phi}_j= -\text{div} (\nabla_x \rho_j  \partial_z  \widetilde{\phi}_j ) + \text{div} ( \partial_z  \rho_j\nabla_x \widetilde{\phi}_j)  + \partial_z U_j.
\end{equation*}
Therefore we have
\begin{align*}
\Vert \partial_z U \Vert_{L^2(I, H^{s-2})} &\leq  \Vert \nabla_x \rho_1  \partial_z  \widetilde{\phi}_1 - \nabla_x \rho_2  \partial_z  \widetilde{\phi}_2 \Vert_{L^2(I, H^{s-1})}  \\
&\quad + \Vert  \partial_z  \rho_1\nabla_x \widetilde{\phi}_1 -  \partial_z  \rho_2\nabla_x \widetilde{\phi}_2 \Vert_{L^2(I, H^{s-1})}.
\end{align*}
Since $s-1>0$, we can argue as above and conclude the estimate by means of product rules.

Therefore we are left with the proof of \eqref{est:phi}. Since $\widetilde{P}_j \widetilde{\phi}_j =0, j=1,2$ we can write
\begin{equation}\label{F+G}
\begin{aligned}
 \widetilde{P}_1 \widetilde{\phi}
 &= (\widetilde{P}_2 -\widetilde{P}_1)\widetilde{\phi}_2 =  F + \partial_z G\\
F&= \text{div} \big( \partial_z( \rho_2-\rho_1) \nabla_x \widetilde{\phi}_2 \big)
+ \text{div} \big( \nabla_x( \rho_1-\rho_2) \partial_z \widetilde{\phi}_2 \big)\\
G&= \nabla_x( \rho_1-\rho_2) \cdot\nabla_x \widetilde{\phi}_2 
+ \Big( \frac{ 1+ \vert \nabla_x \rho_2 \vert^2}{\partial_z \rho_2} 
- \frac{ 1+ \vert \nabla_x \rho_1 \vert^2}{\partial_z \rho_1}\Big) \partial_z \widetilde{\phi}_2.
\end{aligned}
\end{equation}
Arguing exactly as in the proof of \eqref{estimation1} we can write
\begin{equation} 
\Vert F \Vert_{L^2(I, H^{s-2})} + \Vert G \Vert_{L^2(I, H^{s-1})} \leq \text{RHS}.
\end{equation}
 It follows from Proposition \ref{dnint1F} that
 \begin{equation}\label{estim:2}
  \Vert \nabla_{x,z} \widetilde{\phi} \Vert_{L^2(I, H^{s-1})} \leq \mathcal{F}(\Vert \eta_1 \Vert_{H^{s+ \mez}}) \big(\text{RHS} +   \Vert \nabla_{x,z} \widetilde{\phi} \Vert_{X^{-\mez}(I)}\big).
\end{equation}
Then \eqref{est:phi} will be proved if we show that
 \begin{equation}\label{estim:3}
  \Vert \nabla_{x,z} \widetilde{\phi} \Vert_{X^{-\mez}(I)} \leq \text{RHS}. 
\end{equation}
We begin by proving the following estimate.
 \begin{equation}\label{est-v1}
\Vert \nabla_{x,z} \widetilde{\phi} \Vert_{L^2(J, L^2)} \leq \text{RHS}.
\end{equation} 
For this purpose we use the variational characterization of the solutions given in \cite{ABZ1,ABZ3}. 
It is sufficient to know that 
$\widetilde\phi_j = \widetilde{u}_j + \underline{\widetilde{f}}$ where 
$\underline{\widetilde{f}}=e^{z\langle D_x\rangle}f$ and $\widetilde{u}_j$ is such that, 
with the notations 
\begin{align*}
&X=(x,z) \in \widetilde{\Omega} = \{(x,z): x \in \xR^d, -1<z<0\},\\
&\Lambda^j=(\Lambda^j_1,\Lambda^j_2),\qquad \Lambda^j_1=\frac{1}{\partial_z \rho_j}\partial_z,\quad \Lambda^j_2=\nabla_x-\frac{\nabla_x\rho_j}{\partial_z \rho_j}\partial_z,
\end{align*}
we have
\begin{equation}\label{equi-bis}
 \int_{\widetilde{\Omega}}\Lambda ^j\widetilde{u}_j\cdot\Lambda ^j\theta \, J_j\, dX = - \int_{\widetilde{\Omega}}\Lambda ^j \underline{\widetilde{f}} \cdot\Lambda ^j\theta \, J_j\, dX 
 \end{equation}
 for all $\theta \in H^{1,0}(\widetilde{\Omega}),$ where $J_j = \vert \partial_z \rho_j \vert.$
 
Making the difference between the two equations \eqref{equi-bis}, 
  and taking $\theta =  \widetilde{\phi}= \widetilde{u}_1 -\widetilde{u}_2$  one can find a positive constant $C$ such that 
\begin{equation*} 
\int_{\widetilde{\Omega}}\vert \Lambda ^1\widetilde{\phi} \vert^2\, dX \leq C(A_1+\cdots +A_6)
\end{equation*}
where
\begin{equation}\label{n530}
\left\{
\begin{aligned}
&A_1=   \int_{\widetilde{\Omega}}\vert (\Lambda^1 - \Lambda^2) 
\widetilde{u}_2 \vert \vert \Lambda^1 \widetilde{\phi} \vert \,J_1\,  dX, \qquad   
&&A_2
=\int_{\widetilde{\Omega}}\vert (\Lambda^1 - \Lambda^2) \widetilde{\phi} \vert \vert \Lambda^2 \widetilde{u}_2 \vert \,J_1 dX,  \\
&A_3 = \int_{\widetilde{\Omega}}\vert \Lambda^2\widetilde{u}_2\vert \vert  
\Lambda^2 \widetilde{\phi} \vert \,\vert J_1-J_2\vert\, dX, \qquad 
&&A_4 = \int_{\widetilde{\Omega}}
\vert (\Lambda^1 - \Lambda^2) \underline{\widetilde{f}}  
\vert \vert \Lambda^1 \widetilde{u}\vert \,J_1\, dX, \\
& A_5 =  \int_{\widetilde{\Omega}}\vert (\Lambda^1 - \Lambda^2) \widetilde{\phi}
\vert \vert \Lambda^2 \underline{\widetilde{f}}  \vert \,J_1\, dX,
\qquad 
&& A_6 =  
 \int_{\widetilde{\Omega}}\vert \Lambda^2  \underline{\widetilde{f}} \vert \vert  \Lambda^2 \widetilde{\phi} \vert \,\vert J_1-J_2\vert\, dX .
\end{aligned}
\right.
 \end{equation}
 Noticing that $ \Lambda^1 - \Lambda^2  = \beta \partial_z$ we deduce  from Proposition~\ref{p1H} that 
\begin{equation*}
\begin{aligned}
 A_1 &\leq \Vert \Lambda^1 \widetilde{\phi} \Vert_{L^2(\widetilde{\Omega})}\Vert \beta \Vert_{L^2( \widetilde{\Omega})} \Vert \partial_z \widetilde{u}_2 \Vert_{L^\infty(I, L^\infty )} \\
 &\leq \Vert \Lambda^1\widetilde{\phi} \Vert_{L^2(\widetilde{\Omega})} \mathcal{F} (\Vert (\eta_1, \eta_2)\Vert_{H^{s+\mez}\times  H^{s+\mez}}) \Vert \eta_1-\eta_2\Vert_{H^\mez}\left\{\lA f\rA_{H^s}+\lA f\rA_{\holder{r}}\right\}.
 \end{aligned}
 \end{equation*}
Now 
$$
A_2 \leq   \Vert \beta \Vert_{L^2(\widetilde{\Omega})}
\Vert \Lambda^2 \widetilde{u}_2 \Vert_ {L^\infty(\widetilde{\Omega})}\Vert \partial_z \widetilde{\phi}\Vert_{L^2(\widetilde{\Omega})}.
$$
Using Proposition~\ref{p1H} we obtain
\begin{align*}
A_2 &\leq  \mathcal{F} \bigl(\lA (\eta_1,\eta_2)\rA_{H^{s+\mez} \times H^{s+\mez}}\bigr) \lA \eta_1-\eta_2\rA_{H^{ \mez}}\left\{\lA f\rA_{H^s}+\lA f\rA_{\holder{r}}\right\}\Vert \Lambda^1  \widetilde{\phi} \Vert_{L^2(\widetilde{\Omega})}\\
& \leq  \mathcal{F} \bigl(\lA (\eta_1,\eta_2)\rA_{H^{s+\mez} \times H^{s+\mez}}\bigr) \lA \eta_1-\eta_2\rA_{H^{ s- \mez}}\left\{\lA f\rA_{H^s}+\lA f\rA_{\holder{r}}\right\} \Vert \Lambda^1  \widetilde{\phi} \Vert_{L^2(\widetilde{\Omega})}.
\end{align*}
Now we estimate $A_3$ as follows. We have
$$
A_3 \leq \Vert \Lambda^2\widetilde{u}_2 \Vert_ {L^\infty(\widetilde{\Omega})}\Vert \Lambda^2 \widetilde{\phi} \Vert_ {L^ 2(\widetilde{\Omega})}\Vert J_1-J_2\Vert_{L^ 2(\widetilde{\Omega})}.
$$
Then we observe that  $\Vert J_1-J_2\Vert_{L^ 2(\widetilde{\Omega})}\les\Vert \eta_1-\eta_2\Vert_{H^\mez  } \les \Vert \eta_1-\eta_2\Vert_{H^{s-\mez}  }$   and we use the elliptic regularity to obtain
$$ A_3 \leq \mathcal{F}  \bigl(\lA (\eta_1,\eta_2)\rA_{H^{s+\mez} \times H^{s+\mez}}\bigr) \Vert f \Vert_{W^{r,\infty}} \Vert \eta_1-\eta_2\Vert_{H^{s-\mez}  } \Vert \Lambda^2 \widetilde{\phi} \Vert_ {L^ 2(\widetilde{\Omega})}.$$
To estimate $A_4$ and $A_5$ we recall 
that $\underline{\widetilde{f}} = e^{z\langle D_x \rangle}f.$ 
Then we have 
$$
\Vert \beta \partial_z \underline{\widetilde{f}} \Vert_{L^2(I  \times \xR^d)}  \leq  \Vert \beta  \Vert_{L^2(I\times \xR^d) }\Vert \partial_z\underline{\widetilde{f}} \Vert_{L^\infty(I\times \xR^d) }.
 $$
Since $\Vert \partial_z\underline{\widetilde{f}} \Vert_{L^\infty(I \times \xR^d)} 
\leq \Vert  f \Vert_{W^{r,\infty}} $ we obtain 
 $$
A_4+A_5 \leq \Vert \Lambda^1\widetilde{\phi} \Vert_{L^2(\widetilde{\Omega})} 
\mathcal{F} \bigl(\lA (\eta_1,\eta_2)\rA_{H^{s+\mez}}\bigr) 
\lA \eta_1-\eta_2\rA_{H^{s-\mez} \times H^{s-\mez}}\lA f\rA_{W^{r,\infty}}.
$$
The term $A_6$ is estimated like $A_3.$ This proves \eqref{est-v1}. 
    
To complete the proof of \eqref{estim:3}, 
in view of \e{est-v1}, it remains only to prove that 
$\Vert \nabla_{x,z} \widetilde{\phi} \Vert_{L^\infty(I, H^{-\mez})}\le \text{RHS}$. 
First of all by Lemma \ref{lions} we have 
$$
\Vert \nabla_{x} \widetilde{\phi} \Vert_{L^\infty(I, H^{-\mez})}
\leq C\big( \Vert \nabla_{x} \widetilde{\phi} \Vert_{L^2(I, L^2)}
+ \Vert \partial_z \nabla_{x} \widetilde{\phi} \Vert_{L^2(I, H^{-1})} \big) \leq C' \Vert \nabla_{x,z} \widetilde{\phi} \Vert_{L^2(I, L^2)}
$$
and we use \eqref{est-v1} to deduce that 
$\Vert \nabla_{x} \widetilde{\phi} \Vert_{L^\infty(I, H^{-\mez})}\le \text{RHS}$. 
So it remains to prove that, similarly, $\Vert \partial_z \widetilde{\phi} \Vert_{L^\infty(I, H^{-\mez})}\le \text{RHS}$. 
Here, by contrast with the estimate for $\nabla_x\widetilde{\phi}$, one cannot obtain the desired result from 
Lemma \ref{lions}, exploiting the previous bound \e{est-v1}. Indeed, one cannot use the equation satisfied by $\widetilde{\phi}$ to estimate the $L^2(I,H^{-1})$-norm of $\partial_z^2\widetilde{\phi}$. As above, we shall exploit the fact that one can factor out a spatial derivative by working with $U$ instead of $\widetilde{\phi}$. We shall prove that 
$\lA U\rA_{L^\infty(I,H^{-\mez})}\le \text{RHS}$ and then relate $\partial_z\widetilde{\phi}$ and $U$ to complete the proof. 

Lemma~\ref{lions} implies that
$$
\lA U\rA_{L^\infty(I,H^{-\mez})}\les \lA U\rA_{L^2(I,L^2)}+\lA \partial_z U\rA_{L^2(I,H^{-1})}.
$$
The $L^2(L^2)$-norm of $U$ is easily estimated using the bound \e{est-v1} for the $L^2(L^2)$-norm 
of $\nabla_{x,z}\widetilde{\phi}$. To estimate $\lA \partial_z U\rA_{L^2(I,H^{-1})}$, write
$$
\partial_z U=\cnx (\nabla\rho_1\partial_z\widetilde{\phi})+\cnx(\nabla \rho\partial_z 
\widetilde{\phi}_2)-\cnx(\partial_z\rho_1\nabla\widetilde{\phi})-\cnx(\partial_z\rho\nabla\widetilde{\phi}_2)
$$
so
\begin{align*}
\lA \partial_zU\rA_{L^2(I,H^{-1})}&\le 
\lA \nabla_{x,z}\rho_1\rA_{L^\infty(I,L^\infty)}\big\Vert \nabla_{x,z}\widetilde{\phi}\big\Vert_{L^2(I,L^2)}\\
&\quad +\lA \nabla_{x,z}\rho\rA_{L^2(I,L^2)}\big\Vert \nabla_{x,z}\widetilde{\phi}_2\big\Vert_{L^\infty(I,L^\infty)}
\end{align*}
which implies that $\lA U\rA_{L^\infty(I,H^{-\mez})}\le \text{RHS}$. Now, directly from the definition of $U$ one has
$$
\partial_z \widetilde{\phi}=\frac{\partial_z \rho_1}{1+|\nabla\rho_1|^2} 
\Big[ U+\Big(\frac{1+|\nabla\rho_2|^2}{\partial_z\rho_2}-\frac{1+|\nabla\rho_1|^2}{\partial_z\rho_1}\Big)\partial_z\widetilde{\phi}_2+\nabla\rho_1\cdot\nabla\widetilde{\phi}+\nabla\rho\cdot\nabla\widetilde{\phi}_2\Big].
$$
The $L^\infty(I,H^{-\mez})$-norm of the term between brackets is bounded by $\text{RHS}$, using the fact that the product 
is bounded from $H^{s-\mez}\times H^{-\mez}$ (resp.\ $L^2\times H^{s-1}$) to $H^{-\mez}$ (resp.\ $H^{-\mez}$) 
in order to estimate $\nabla\rho_1\cdot \nabla\widetilde{\phi}$ (resp.\ the other terms). 
Since the coefficient 
$\frac{\partial_z \rho_1}{1+|\nabla\rho_1|^2} $ belongs to $L^\infty(I,H^{s-\mez})$ and since the product 
is bounded from $H^{s-\mez}\times H^{-\mez}$ to $H^{-\mez}$, this concludes the proof.
\end{proof}

\subsection{Paralinearization of the equations}Recall from Proposition~\ref{prop:newS} that 
\begin{equation}\label{BVzeta0}
\left\{
\begin{aligned}
&(\partial_{t}+V_j\cdot\partialx)\B_j=\ma_j-g,\\
&(\partial_t+V_j\cdot\partialx)V_j+\ma_j\zeta_j=0,\\
&(\partial_{t}+V_j\cdot\partialx)\zeta_j=G(\eta_j)V_j+ \zeta_j G(\eta_j)\B_j+\gamma_j,\quad \zeta_j=\nabla \eta_j,
\end{aligned}
\right.
\end{equation}
where~$\gamma_j$ is the remainder term given by \eqref{eq:zeta}. 
We now compute and paralinearize the equations satisfied by the differences
$$
\zeta= \zeta_1 - \zeta_2, \quad V= V_1-V_2, \quad B=B_1- B_2.
$$

In \cite{ABZ3}, assuming that $s>1+d/2$, we deduced that 
\begin{equation}\label{n531}
\left\{
\begin{aligned}
&(\partial_t +V_1\cdot\partialx)(\dV+\zeta_1 B)+\ma_2 \dzeta =F_1 ,\\
&(\partial_{t}+V_2\cdot\partialx)\dzeta-G(\eta_1)\dV- \zeta_1 G(\eta_1)\dB=F_2,
\end{aligned}
\right.
\end{equation}
for some remainders such that
$$
\lA (F_1,F_2)\rA_{L^{p}([0,T];H^{s-1}\times H^{s-\tdm})} \le \K(M_1,M_2)N(T).
$$
However, the estimate for $f_2$ no longer hold 
when $s<3/2$ for the reason 
explained in Remark~\ref{T522}. 
To overcome this problem, the key point is that one obtains the desired result by replacing 
$\partial_{t}+V_2\cdot\partialx$ (resp.\ $\zeta_1 G(\eta_1)\dB$) 
by $\partial_{t}+T_{V_2}\cdot\partialx$ (resp.\ $T_{\zeta_1}
G(\eta_1)\dB$ in the equation for $\zeta$ (there is a cancellation when one adds the remainders).

\begin{lemm}\label{lemm:symmind}
The differences~$\zeta,B,V$ satisfy a system of the form
\begin{equation}\label{syst:delta}
\left\{
\begin{aligned}
&(\partial_t +T_{V_1}\cdot\partialx)(\dV+\zeta_1 B)+\ma_2 \dzeta =f_1 ,\\
&(\partial_{t}+T_{V_2}\cdot\partialx)\dzeta-G(\eta_1)\dV- T_{\zeta_1}
G(\eta_1)\dB=f_2,
\end{aligned}
\right.
\end{equation}
for some remainders such that
$$
\lA (f_1,f_2)\rA_{L^{p}([0,T];H^{s-1}\times H^{s-\tdm})} \le \K(M_1,M_2)N(T).
$$
\end{lemm}
\begin{proof}
Directly from \e{BVzeta0}, it is easily verified that
\begin{equation}\label{n535}
\left\{
\begin{aligned}
&\partial_{t}\dB+V_1\cdot\partialx\dB=\da +R_1 ,\\
&\partial_t \dV+V_1\cdot\partialx\dV+\ma_2 \dzeta+ \da \zeta_1=R_2,\\
\end{aligned}
\right.
\end{equation}
where $a=a_1-a_2$ and
\begin{equation*}
R_1=-V\cdot \nabla B_2,\quad R_2=-\dV \cdot\nabla V_2.
\end{equation*}
Now we use the first equation of \e{n535} to express $\da$ in terms of 
$\partial_{t}\dB+V_1\cdot\partialx\dB$ in the last term of the left-hand side of 
the second equation of \e{n535}. Replacing 
$V_1\cdot\nabla$ in the right-hand side by $T_{V_1}\cdot \nabla$, 
we thus obtain the first equation of 
\e{syst:delta} with
$$
f_1=\dB (\partial_t \zeta_1+V_1\cdot\nabla \zeta_1)+R_1\zeta_1+R_2
-T_{ \partialx(\dV+\zeta_1 B)}\cdot V_1-R\big(V_1,\partialx(\dV+\zeta_1 B)\big).
$$
To estimate $\partial_t \zeta_1+V_1\cdot\nabla \zeta_1$ we use the identity 
\e{n38} to deduce
$$
\partial_t \zeta_1+V_1\cdot\nabla \zeta_1=G(\eta_1)V_1-(\cnx V_1)\zeta_1.
$$
Then, using \e{ests+2} to estimate $G(\eta_1)V_1$ and the fact that the product 
is bounded from $H^{s-1}\times H^{s-\mez}$ into $H^{s-1}$, we get
\begin{align*}
\lA \partial_t \zeta_1+V_1\cdot\nabla \zeta_1\rA_{H^{s-1}}
&\le \lA G(\eta_1)V_1\rA_{H^{s-1}}+\lA \cnx V_1\rA_{H^{s-1}}\lA \zeta_1\rA_{H^{s-\mez}}\\
&\le \mathcal{F}\big(\lA \eta_1\rA_{H^{s+\mez}}\big)\lA V_1\rA_{H^s}.
\end{align*}
Also, it follows from the $L^\infty$-estimate \e{n232} for $G(\eta)f$ that
$$
\lA \partial_t \zeta_1+V_1\cdot\nabla \zeta_1\rA_{L^\infty}\le 
\mathcal{F}\big(\lA \eta_1\rA_{H^{s+\mez}}\big)\big\{\lA V_1\rA_{H^s}+\lA V_1\rA_{\holder{r}}\big\}.
$$

Then the estimate of the $H^{s-1}$-norm of $f_1$ follows 
from the usual tame 
estimate in Sobolev space (see \e{prS2}), the rule \e{niS} 
applied with 
$m=1$ to estimate $T_{ \partialx(\dV+\zeta_1 B)}\cdot V_1$, 
the rule \e{Bony3} applied with $\alpha=1$ and $a=V_1$ to 
estimate the remainder $R\big(V_1,\partialx(\dV+\zeta_1 B)\big)$ as well as the estimates for the 
Dirichlet-Neumann operator given by 
Propositions~\ref{T:DN-Hs} and \ref{p1H}.

To compute and to estimate $f_2$ 
we shall rewrite the equation for $\zeta_j$, $j=1,2$, by using 
the identity \e{n38} written under the form
$$
\partial_t \zeta_j=G(\eta_j)V_j-\Theta_j
$$ 
where $\Theta_j$ is the function with values in $\xR^d$ whose coordinates 
$\Theta_j^k$ is given by $\Theta_j^k=\cnx (V_j\zeta_j^k)$. Now write
$$
\Theta_j^k=\cnx (V_j\zeta_j^k)
=\cnx( T_{V_j} \zeta_j^k)+\cnx (T_{\zeta_j^k}V_j)+\cnx (R(\zeta_j^k,V_j))
$$
and use the Leibniz rule $\partial_{\alpha}T_a b=T_{\partial_\alpha a} b
+T_a\partial_{\alpha}b$ to obtain
$$
\Theta_j^k=T_{V_j} \nabla \zeta_j^k +T_{\zeta_j^k}\cnx V_j +F_j^k
$$
with
$$
F_j^k=T_{\cnx V_j} \zeta_j^k+T_{\nabla\zeta_j^k}\cdot V_j+\cnx (R(\zeta_j^k,V_j)).
$$
Using the identity $G(\eta_j)B_j=-\cnx V_j +\widetilde{\gamma}_j$ where 
$\widetilde{\gamma}_j$ is estimated by means of \e{n32}, we obtain
$$
\Theta_j=T_{V_j}\cdot\nabla\zeta_j-T_{\zeta_j}G(\eta_j)B_j
+T_{\zeta_j}\widetilde{\gamma}_j+F_j
$$
so
$$
\partial_t \zeta_j+T_{V_j}\cdot\nabla \zeta_j 
=G(\eta_j)V_j+T_{\zeta_j}G(\eta_j)B_j-T_{\zeta_j}\widetilde{\gamma}_j-F_j.
$$
Subtracting the equation for $j=1$ and the one for $j=2$ 
we obtain the desired equation for 
$\zeta=\zeta_1-\zeta_2$ with
\begin{align*}
f_2&\defn (I)+(II)+(III)-F\\
(I)&\defn \bigl(G(\eta_1)-G(\eta_2)\big)V_2+T_{\zeta_2}\big(G(\eta_1)-G(\eta_2)\big)
B_2,\\
(II)&\defn -T_V\cdot\nabla\zeta_1+T_\zeta G(\eta_1)B_2,\\
(III)&\defn -T_{\zeta_1}\widetilde{\gamma}-T_\zeta\widetilde{\gamma}_2,\\
F&\defn F_1-F_2.
\end{align*}
The term $(I)$ is estimated by means of Proposition~\ref{L:p60}. 
To estimate $(II)$, write
$$
\lA T_V\cdot\nabla\zeta_1\rA_{H^{s-\tdm}}\les \lA V\rA_{L^\infty}\lA \zeta_1\rA_{H^{s-\mez}}
$$
and
$$
\lA T_\zeta G(\eta_1)B_2\rA_{H^{s-\tdm}}\les 
\lA \zeta\rA_{\zygmund{-\mez}}\lA G(\eta_1)B_2\rA_{H^{s-1}}
\le \lA \zeta\rA_{\zygmund{-\mez}}\mathcal{F}\big(\lA \eta_1\rA_{H^{s+\mez}}\big) 
\lA B_2\rA_{H^{s}}
$$
where we used Proposition~\ref{T:DN-Hs} to estimate $G(\eta_1)B_2$. 

It remains to estimate $F$ given by 
\begin{align*}
F=T_{\cnx V_2}\zeta+T_{\cnx V}\zeta_1+T_{\nabla\zeta_2}\cdot V+T_{\nabla\zeta}\cdot V_1+\cnx\big(R(\zeta,V_2)\big)+\cnx \big(R(\zeta_1,V)\big).
\end{align*}
Write
\begin{align*}
&\lA T_{\cnx V_2}\zeta\rA_{H^{s-\tdm}}\les 
\lA \cnx V_2\rA_{L^\infty}\lA \zeta\rA_{H^{s-\tdm}}\les 
\lA V_2\rA_{\holder{r}}\lA \eta\rA_{H^{s-\mez}}
,\\
&\lA T_{\cnx V}\zeta_1\rA_{H^{s-\tdm}}\les \lA \cnx V\rA_{\zygmund{-1}}\lA \zeta_1\rA_{H^{s-\mez}}\les \lA V\rA_{\holder{r-1}}\lA \eta_1\rA_{H^{s\mez}},\\
&\lA T_{\nabla\zeta_2}\cdot V\rA_{H^{s-\tdm}}\les 
\lA \nabla\zeta_2\rA_{\zygmund{-\mez}}\lA V\rA_{H^{s-1}}
\les \lA \eta\rA_{\holdertdm}\lA V\rA_{H^{s-1}}.
\end{align*}
Now the key point of this proof is that one has the following obvious inequalities
$$
\lA \cnx\big(R(\zeta,V_2)\big)+\cnx \big(R(\zeta_1,V)\big)\rA _{H^{s-\tdm}}
\le \lA R(\zeta,V_2)+R(\zeta_1,V)\rA _{H^{s-\mez}}.
$$
Since $s-1/2>0$ (by contrast with $s-3/2$ which might be negative), one can use \e{Bony3} to deduce
\begin{align*}
&\lA R(\zeta,V_2)\rA _{H^{s-\mez}}\les \lA \zeta\rA_{\zygmund{-\mez}}\lA 
V_2\rA_{H^s}\les \lA \eta\rA_{\holder{r-\mez}}\lA V_2\rA_{H^s},\\
&\lA R(\zeta_1,V)\rA _{H^{s-\mez}}\les \lA \zeta_1\rA_{H^{s-\mez}}\lA V\rA_{L^\infty}
\les \lA \eta_1\rA_{H^{s+\mez}}\lA V\rA_{\holder{r-1}}.
\end{align*}
This completes the proof.
\end{proof}

Once Proposition~\ref{L:p60} and Lemma~\ref{lemm:symmind} are established, the end of the proof 
of the contraction estimate in Theorem~\ref{th.lipschitz} follows from the analysis in \cite{ABZ3}. 
We shall recall the scheme of the proof for the sake of completeness.

Our goal is to prove an 
estimate of the form 
\begin{equation}\label{NM1M2}
N(T)\le \mathcal{F}(M_1, M_2) N(0) +T^\delta \, \mathcal{F}(M_1,M_2) N (T),
\end{equation}
for some $\delta>0$ and some function~$\mathcal{F}$ depending only on~$s$ and~$d$. 
This implies $N(T_1)\leq 2\mathcal{F}(M_1, M_2) N(0)~$ for~$T_1$ 
small enough (depending on~$T$ and~$\mez \mathcal{F}(M_1,M_2)$), and iterating 
the estimate between~$[T_1, 2T_1]$,\dots,~$[T- T_1, T]$ implies Theorem~\ref{th.lipschitz}. 

Firstly, one symmetrize the equations \e{syst:delta}. 
\begin{lemm} \label{symetr}
Set~$I = [0,T]$ and
$$
\lambda_1\defn \sqrt{(1+|\nabla \eta_1|^2)|\xi|^2-(\nabla\eta_1\cdot\xi)^2},
$$
and
$$
\ell \defn \sqrt{ \lambda_1 \ma_2},\quad \varphi\defn T_{\sqrt{\lambda_1}}(\dV+\zeta_1 B),\quad \vartheta\defn T_{\sqrt{\ma_2}} \dzeta.
$$
Then 
\begin{align}
&(\partial_t +T_{V_1}\cdot\partialx)\varphi + T_{\ell}\vartheta =g_1 ,\label{syst:delta5}\\
&(\partial_{t}+T_{V_2}\cdot\partialx)\vartheta -T_{\ell }\varphi =g_2,\label{syst:delta5bis}
\end{align}
where
$$
\lA (g_1,g_2)\rA_{L^{p}(I;H^{s-\tdm}\times H^{s-\tdm})} \le \K(M_1,M_2)N(T).
$$
\end{lemm}
The previous result is proved following the proof of Lemma~$5.6$ in \cite{ABZ3}.

We then use the previous result and classical arguments to obtain a Sobolev estimate. 

\begin{lemm}\label{L63.a}
Set
$$
M'(T)\defn  \sup_{t\in I} \big\{\lA \vartheta(t)\rA_{H^{s-\tdm}}
+ \lA \varphi(t)\rA_{H^{s-\tdm}}\big \}.
$$
We have
\begin{equation}\label{esti:N'}
M'(T)\le \mathcal{K}(M_1,M_2)\bigl( N(0)+T^\delta N(T)\bigr)
\end{equation}
for some $\delta>0$.
\end{lemm}
This follows by using {\em mutatis mutandis} the arguments used in the proof 
of Lemma~$5.7$ in \cite{ABZ3}. 

Then the end of the proof of Theorem~\ref{th.lipschitz} is in two steps. Firstly, using \e{esti:N'}, 
one deduces a Sobolev estimate for the original unknown $(\eta,\psi,V,B)$. 
Again, this follows from the analysis in \cite{ABZ3}. 
Secondly, one has to estimate the H\"older norms. To do so, we use 
Theorem~\ref{T20}. To use this theorem, one has to reduce the analysis to a scalar 
equation. Notice that System~\e{syst:delta5}--\e{syst:delta5bis} involves two 
different transport operators $\partial_t+V_1\cdot\nabla$ and 
$\partial_t+V_2\cdot\nabla$. This is used in the proof of 
Lemma~$5.6$ in \cite{ABZ3} to bound the commutator 
$[T_{\sqrt{a_2}},\partial_t+V_2\cdot\nabla]$ in terms of the $L^\infty$-norm 
of $\partial_t a_2+V_2\cdot\nabla a_2$. However, once this symmetrization is done and 
Lemma~\ref{symetr}Ê
is proved, one 
can freely replace $\partial_t+T_{V_2}\cdot\nabla$ by 
$\partial_t+T_{V_1}\cdot\nabla$ in the equation for $\vartheta$. Indeed, this produces 
an error 
$T_{V_1-V_2}\cdot\nabla\vartheta $ that we estimate writing 
$$
T_{V_1-V_2}\cdot\nabla\vartheta =
\Big\{T_{V_1-V_2}\cdot\nabla T_{\sqrt{a_2}}\zeta_1 \Big\}
-\Big\{T_{V_1-V_2}\cdot\nabla T_{\sqrt{a_2}}\zeta_2 \Big\}
$$
and we estimate separately the contribution of each term, writing 
$$
\lA T_{V_1-V_2}\cdot\nabla T_{\sqrt{a_2}}\zeta_1 \rA_{H^{s-\tdm}}
\les \lA V_1-V_2\rA_{L^\infty}\lA \sqrt{a_2}\rA_{L^\infty}\lA \zeta_1\rA_{H^{s-\mez}}
$$
together with a similar estimate for the other term. By so doing, it follows from 
\e{syst:delta5bis} that
$$
(\partial_{t}+T_{V_1}\cdot\partialx)\vartheta -T_{\ell }\varphi =G_2,
$$
where $G_2$ satisfies $\lA G_2\rA_{L^{p}(I;H^{s-\tdm}\times H^{s-\tdm})} \le \K(M_1,M_2)N(T)$. 
Now we find that $u=\vartheta+i\varphi$ satisfies 
$$
\partial_t u+T_{V_1}\cdot\nabla u +iT_{\ell}u=g
$$
with 
$$
\lA g\rA_{L^{p}(I;H^{s-\tdm}\times H^{s-\tdm})} \le \K(M_1,M_2)N(T).
$$
Since the $L^{p}(I;H^{s-\tdm}\times H^{s-\tdm})$-norm of 
$T_{\cnx V_1} u$ is bounded by $\K(M_1,M_2)N(T)$, one can further 
reduce the analysis 
to 
$$
\partial_t u+\mez\big( T_{V_1}\cdot\nabla +\nabla\cdot T_{V_1}\big)u
+iT_{\ell}u=g'
$$
with 
$$
\lA g'\rA_{L^{p}(I;H^{s-\tdm})} \le \K(M_1,M_2)N(T).
$$
Then one is in position to apply Theorem~\ref{T20} with $s$ replaced by $s-3/2$ (notice that $s$ in any real number in 
Theorem~\ref{T20}) to obtain 
\begin{align*}
&\Vert  u  \Vert_{L^p(I;  \holder{r-\mez})}\\
&\qquad \le \Vert  u  \Vert_{L^p(I;  \zygmund{s-\tdm-\frac{d}{2}+\mu})}\\
&\qquad\leq    \mathcal{F}\big(\Vert V_1\Vert_{E_0} +  \mathcal{N}_k(\sqrt{ \lambda_1 \ma_2})\big) 
\Big\{  \Vert  g'\Vert_{L^p(I; H^{s-\tdm})} 
+ \Vert  u \Vert_{C^0(I; H^{s-\tdm})}\Big\}\\
&\qquad \le \K(M_1,M_2)N(T).
\end{align*}
Then, proceeding as above, we recover an estimate for the original unknowns, that is an estimate for 
$\lA (\eta,V,B)\rA_{L^p([0,T];\holder{r-\mez}\times \holder{r-1}\times \holder{r-1})}$. 
This completes the proof of Theorem~\ref{th.lipschitz}.

\section{Passing to the limit in the equations}\label{S:53}

Below we shall obtain rough solutions of the water waves system as limits of smoother solutions. 
This requires to prove that one can pass to the limit in the equations. 
Here we shall prove that it is possible to do so even under very mild assumptions.

Firstly, we prove a strong continuity property of the Dirichlet Neumann operator at 
the minimal level of regularity required to prove that 
$G(\eta)\psi$ is well-defined,  
that is for any 
Lipschitz function $\eta$ and any $\psi$ in $H^{\mez}(\xR^d)$, recalling that
\be\label{n539}
\lA G(\eta)f\rA_{H^{-\mez}}\le \mathcal{F} \big(\lA \eta\rA_{W^{1,\infty}}\big)\lA f\rA_{H^{\mez}}.
\ee
We have the following result which complements Proposition \ref{L:p60}. 
\begin{prop}\label{th.25}
There exists a non decreasing function $\mathcal{F}\colon\xR_+\rightarrow \xR_+$ such that, 
for all $\eta_j\in W^{1,\infty}({\mathbf{R}}^d)$,~$j=1,2$ and all $f\in H^{\mez} ({\mathbf{R}}^d)$, 
\be\label{n540}
\big\Vert \bigl(G(\eta_1)- G(\eta_2)\bigr) f \big\Vert_{H^{-\mez} }
\leq \mathcal{F} \bigl(\|(\eta_1,\eta_2)\|_{W^{1,\infty}\times W^{1,\infty}}\bigr)
\|\eta_1 - \eta_2\|_{W^{1,\infty}} \|f \|_{H^\mez}.
\ee
\end{prop}
\begin{proof}
The proof follows the one of Proposition \ref{L:p60}. 
In particular, we shall use the variational formulation 
of the harmonic extension of $f$ used in the last part of the proof of 
Proposition \ref{L:p60}.

Recall that, for~$j=1,2$, we introduce 
$\rho_j(x,z)$ defined by~\eqref{diffeo} (replacing of course $\eta$ by $\eta_j$). 
Recall also 
(see the paragraph below \e{est-v1}) 
that we have set
 \begin{align}
\Lambda_1^j = \frac{1}{\partial_z \rho_j}\partial_z, \quad \Lambda_2^j 
= \nabla_x - \frac{\nabla_x \rho_j}{\partial_z \rho_j}\partial_z.
\end{align}
and
$\widetilde{\phi}_j(x,z) = \phi_j(x,\rho_j(x,z))$ (where $\Delta_{x,y}\phi_j =0$ in $\Omega_j, \phi_j\arrowvert_{\Sigma_j} = f $). 
As already seen, we then have
\begin{equation}\label{Getaj}
 G(\eta_j)f = U_j\arrowvert_{z=0}, \quad U_j= \Lambda_1^j \widetilde{\phi}_j - \nabla_x \rho_j\cdot \Lambda_2^j \widetilde{\phi}_j.
\end{equation}
We shall make repeated use of the bound
\be\label{est-rhoiii}
\Vert \nabla_{x,z}(\rho_1 -\rho_2)\Vert_{L^\infty(I,L^\infty(\xR^d))}
\leq C\Vert \eta_1-\eta_2\Vert_{W^{1,\infty}(\xR^d)}.
\ee
This implies that
\begin{equation}\label{diff-lambda}
\left\{
 \begin{aligned}
 &(i) \quad \Lambda^1_k - \Lambda^2_k = \beta_k \partial_z, 
 \quad \text{with }\supp\beta_k \subset \xR^d\times I, \text{where }  I = (-1,0), \\
 &(ii) \quad \Vert \beta_k \Vert_{L^\infty(I\times \xR^d)}
 \leq\mathcal{F}(\Vert (\eta_1, \eta_2)\Vert_{W^{1,\infty} \times W^{1,\infty} }) 
 \Vert \eta_1-\eta_2\Vert_{W^{1,\infty} }.
 \end{aligned}
 \right.
 \end{equation}

\begin{lemm}\label{dem:estvar}
Set $I=(-1,0)$, $v = \widetilde{\phi}_1 - \widetilde{\phi}_2,$ and $ \Lambda^j = (\Lambda^j_1, \Lambda^j_2) $. 
There exists a non decreasing function $\mathcal{F}:\xR^+\to \xR^+$ such that 
\begin{equation}\label{est2:var}
\Vert \Lambda^j v \Vert_{L^2(I;L^2(\xR^d))}
\leq \mathcal{F}( \Vert(\eta_1, \eta_2)\Vert_{W^{1,\infty} \times W^{1,\infty} })
\Vert \eta_1-\eta_2\Vert_{W^{1,\infty} }
\Vert f\|_{H^{\mez}}.
\end{equation}
\end{lemm}
Let us show how this Lemma implies 
Theorem \ref{th.25}. According to \eqref{Getaj} we have
\begin{equation}\label{formeU}
\begin{aligned}
 U_1-U_2 &= (1) +(2) +(3)+ (4)+ (5) \quad \text{where}\\
(1) &= \Lambda_1^1 \widetilde{\phi}, \quad (2) 
= (\Lambda_1^1 -\Lambda_1^2)\widetilde{\phi}_2, \quad (3) =  - \nabla_x (\rho_1 -   \rho_2) \Lambda_2^1 \widetilde{\phi}_1\\
(4)&= - (\nabla_x \rho_2) \Lambda_2^1\widetilde{\phi}, \quad (5) =   - (\nabla_x \rho_2)(\Lambda_2^1- \Lambda_2^2)\widetilde{\phi}_2.
\end{aligned}
 \end{equation}

The $L^2(I, L^2(\xR^d))$ norms of $(1)$ and $(4)$ are estimated 
using \eqref{est2:var}. Also, the $L^2(I, L^2(\xR^d))$ norms of $(2)$ and $(5)$ are estimated 
by the right hand side of \eqref{est2:var} using \eqref{diff-lambda} and \eqref{est:phi}. 
Eventually the $L^2(I, L^2(\xR^d))$ norm of $(3)$ is also estimated by the right hand side 
of \eqref{est2:var} using \eqref{est-rhoiii} and \eqref{est:phi}. It follows that 
\begin{equation}\label{est:Ujb}
\Vert U_1-U_2 \Vert_{L^2(I,L^2 )} 
\leq \mathcal{F}( \Vert(\eta_1, \eta_2)\Vert_{W^{1,\infty} \times W^{1,\infty}})
\|\eta_1 - \eta_2\|_{W^{1, \infty}}  \| f\|_{H^{\mez}}.
\end{equation}
Now according to \eqref{dzU} we have
\begin{equation}\label{formedzU}
\partial_z(U_1-U_2)
= -\nabla_x\big(  \partial_z(\rho_1- \rho_2 )\Lambda_2^1 \widetilde{\phi}_1 
+ (\partial_z\rho_2)(\Lambda_2^1-\Lambda_2^2) \widetilde{\phi}_1 +(\partial_z\rho_2)\Lambda_2^2\widetilde{\phi}\big).
\end{equation}
Therefore using the same estimates as above we see easily that
\begin{equation}\label{est:dzUj}
\Vert  \partial_z(U_1-U_2)\Vert_{L^2(I,H^{-1} )}
\leq \mathcal{F}( \Vert(\eta_1, \eta_2)\Vert_{W^{1,\infty} \times W^{1,\infty} }) \|\eta_1 - \eta_2\|_{W^{1, \infty}}  \| f\|_{H^{\mez}}.
\end{equation}
Then Theorem \ref{th.25} follows from \eqref{est:Ujb}, \eqref{est:dzUj} and Lemma \ref{lions}.
\end{proof}
\begin{proof}[Proof of Lemma \ref{dem:estvar}]
We proceed as in the end of the proof of Proposition~\ref{L:p60}. 
Namely we use the inequality 
\begin{equation*} 
\int_{\widetilde{\Omega}}\vert \Lambda ^1\widetilde{\phi} \vert^2\, dX \leq C(A_1+\cdots +A_6)
\end{equation*}
where $A_1,\ldots,A_6$ are given by \e{n530}, and we recall their expressions for the reader convenience:
\begin{equation*}
\left\{
\begin{aligned}
&A_1=   \int_{\widetilde{\Omega}}\vert (\Lambda^1 - \Lambda^2) 
\widetilde{u}_2 \vert \vert \Lambda^1 \widetilde{\phi} \vert \,J_1\,  dX, \qquad   
&&A_2
=\int_{\widetilde{\Omega}}\vert (\Lambda^1 - \Lambda^2) \widetilde{\phi} \vert \vert \Lambda^2 \widetilde{u}_2 \vert \,J_1 dX,  \\
&A_3 = \int_{\widetilde{\Omega}}\vert \Lambda^2\widetilde{u}_2\vert \vert  
\Lambda^2 \widetilde{\phi} \vert \,\vert J_1-J_2\vert\, dX, \qquad 
&&A_4 = \int_{\widetilde{\Omega}}
\vert (\Lambda^1 - \Lambda^2) \underline{\widetilde{f}}  
\vert \vert \Lambda^1 \widetilde{u}\vert \,J_1\, dX, \\
& A_5 =  \int_{\widetilde{\Omega}}\vert (\Lambda^1 - \Lambda^2) \widetilde{\phi}
\vert \vert \Lambda^2 \underline{\widetilde{f}}  \vert \,J_1\, dX,
\qquad 
&& A_6 =  
 \int_{\widetilde{\Omega}}\vert \Lambda^2  \underline{\widetilde{f}} \vert \vert  \Lambda^2 \widetilde{\phi} \vert \,\vert J_1-J_2\vert\, dX,
\end{aligned}
\right.
 \end{equation*}
where $\underline{\widetilde{f}}$ and 
$\widetilde{u}_j$ ($\widetilde{u}\defn\widetilde{u}_1-\widetilde{u}_2$) are such that $\widetilde\phi_j = \widetilde{u}_j + \underline{\widetilde{f}}$ with  
$\underline{\widetilde{f}}=e^{z\langle D_x\rangle}f$.
 
Using \eqref{diff-lambda}, \eqref{est:phi}, \eqref{est-rhoiii} we can write
\begin{equation}\label{est:A1}  
\begin{aligned}
\vert A_1 \vert &\leq \Vert \beta \Vert_{L^{\infty}(I\times \xR^d)}
\Vert J_1 \Vert_{L^{\infty}(I\times \xR^d)}\Vert \partial_z \widetilde{u}_2\Vert_{L^{2}(I\times \xR^d)}
\Vert \Lambda^1 \widetilde{\phi}\Vert_{L^{2}(\widetilde{\Omega})}\\
& \leq \mathcal{F}( \| (\eta_1, \eta_2)\|_{W^{1, \infty}\times W^{1, \infty}})
\|\eta_1 - \eta_2\|_{W^{1, \infty}}  \| f\|_{H^{\mez}}\Vert \Lambda^1 \widetilde{\phi}\Vert_{L^{2}(\widetilde{\Omega})}.
\end{aligned}
\end{equation}
Since $\Lambda_j^1 - \Lambda_j^2 = \beta_j \partial_z \rho_1\Lambda_1^1$ 
the term $A_2$ can be bounded by the right hand side of \eqref{est:A1}.

Now we have $\Vert J_1 - J_2 \Vert _{L^{\infty}(I\times \xR^d)} \leq C\Vert \eta_1 -\eta_2 \Vert_{W^{1,\infty}(\xR^d)}$ and 
$$
\Vert \Lambda^2 \widetilde{\phi}\Vert_{L^{2}(\widetilde{\Omega})}
\leq \mathcal{F}( \| (\eta_1, \eta_2)\|_{W^{1, \infty}\times W^{1, \infty}}) \Vert \Lambda^1\widetilde{\phi} \Vert _{L^{2}(\widetilde{\Omega})}.
$$
So using \eqref{est:phi} we see that the term $A_3$ 
can be also estimated by the right hand side of \eqref{est:A1}. 
To estimate the terms $A_4$ to $A_6$ we use the same arguments and also 
$\Vert \underline{\widetilde{f}} \Vert_{H^1(\widetilde{\Omega})} 
\les \Vert f\Vert_{H^\mez({\mathbf{R}}^d)}$. This completes the proof.
\end{proof}

\begin{nota}
Given two functions $\eta,\psi$ one writes
$$
B(\eta)\psi=\frac{G(\eta)\psi+\nabla\eta\cdot\nabla\psi}{1+|\nabla\eta|^2},\quad 
V(\eta)\psi=\nabla\psi-(B(\eta)\psi)\nabla\eta.
$$
\end{nota}

\begin{coro}\label{T534}
Fix $s>\mez +\frac{d}{2}$. 
Consider two functions 
$\eta$ and $\psi$ in
$H^{s+\mez}(\xR^d)$ as well as two 
sequences $(\eta_n)_{n\in \xN}$ and $(\psi_n)_{n\in\xN}$ such that the following properties hold:
\begin{enumerate}
\item $(\eta_n)_{n\in \xN}$ and $(\psi_n)_{n\in\xN}$ 
are bounded sequences in $H^{s+\mez}(\xR^d)$;
\item $(\eta_n)_{n\in \xN}$ converges to $\eta$ 
in $W^{1,\infty}(\xR^d)$;
\item $(\psi_n)_{n\in \xN}$ converges to $\psi$ 
in $H^{\mez}(\xR^d)$.
\end{enumerate}
Then $G(\eta_n)\psi_n$ (resp.\ $B(\eta_n)\psi_n$, resp.\ 
$V(\eta_n)\psi_n$) converges in $H^{-\mez}(\xR^d)$ to 
$G(\eta)\psi$ (resp.\ $B(\eta)\psi$, resp.\ $V(\eta)\psi$).
\end{coro}
\begin{proof}
The proof is straightforward. 
Write
$$
G(\eta_n)\psi_n-G(\eta)\psi=\big(G(\eta_n)-G(\eta)\big)\psi_n
+G(\eta)(\psi_n-\psi).
$$
The inequality \e{n539} (resp.\ \e{n540}) then implies that the 
second (resp.\ first) term in the right-hand side converges to 
$0$ in $H^{-\mez}(\xR^d)$. 

To study the limit of $B(\eta_n)\psi_n$, we first prove that 
$\nabla\eta_n\cdot \nabla\psi_n$ converges to $\nabla\eta\cdot\nabla \psi$ in $H^{-\mez}$. To do so, one makes the difference and then use the fact that the product is bounded from $H^{s-\mez}\times H^{-\mez}$ to $H^{-\mez}$ to obtain
$$
\lA \nabla\eta_n \cdot \nabla(\psi_n-\psi)\rA_{H^{-\mez}}
\les \lA \eta_n\rA_{H^{s+\mez}}\lA \psi_n-\psi\rA_{H^{\mez}}.
$$
On the other hand
$$
\lA \nabla (\eta_n-\eta)\cdot\nabla\psi\rA_{H^{-\mez}}
\le \lA \nabla (\eta_n-\eta)\cdot\nabla\psi\rA_{L^{2}}
\le \lA \eta_n-\eta\rA_{W^{1,\infty}}\lA \psi\rA_{H^{s+\mez}}.
$$
This proves that $\nabla\eta_n\cdot \nabla\psi_n$ converges to $\nabla\eta\cdot\nabla \psi$ in $H^{-\mez}$. 

Now set $a_n=(1+|\nabla\eta_n|^2)^{-1}$, 
$b_n=G(\eta_n)\psi_n+\nabla\eta_n\cdot \nabla\psi_n$. We have 
proved that $b_n$ converges to its limit $b=G(\eta)\psi+\nabla\eta\cdot \nabla\psi$ 
in $H^{-\mez}$. 
It is easily checked that $a_n$ converges to 
$a=(1+|\nabla\eta|^2)^{-1}$ in $L^{\infty}$ and that $a-1$ 
belongs to $H^{s-\mez}$. So, as above, one easily verify that $B(\eta_n)\psi_n=a_nb_n$ converges to $B(\eta)\psi=ab$ in 
$H^{-\mez}$. This in turn easily implies that 
$V(\eta_n)\psi_n$ converges to 
$V(\eta)\psi$ in $H^{-\mez}$.
\end{proof}

\section{Existence and uniqueness}\label{S:54}

We have already proved the uniqueness of solutions 
(which is a straightforward consequence of Theorem~\ref{th.lipschitz}) so, 
to complete the proof of Theorem~\ref{main}, it remains to prove the existence. 
This is done by means of standard arguments together with a sharp blow up criterion 
proved by de Poyferr\'e \cite{TdP}. Namely, it follows from his result 
that, 
if the lifespan $T^*$ of a smooth solution of the water waves system is finite, then
\be\label{n537}
\lim_{T\rightarrow T^*} \bigl(M_s(T)+Z_r(T)\bigr)=+\infty, 
\ee
with notations as above (with the same assumptions on $s$ and $r$, see \e{n416}). 

We use this criterion to obtain solutions to the water waves system 
as limits of smooth solutions. Namely, consider a family of initial data $(\psi_0^\eps,\eta_0^\eps)$ in 
$H^\infty({\mathbf{R}}^d)^2$ 
converging to $(\psi_0,\eta_0)$. 
It follows that 
the Cauchy problem 
has a unique smooth solution~$(\psi^\eps,\eta^\eps)$ defined on some time interval~$[0,T_\eps^*)$ 
(this follows from 
the Cauchy result in \cite{ABZ3} (see also \cite{WuInvent,WuJAMS} and \cite{LannesJAMS}). 
The question is to prove that this family of smooth solutions exists on a uniform time interval and that, in addition, it 
converges to a solution of the water waves system. 
By applying our a priori estimate \e{n417}, it follows that there exists a function~$\mathcal{F}$ such that, for all 
$\eps\in (0,1]$ and all~$T<T_\eps$, we have
\begin{equation}\label{apriorieps}
M_s^\eps(T)+Z_r^\eps(T)\le \mathcal{F}\bigl(\mathcal{F}(M_{s,0})+T\mathcal{F}\bigl(M_s^\eps(T)+Z_r^\eps(T)\bigr)\bigr),
\end{equation}
with obvious notations. 

Then, 
by standard arguments and using the blow up criterion~\e{n537},  
we infer that the lifespan is 
bounded from below by a positive time~$T$ 
independent of~$\eps$ 
and that we have uniform estimates 
on~$[0,T]$ for $M_s^\eps(T)+Z_r^\eps(T)$. 

For $\sigma\in\xR$ and $a\in [0,+\infty]$, set
$$
\mathcal{H}^\sigma\defn H^{\sigma+\mez}\times H^{\sigma+\mez}\times H^\sigma\times H^\sigma, 
\quad 
\mathcal{W}^a\defn 
\holder{a+\mez}\times \holder{a+\mez} \times \holder{a}\times \holder{a}.
$$
Since $(\eta_\eps,\psi_\eps,V_\eps,B_\eps)$ is uniformly bounded in 
$X\defn L^\infty([0,T];\mathcal{H}^{s})\cap L^p([0,T];\mathcal{W}^{r})$ and since $X$ is the dual of a Banach space, 
it has (after extraction of a subsequence) a weak limit $(\eta,\psi,V,B)$ in $X$.   
Moreover, 
the contraction estimate~\eqref{eq.lipschitz} implies that 
$(\eta_\eps,\psi_\eps,B_\eps,V_\eps)$ 
is a Cauchy sequence in 
$L^\infty([0,T];\mathcal{H}^{s-1})\cap L^p([0,T];\mathcal{W}^{r-1})$. 
Therefore $(\eta_\eps,\psi_\eps)$ converges to its limit 
$(\eta,\psi)$ strongly in $L^\infty([0,T];H^{s-\mez}\times H^{s-1})$. 
Since $(\eta_\eps,\psi_\eps)$ is uniformly bounded in 
$L^\infty([0,T];H^{s+\mez}\times H^{s})$, by interpolation, 
$(\eta_\eps,\psi_\eps)$ converges also strongly in $L^\infty([0,T];H^{s'+\mez}\times H^{s'})$ for any $s'<s$. In particular, $(\eta_\eps,\psi_\eps)$ converges strongly to $(\eta,\psi)$ in $L^\infty([0,T];W^{1,\infty}
\times H^{\mez})$. 
As a result (see Corollary \ref{T534}) 
$$
G(\eta_\eps)\psi_\eps, \quad 
B_\eps = \frac{G(\eta_\eps)\psi_\eps+\nabla\eta_\eps\cdot\nabla\psi_\eps}{1+|\nabla\eta_\eps|^2} ,
\quad
V_\eps=\nabla \psi_\eps-B_\eps\nabla \eta_\eps
$$
converge, respectively, to 
$$
G(\eta)\psi,\quad 
\frac{G(\eta)\psi+\nabla\eta\cdot\nabla\psi}{1+|\nabla\eta|^2}, \quad 
\nabla \psi-B\nabla \eta.
$$
This proves that the weak limits $B,V$ of $B_\eps,V_\eps$ satisfy
$$
B=\frac{G(\eta)\psi+\nabla\eta\cdot\nabla\psi}{1+|\nabla\eta|^2}, \quad 
 V=\nabla \psi-B\nabla \eta,
$$
as well as the fact that one can pass 
to the limit in the equations. 
We thus obtain a solution $(\eta,\psi)$ 
such that 
$(\eta,\psi,V,B)$ is in $L^\infty([0,T];\mathcal{H}^s)\cap L^p([0,T];\mathcal{W}^r)$. By interpolation, 
$(\eta,\psi,V,B)$ is continuous in time 
with values in $\mathcal{H}^{s'}$  for any $s'<s$. 
It remains to prove that the solution is continuous in time with values in $\mathcal{H}^{s}$. This was done in details in \cite{ABZ1} for the case with surface tension. 
For the case without surface tension, this is done by Nguyen \cite{Nguyen} 
following the Bona-Smith' strategy.

\appendix
\chapter{Paradifferential calculus}\label{sec:2}

\section{Notations and classical results}

For the reader convenience, we recall notations as well as estimates for Bony's paradifferential operators 
(following \cite{Bony,MePise,Meyer,Taylor}). 
We also gather various estimates in H\"older or Zygmund spaces.
 
For~$k\in\xN$, we denote by $W^{k,\infty}({\mathbf{R}}^d)$ the usual Sobolev spaces.
For $\rho= k + \sigma$, $k\in \xN, \sigma \in (0,1)$ denote 
by~$W^{\rho,\infty}({\mathbf{R}}^d)$ 
the space of functions whose derivatives up to order~$k$ are bounded and uniformly H\"older continuous with 
exponent~$\sigma$. 

\begin{defi}\label{T:5}
Given~$\rho\in [0, 1]$ and~$m\in\xR$,~$\Gamma_{\rho}^{m}({\mathbf{R}}^d)$ denotes the space of
locally bounded functions~$a(x,\xi)$
on~${\mathbf{R}}^d\times({\mathbf{R}}^d\setminus 0)$,
which are~$C^\infty$ functions of $\xi$ outside the origin 
and
such that, for any~$\alpha\in\xN^d$ and any~$\xi\neq 0$, the function
$x\mapsto \partial_\xi^\alpha a(x,\xi)$ is in $W^{\rho,\infty}({\mathbf{R}}^d)$ and there exists a constant
$C_\alpha$ such that,
\begin{equation*}
\forall\la \xi\ra\ge \mez,\quad 
\lA \partial_\xi^\alpha a(\cdot,\xi)\rA_{W^{\rho,\infty}(\xR^d)}\le C_\alpha
(1+\la\xi\ra)^{m-\la\alpha\ra}.
\end{equation*}
\end{defi}
Given a symbol~$a$ in one such symbol class, one defines
the paradifferential operator~$T_a$ by
\begin{equation}\label{eq.para}
\widehat{T_a u}(\xi)=(2\pi)^{-d}\int \chi(\xi-\eta,\eta)\widehat{a}(\xi-\eta,\eta)\psi(\eta)\widehat{u}(\eta)
\, d\eta,
\end{equation}
where
$\widehat{a}(\theta,\xi)=\int e^{-ix\cdot\theta}a(x,\xi)\, dx$
is the Fourier transform of~$a$ with respect to the first variable; 
$\chi$ and~$\psi$ are two fixed~$C^\infty$ functions such that:
\begin{equation}\label{cond.psi}
\psi(\eta)=0\quad \text{for } \la\eta\ra\le 1,\qquad
\psi(\eta)=1\quad \text{for }\la\eta\ra\geq 2,
\end{equation}
and~$\chi(\theta,\eta)$ 
satisfies, for some small enough positive numbers $\eps_1<\eps_2$,
$$
\chi(\theta,\eta)=1 \quad \text{if}\quad \la\theta\ra\le \eps_1\la \eta\ra,\qquad
\chi(\theta,\eta)=0 \quad \text{if}\quad \la\theta\ra\geq \eps_2\la\eta\ra,
$$
and
$$
\forall (\theta,\eta)\,:\qquad \la \partial_\theta^\alpha \partial_\eta^\beta \chi(\theta,\eta)\ra\le 
C_{\alpha,\beta}(1+\la \eta\ra)^{-\la \alpha\ra-\la \beta\ra}.
$$

Given a symbol~$a\in \Gamma^m_{\rho}(\xR^d)$, we set
\begin{equation}\label{defi:norms}
M_{\rho}^{m}(a)= 
\sup_{\la\alpha\ra\le 1+2d+\rho~}\sup_{\la\xi\ra \ge 1/2~}
\lA (1+\la\xi\ra)^{\la\alpha\ra-m}\partial_\xi^\alpha a(\cdot,\xi)\rA_{W^{\rho,\infty}(\xR)}.
\end{equation}

Notice that the cut-off function $\chi$ can be so chosen that the definition of $T_a$ coincides with the 
usual one of a paraproduct (in terms of Littlewood-Paley decomposition) when the symbol $a$ 
depends on $x$ only. 

\section{Symbolic calculus}\label{sec.2.2}
We shall use quantitative results from \cite{MePise} about operator norms estimates in symbolic calculus. 
To do so, introduce the following semi-norms.
\begin{defi}[Zygmund spaces]\label{T:Zygmund}
Consider a dyadic decomposition of the identity: 
$I=\Delta_{-1}+\sum_{q=0}^\infty \Delta_q$. 
If~$s$ is any real number, the Zygmund class~$C^{s}_*({\mathbf{R}}^d)$ is the 
space of tempered distributions~$u$ such that
$$
\lA u\rA_{C^{s}_*}\defn \sup_q 2^{qs}\lA \Delta_q u\rA_{L^\infty}<+\infty.
$$
\end{defi}
\begin{rema}\label{R:Zygmund}
It is known that~$C^{s}_*({\mathbf{R}}^d)$ is the usual H\"older 
space~$W^{s,\infty}({\mathbf{R}}^d)$ if~$s>0$ is not an integer. 
\end{rema}

\begin{defi}\label{defi:order}
Let~$m\in\xR$.
An operator~$T$ is said to be of  order~$\leo m$ if, for all~$\mu\in\xR$,
it is bounded from~$H^{\mu}$ to~$H^{\mu-m}$ and from~$C^{\mu}_*$ to~$C^{\mu-m}_*$.
\end{defi}

The main features of symbolic calculus for paradifferential operators are given by the following theorem.
\begin{theo}\label{theo:sc0}
Let~$m\in\xR$ and~$\rho\in [0,1]$. 

$(i)$ If~$a \in \Gamma^m_0({\mathbf{R}}^d)$, then~$T_a$ is of order~$\leo m$. 
Moreover, for all~$\mu\in\xR$ there exists a constant~$K$ such that
\begin{equation}\label{esti:quant1}\lA T_a \rA_{H^{\mu}\rightarrow H^{\mu-m}}\le K M_{0}^{m}(a),\qquad  
\lA T_a \rA_{C^{\mu}_*\rightarrow C^{\mu-m}_*}\le K M_{0}^{m}(a).
\end{equation}
$(ii)$ If~$a\in \Gamma^{m}_{\rho}({\mathbf{R}}^d), b\in \Gamma^{m'}_{\rho}({\mathbf{R}}^d)$ then 
$T_a T_b -T_{a b}$ is of order~$\leo m+m'-\rho$. 
Moreover, for all~$\mu\in\xR$ there exists a constant~$K$ such that
\begin{equation}\label{esti:quant2}
\begin{aligned}
\lA T_a T_b  - T_{a b}   \rA_{H^{\mu}\rightarrow H^{\mu-m-m'+\rho}}&\le 
K M_{\rho}^{m}(a)M_{0}^{m'}(b)+K M_{0}^{m}(a)M_{\rho}^{m'}(b),\\
\lA T_a T_b  - T_{a b}   \rA_{C^{\mu}_*\rightarrow C^{\mu-m-m'+\rho}_*}&\le 
K M_{\rho}^{m}(a)M_{0}^{m'}(b)+K M_{0}^{m}(a)M_{\rho}^{m'}(b).
\end{aligned}
\end{equation}
$(iii)$ Let~$a\in \Gamma^{m}_{\rho}({\mathbf{R}}^d)$. Denote by 
$(T_a)^*$ the adjoint operator of~$T_a$ and by~$\overline{a}$ the complex conjugate of~$a$. Then 
$(T_a)^* -T_{\overline{a}}$ is of order~$\leo m-\rho$. 
Moreover, for all~$\mu$ there exists a constant~$K$ such that
\begin{equation}\label{esti:quant3}
\begin{aligned}
&\lA (T_a)^*   - T_{\overline{a}}   \rA_{H^{\mu}\rightarrow H^{\mu-m+\rho}}\le 
K M_{\rho}^{m}(a), \\
&\lA (T_a)^*   - T_{\overline{a}}   \rA_{C^{\mu}_*\rightarrow C^{\mu-m+\rho}_*}\le 
K M_{\rho}^{m}(a).
\end{aligned}
\end{equation}
\end{theo}
We also need in this article to consider paradifferential operators with negative regularity. 
As a consequence, we need to extend our previous definition.
\begin{defi}
For~$m\in \xR$ and~$\rho\in (-\infty, 0)$,~$\Gamma^m_\rho(\xR^d)$ denotes the space of
distributions~$a(x,\xi)$
on~${\mathbf{R}}^d\times(\xR^d\setminus 0)$,
which are~$C^\infty$ with respect to~$\xi$ and 
such that, for all~$\alpha\in\xN^d$ and all~$\xi\neq 0$, the function
$x\mapsto \partial_\xi^\alpha a(x,\xi)$ belongs to $C^\rho_*({\mathbf{R}}^d)$ and there exists a constant
$C_\alpha$ such that,
\begin{equation}
\forall\la \xi\ra\ge \mez,\quad \lA \partial_\xi^\alpha a(\cdot,\xi)\rA_{C^\rho_*}\le C_\alpha
(1+\la\xi\ra)^{m-\la\alpha\ra}.
\end{equation}
For~$a\in \Gamma^m_\rho$, we define 
\begin{equation}
M_{\rho}^{m}(a)= 
\sup_{\la\alpha\ra\le \frac {3d} 2 +\rho+1 ~}\sup_{\la\xi\ra \ge 1/2~}
\lA (1+\la\xi\ra)^{\la\alpha\ra-m}\partial_\xi^\alpha a(\cdot,\xi)\rA_{C^{\rho}_*({\mathbf{R}}^d)}.
\end{equation}
\end{defi}
\section{Paraproducts and product rules}\label{sec.2.3}
We recall here some properties of paraproducts (a paraproduct is a paradifferential operator 
$T_a$ whose symbol $a=a(x)$ is a function of~$x$ only). 
For our purposes, a key feature is that one can define paraproducts~$T_a$ for rough functions~$a$ which do not belong to~$L^\infty(\xR^d)$ but merely %
$\zygmund{-m}({\mathbf{R}}^d)$ with~$m>0$. 
\begin{defi}
Given two functions~$a,b$ defined on~${\mathbf{R}}^d$ we define the remainder 
$$
R(a,u)=au-T_a u-T_u a.
$$
\end{defi}
We record here various estimates about paraproducts (see chapter 2 in~\cite{BCD}).
\begin{theo}
\begin{enumerate}[i)]
\item  Let~$\alpha,\beta\in \xR$. If~$\alpha+\beta>0$ then
\begin{align}
&\lA R(a,u) \rA _{H^{\alpha + \beta-\frac{d}{2}}({\mathbf{R}}^d)}
\leq K \lA a \rA _{H^{\alpha}({\mathbf{R}}^d)}\lA u\rA _{H^{\beta}({\mathbf{R}}^d)},\label{Bony} \\ 
&\lA R(a,u) \rA _{C^{\alpha + \beta}_*({\mathbf{R}}^d)} 
\leq K \lA a \rA _{C^{\alpha}_*({\mathbf{R}}^d)}\lA u\rA _{C^{\beta}_*({\mathbf{R}}^d)}\label{Bony2},\\
&\lA R(a,u) \rA _{H^{\alpha + \beta}({\mathbf{R}}^d)} \leq K \lA a \rA _{C^{\alpha}_*({\mathbf{R}}^d)}\lA u\rA _{H^{\beta}({\mathbf{R}}^d)}.\label{Bony3}
\end{align}

\item   Let~$m>0$ and~$s\in \xR$. Then
\begin{align}
&\lA T_a u\rA_{H^{s-m}}\le K \lA a\rA_{C^{-m}_*}\lA u\rA_{H^{s}},\label{niS}\\
&\lA T_a u\rA_{C^{s-m}_*}\le K \lA a\rA_{C^{-m}_*}\lA u\rA_{C^{s}_*}.\label{niZ}\\
&\lA T_a u\rA_{C^{s}_*}\le K \lA a\rA_{L^\infty}\lA u\rA_{C^{s}_*}.\label{niZZ}
\end{align}
\item Let~$s_0,s_1,s_2$ be such that 
$s_0\le s_2$ and~$s_0 < s_1 +s_2 -\frac{d}{2}$, 
then
\begin{equation}\label{boundpara}
\lA T_a u\rA_{H^{s_0}}\le K \lA a\rA_{H^{s_1}}\lA u\rA_{H^{s_2}}.
\end{equation}
\end{enumerate}
\end{theo}

By combining the two previous points with the embedding~$H^\mu({\mathbf{R}}^d)\subset C^{\mu-d/2}_*({\mathbf{R}}^d)$ (for any~$\mu\in\xR$) 
we immediately obtain the following results.
\begin{prop}\label{lemPa}
Let~$r,\mu\in \xR$ be such that~$r+\mu>0$. If~$\gamma\in\xR$ satisfies 
$$
\gamma\le r  \quad\text{and}\quad \gamma < r+\mu-\frac{d}{2},
$$
then there exists a constant~$K$ such that, for all 
$a\in H^{r}({\mathbf{R}}^d)$ and all~$u\in H^{\mu}({\mathbf{R}}^d)$, we have
$$
\lA au - T_a u\rA_{H^{\gamma}}\le K \lA a\rA_{H^{r}}\lA u\rA_{H^\mu}.
$$
\end{prop} 
\begin{coro}
\begin{enumerate}[i)]
\item \label{it.1} If~$u_j\in H^{s_j}({\mathbf{R}}^d)$ ($j=1,2$) with $s_1+s_2> 0$ then $u_1 u_2\in H^{s_0}({\mathbf{R}}^d)$ and 
\begin{equation}\label{pr}
\lA u_1 u_2 \rA_{H^{s_0}}\le K \lA u_1\rA_{H^{s_1}}\lA u_2\rA_{H^{s_2}},
\end{equation}
if 
$$
s_0\le s_j, \quad j=1,2,\quad\text{and}\quad s_0< s_1+s_2-d/2.
$$

\item \label{it.2} (Tame estimate in Sobolev spaces) If~$s\ge 0$ then
\begin{equation}\label{prS2}
\lA u_1 u_2 \rA_{H^{s}}\le K \bigl(\lA u_1\rA_{H^{s}}\lA u_2\rA_{L^{\infty}}+ \lA u_1\rA_{L^{\infty}} \lA u_2\rA_{H^{s}}\bigr).
\end{equation}
\item \label{it.3} (Tame estimate in Zygmund spaces) If~$s\ge 0$ then 
\begin{equation}\label{prZ1}
\lA u_1 u_2 \rA_{C^{s}_*}\le K \bigl(\lA u_1\rA_{C^{s}_*}\lA u_2\rA_{L^{\infty}}+ \lA u_1\rA_{L^{\infty}} \lA u_2\rA_{C^{s}_*}\bigr).
\end{equation}

\item \label{it.4} Let~$\mu,m\in \xR$ be such that~$\mu,m> 0$ and~$m\not \in \xN$. Then
\begin{equation}\label{prZ2}
\lA u_1 u_2\rA_{H^\mu}\le K \bigl(\lA u_1\rA_{L^\infty}\lA u_2\rA_{H^\mu}+\lA u_2\rA_{C^{-m}_*}\lA u_1\rA_{H^{\mu+m}}\bigr).
\end{equation}

\item \label{it.5} Let~$\beta>\alpha >0~$. Then
\begin{equation}\label{prZ3}
\lA u_1 u_2\rA_{C^{-\alpha}_*}\le K \lA u_1\rA_{C^\beta_*} \lA u_2\rA_{C^{-\alpha}_*}.
\end{equation}

\item \label{it.6} Let~$s>d/2$ and consider~$F\in C^\infty(\xC^N)$ such that~$F(0)=0$. 
Then there exists a non-decreasing function~$\mathcal{F}\colon\xR_+\rightarrow\xR_+$ 
such that
\begin{equation}\label{esti:F(u)}
\lA F(U)\rA_{H^s}\le \mathcal{F}\bigl(\lA U\rA_{L^\infty}\bigr)\lA U\rA_{H^s},
\end{equation}
for any~$U\in H^s({\mathbf{R}}^d)^N$. 
\item \label{it.7}
 Let~$s\geq 0$ and consider~$F\in C^\infty(\xC^N)$ such that~$F(0)=0$. 
Then there exists a non-decreasing function~$\mathcal{F}\colon\xR_+\rightarrow\xR_+$ such that
\begin{equation}\label{esti:F(u)bis}
\lA F(U)\rA_{C^{s}_*}\le \mathcal{F}\bigl(\lA U\rA_{L^\infty}\bigr)\lA U\rA_{C^{s}_*},
\end{equation}
for any~$U\in C^{s}_*({\mathbf{R}}^d)^N$. 
\end{enumerate}
\end{coro}
\begin{proof}
The first three estimates  are well-known, 
see H\"ormander \cite{Hormander} or \cite{BCD}. 
To prove $\ref{it.4}$)  and $\ref{it.5}$) we write
$$
u_1u_2=T_{u_1} u_2+T_{u_2} u_1 +R(u_1,u_2). 
$$
Then \eqref{prZ2} follows from
\begin{alignat*}{2}
&\lA T_{u_1} u_2\rA_{H^\mu} \les \lA u_1\rA_{L^\infty}\lA u_2\rA_{H^\mu} \qquad &&(\text{see }\eqref{esti:quant1}),\\
&\lA T_{u_2} u_1\rA_{H^\mu} \les \lA u_2\rA_{C^{-m}_*}\lA u_1\rA_{H^{\mu+m}} 
&&(\text{see }\eqref{niS}),\\
&\lA R(u_1,u_2)\rA_{H^\mu} \les \lA u_2\rA_{C^{-m}_*}\lA u_1\rA_{H^{\mu+m}} \quad
&&(\text{see }\eqref{Bony3}),
\end{alignat*}
while similarly \eqref{prZ3} follows from
\begin{alignat*}{2}
&\lA T_{u_1} u_2\rA_{C^{-\alpha}_*} \les \lA u_1\rA_{L^\infty}\lA u_2\rA_{C^{-\alpha}_*} 
\les \lA u_1\rA_{C^\beta_*}\lA u_2\rA_{C^{-\alpha}_*} ,\\
&\lA T_{u_2} u_1\rA_{C^{-\alpha}_*} 
\les \lA u_2\rA_{C^{-\alpha}_*}\lA u_1\rA_{C^{0}_*} \le \lA u_2\rA_{C^{-\alpha}_*}\lA u_1\rA_{C^{\beta}_*},\\ 
&\lA R(u_1,u_2)\rA_{C^{-\alpha}_*}\le \lA R(u_1,u_2)\rA_{C^{\beta-\alpha}_*} 
\les \lA u_2\rA_{C^{-\alpha}_*}\lA u_1\rA_{C^{\beta}_*} .
\end{alignat*}
(With regards to the last inequality, to apply \eqref{Bony2} we do need~$\beta>\alpha>0$.) 
Finally, $\ref{it.6}$) and $\ref{it.7}$) are due to Meyer~\cite[Th\'eor\`eme 2.5 and remarque]{Meyer}, 
in the line of the work by Bony~\cite{Bony}.
\end{proof}

Finally, we recall Prop. $2.12$ in \cite{ABZ3} which 
is a generalization of~\eqref{niS}. 
\begin{prop}\label{prop.niSbis}
Let~$\rho<0$,~$m\in \xR$ and~$a\in \dot{ \Gamma}^m_\rho$. Then the operator~$T_a$ is of order~$m-\rho$:
\begin{equation}\label{niSbis}
\begin{aligned}
\| T_a \|_{H^s \rightarrow H^{s-(m- \rho)}}&\leq C M_{\rho}^{m}(a),\\
\| T_a \|_{C^s_* \rightarrow C^{s-(m- \rho)}_*}&\leq C M_{\rho}^{m}(a).
\end{aligned}
\end{equation}
\end{prop}
We also need the following technical result.

\begin{prop}\label{comut}
Set~$ \langle D_x \rangle=(I-\Delta)^{1/2}$. 

\begin{enumerate}[i)]
\item \label{itt.1} Let~$s>\mez + \frac{d}{2}$ and~$\sigma \in \xR$ be such that~$\sigma \leq s$. 
Then there exists~$K>0$ such that for all~$V \in W^{1, \infty}({\mathbf{R}}^d) \cap H^s({\mathbf{R}}^d)$ 
and ~$u \in H^{\sigma- \mez}({\mathbf{R}}^d)$  one has
$$
\Vert [\lDx{\sigma}, V] u \Vert_{L^2({\mathbf{R}}^d)} 
\leq K \big \{ \Vert V \Vert_{W^{1, \infty}({\mathbf{R}}^d)} 
+ \Vert V \Vert_{H^s({\mathbf{R}}^d) }\big \} \Vert u \Vert_{H^{\sigma- \mez}({\mathbf{R}}^d)}.
$$
\item \label{itt.2} Let~$s>1+\frac{d}{2}$ and~$\sigma \in \xR$ be such that~$\sigma \leq s$. 
Then there exists~$K>0$ such that for all~$V \in   H^s({\mathbf{R}}^d)$  and ~$u \in H^{\sigma- 1}({\mathbf{R}}^d)$  one has
$$
\Vert [\lDx{\sigma}, V] u \Vert_{L^2({\mathbf{R}}^d)} \leq K \Vert V \Vert_{H^s({\mathbf{R}}^d) }  \Vert u \Vert_{H^{\sigma- 1}({\mathbf{R}}^d)}.
$$
\item \label{itt.3} Let~$s> \mez + \frac d 2~$ and~$V\in H^s( \xR^d)$. Then 
$$ \Vert [\lDx{\mez}, V] u \Vert_{L^\infty(\xR^d)}
\leq K \| V\|_{H^s(\xR^d)} \| u \|_{L^\infty(\xR^d)}.
$$
\end{enumerate}
\end{prop}
\begin{proof}The first two statements are proved in \cite{ABZ3}. To prove~$\ref{itt.3}$) we 
use~\eqref{esti:quant2} 
with~$m = \mez, m'=0, \rho = \mez + \epsilon$ to obtain
$$
\| [\lDx{\mez}, T_V]u  \|_{C^\epsilon_*}
\leq C\|V\|_{H^s}\| u\|_{C^0_*} \leq C\|V\|_{H^s}\| u\|_{L^\infty}.
$$
On the other hand, 
$$
[\lDx{\mez}, V-T_V]u = \lDx{\mez} (V-T_V) u - (V- T_V) \lDx{\mez} u.
$$
Let~$\mez < r< s- \frac d 2$ so that 
$$ \| V\|_{\zygmund{r}} \leq C \| V\|_{H^s}.$$
According to~\eqref{niZZ} and~\eqref{Bony2},~$V- T_V$ 
is bounded from~$L^\infty$ to~$\zygmund{r}$ by~$K\|V\|_{\zygmund{r}}$ 
and according to~\eqref{niZ} and~\eqref{Bony2}, 
from~$C^{-\mez}_*$ to~$C^{r- \mez}_*$ by~$K\|V\|_{\zygmund{r}}$, 
which implies 
$$
\|[\lDx{\mez}, V-T_V]u\| _{C^{r- \mez}_*}\leq K \| V\|_{H^s} \| u\|_{L^\infty}.
$$
This completes the proof.
\end{proof}

We need elementary estimates on the solutions 
of transport equations that we recall now.
\begin {prop}\label{estclass}
Let~$I=[0,T]$ and consider the Cauchy problem
\begin{equation}\label{transport}
\left\{
\begin{aligned}
&\partial_t u + V\cdot \nabla u =f, \quad t \in I,\\
&u\arrowvert_{t=0} = u_0.
\end{aligned}
\right.
\end{equation}
We have the following estimates
\begin{equation}\label{eq.i}  
\Vert u(t) \Vert_{L^\infty({\mathbf{R}}^d)} 
\leq \Vert u_0 \Vert_{L^\infty({\mathbf{R}}^d)} + \int_0^t \Vert f(\sigma,\cdot) \Vert_{L^\infty({\mathbf{R}}^d)}d\sigma.
\end{equation}
There exists a non decreasing function  $\mathcal{F}:\xR^+ \to \xR^+$ such that
\begin{equation}
\Vert u(t) \Vert_{L^2({\mathbf{R}}^d)} 
\leq
\mathcal{F}\big(\Vert V \Vert_{L^1(I; W^{1, \infty}({\mathbf{R}}^d))}\big)
\big(\Vert u_0 \Vert_{L^2({\mathbf{R}}^d)} + \int_0^t \Vert f(t',\cdot)\Vert_{L^2({\mathbf{R}}^d)}\,
dt' \big).\label{eq.ii}
\end{equation}
If~$s>1+\frac{d}{2}$ and~$\sigma \leq s$ 
there exists a non decreasing function~$\mathcal{F}:\xR^+ \to \xR^+$ 
such that 
\begin{equation}\label{eq.iii}
\Vert u(t) \Vert_{H^{\sigma}({\mathbf{R}}^d)}  \\
\leq \mathcal{F} \bigl( \Vert V\Vert_{L^1(I; H^s({\mathbf{R}}^d))}\bigr)
\bigl(\Vert u_0 \Vert_{H^{\sigma}({\mathbf{R}}^d)} 
+ \int_0^t \Vert f(t',\cdot) \Vert_{H^{\sigma}({\mathbf{R}}^d)}\,dt' \big)  .
\end{equation}
\end{prop}

\vspace{3mm}

\noindent\textbf{Thomas Alazard}\\
\noindent DMA, \'Ecole normale sup\'erieure et CNRS UMR 8553, 
45 rue dÕUlm, 75005 Paris, France

\vspace{1mm}

\noindent\textbf{Nicolas Burq}\\
\noindent Univiversit\'e Paris-Sud, D\'epartement de Math\'ematiques, 91405 Orsay, France

\vspace{1mm}

\noindent\textbf{Claude Zuily}\\
\noindent 
Universit\'e Paris-Sud, D\'epartement de Math\'ematiques, 91405 Orsay, France


\begin{thebibliography}{10}
\small


\bibitem{ABZ1}
Thomas Alazard, Nicolas Burq and Claude Zuily.
\newblock On the water-wave equations with surface tension.
\newblock {\em Duke Math. J.}, 158(3):413--499, 2011.

\bibitem{ABZ2}
Thomas Alazard, Nicolas Burq and Claude Zuily.
\newblock Strichartz estimates for water waves.
\newblock {\em Ann. Sci. {\'E}c. Norm. Sup{\'e}r. (4)},   t.44, 855--903, 2011.

\bibitem{ABZ3}
Thomas Alazard, Nicolas Burq and Claude Zuily.
\newblock On the Cauchy problem for gravity water waves.
\newblock arXiv:1212.0626.

\bibitem{Bertinoro}
Thomas Alazard, Nicolas Burq and Claude Zuily.
\newblock The water-waves equations: from Zakharov to Euler. 
\newblock Studies in Phase Space Analysis with Applications to PDEs. 
\newblock Progress in Nonlinear Differential Equations and Their Applications Volume 84, 2013, pp 1-20. 

\bibitem{ABZ4}
Thomas Alazard, Nicolas Burq and Claude Zuily.
\newblock Cauchy theory for the gravity water waves system with non localized initial data.
\newblock arXiv:1305.0457.

\bibitem{AM}
Thomas Alazard and Guy M{\'e}tivier.
\newblock Paralinearization of the {D}irichlet to {N}eumann operator, and
  regularity of three-dimensional water waves.
\newblock {\em Comm. Partial Differential Equations}, 34(10-12):1632--1704,
  2009.

\bibitem{Alipara}
Serge Alinhac.
\newblock Paracomposition et op\'erateurs paradiff\'erentiels.
\newblock {\em Comm. Partial Differential Equations}, 11(1):87--121, 1986.

\bibitem{Ali}
Serge Alinhac.
\newblock Existence d'ondes de rar\'efaction pour des syst\`emes
  quasi-lin\'eaires hyperboliques multidimensionnels.
\newblock {\em Comm. Partial Differential Equations}, 14(2):173--230, 1989.

\bibitem{BaCh}
Hajer Bahouri and Jean-Yves Chemin.
\newblock \'{E}quations d'ondes quasilin\'eaires et estimations de
  {S}trichartz.
\newblock {\em Amer. J. Math.}, 121(6):1337--1377, 1999.

\bibitem{BCD}
Hajer Bahouri, Jean-Yves Chemin, and Rapha{\"e}l Danchin.
\newblock {\em Fourier analysis and nonlinear partial differential equations},
  volume 343 of {\em Grundlehren der Mathematischen Wissenschaften [Fundamental
  Principles of Mathematical Sciences]}.
\newblock Springer, Heidelberg, 2011.


\bibitem{Be}
M.S Berger, 
\newblock {Non linearity and functionnal analysis}, 
\newblock {\em Academic Press New York}, 1977.

\bibitem{BO}
T.~Brooke Benjamin and Peter~J. Olver.
\newblock Hamiltonian structure, symmetries and conservation laws for water
  waves.
\newblock {\em J. Fluid Mech.}, 125:137--185, 1982.

\bibitem{Blair}
Matthew Blair.
\newblock Strichartz estimates for wave equations with coefficients of
  {S}obolev regularity.
\newblock {\em Comm. Partial Differential Equations}, 31(4-6):649--688, 2006.

\bibitem{Bony}
Jean-Michel Bony.
\newblock Calcul symbolique et propagation des singularit\'es pour les
  \'equations aux d\'eriv\'ees partielles non lin\'eaires.
\newblock {\em Ann. Sci. \'Ecole Norm. Sup. (4)}, 14(2):209--246, 1981.

\bibitem{BGT1}
Nicolas Burq, Patrick G\'erard, and Nikolay Tzvetkov.
\newblock Strichartz inequalities and the nonlinear {S}chr\"odinger equation on
  compact manifolds.
\newblock {\em Amer. J. Math.}, 126(3):569--605, 2004.

\bibitem{CL}
Angel Castro and David Lannes. 
\newblock Well-posedness and shallow-water stability for a new Hamiltonian formulation of the water waves equations with vorticity. 
\newblock  arXiv:1402.0464.

\bibitem{CaARMA}
R{\'e}mi Carles.
\newblock Geometric optics and instability for semi-classical {S}chr\"odinger
  equations.
\newblock {\em Arch. Ration. Mech. Anal.}, 183(3):525--553, 2007.

\bibitem{CCT}
Michael Christ, James Colliander, et Terence Tao.
\newblock Asymptotics, frequency modulation, and low regularity ill-posedness
  for canonical defocusing equations.
\newblock {\em Amer. J. Math.}, 125(6):1235--1293, 2003.

\bibitem{CMSW}
Robin~Ming Chen, Jeremy~L. Marzuola, Daniel Spirn, and J.~Douglas Wright.
\newblock On the regularity of the flow map for the gravity-capillary
  equations.
\newblock {\em J. Funct. Anal.}, 264(3):752--782, 2013.

\bibitem{CHS}
Hans Christianson, Vera~Mikyoung Hur, and Gigliola Staffilani.
\newblock Strichartz estimates for the water-wave problem with surface tension.
\newblock {\em Comm. Partial Differential Equations}, 35(12):2195--2252, 2010.

\bibitem{ChLi}
Demetrios Christodoulou and Hans Lindblad.
\newblock On the motion of the free surface of a liquid.
\newblock {\em Comm. Pure Appl. Math.}, 53(12):1536--1602, 2000.

\bibitem{CoCoGa}
Antonio C{\'o}rdoba, Diego C{\'o}rdoba, and Francisco Gancedo.
\newblock Interface evolution: water waves in 2-D.
\newblock Adv. Math., 223(1):120Ð173, 2010.

\bibitem{CS}
Daniel Coutand and Steve Shkoller.
\newblock Well-posedness of the free-surface incompressible {E}uler equations
  with or without surface tension.
\newblock {\em J. Amer. Math. Soc.}, 20(3):829--930 (electronic), 2007.




\bibitem{Craig1985}
Walter Craig.
\newblock {An existence theory for water waves and the Boussinesq and
  Korteweg-deVries scaling limits}.
\newblock {\em Communications in Partial Differential Equations},
  10(8):787--1003, 1985.
\bibitem{CSS}
Walter Craig, Ulrich Schanz, and Catherine Sulem.
\newblock The modulational regime of three-dimensional water waves and the
  {D}avey-{S}tewartson system.
\newblock {\em Ann. Inst. H. Poincar\'e Anal. Non Lin\'eaire}, 14(5):615--667,
  1997.




\bibitem{CrSu}
Walter Craig and Catherine Sulem. 
\newblock Numerical simulation of gravity waves. 
\newblock {\em J. Comput. Phys.} 108(1):73Ð83, 1993.

\bibitem{DK}
Bj\"orn E.J. Dahlberg and Carlos E. Kenig, 
\newblock Harmonic Analysis and PDE's, 1985--1996, 
\newblock  http://www.math.chalmers.se/Math/Research/GeometryAnalysis/Lecturenotes/

\bibitem{Ebin}
David, G. Ebin,
\newblock The equations of motion of a perfect fluid with free
boundary are not well posed,
\newblock {\em Communications in Partial Differential Equations} 10 (12): 1175--1201,1987.

\bibitem{GV}
Jean Ginibre and Giorgio Velo. 
\newblock Generalized Strichartz inequalities for the wave equation.
\newblock {\em J. Funct. Anal.} 133:50Ð68, 1995. 

\bibitem{HN}
Bei Hu and David P. Nicholls. 
\newblock Analyticity of Dirichlet-Neumann operators on Hšlder and Lipschitz domains. 
\newblock {\em SIAM J. Math. Anal.}, 37(1):302-320, 2005.

\bibitem{HIT}
John Hunter, Mihaela Ifrim, Daniel Tataru.
\newblock Two dimensional water waves in holomorphic coordinates.
\newblock arXiv:1401.1252.


\bibitem{Hormander}
Lars H{\"o}rmander.
\newblock {\em Lectures on linear hyperbolic differential equations},
 {\em Math{\'e}matiques {\&} Applications (Berlin) [Mathematics \&
  Applications]}, volume 26.
\newblock Springer-Verlag, Berlin, 1997.


%

\bibitem{KT}
Markus Keel and Terence Tao.
\newblock Endpoint Strichartz estimates.
\newblock Amer. J. Math. 120(5):955Ð980, 1998.

\bibitem{KRS}
Sergiu Klainerman, Igor Rodnianski, Jeremie Szeftel. 
\newblock Overview of the proof of the Bounded $L^2$ Curvature Conjecture.
\newblock  arXiv:1204.1772.



\bibitem{LannesKelvin}
David Lannes.
\newblock A stability criterion for two-fluid interfaces and applications.
\newblock {\em Arch. Ration. Mech. Anal.}, 208(2):481--567, 2013.


\bibitem{LannesJAMS}
David Lannes.
\newblock Well-posedness of the water-waves equations.
\newblock {\em J. Amer. Math. Soc.}, 18(3):605--654 (electronic), 2005.

\bibitem{Lebeau92}
Gilles Lebeau.
\newblock Singularit\'es des solutions d'\'equations d'ondes semi-lin\'eaires.
\newblock {\em Ann. Sci. \'Ecole Norm. Sup. (4)}, 25(2):201--231, 1992.

\bibitem{LebeauKH}
Gilles Lebeau.
\newblock R\'egularit\'e du probl\`eme de {K}elvin-{H}elmholtz pour
  l'\'equation d'{E}uler 2d.
\newblock {\em ESAIM Control Optim. Calc. Var.}, 8:801--825 (electronic), 2002.
\newblock A tribute to J. L. Lions.

\bibitem{LebeauSc}
Gilles  Lebeau.
\newblock  Contr\^ole de l'\'equation de Schr\" odinger.
\newblock {\em  J. Math. Pures Appl.}, (9) 71 (1992), no. 3, 267Ð291.

\bibitem{LindbladAnnals}
Hans Lindblad.
\newblock Well-posedness for the motion of an incompressible liquid with free
  surface boundary.
\newblock {\em Ann. of Math. (2)}, 162(1):109--194, 2005.

\bibitem{MR}
Nader Masmoudi and Fr{\'e}d{\'e}ric Rousset.
\newblock Uniform regularity and vanishing viscosity limit for the free surface Navier-Stokes equations. 
\newblock arXiv:1202.0657.

\bibitem{MePise}
Guy M{\'e}tivier.
\newblock {\em Para-differential calculus and applications to the {C}auchy
  problem for nonlinear systems}, volume~5 of {\em Centro di Ricerca Matematica
  Ennio De Giorgi (CRM) Series}.
\newblock Edizioni della Normale, Pisa, 2008.

\bibitem{Meyer}
Yves Meyer.
\newblock Remarques sur un th\'eor\`eme de {J}.-{M}. {B}ony.
\newblock In {\em Proceedings of the {S}eminar on {H}armonic {A}nalysis
  ({P}isa, 1980)}, number suppl. 1, pages 1--20, 1981.



\bibitem{Nalimov}
V.~I. Nalimov.
\newblock The {C}auchy-{P}oisson problem.
\newblock {\em Dinamika Splo\v sn. Sredy}, (Vyp. 18 Dinamika Zidkost. so 
Svobod. Granicami):104--210, 254, 1974.

\bibitem{Nguyen}
Quang Huy Nguyen.
\newblock Work in preparation.



\bibitem{RoZu}
Luc Robbiano and Claude Zuily.
\newblock Strichartz estimates for {S}chr\"odinger equations with variable
  coefficients.
\newblock {\em M\'em. Soc. Math. Fr. (N.S.)}, (101-102):vi+208, 2005.

\bibitem{Safonov}
Mikhail V. Safonov.
\newblock Boundary estimates for positive solutions to second order elliptic equations.
\newblock arXiv:0810.0522



\bibitem{SZ}
Jalal Shatah and Chongchun Zeng.
\newblock Geometry and a priori estimates for free boundary problems of the
  {E}uler equation.
\newblock {\em Comm. Pure Appl. Math.}, 61(5):698--744, 2008.


\bibitem{Smith}
Hart~F. Smith.
\newblock A parametrix construction for wave equations with {$C^{1,1}$}
  coefficients.
\newblock {\em Ann. Inst. Fourier (Grenoble)}, 48(3):797--835, 1998.
 Smith, Hart F.; Tataru, Daniel Sharp local well-posedness results for the nonlinear wave equation. Ann. of Math. (2) 162 (2005), no. 1, 291Ð366

\bibitem{StTa}
Gigliola Staffilani and Daniel Tataru.
\newblock Strichartz estimates for a {S}chr\"odinger operator with nonsmooth
  coefficients.
\newblock {\em Comm. Partial Differential Equations}, 27(7-8):1337--1372, 2002.


\bibitem{TataruNS}
Daniel Tataru.
\newblock Strichartz estimates for operators with nonsmooth coefficients and
  the nonlinear wave equation.
\newblock {\em Amer. J. Math.}, 122(2):349--376, 2000.

\bibitem{TataruNSII}
Daniel Tataru.
\newblock Strichartz estimates for second order hyperbolic operators with
  nonsmooth coefficients. {II}.
\newblock {\em Amer. J. Math.}, 123(3):385--423, 2001.

\bibitem{Taylor}
Michael~E. Taylor.
\newblock {\em pseudo-differential operators and nonlinear {PDE}}, volume 100 of
  {\em Progress in Mathematics}.
\newblock Birkh\"auser Boston Inc., Boston, MA, 1991.

\bibitem{TdP}
Thibault de Poyferr{\'e}. 
\newblock Work in preparation.

\bibitem{Tougeron}
Monique Sabl{\'e}-Tougeron.
\newblock R\'egularit\'e microlocale pour des probl\`emes aux limites non
  lin\'eaires.
\newblock {\em Ann. Inst. Fourier (Grenoble)}, 36(1):39--82, 1986.

\bibitem{WaZh}
Chao Wang and Zhifei Zhang.
\newblock Break-down criterion for the water-wave equation.
\newblock  arXiv:1303.6029.

\bibitem{WuInvent}
Sijue Wu.
\newblock Well-posedness in {S}obolev spaces of the full water wave problem in
  2-{D}.
\newblock {\em Invent. Math.}, 130(1):39--72, 1997.

\bibitem{WuJAMS}
Sijue Wu.
\newblock Well-posedness in {S}obolev spaces of the full water wave problem in
  3-{D}.
\newblock {\em J. Amer. Math. Soc.}, 12(2):445--495, 1999.


\bibitem{Yosihara}
Hideaki Yosihara.
\newblock Gravity waves on the free surface of an incompressible perfect fluid
  of finite depth.
\newblock {\em Publ. Res. Inst. Math. Sci.}, 18(1):49--96, 1982.

\bibitem{Zakharov1968}
Vladimir~E. Zakharov.
\newblock Stability of periodic waves of finite amplitude on the surface of a
  deep fluid.
\newblock {\em Journal of Applied Mechanics and Technical Physics},
  9(2):190--194, 1968.

\end{thebibliography}
\end{document}